\documentclass[10pt,a4paper]{article}
\usepackage[utf8]{inputenc}
\usepackage{amsmath, amsthm}
\usepackage{esint}
\usepackage{amsfonts}
\usepackage{amssymb}
\usepackage{hyperref}
\usepackage{cleveref}
\usepackage{graphicx}
\usepackage[T1]{fontenc}
\usepackage[francais,english]{babel}
\usepackage{listings}
\usepackage{vmargin} 
\usepackage{dsfont}
\usepackage{color}

\usepackage[final]{pdfpages}
\numberwithin{equation}{section}

\newtheorem{de}{Definition}[section]
\newtheorem{thm}[de]{Theorem}
\newtheorem{cor}[de]{Corollary}
\newtheorem{prop}[de]{Proposition}
\newtheorem{lem}[de]{Lemma}
\newtheorem{rem}[de]{Remark}

\newcommand\eps{\varepsilon}
\newcommand\E{\mathcal{E}}
\newcommand\F{\mathcal{F}}
\newcommand\V{\mathcal{V}}

\newcommand\G{\mathcal{G}}

\newcommand\calK{\mathcal{K}}

\newcommand\M{\mathcal{M}}
\newcommand\He{\mathcal{H}}

\newcommand\trho{\tilde{\rho}}

\newcommand\tilm{{\widetilde{\mu}}}

\renewcommand{\textbf}[1]{\begingroup\bfseries\mathversion{bold}#1\endgroup}

\DeclareMathOperator*{\dive}{div}
\DeclareMathOperator*{\argmin}{argmin}
\DeclareMathOperator*{\Rn}{\mathbb{R}^\textit{n}}
\DeclareMathOperator*{\R}{\mathbb{R}}
\DeclareMathOperator*{\Pa}{\mathcal{P}}
\DeclareMathOperator*{\Paa}{\mathcal{P}^{ac}}
\DeclareMathOperator*{\Paad}{\mathcal{P}^{ac}_2}

\DeclareMathOperator*{\Le}{\mathcal{L}}

\title{On cross-diffusion systems for two populations subject to a common congestion effect}
\date{}
\author {Maxime Laborde \thanks{\scriptsize Department of Mathematics and Statistics, McGill University, Montreal, CANADA
(\texttt{maxime.laborde@mcgill.ca})}}
\begin{document}

\maketitle

\begin{abstract}
In this paper, we investigate the existence of solution for systems of Fokker-Planck equations coupled through a common nonlinear congestion. Two different kinds of congestion are considered: a porous media congestion or {\it soft} congestion and the {\it hard} congestion given by the constraint $\rho_1+\rho_2 \leqslant 1$. We show that these systems can be seen as gradient flows in a Wasserstein product space and then we obtain a constructive method to prove the existence of solutions. Therefore it is natural to apply it for numerical purposes and some numerical simulations are included.
\end{abstract}



\textbf{MS Classification:} 35K40, 49J40, 49J45.

\section{Introduction}

The modelling of crowd behaviour has become a very active field of applied mathematics in recent years. These models permit to understand many phenomena such as cell migration, tumor growth, etc. Several models already exist to tackle this problem. The first one, microscopic, consists in seeing a population as a high number of individuals which satisfy ODEs, see for instance \cite{MV} and the second is macroscopic and consists in describing a population by a density $\rho$ satisfying a PDE, where $\rho(t,x)$ represents the density of individuals in $x$ at time $t$. In the latter framework, different methods to handle the congestion effect have been proposed. The first one consists in saying that the motion has to be slower when the density is very high, see for example \cite{crippalm,clm,clm2} for a  different approach with applications to crowd dynamics. Another way of modelling the congestion effect is to use a threshold: the density evolves as we would expect until it touches a maximal level and then the motion has to be adapted in these regions (to not increase the density there), see for example \cite{MRCS} for crowd motion model and \cite{MRCS1} for application to dendritic growth. For a comparison between microscopic and macroscopic models, we refer to \cite{MRCSV}. In \cite{MS}, M{\'e}sz{\'a}ros and Santambrogio proposed a model for hard congestion where individuals are subject to a Brownian diffusion. This corresponds to modified a Fokker-Planck equation with an $L^\infty$ constraint on the density.\\

 Since in macroscopic models, we have mass conservation, the theory of optimal transportation is a very natural tool to attack them. In \cite{MRCS}, the authors investigated a model of room evacuation. They showed that if the desired velocity field of the individuals is given by a gradient, say $V=\nabla D$, where $D$ is the distance to a given target, then the problem has a gradient flow structure in the Wasserstein space and the velocity field has to be adapted by a pressure field to handle congestion effect. More recently in \cite{MS}, a splitting scheme has been introduced to handle velocity fields which are not necessarily gradient field. The scheme consists in combining steps where the density follows the unconstrained Fokker-Planck equation and Wasserstein projections onto the set of densities which cannot exceed $1$.\\ 

A natural variant of the model of \cite{MS}, consists in considering two (or more) populations, each of whom is subject to an advection term coming from different potential gradients but coupled through the constraint that the total density cannot exceed a given threshold, say $1$, and then subject to a common pressure field.  Note that variant problems with  total density {\em equal} to $1$ are treated in \cite{DMMRC,BE,guittet,CGM} and for more general cross-diffusion systems, we refer, for instance, to \cite{LPR,DLMT,J,JZ,KMVcrossdiff}. For a linear diffusion (corresponding to a Brownian noise on each species), the two-species crowd dynamic is expressed by the PDEs 

\begin{equation}
\label{crossfiff-system crowd}
\left\{\begin{array}{lll}
\partial_t \rho_1 -\Delta \rho_1- \dive(\rho_1 (\nabla V_1+\nabla p))=0,\\
\partial_t \rho_2 -\Delta \rho_2- \dive(\rho_2 (\nabla V_2+\nabla p))=0,\\
p\geqslant 0, \rho_1+\rho_2 \leqslant 1, \; p(1-\rho_1-\rho_2)=0,\\
\rho_1(0, \cdot) = \rho_{1,0}, \, \rho_2(0, \cdot) = \rho_{2,0}
\end{array}\right.
\end{equation}
on $\Omega$ a convex compact subset of $\R^n$ with smooth boundary such that
\begin{eqnarray}
\label{ass:size subset}
|\Omega | >2.
\end{eqnarray}
The assumption \eqref{ass:size subset} is made to ensure that the subset 
$$\mathcal{K} := \left\{ (\rho_1, \rho_2) \in \Paa (\Omega)^2 \, : \, \rho_1 +\rho_2 \leqslant 1 \, a.e.\right\}$$
 is neither empty nor trivial. We put no-flux boundary conditions to preserve the mass in $\Omega$, 
$$ (\nabla \rho_1 + \rho_1 (\nabla V_1+\nabla p) ) \cdot \nu =0 \text{ and } (\nabla \rho_2 + \rho_2 (\nabla V_2+\nabla p) )\cdot \nu =0 \text{ a.e. on } {\R}^+ \times \partial \Omega,
$$
where $\nu$ is the outward unit normal to $\partial \Omega$.\\

In this paper, we show that this system is the gradient flow for the Wasserstein product distance of the energy
\[\E_\infty(\rho_1, \rho_2):=\left\{\begin{array}{ll}
\sum_{i=1}^2\int_{\Omega} (\rho_i \log(\rho_i)+ V_i \rho_i)+\int_{\Omega} \chi_{[0,1]} (\rho_1(x)+\rho_2(x))\mbox{d}x &\text{ if } \rho_i \log(\rho_i) \in L^1(\Omega),\\
+\infty, &\text{ otherwise, }
\end{array}\right.\]
where $\chi_{[0,1]}$ is the indicator function of $[0,1]$,
$$ \chi_{[0,1]}(z):=\left\{\begin{array}{ll}
0 & \text{ if } z \in [0,1],\\
+\infty &\text{ otherwise. }
\end{array}\right.$$
In addition, for a different energy of the form
\begin{multline*}
\E_m(\rho_1, \rho_2):=\\
\left\{ \begin{array}{ll}
 \sum_{i=1}^2 \int_{\Omega} (\rho_i\log(\rho_i)+V_i \rho_i )+\int_\Omega \frac{1}{m-1}(\rho_1(x)+\rho_2(x))^{m} \,dx &\text{ if } \rho_i \log(\rho_i),  (\rho_1+\rho_2)^{m}\in L^1(\Omega),\\
 +\infty & \text{ otherwise, }
\end{array}\right.
\end{multline*}
for $m > 1$,
the gradient flow of $\E_m$ leads to the following nonlinear system
\begin{equation}
\label{systdiffcoup}
\left\{ \begin{array}{ll}
\partial_t \rho_1= \Delta \rho_1 + \dive\left(\rho_1 \nabla\left( V_1 +\frac{m}{m-1}(\rho_1+\rho_2)^{m-1}\right)\right)\\
\partial_t \rho_2= \Delta \rho_2 + \dive\left(\rho_2 \nabla\left( V_2 +\frac{m}{m-1}(\rho_1+\rho_2)^{m-1}\right)\right)\\
\rho_1(0, \cdot) = \rho_{1,0}, \, \rho_2(0, \cdot) = \rho_{2,0}
\end{array}\right.
\end{equation}
with no flux boundary conditions.
Then for a given small time step $h>0$, the JKO scheme for this energy reads, 
\begin{equation}\label{jkosyste}
 (\rho_1^{k+1}, \rho_2^{k+1})=\argmin_{(\rho_1,\rho_2)} \Big\{\sum_{i=1}^2 \frac{1}{2h} W_2^2(\rho_i, \rho_i^k)+\E_m(\rho_1, \rho_2) \Big\}
 \end{equation}
which, in the particular case of the linear diffusion crowd motion problem with two species, takes the form

\[ (\rho_1^{k+1}, \rho_2^{k+1})=\argmin_{\ \rho_1+\rho_2 \leqslant 1} \Big\{\sum_{i=1}^2 \left(\frac{1}{2h} W_2^2(\rho_i, \rho_i^k)+\int_{\Omega} (\rho_i \log(\rho_i)+ V_i \rho_i) \right) \Big\}.\]
We want to mention that the results in this paper have been obtained in the authors's PhD thesis, \cite{L_these}, back in 2016. Note that recently, in \cite{KM}, Kim and M{\'e}sz{\'a}ros studied problems \eqref{systdiffcoup} and \eqref{crossfiff-system crowd} without individual diffusions. They prove existence of weak solution in dimension $1$ for segregated initial conditions and ordered drifts. In any dimension, they prove existence of very weak solutions. The difficulty is to handle the cross diffusive term which needs to have strong compactness in $\rho_1,\rho_2$ and $\rho_1 +\rho_2$. Here, this difficulty is overcome by assuming that individuals of each populations are subject to a Brownian diffusion. This allows us to obtain separated estimates on $\rho_i$ and $\rho_1+\rho_2$. In \cite{LM}, Laurençot and Matioc give a similar result in $\R$ and $m=2$. In this paper, we extend this result on $\Omega \subset \R^n$ and with $m\in [1,+\infty]$. Furthermore, taking advantage of the gradient flow structure, we give numerical simulations implemented by the augmented Lagrangian scheme introduced in \cite{BCL}. We want to point out that uniqueness of systems \eqref{crossfiff-system crowd} and \eqref{systdiffcoup} is still an open question due to the lack of geodesic convexity of the common energy and we do not adress this problem in this paper. We refer to \cite{KM} for further discussions on this subject.
\\

 This paper is organized as follows. In section \ref{crossfiff-section 2: preliminaries and main results}, we introduce our assumptions and we state our main results. In section \ref{crossfiff-section 3: proof entropy}, we prove the existence of a weak solution for system \eqref{systdiffcoup}. The key ingredient is the flow interchange argument (see \cite{MMCS,DFM,L} for example) which gives separated estimates on the gradient of $\rho_1+\rho_2$ and on the gradient of $\rho_i$. Section \ref{crossfiff-section 4: proof crowd} provides the proof of existence of a weak solution for system \eqref{crossfiff-system crowd}. In this section we use again the flow interchange argument to obtain stronger estimates. In section \ref{crossfiff-section 5: link and uniqueness}, we focus on the particular case where $\nabla V_1= \nabla V_2$. In this case, we are able to show the convergence when $m \rightarrow +\infty$ of a solution to \eqref{systdiffcoup} to a solution to \eqref{crossfiff-system crowd} and we prove a $L^1$-contraction theorem. In the final section \ref{crossfiff-secction 6: simulations}, numerical simulations are presented. 

\section{Preliminaries and main results}
\label{crossfiff-section 2: preliminaries and main results}
 Throughout the paper, $\Omega$ is a smooth convex bounded subset of $\Rn$. We start to recall some results from the optimal tranportation theory and then we will state our main results. 
\subsection{Wasserstein space}
For a detailed exposition, we refer to reference textbooks \cite{V1,V2,AGS,S}. We denote $\mathcal{M}^+(\Omega)$ the set of nonnegative finite Radon measures on $\Omega$, $\Pa(\Omega)$ the space of probability measures on $\Omega$, and $\Paa(\Omega)$, the subset of $\Pa(\Omega)$ of probability measures on $\Omega$ absolutely continuous with respect to the Lebesgue measure. \\
For all $\rho,\mu \in \Pa(\Omega)$, we denote $\Pi(\rho,\mu)$, the set of probability measures on $\Omega \times \Omega$ having $\rho$ and $\mu$ as first and second marginals, respectively, and an element of $\Pi(\rho,\mu)$ is called a transport plan between $\rho$ and $\mu$. Then for all $\rho,\mu \in \Pa(\Omega)$, we denote by $W_2(\rho,\mu)$ the Wasserstein distance between $\rho$ and $\mu$, defined as
$$ W_2^2(\rho,\mu)= \min\left\{ \iint_{\Omega \times \Omega} |x-y|^2 \,d\gamma(x,y) \, : \, \gamma \in \Pi(\rho,\mu)\right\}.$$
Since this optimal transportation problem is a linear problem under linear constraints, it admits a dual formulation given by
$$ W_2^2(\rho,\mu)=\sup\left\{ \int_\Omega \varphi(x) \,d\rho(x) +\int_\Omega \psi(y) \, d\mu(y) \, : \varphi, \psi \in \mathcal{C}(\Omega) \text{ s.t. }  \varphi(x) +\psi(y) \leqslant |x-y|^2 \right\}.$$
Optimal solutions to the dual problem are called Kantorovich potentials between $\rho$ and $\mu$. If $\rho \in \Paa(\Omega)$, a well-known result proved by Brenier, \cite{B}, states that the optimal transport plan, $\gamma$, is unique and is induced by an optimal transport map, $T$, i.e  $\gamma$ is of the form  $(Id \times T)_{\#} \rho$, where $T_{\#}\rho =\mu$ and $T$ is the gradient of a convex function. Moreover, the optimal transport map is given by $T=Id -\nabla \varphi$ where $(\varphi , \psi)$ is a pair of Kantorovich potentials between $\rho$ and $\mu$.\\
It is well known that $\Pa(\Omega)$ endowed with the Wasserstein distance defines a metric space and since $\Omega$ is compact, $W_2$ metrizes the narrow convergence of probability measures.

\subsection{Assumptions and main results }

 For $i \in \{1,2\}$, we define $\V_i \, : \, \Pa(\Omega) \rightarrow \R$ the potential energy associated to $V_i \in W^{1,\infty}(\Omega)$ as
$$ \V_i(\rho) := \int_\Omega V_i(x) \, d\rho(x).$$  
We introduce the Entropy $\He$ defined, for all probabilty measures $\rho$, as
$$ \He(\rho) := \left\{ \begin{array}{ll}
\int_{\Omega} H(\rho(x) ) \, dx & \text{ if } \rho \ll \Le_{|\Omega},\\
+\infty & \text{ otherwise, } 
\end{array} \right., \qquad H(z):=z\log(z) \text{ for all }z\in \mathbb{R}^+.$$
Finally, for $m \in [1,+\infty)$, we define $\F_{m} \, : \, \M^+(\Omega) \rightarrow \R \cup \{ +\infty \}$ as
\begin{eqnarray}
\label{part3-definefm}
\F_m(\rho) := \left\{ \begin{array}{ll}
\int_{\Omega} F_{m}(\rho(x) ) \, dx & \text{ if } \rho \ll \Le_{|\Omega},\\
+\infty & \text{ otherwise, } 
\end{array} \right., \qquad
 F_{m}(z):= \left\{ \begin{array}{ll}
z \log z & \text{ if } m=1, \\
\frac{z^{m}}{m-1} & \text{ if } m >1.
\end{array}\right. \text{for all }z\in \mathbb{R}^+,
\end{eqnarray}
and, for $m=+\infty$, $\F_\infty \, : \, \M^+(\Omega) \rightarrow \R \cup \{ +\infty \}$ is defined by
$$ \F_\infty(\rho) := \left\{ \begin{array}{ll}
0 & \text{ if } \| \rho \|_\infty \leqslant 1 ,\\
+\infty & \text{ otherwise. } 
\end{array} \right.$$
\begin{de}[Weak solution]~
\label{weak solution}
\begin{itemize}
\item
We say that $(\rho_1,\rho_2) \, : \, [0, +\infty) \rightarrow \Paa(\Omega)^2$ is a weak solution to \eqref{systdiffcoup} if for all $i \in \{1,2\}$ and for all $T<+\infty$, $\rho_i \in \mathcal{C}^{0,1/2}([0,T],\Paa(\Omega)) \cap L^{2-1/m}((0,T),W^{1,2-1/m}(\Omega))\cap L^{2m-1}((0,T) \times \Omega) $, $\rho_i \nabla F_m'(\rho_1 + \rho_2) \in L^{2-1/m}((0,T)\times \Omega)$ and for all $\phi \in \mathcal{C}_c^\infty([0,+\infty) \times \Rn)$,

\begin{eqnarray*}
\begin{split}
 \int_0^{+\infty}\int_\Omega \left[ \rho_i \partial_t \phi - (\rho_i \nabla V_i +\rho_i \nabla F_m'(\rho_1 + \rho_2)+ \nabla \rho_i)\cdot \nabla \phi \right]  \,dx dt = -\int_\Omega \phi(0,x) \rho_{i,0}(x)\,dx.
 \end{split}
\end{eqnarray*}

\item We say that $(\rho_1,\rho_2,p) \, : \, [0, +\infty) \rightarrow \Paa(\Omega)^2 \times H^1(\Omega)$ is a weak solution to \eqref{crossfiff-system crowd} if for all $i \in \{1,2\}$ and for all $T<+\infty$, $\rho_i \in \mathcal{C}^{0,1/2}([0,T],\Paa(\Omega)) \cap L^2((0,T),H^1(\Omega))$, $p \in L^2((0,T),H^1(\Omega))$ with $p \geqslant 0$, $\rho_1 + \rho_2 \leqslant 1$ and $p(1-\rho_1-\rho_2)=0$ a.e. in $[0,T] \times \Omega$. In addition, for all $\phi \in \mathcal{C}_c^\infty([0,+\infty) \times \Rn)$,

$$ \int_0^{+\infty}\int_\Omega  \left[ \rho_i \partial_t \phi - (\rho_i \nabla V_i +\rho_i \nabla p + \nabla \rho_i)\cdot \nabla \phi \right]  \,dx dt = -\int_\Omega \phi(0,x) \rho_{i,0}(x)\,dx.$$
\end{itemize}
\end{de}
The main results of this paper are

\begin{thm}
\label{crossfiff-existence entropy}
Assume that $\rho_{1,0},\rho_{2,0} \in \Paa(\Omega)$ satisfy
\begin{eqnarray}
\label{crossfiff-CI}
\He(\rho_{1,0})+\He(\rho_{2,0})+\F_m(\rho_{1,0}+\rho_{2,0}) <+\infty,
\end{eqnarray}
then \eqref{systdiffcoup} admits at least one weak solution. 
\end{thm}
and
\begin{thm}
\label{crossfiff-existence crowd motion}
Assume that $\Omega$ satisfies \eqref{ass:size subset}. If $(\rho_{1,0},\rho_{2,0})\in \calK$ satisfies 
$$ \He(\rho_{1,0})+\He(\rho_{2,0}) <+\infty ,$$
then there exists at least one weak solution to \eqref{crossfiff-system crowd}. 
\end{thm}

\begin{rem}[Remarks on possible extensions:]~\\

\begin{itemize}
\item These models can be generalized to more than two species. Moreover, instead of assuming that individuals of different populations take the same space, we can generalize to densities evolving under the constraints on $\alpha_1 \rho_1 + \alpha_2 \rho_2$, for $\alpha_1, \alpha_2 >0$. Then system \eqref{systdiffcoup} becomes

\begin{equation*}
\partial_t \rho_i= \dive(\rho_i \nabla V_i) +\Delta \rho_i +\alpha_i\dive(\rho_i\nabla F'_{m}(\alpha_1\rho_1+\alpha_2\rho_2)), \; i=1, \; 2.
\end{equation*}
and system with hard congestion becomes 
\begin{equation*}
\left\{\begin{array}{lll}
\partial_t \rho_1 -\Delta \rho_1- \dive(\rho_1 (\nabla V_1+\nabla p))=0,\\
\partial_t \rho_2 -\Delta \rho_2- \dive(\rho_2 (\nabla V_2+\nabla p))=0,\\
p\geqslant 0, \alpha_1\rho_1+\alpha_2 \rho_2 \leqslant 1, \; p(1-\alpha_1 \rho_1-\alpha_2 \rho_2)=0. 
\end{array}\right.
\end{equation*}
\item These results can be generalized to more general velocities. Indeed, using the semi-implicit scheme introduced by DiFrancesco and Fagioli in \cite{DFF} and developped in \cite{L} or the splitting method introduced in \cite{CL1}, we can treat vector fields depending on the densities and which come not necessarily from a potential. These extensions allow to treat nonlocal interactions between different species, of the form $V_i[\rho_1,\rho_2]=K_{i,1} \ast \rho_1 + K_{i,2} \ast \rho_2$ where $K_{i,j} \in W^{1,\infty}$, which are subject to a common congestion effect .

\item To simplify the exposition, during the whole paper, we deal with linear self-diffusion terms but it is possible to extend Theorems \ref{crossfiff-existence entropy} and \ref{crossfiff-existence crowd motion} to nonlinear self-diffusions. In particular, we can deal with porous medium diffusion of the form $\Delta \rho_i^{q_i}$. This can be done replacing the Entropy $\He(\rho_i)$ by the functional $\F_{q_i}(\rho_i)$. In the analysis, the individual estimates found in Proposition \ref{crossfiff-estimation grad} and in Proposition \ref{crossfiff-estimate grad crowd} become $L^2((0,T),H^1(\Omega))$ estimates on  $\rho_{1,h}^{q_1 /2}$ and $\rho_{2,h}^{q_2 /2}$ (see for example \cite{L}) without modifying the joint estimate. In addition, discret solutions are not globally supported anymore, i.e. Lemma \ref{crossfiff-rho positif} and Lemma \ref{lem:rho positif crowd} do not hold, but Proposition \ref{prop:variation premiere} and Proposition \ref{prop:Euler-Lagrange} can be recovered, see for example \cite{S} for $m <+\infty$ and \cite{KM} in the case $m=\infty$.

\end{itemize}
\end{rem}

\section{Coupling through common soft congestion}
\label{crossfiff-section 3: proof entropy}

In this section, we prove Theorem \ref{crossfiff-existence entropy} using the implicit JKO scheme, firstly introduced by Jordan, Kinderlherer and Otto in \cite{JKO}. Given a time step $h>0$, we construct by induction two sequences $\rho_{1,h}^k$ and $\rho_{2,h}^k$ with the following scheme: $\rho_{i,h}^0=\rho_{i,0}$ and for all $k \geqslant 0$,
\begin{eqnarray}
\label{crossfiff-scheme entropy}
(\rho_{1,h}^{k+1},\rho_{2,h}^{k+1}) \in \argmin_{(\rho_1,\rho_2) \in \Paa(\Omega)^2} \left\{ \sum_{i=1}^2 \left( W_2^2(\rho_i,\rho_{i,h}^{k})+2h\left(\He(\rho_i) +\V_i(\rho_i) \right) \right) +2h \F_m( \rho_1+ \rho_2)  \right\}.
\end{eqnarray}
These sequences are well-defined by standard compactness and l.s.c argument. Then we define the piecewise constant interpolations $\rho_{i,h} \, : \, \R^+ \rightarrow \Paa(\Omega)$ by
$$ \rho_{i,h}(t):= \rho_{i,h}^{k+1}, \qquad \text{if } t \in (kh,(k+1)h].$$

In the first part of this section, we study the convergence of these sequences and then we give the proof of Theorem \ref{crossfiff-existence entropy}.
\subsection{Estimates and convergences}
We start retrieving classical estimates coming from the JKO scheme, \cite{JKO}, and then, we develop stronger estimates using the flow interchange argument, \cite{MMCS,DFM}. First, the minimization scheme gives

\begin{prop}
\label{crossfiff-estimation sum}
For all $T <+\infty$ and for all $i \in \{ 1 ,2 \}$, there exists a constant $C <+\infty$ such that for all $k \in \mathbb{N}$ and for all $h$ with $kh\leqslant T$ and let $N=\lfloor \frac{T}{h} \rfloor$, we have
\begin{eqnarray}
& \He(\rho_{i,h}^k) \leqslant C, \label{crossfiff-estimation moment}\\
& \F_m(\rho_{1,h}^k+\rho_{2,h}^{k}) \leqslant C,\label{crossfiff-estimation entropie}\\
& \sum_{k=0}^{N-1} W_2^2(\rho_{i,h}^{k},\rho_{i,h}^{k+1}) \leqslant Ch.\label{crossfiff-estimation distance}
\end{eqnarray}
\end{prop}

\begin{proof}
These results are obtained easily taking $\rho_i=\rho_{i,h}^k$ as competitors in \eqref{crossfiff-scheme entropy}, see \cite{JKO}. 
\end{proof}

\begin{rem}
\label{crossfiff-estimate distance ind m}
Notice that estimate \eqref{crossfiff-estimation distance} does not depend on $m$. This Remark will be useful in section \ref{crossfiff-section 5: link and uniqueness} to show that a solution to \eqref{systdiffcoup} converges to a solution to \eqref{crossfiff-system crowd}.
\end{rem}

In the next proposition, stronger estimates are obtained in order to pass to the limit in the nonlinear diffusive term. The main argument to prove this proposition is the flow interchange argument, introduced in \cite{MMCS}. First we recall the definition of a $\kappa$-flow.
\begin{de}
A semigroup $\mathfrak{S}_\Psi \, : \, \R^+ \times \mathcal{P}^{ac}(\Omega) \rightarrow \mathcal{P}^{ac}(\Omega)$ is a $\kappa$-flow for the functional $\Psi \, :\,\mathcal{P}^{ac}(\Omega)\rightarrow \R \cup \{ +\infty\}$ with respect to  $W_2$ if, for all $\rho \in \mathcal{P}^{ac}(\Omega)$, the curve $s \mapsto \mathfrak{S} _\Psi^s[\rho]$ is absolutely continuous on $\R^+$ and satisfies the evolution variational inequality (EVI) 
\begin{eqnarray}
\label{eq:EVI}
\frac{1}{2}\frac{d^+}{d\sigma}\mid_{\sigma=s} W_2^2(\mathfrak{S} _\Psi^s[\rho], \tilde{\rho}) +\frac{\kappa}{2}W_2^2(\mathfrak{S} _\Psi^s[\rho], \tilde{\rho}) \leqslant \Psi(\tilde{\rho})-\Psi(\mathfrak{S} _\Psi^s[\rho]),
\end{eqnarray}
for all $s>0$ and for all $\tilde{\rho} \in \mathcal{P}^{ac}(\Omega)$ such that $\Psi(\tilde{\rho}) < +\infty$, where 
$$ \frac{d^+}{dt}f(t) := \limsup_{s\rightarrow 0^+} \frac{f(t+s) -f(t)}{s}.$$
\end{de}

In \cite{AGS}, the authors showed that the fact a functional admits a $\kappa$-flow is equivalent to $\kappa$-displacement convexity.

\begin{prop}
\label{crossfiff-prop f.g.i sum}
For all $T>0$, there exists a constant $C_T>0$ such that,  
\begin{eqnarray}
\label{crossfiff-estimation grad}
\| \rho_{1,h}^{1/2} \|_{L^2((0,T),H^1(\Omega))}^2 +\| \rho_{2,h}^{1/2} \|_{L^2((0,T),H^1(\Omega))}^2 +\frac{1}{m}\| (\rho_{1,h}+\rho_{2,h})^{m/2} \|_{L^2((0,T),H^1(\Omega))}^2 \leqslant C_T.
\end{eqnarray}
\end{prop}

\begin{proof}
We use the flow interchange argument, introduced in \cite{MMCS}, to find a stronger estimate as in \cite{DFM,L}. In other words, we perturb $\rho_{1,h}^{k}$ and $\rho_{2,h}^{k}$ by the heat flow. Let $\eta_i$ be the solution to
\begin{eqnarray}
\label{crossfiff-chaleur Flow Inter}
\left\{\begin{array}{ll}
\partial_t \eta_i = \Delta \eta_i & \text{ in } (0,T) \times \Omega,\\
\nabla \eta_i \cdot \nu =0 & \text{ in } (0,T) \times \partial \Omega,\\
{\eta_i}_{|t=0}=\rho_{i,h}^k.
\end{array}\right.
\end{eqnarray}
\\
Since the Entropy is geodesically convex then the heat flow is a $0$-flow of the Entropy $\He$, and satisfies the Evolution Variational Inequality, \eqref{eq:EVI}, see \cite{JKO,V1,AGS,DS,S},
\begin{eqnarray}
\label{crossfiff-E.V.I}
\frac{1}{2}\frac{d^+}{d\sigma}_{|\sigma=s} W_2^2(\eta_i(s),\rho) \leqslant \He(\rho) - \He(\eta_i(s)),
\end{eqnarray}
for all $s>0$ and $\rho \in \Paad(\Omega)$.
\\
Taking $(\eta_1(s),\eta_2(s))$ as a competitor in the minimization \eqref{crossfiff-scheme entropy}, we get 
\begin{multline}
\label{crossfiff-derivative scheme}
\sum_{i=1}^2\frac{1}{2}{\frac{d^+}{d s}  W_2^2(\eta_i(s),\rho_{i,h}^{k-1})}_{|s=0}
+h{\frac{d^+}{d s}\left( \sum_{i=1}^2 \left(\He(\eta_i(s)) +\V_i(\eta_i(s)) \right) +\F_m(\eta_1(s)+ \eta_2(s))\right)}_{|s=0} \geqslant 0.  
\end{multline}
Since $\eta_i(s)$ is a smooth positive function for $s>0$, the following computations are justified
\begin{eqnarray}
\begin{split} 
\partial_s \Big(\sum_{i=1}^2\big(  &\He(\eta_i(s)) \left. +\V_i(\eta_i(s)) \right) + \F_m(\eta_1(s)+ \eta_2(s) ) \Big) \\
&=\sum_{i=1}^2 \left(\int_\Omega \Delta \eta_i(s) ((1+\log(\eta_i(s))) +V_i \right) + \int_\Omega \Delta (\eta_1(s)+\eta_2(s)) F_m'(\eta_1(s)+\eta_2(s))\\
&=- \sum_{i=1}^2 \left( \int_\Omega  \frac{|\nabla \eta_i(s)|^2}{\eta_i(s)} + \int_\Omega \nabla V_i \cdot \nabla \eta_i(s) \right)- \int_\Omega |\nabla (\eta_1(s)+\eta_2(s))|^2F_m''(\eta_1(s)+\eta_2(s)).
\end{split}
\end{eqnarray}
\\
In addition, Young's inequality gives
\begin{eqnarray*}
-\int_\Omega \nabla V_i(s) \cdot \nabla \eta_i & \leqslant & \int_\Omega |\nabla V_i||\nabla \eta_i(s)|
\leqslant  \frac{1}{2} \int_\Omega |\nabla V_i|^2\eta_i(s) +\frac{1}{2}\int_\Omega \frac{|\nabla \eta_i(s)|^2}{\eta_i(s)}\\
\end{eqnarray*}
Then, we have
\begin{multline}
\partial_s \Big(\sum_{i=1}^2 \left(\He(\eta_i(s)) +\V_i(\eta_i(s)) \right) + \F_m(\eta_1(s)+ \eta_2(s)) \Big) \\
\leqslant \sum_{i=1}^2 \left( - \frac{1}{2}\int_\Omega  \frac{|\nabla \eta_i(s)|^2}{\eta_i(s)} +\frac{1}{2} \int_\Omega |\nabla V_i|^2 \eta_i(s)  \right)- \int_\Omega |\nabla (\eta_1(s)+\eta_2(s))|^2F_m''(\eta_1(s)+\eta_2(s)).
\end{multline}
By definition of $F_m$, for $m\geqslant 1$, $F_m''(z) =m z^{m-2}$ for all $z\geqslant0$ and, since $V_i \in W^{1,\infty}(\Omega)$,
\begin{multline}
\partial_s \Big(\sum_{i=1}^2 \left(\He(\eta_i(s))  +\V_i(\eta_i(s)) \right) + \F_m(\eta_1(s)+ \eta_2(s)) \Big) \\
\leqslant C- \frac{1}{2} \sum_{i=1}^2  \int_\Omega |\nabla \eta_i(s)^{1/2}|^2 - \frac{4}{m}\int_\Omega |\nabla (\eta_1(s)+\eta_2(s))^{m/2}|^2.
\end{multline}
By a lower semi-continuity argument,
\begin{multline*}
 \frac{1}{2}\sum_{i=1}^2  \int_\Omega |\nabla (\rho_{i,h}^k)^{1/2}|^2 + \frac{4}{m}\int_\Omega |\nabla (\rho_{1,h}^k+\rho_{2,h}^k)^{m/2}|^2 \\
  \leqslant C -{\frac{d^+}{d s}\left( \sum_{i=1}^2 (\He(\eta_i(s)) +\V_i(\eta_i(s)) ) +\F_m(\eta_1(s)+ \eta_2(s))\right)}_{|s=0} .
\end{multline*}
Combining with \eqref{crossfiff-derivative scheme} and \eqref{crossfiff-E.V.I}, we obtain
\begin{eqnarray*}
h \sum_{i=1}^2  \int_\Omega |\nabla (\rho_{i,h}^k)^{1/2}|^2 + \frac{4h}{m}\int_\Omega |\nabla (\rho_{1,h}^k+\rho_{2,h}^k)^{m/2}|^2 \leqslant \sum_{i=1}^2 \left( \He(\rho_{i,h}^{k-1}) -\He(\rho_{i,h}^{k}) \right) +Ch.
\end{eqnarray*}
Then summing over $k$, we obtain
$$
\| \rho_{1,h}^{1/2} \|_{L^2((0,T),H^1(\Omega))}^2 +\| \rho_{2,h}^{1/2} \|_{L^2((0,T),H^1(\Omega))}^2 +\frac{1}{m}\| (\rho_{1,h}+\rho_{2,h})^{m/2} \|_{L^2((0,T),H^1(\Omega))}^2 \leqslant C_T,
$$
where we use the fact that $\| \rho_{i,h}^{1/2} \|_{L^2((0,T) \times \Omega)}^2 =T$ and $\frac{1}{m}\| (\rho_{1,h}+\rho_{2,h})^{m/2} \|_{L^2((0,T)\times \Omega)}^2 \leqslant CT$ by \eqref{crossfiff-estimation entropie}.
\end{proof}

\begin{rem}
The bound on $\| \rho_{i,h}^{1/2} \|_{L^2((0,T),H^1(\Omega))}$ does not depend on $m$. However, if we multiply the Entropy $\He$ by a small parameter $\eps >0$ in the JKO scheme \eqref{crossfiff-scheme entropy}, individual bounds blow up as $\eps$ goes to $0$.
\end{rem}

Now we can deduce the following convergences.

\begin{prop}
\label{prop:strong-convergence entropy}
For all $T<+\infty$, there exist $\rho_1$ and $\rho_2$ in $\mathcal{C}^{0,1/2}([0,T],\Paa(\Omega))$ such that, up to a subsequence,
\begin{enumerate}
\item $\rho_{i,h}$ converges to $\rho_i$ in $L^\infty([0,T],\Paa(\Omega))$, \label{uniform convergence in time entropy}
\item $\rho_{i,h}$ converges strongly to $\rho_i$ in $L^1((0,T) \times \Omega)$,
\item $(\rho_{1,h}+\rho_{2,h})^{m/2}$ converges strongly to $(\rho_1 +\rho_2)^{m/2}$ and $\nabla (\rho_{1,h}+\rho_{2,h})^{m/2}$ converges weakly to $\nabla (\rho_1 +\rho_2)^{m/2}$ in $L^2((0,T) \times \Omega)$. 
\end{enumerate}

\end{prop}

\begin{proof}
\begin{enumerate}
\item The first convergence is classical. We use the refined version of Ascoli-Arzelà's Theorem, \cite[Proposition 3.3.1]{AGS}, and we immediately deduce that there exists a subsequence such that, for $i=1,2$, $\rho_{i,h}$ converges to $\rho_i \in \mathcal{C}^{1/2}([0,T],\Paa(\Omega))$ in $L^\infty([0,T],\Paa(\Omega))$.\\

The next two strong convergence results are obtained applying an extension of the Aubin-Lions Lemma proved by Rossi and Savaré in \cite[Theorem 2]{RS}. In the sequel, we work with the convergent subsequence obtained in the first step. 

\item
Let $\G \, : \, L^1(\Omega) \rightarrow (-\infty, +\infty]$ and $g \, : \, L^1(\Omega)\times L^1(\Omega)\rightarrow [0, +\infty]$   defined by
$$ \G(\rho):=\left\{ \begin{array}{ll}
\|\rho^{1/2} \|_{H^1(\Omega)} & \text{ if } \rho \in \Paa(\Omega) \text{ and } \rho^{1/2} \in H^1(\Omega)\\
+\infty & \text{ otherwise, }
\end{array}\right.$$
and 
$$ g(\rho,\mu):=\left\{ \begin{array}{ll}
W_2(\rho,\mu) & \text{ if }  \rho,\mu \in \Pa(\Omega)\\
+\infty & \text{ otherwise, }
\end{array}\right.$$

$\G$ is l.s.c and its sublevels are relatively compact in $L^1(\Omega)$ (see \cite{DFM,L}) and $g$ is a pseudo-distance. According to \eqref{crossfiff-estimation distance} and \eqref{crossfiff-estimation grad}, we have 
$$ \sup_{h\leqslant 1}\int_0^T \G(\rho_{i,h}(t) ) \, dt <+\infty, \text{ and } \lim_{\tau \searrow 0}\sup_{h\leqslant 1}\int_0^{T-\tau} g(\rho_{i,h}(t+\tau), \rho_{i,h}(t)) \, dt =0,$$
then applying Rossi-Savaré's Theorem, there exists a subsequence, not-relabeled, such that for $i=1,2$, $\rho_{i,h}$ converges in measure with respect to $t$ in $L^1(\Omega)$ to $\rho_i$. Moreover by Lebesgue's dominated convergence Theorem, $\rho_{i,h}$ converges to $\rho_i$ strongly in $L^1((0,T)\times \Omega)$ .

\item With the same argument, we get a strong convergence on a nonlinear quantity of $\rho_{1,h}+\rho_{2,h}$. Let $\G$  define by
$$ \G(\rho):=\left\{ \begin{array}{ll}
\|\rho^{m/2} \|_{H^1(\Omega)} & \text{ if } \rho \in \Paa(\Omega) \text{ and } \rho^{m/2} \in H^1(\Omega)\\
+\infty & \text{ otherwise, }
\end{array}\right.$$
and $g$ defined as before. We want to apply Theorem 2 of \cite{RS} in $L^m(\Omega)$ over the sequence $\frac{\rho_{1,h}+\rho_{2,h}}{2}$. By \eqref{crossfiff-estimation grad}, we obtain
$$ \sup_{h\leqslant 1}\int_0^T \G\left(\frac{\rho_{1,h}(t)+\rho_{2,h}(t)}{2} \right) \, dt <+\infty.$$
Since, it is well-known that for all $\rho_1,\rho_2,\mu_1,\mu_2 \in \Paa(\Omega)$,
$$ W_2^2\left(\frac{\rho_1+\rho_2}{2},\frac{\mu_1+\mu_2}{2}\right) \leqslant \frac{1}{2} W_2^2(\rho_1,\mu_1) + \frac{1}{2} W_2^2(\rho_2,\mu_2),$$
by \eqref{crossfiff-estimation distance}, we obtain
$$\lim_{\tau \searrow 0}\sup_{h\leqslant 1}\int_0^{T-\tau} g\left(\frac{\rho_{1,h}+\rho_{2,h}}{2}(t+\tau), \frac{\rho_{1,h}+\rho_{2,h}}{2}(t)\right) \, dt =0.$$
Theorem 2 in \cite{RS} and Lebesgue's dominated convergence Theorem imply that $\rho_{1,h}+\rho_{2,h}$ converges strongly to $\rho_{1}+\rho_{2}$ in $L^m((0,T)\times \Omega)$. In addition, Krasnoselskii's Theorem, \cite[Chapter 2]{DF}, implies that $(\rho_{1,h}+\rho_{2,h})^{m/2}$ converges to $(\rho_{1}+\rho_{2})^{m/2}$ in $L^2((0,T)\times \Omega)$. To conclude, $\nabla (\rho_{1,h}+\rho_{2,h})^{m/2}$ is bounded in $L^2((0,T)\times \Omega)$, thanks to \eqref{crossfiff-estimation grad}, then $\nabla (\rho_{1,h}+\rho_{2,h})^{m/2}$ weakly converges to $\nabla (\rho_{1}+\rho_{2})^{m/2}$ in $L^2((0,T)\times \Omega)$. 
\end{enumerate}
\end{proof}

\begin{rem}
It is possible to obtain a strong convergence result in $L^1((0,T)\times \Omega)$ for the pressure $F'_m(\rho_{1,h} +\rho_{2,h})$. Indeed, since $\rho_{1,h} +\rho_{2,h}$ strongly converges in $L^m((0,T)\times \Omega)$, then up to a subsequence, $F_m'(\rho_{1,h} +\rho_{2,h}) \rightarrow F_m'(\rho_{1} +\rho_{2})$ a.e. In addition using De La Vallée Poussin's Theorem, we show that $(F'_m(\rho_{1,h} +\rho_{2,h}))_h$ is uniformly integrable. We conclude applying Vitali's convergence Theorem.
\end{rem}

\begin{rem}
\label{rem: drop diffusion porous}
Notice that we can drop one individual diffusion. Assume that we drop the individual Entropy in the JKO scheme \eqref{crossfiff-scheme entropy} for one of the two densities, for instance $\rho_2$. The difficulty is to obtain a strong convergence for the sequence $(\rho_{2,h})_h$. Proposition \ref{prop:strong-convergence entropy} gives the strong convergence of $\rho_{1,h}$ and $\rho_{1,h} +\rho_{2,h}$ in $L^1((0,T) \times \Omega)$ and $L^m((0,T)\times \Omega)$ respectively, and then pointwise on $(0,T) \times \Omega$. Consequently, $\rho_{2,h}= (\rho_{1,h}+\rho_{2,h}) - \rho_{1,h}$ converges pointwise on $(0,T) \times \Omega$. Moreover,
$$ \int_0^T \int_\Omega \rho_{2,h}(t,x)^m \,dxdt \leqslant  \int_0^T \int_\Omega (\rho_{1,h}(t,x)+\rho_{2,h}(t,x))^m \,dxdt \leqslant C_T.$$
Then Vitali's convergence Theorem implies that $\rho_{2,h}$ strongly converges to $\rho_2$ in $L^1((0,T) \times \Omega)$.
\end{rem}

\subsection{Existence of weak solutions to \eqref{systdiffcoup}}

In this section, we start by giving the optimality conditions for \eqref{crossfiff-scheme entropy}. Instead of using {\it horizontal perturbations}, $\rho_{i,\eps} = {\Phi_\eps}_{\#} \rho_{i,h}^{k+1}$, as introduced in \cite{JKO} by Jordan, Kinderlherer and Otto, we will perturb $\rho_{i,h}^{k+1}$ with {\it vertical perturbations} introduced in 
 \cite{BS,CS}, and revisited in \cite{Sa,S}, which consist in taking $\rho_{i,\eps}= (1-\eps) \rho_{i,h}^{k+1} + \eps \trho_i$, for any $\trho_i \in L^\infty(\Omega)$. Before giving the optimality conditions for \eqref{crossfiff-scheme entropy}, we state the following Lemma.

\begin{lem}
\label{crossfiff-rho positif}
For all $k \geqslant 1$, $\rho_{i,h}^k>0$ a.e. and $\log(\rho_{i,h}^k)  \in L^1(\Omega)$.
\end{lem}

\begin{proof}
The proof is the same as \cite[Lemma 8.6]{S}.

\end{proof}

This Lemma ensures the uniqueness (up to a constant) of the Kantorovich potential in the transport from $\rho_{i,h}^{k+1}$ to $\rho_{i,h}^k$ and then, we can easily compute the first variation of $W_2(\cdot,\rho_{i,h}^k)$ according to \cite[Proposition 7.17]{S}. 

\begin{prop}
\label{prop:variation premiere}
For $i \in \{1,2\}$, $\rho_{i,h}^{k+1}$ satisfies
\begin{eqnarray}
\label{crossfiff-variation premiere}
\nabla V_i +\nabla \log(\rho_{i,h}^{k+1}) + \nabla F_m'(\rho_{1,h}^{k+1}+\rho_{2,h}^{k+1}) + \frac{\nabla \varphi_{i,h}^{k+1}}{h} = 0 \qquad \rho_{i,h}^{k+1}-a.e,
\end{eqnarray}
where $\varphi_{i,h}^{k+1}$ is the (unique) Kantorovich potential from $\rho_{i,h}^{k+1}$ to $\rho_{i,h}^{k}$.

\end{prop}

\begin{proof}
The proof is a straightforward adaptation of classical result, see for instance \cite{S}.

\end{proof}

A classical consequence of the previous Proposition is that $\rho_{1,h}$ and $\rho_{2,h}$ are solutions to a discrete approximation of system \eqref{systdiffcoup}. 
\begin{prop}
\label{crossfiff-equation discrete entropie}
Let $h>0$, for all $T>0$, let $N$ such that $N=\lfloor \frac{T}{h} \rfloor $. Then for all $(\phi_1, \phi_2) \in \mathcal{C}^\infty_c ([0,T)\times \Rn)^2$ and for all $i \in \{1,2\}$,
\begin{eqnarray*}
\begin{split}
\int_0^T & \int_{\Omega} \rho_{i,h}(t,x)  \partial_t \phi_i(t,x) \,dxdt + \int_{\Omega} \rho_{i,0}(x) \phi_i(0,x) \, dx\\
&= h\sum_{k=0}^{N-1}\int_{\Omega} \nabla V_i(x) \cdot  \nabla \phi_i (t_{k},x) \rho_{i,h}^{k+1}(x)\,dx +h\sum_{k=0}^{N-1}\int_{\Omega} \nabla \rho_{i,h}^{k+1}(x) \cdot  \nabla \phi_i (t_{k},x)\,dx\\
&+h\sum_{k=0}^{N-1}\int_{\Omega} \nabla F_m'(\rho_{1,h}^{k+1} + \rho_{2,h}^{k+1}) \cdot  \nabla \phi_i(t_{k},x) \rho_{i,h}^{k+1}(x)\,dx+\sum_{k=0}^{N-1}\int_{\Omega \times \Omega} \mathcal{R}[\phi_i(t_{k},\cdot)](x,y) d\gamma_{i,h}^k (x,y)
\end{split}
\end{eqnarray*}
\\
where $t_k=hk$ ($t_N :=T$) and $\gamma_{i,h}^k$ is the optimal transport plan in $W_2(\rho_{i,h}^{k},\rho_{i,h}^{k+1})$.
Moreover, $ \mathcal{R}$ is defined such that, for all $\phi \in \mathcal{C}^\infty_c([0,T) \times \Rn)$,
$$ |\mathcal{R}[\phi](x,y)| \leqslant \frac{1}{2} \|D^2 \phi \|_{L^\infty ([0,T) \times \Rn)} |x- y|^2.$$

\end{prop}
\begin{proof}
 We multiply by $\rho_{i,h}^{k+1}$ and take the $L^2$-inner product between the l.h.s. of \eqref{crossfiff-variation premiere} and $\nabla \phi_i(t_k, \cdot)$, for any $\phi_i \in \mathcal{C}^\infty_c ([0,T)\times \Rn)$ and the proof is the same as in \cite{A,L}, for example.
\end{proof}

Another consequence of \eqref{crossfiff-variation premiere} is an improvment of the regularity of $\rho_{i,h}$.

\begin{prop}
\label{prop:improvment regularity}
For all $T>0$ and $i=1,2$, we have
\begin{itemize}
\item $(\rho_{1,h} +\rho_{2,h})^{1/2} \nabla F_m'(\rho_{1,h} +\rho_{2,h}) $ is bounded in $L^2((0,T)\times \Omega)$,

\item $\rho_{i,h}, \rho_{1,h} + \rho_{2,h}$ are bounded  in $L^{2m-1}(((0,T)\times \Omega)$,

\item $\nabla F_m'(\rho_{1,h} +\rho_{2,h}) \rho_{i,h}$ is bounded in $L^{2-1/m}((0,T) \times \Omega)$ and $\rho_{i,h}$ is bounded in $L^{2-1/m}((0,T), W^{1,2-1/m}(\Omega))$.
\end{itemize}
\end{prop}

\begin{proof}
The first item is a direct consequence of \eqref{crossfiff-variation premiere}, using Proposition \ref{prop:strong-convergence entropy}, see for example \cite{KM}, and by Poincaré-Wirtinger inequality, we prove the second item. Now we will prove the third item. The first part is straightforward applying Hölder's inequality,
$$ \| \nabla F_m'(\rho_{1,h} +\rho_{2,h}) \rho_{i,h} \|_{L^{2-1/m}} \leqslant \|\nabla F_m'(\rho_{1,h} +\rho_{2,h}) \rho_{i,h}^{1/2} \|_{L^2}^{1-1/2m} \| \rho_{i,h} \|_{L^{2m-1}}^{1/2m} < +\infty.$$ 
According to \eqref{crossfiff-variation premiere}, we obtain a.e.
$$ |\nabla \rho_{i,h}^{k+1} |^{2-1/m} \leqslant C\left(  \left| \frac{\nabla \varphi_{i,h}^{k+1}\rho_{i,h}^{k+1}}{h} \right|^{2-1/m} + (|\nabla V_i|\rho_{i,h}^{k+1})^{2-1/m} + (|\nabla F_m'(\rho_{1,h}^{k+1}+\rho_{2,h}^{k+1})|\rho_{i,h}^{k+1})^{2-1/m} \right).$$
We have already seen that $\nabla F_m'(\rho_{1,h} +\rho_{2,h}) \rho_{i,h}$ is bounded in $L^{2-1/m}((0,T) \times \Omega)$ and since $\rho_{i,h} \in L^1 \cap L^{2m-1} ((0,T) \times \Omega)$, $ \| \nabla V_i\rho_{i,h} \|_{L^{2-1/m}} \leqslant C$. To deal with the last term, notice that by Hölder's inequality,
$$ \int_\Omega \left| \frac{\nabla \varphi_{i,h}^{k+1}\rho_{i,h}^{k+1}}{h} \right|^{2-1/m} \leqslant \frac{1}{h^{2-1/m}} W_2(\rho_{i,h}^k, \rho_{i,h}^{k+1})^{2-1/m} \|\rho_{i,h}^{k+1} \|_{L^{2m-1}}^{(2m-1)/2m},$$
and then,
\begin{eqnarray*}
h\sum_{k=0}^{N-1}\int_\Omega \left|\frac{\nabla \varphi_{i,h}^{k+1}\rho_{i,h}^{k+1}}{h} \right|^{2-1/m} & \leqslant & C h^{1/m -1} N^{1/2m} \left( \sum_{k=0}^{N-1} W_2^2(\rho_{i,h}^k, \rho_{i,h}^{k+1}) \right)^{(2m-1)/2m}\\
& \leqslant & C T^{1/2m}\left( \frac{\sum_{k=0}^{N-1} W_2^2(\rho_{i,h}^k, \rho_{i,h}^{k+1})}{h} \right)^{(2m-1)/2m}\\
&\leqslant & C_T,
\end{eqnarray*}
by \eqref{crossfiff-estimate distance ind m} where $T=Nh$. Then $\nabla \rho_{i,h}$ is bounded in $L^{2-1/m}((0,T) \times \Omega)$ and we conclude the proof with Poincaré-Wirtinger inequality.

\end{proof}

Now we are able to prove Theorem \ref{crossfiff-existence entropy}.

\begin{proof}[Proof of Theorem \ref{crossfiff-existence entropy}]

We have to pass to the limit in all terms in Proposition \ref{crossfiff-equation discrete entropie} as $h \searrow 0$. The remainder term converges to $0$ using the total square distance estimate \eqref{crossfiff-estimation distance} and the linear term converges to 
$$ \int_0^T \int_\Omega \rho_i \partial_t \phi_i - \int_0^T \int_\Omega \nabla V_i \cdot \nabla \phi_i \rho_i,$$
when $h$ goes to $0$ thanks to Proposition \ref{prop:strong-convergence entropy}.

Furthermore, since $\nabla \rho_{i,h}$ is bounded in $L^{2-1/m}((0,T)\times \Omega)$, because of Proposition \ref{prop:improvment regularity} and the fact that $\rho_{i,h}$ strongly converges to $\rho_i$ in $L^1( (0,T) \times \Omega)$, we conclude that $\nabla \rho_{i,h}$ converges weakly to $\nabla \rho_i$ in $L^{2-1/m}((0,T)\times \Omega)$. This implies that the individual diffusion term converges to 
$$ \int_0^T \int_\Omega \nabla \phi_i \cdot \nabla\rho_i \,dxdt .$$

It remains to study the convergence of the nonlinear cross diffusion term. First, we remark that $\nabla F_m'(\rho_{1,h}^{k+1} +\rho_{2,h}^{k+1})$ can be rewritten as
$$\nabla F_m'(\rho_{1,h}^{k+1} +\rho_{2,h}^{k+1})=2 \frac{ (\rho_{1,h}^{k+1} +\rho_{2,h}^{k+1})^{m/2}}{\rho_{1,h}^{k+1} +\rho_{2,h}^{k+1}} \nabla (\rho_{1,h}^{k+1} +\rho_{2,h}^{k+1})^{m/2}.$$
Then 
$$ \nabla F_m'(\rho_{1,h}^{k+1} +\rho_{2,h}^{k+1})\rho_{i,h}^{k+1} = 2 G_{1-m/2}(\rho_{1,h}^{k+1},\rho_{2,h}^{k+1})\nabla (\rho_{1,h}^{k+1} +\rho_{2,h}^{k+1})^{m/2},$$
where $G_\alpha \, : \, \R^+ \times \R^+ \rightarrow \R$ is the continuous function (for $\alpha <1$) defined by
$$ G_\alpha(x,y):= \left\{\begin{array}{ll}
\frac{x}{(x+y)^\alpha} & \text{    if }  x>0, y \geqslant 0,\\
0 & \text{    otherwise. }
\end{array}\right.$$
As $m\geqslant 1$, $1-\frac{m}{2} <1$ so $G_{1-m/2}$ is continuous and since, up to a subsequence, $\rho_{i,h}$ converges to $\rho_i$ a.e., we obtain that $G_{1-m/2}(\rho_{1,h},\rho_{2,h})$ converges to $G_{1-m/2}(\rho_{1},\rho_{2})$ a.e. in $(0,T)\times \Omega$. In addition, 
\begin{equation}
\label{crossfiff-majoration G}
\left| G_{1-m/2}(\rho_{1,h},\rho_{2,h}) \right| = \left| (\rho_{1,h}+\rho_{2,h})^{m/2} \frac{\rho_{1,h}}{\rho_{1,h}+\rho_{2,h}} \right| \leqslant (\rho_{1,h}+\rho_{2,h})^{m/2}.
\end{equation} 
Up to a subsequence, $\rho_{i,h}$ and $\rho_{1,h}+\rho_{2,h}$ converge a.e. in $(0,T) \times \Omega$, and, since $  (\rho_{1,h}+\rho_{2,h})^{m/2}$ converges to  $(\rho_{1}+\rho_{2})^{m/2}$ in $L^2((0,T)\times \Omega)$, there exists a function $g \in L^2((0,T)\times \Omega)$ such that,
$$ |  (\rho_{1,h}+\rho_{2,h})^{m/2} | \leqslant g.$$
Then Lebesgue's dominated convergence Theorem implies that $G_{1-m/2}(\rho_{1,h},\rho_{2,h})$ converges strongly in $L^2((0,T)\times \Omega)$ to $G_{1-m/2}(\rho_{1},\rho_{2})$. Moreover, $\nabla (\rho_{1,h}^{k+1} +\rho_{2,h}^{k+1})^{m/2}$ converges weakly in $L^2((0,T)\times \Omega)$, by Proposition \ref{prop:strong-convergence entropy}, then $ \nabla F_m'(\rho_{1,h} +\rho_{2,h})\rho_{i,h}$ converges weakly in $L^1((0,T)\times \Omega)$ to $ \nabla F_m'(\rho_{1} +\rho_{2})\rho_{i}$ and

$$ h\sum_{k=0}^{N-1}\int_{\Omega} \nabla F_m'(\rho_{1,h}^{k+1} + \rho_{2,h}^{k+1}) \cdot  \nabla \phi_i(t_{k},x) \rho_{i,h}^{k+1}(x)\,dx \rightarrow \int_0^T \int_\Omega  \nabla F_m'(\rho_{1} +\rho_{2}) \cdot \nabla \phi_i \rho_{i}\, dxdt.$$
In addition, by Proposition \ref{prop:improvment regularity}, we obtain that $ \nabla F_m'(\rho_1 +\rho_2) \rho_i \in L^{2-1/m}((0,T)\times \Omega)$, which concludes the proof.

\end{proof}

\section{Coupling by hard congestion}
\label{crossfiff-section 4: proof crowd}

In this section we prove the existence of a weak solution to \eqref{crossfiff-system crowd}, i.e. Theorem \ref{crossfiff-existence crowd motion}. This system can be seen as gradient flow in a Wasserstein product space. Using the Jordan-Kinderlherer-Otto scheme, we construct two sequences defined in the following way: let $h>0$ be a time step, we construct a sequence $( \rho_{1,h}^k,\rho_{2,h}^k)$ with $ ( \rho_{1,h}^0,\rho_{2,h}^0)=
( \rho_{1,0},\rho_{2,0})$ and $( \rho_{1,h}^{k+1},\rho_{2,h}^{k+1})$ is a solution to
\begin{eqnarray}
\label{crossfiff-JKO crowd}
\inf_{(\rho_1,\rho_2) \in \calK} \sum_{i=1}^2 
\left[ \frac{1}{2h} W_2^2(\rho_i,\rho_{i,h}^{k}) + \He(\rho_i) + \V_i(\rho_i) \right],
\end{eqnarray}
\\
where $\calK:=\left\{ (\rho_1,\rho_2) \in \Paa(\Omega)^2 \, : \, \rho_1 +\rho_2 \leqslant 1 \right\}$ and $|\Omega|>2$. The direct method shows that these sequences are well-defined. As before, we define the piecewise constant interpolations $\rho_{i,h} \, : \, \R^+ \rightarrow \Paa(\Omega)$ by
$$ \rho_{i,h}(t):= \rho_{i,h}^{k+1}, \qquad \text{if } t \in (kh,(k+1)h].$$

\subsection{Estimates and convergences}

In the following proposition, we list the classical estimates coming from the Wasserstein gradient flow theory.

\begin{prop}
\label{crossfiff-standart estimates crowd}
Let $T>0$. Then there exists $C>0$ such that for $i \in \{1,2\}$ and for all $k\geqslant 0$ such that $k\leqslant N:=\lfloor \frac{T}{h} \rfloor$,
\begin{eqnarray}
\label{eq:standart estimates crowd}
\rho_{1,h}^k + \rho_{2,h}^k \leqslant 1,\qquad
 \He(\rho_{i,h}^k) \leqslant C, \qquad
 \sum_{k=0}^{N-1} W_2^2(\rho_{i,h}^k,\rho_{i,h}^{k+1})\leqslant Ch.
\end{eqnarray}
\end{prop}

As in the previous section, we need stronger estimates in order to handle the very degenerate cross diffusion term, $\dive(\rho_i\nabla p)$.

\begin{prop}
For all $T>0$, there exists a constant $C_T>0$ such that
\begin{eqnarray}
\label{crossfiff-estimate grad crowd}
\| \rho_{1,h}^{1/2} \|_{L^2((0,T),H^1(\Omega))} +\| \rho_{2,h}^{1/2} \|_{L^2((0,T),H^1(\Omega))}  \leqslant C_T.
\end{eqnarray}
\end{prop}

\begin{proof}
We apply the flow interchange technique as previously, Proposition \ref{crossfiff-prop f.g.i sum}. Keeping the same notations as in the previous section, we denote by $\eta_i$ the heat flow with initial condition $\rho_{i,h}^k$. Since the heat flow decreases the $L^\infty$-norm, $(\eta_1(s),\eta_2(s))$, defined in \eqref{crossfiff-chaleur Flow Inter}, is admissible for the minimization problem \eqref{crossfiff-JKO crowd}, for all $s \geqslant 0$. Then the same computations as in Proposition \ref{crossfiff-prop f.g.i sum} give the result. 
\end{proof}

Consequently, we deduce the following convergences.

\begin{prop}
\label{prop:convergence rho}
For all $T>0$, there exist $\rho_1$ and $\rho_2$ in $\mathcal{C}^{0,1/2}([0,T],\Paa(\Omega))$ such that, up to a subsequence,
\begin{enumerate}
\item $\rho_{i,h}$ converges to $\rho_i$ in $L^\infty([0,T],\Paa(\Omega))$,
\item $\rho_{i,h}$ converges strongly to $\rho_i$ in $L^p((0,T) \times \Omega)$, for all $p \in [1, +\infty)$ and $\nabla \rho_{i,h}$ converges narrowly to $\nabla \rho_{i}$.
\end{enumerate}
\end{prop}

\begin{proof}
The total square distance estimate \eqref{eq:standart estimates crowd} and the refined version of Ascoli-Arzelà's Theorem, \cite[Proposition 3.3.1]{AGS}, implies that $\rho_{i,h}$ converges to $\rho_i \in \mathcal{C}^{1/2}([0,T],\Paa(\Omega))$ in $L^\infty([0,T],\Paa(\Omega))$. 
As in Proposition \ref{prop:strong-convergence entropy}, applying \cite[Theorem 2]{RS}, we obtain that $\rho_{i,h}$ converges strongly to $\rho_i$ in $L^1((0,T) \times \Omega)$. And noticing that $\rho_{i,h},\rho_i \leqslant 1$ a.e., we deduce that the strong convergence holds in $L^p((0,T) \times \Omega)$, for all $p \in [1, +\infty)$. To conclude, we remark that $\nabla \rho_{i,h} = 2 \rho_{i,h}^{1/2} \nabla \rho_{i,h}^{1/2}$, $ \rho_{i,h}^{1/2}$ strongly converges to $\rho_{i}^{1/2}$ in $L^2((0,T)\times \Omega)$ and $ \nabla \rho_{i,h}^{1/2}$ weakly converges to $\nabla \rho_{i}^{1/2}$ in $L^2((0,T)\times \Omega)$. 
\end{proof}

We end this section by a lemma implying the uniqueness of the pair of Kantorovich potentials from $\rho_{i,h}^{k+1}$ to $\rho_{i,h}^k$ and then the existence of the first variation of $W_2^2(\cdot,\rho_{i,h}^k)$ (Propositions 7.18 and 7.17 from \cite{S}).\\

\begin{lem}
\label{lem:rho positif crowd}
Minimizers of \eqref{crossfiff-JKO crowd} satisfy $\rho_{i,h}^k >0$ a.e. and $\log(\rho_{i,h}^k) \in L^1(\Omega)$. 
\end{lem}

\begin{proof}
The proof is the same as in \cite[Lemma 8.5]{S}. Indeed we can use a constant perturbation $\trho$ because $(\trho, \trho)$ is admissible in \eqref{crossfiff-JKO crowd} ($\trho +\trho = 2/|\Omega| \leqslant 1$ by \eqref{ass:size subset}).
\end{proof}

\subsection{Pressure field associated to the constraint}

In this section, we introduce a discrete pressure associated to the constraint $\rho_{1,h}^{k+1} +\rho_{2,h}^{k+1} \leqslant 1$. This common pressure is obtained arguing as in \cite{MRCS} in the case of one population.

\begin{lem}
Let $(\rho_{1,h}^{k+1},\rho_{2,h}^{k+1})$ be the unique solution to \eqref{crossfiff-JKO crowd}. Then for all $(\rho_{1},\rho_{2}) \in \calK $,
\begin{equation}
\label{crossfiff-inequality measure}
\int_{\Omega} \psi_{1,h}^{k+1} (\rho_1 - \rho_{1,h}^{k+1}) + \int_{\Omega} \psi_{2,h}^{k+1} (\rho_2 - \rho_{2,h}^{k+1}) \geqslant 0,
\end{equation}
 
where $\psi_{i,h}^{k+1} = \frac{\varphi_{i,h}^{k+1}}{h} +V_i + 1 +\log(\rho_{i,h}^{k+1})$ and $\varphi_{i,h}^{k+1}$ is the optimal (up to a constant) Kantorovich potential in $W_2(\rho_{i,h}^{k+1},\rho_{i,h}^{k})$.
\end{lem}

\begin{proof}
The proof of this result is the same as Lemma 3.1 in \cite{MRCS}. 
\end{proof}

\begin{rem}
\label{rem: rewrite linearization to find pressure}
Notice that \eqref{crossfiff-inequality measure} can be rewritten as
$$ \int_{\Omega} \psi_{1,h}^k f_1 + \int_{\Omega} \psi_{2,h}^k f_2 \geqslant 0,$$
for all functions $f_1 ,f_2 \in L^\infty(\Omega)$ such that

\begin{eqnarray}
\label{crossfiff-condition h}
f_1 +f_2 \leqslant \frac{1-\rho_{1,h}^k -\rho_{2,h}^k}{\eps}, \quad f_i \geqslant \frac{-\rho_{i,h}^k}{\eps} \text{ and } \int_\Omega f_i =0,
\end{eqnarray}
for all $0<\eps \ll 1$.
\end{rem}

In the next proposition, we introduce the common discrete pressure.

\begin{prop}
\label{prop:Euler-Lagrange}
There exists $p_h^k \geqslant 0 $ such that for all, $k \geqslant 1$,
$$p_h^k(1-\rho_{1,h}^k - \rho_{2,h}^k)=0 \qquad a.e. $$
In addition, $p_h^k $ satisfies 
\begin{eqnarray}
\label{crossfiff-first variation pressure}
\nabla p_h^k = - \frac{\nabla \varphi_{i,h}^{k}}{h} -\nabla V_i - \nabla \log(\rho_{i,h}^{k}) \qquad a.e,
\end{eqnarray}
for $i=1,2$. 
\end{prop}

\begin{proof}
 Let $S:=\{ \rho_{1,h}^k +\rho_{2,h}^k =1 \}$ be the set where the constraint is saturated. Firstly, we choose $f_2=0$ on $\Omega$ and $f_1  =0$ on $S$ in Remark \ref{rem: rewrite linearization to find pressure}. Then we have
$$ \int_{S^c} \psi_{1,h}^k f_1  \geqslant 0,$$
for all $f_1  \in L^\infty(\Omega)$. This implies that there exists a constant $C_1$ such that $\psi_{1,h}^k=C_1$ a.e. on $S^c$. Applying the same argument with $f_1 =0$ on $\Omega$ and $f_2 =0$ on $S$, we find a constant $C_2$ such that $\psi_{2,h}^k=C_2$ a.e. on $S^c$.
And since $f_1$ and $f_2$ satisfy \eqref{crossfiff-condition h}, we have
$$ \int_\Omega (\psi_{1,h}^k -C_1 )f_1  +\int_\Omega (\psi_{2,h}^k -C_2 )f_2  \geqslant 0.$$
Now, choosing $f_1 =f$ and $f_2=-f$ on $S$ and by symmetry ($f_1  =-f$ and $f_2=f$), we find
$$ \int_S ((\psi_{1,h}^k -C_1 ) -(\psi_{2,h}^k -C_2 ) ) f =0,$$
for all $f \in L^\infty(\Omega)$. We conclude that $(\psi_{1,h}^k -C_1 )=(\psi_{2,h}^k -C_2 ) =:\psi_h^k$ a.e. on $S$ and consequently
$$ \int_S \psi_h^k(f_1  +f_2) \geqslant 0.$$
On the other hand, since $f_1  +f_2 \leqslant 0$ on $S$, $\psi_h^k \leqslant 0$ a.e on $S$, then we define $p_h^k$ by
$$ p_h^k := (C_1 - \psi_{1,h}^k)_+ = (C_2 - \psi_{2,h}^k)_+.$$
By definition, we have $p_h^k(1-\rho_{1,h}^k -\rho_{2,h}^k)=0$ a.e. and since $\psi_{i,h}^k$ is differentiable a.e., the proof is completed.
\end{proof}

Now, we define the piecewise interpolation $p_{h} \, : \, \R^+ \rightarrow L^{1}(\Omega)$ by
$$ p_{h}(t):= p_{h}^{k+1}, \qquad \text{if } t \in (kh,(k+1)h].$$
Notice that $p_{h}(t) \geqslant 0$ and for all $t\geqslant 0$, $p_h(t)(1-\rho_{1,h}(t)-\rho_{2,h}(t))=0$ a.e.
Therefore, we immediately deduce the following estimate on the pressure.

\begin{prop}
\label{crossfiff-estimate pressure}
For all $T>0$, $  p_h$ is bounded in $L^2((0,T) , H^1( \Omega))$.
\end{prop}

\begin{proof}
First, we prove that $\nabla p_h$ is bounded in $L^2((0,T)\times \Omega)$ and then we will conclude using Poincaré's inequality.
By definition of $p_h^{k+1}$, we have
\begin{eqnarray*}
\int_\Omega |\nabla p_h^{k+1}|^2(\rho_{1,h}^{k+1} +\rho_{2,h}^{k+1}) &= & \sum_{i=1}^2 \int_\Omega |\nabla \psi_{i,h}^{k+1}|^2\rho_{i,h}^{k+1}\\
&\leqslant & C \sum_{i=1}^2 \left( \int_\Omega \left| \frac{\nabla \phi_{i,h}^{k+1}}{h}\right|^2\rho_{i,h}^{k+1} + \int_\Omega |\nabla V_i |^2 \rho_{i,h}^{k+1} + \int_\Omega \frac{|\nabla \rho_{i,h}^{k+1}|^2}{\rho_{i,h}^{k+1}}\right)\\
&\leqslant &C \sum_{i=1}^2\left( \frac{1}{h^2} W_2^2(\rho_{i,h}^k,\rho_{i,h}^{k+1}) +C +\| (\rho_{i,h}^{k+1})^{1/2}\|_{H^1(\Omega)}\right),
\end{eqnarray*}
where the last line is obtained using the fact that $\nabla V_i \in L^\infty(\Omega)$. Summing the previous inequalities over $k$ and by \eqref{crossfiff-standart estimates crowd} and \eqref{crossfiff-estimate grad crowd}, we obtain that
$$ \int_0^T \int_\Omega | \nabla p_h(t)|^2 (\rho_{1,h}(t) + \rho_{2,h}(t)) \leqslant C.$$
Since $p_h(t) =0$ a.e. on $\{\rho_{1,h}(t) + \rho_{2,h}(t) <1\}$, we deduce
$$ \int_0^T \int_\Omega | \nabla p_h(t)|^2 = \int_0^T \int_\Omega | \nabla p_h(t)|^2 (\rho_{1,h}(t) + \rho_{2,h}(t))\leqslant C.$$
\\
We conclude with the same argument as \cite{MS}. Using Poincaré's inequality, since $|\{p_h(t) =0 \} | \geqslant |\{\rho_{1,h}(t) + \rho_{2,h}(t) <1 \} | \geqslant |\Omega| -2 >0$, by \eqref{ass:size subset}, we obtain that $p_h$ is bounded in $L^2((0,T), H^1(\Omega))$.
\end{proof}
Using Proposition \ref{crossfiff-estimate pressure}, the regularity of $\rho_i$ can be improved.
\begin{cor}
For all $T>0$ and $i=1,2$, $\rho_{i,h}$ is bounded in $L^2((0,T) , H^1( \Omega))$.
\end{cor}

\begin{proof}
By \eqref{crossfiff-first variation pressure} combined with $\rho_{i,h}^{k+1}\leqslant 1$, we obtain that
$$ |\nabla \rho_{i,h}^{k+1} |^2 \leqslant C\left(\frac{|\nabla \varphi_{i,h}^{k+1}|^2}{h^2} \rho_{i,h}^{k+1} +|\nabla V_i|^2\rho_{i,h}^{k+1} +|\nabla p_{h}^{k+1} |^2 \right) \qquad a.e.$$
Since, by Proposition \ref{crossfiff-estimate pressure}, $\nabla p_{h}$ is bounded in $L^2((0,T) \times \Omega)$ and 
$$ h\sum_{k=0}^{N-1} \int_\Omega \frac{|\nabla \varphi_{i,h}^{k+1}|^2}{h^2} \rho_{i,h}^{k+1} \leqslant C,$$
because of \eqref{crossfiff-standart estimates crowd}, we have
$$ \|\nabla \rho_{i,h} \|_{L^2((0,T)\times \Omega)} \leqslant C.$$
The proof is concluded noticing that  
$$ \| \rho_{i,h} \|_{L^2((0,T) \times \Omega)} \leqslant \| \rho_{i,h} \|_{L^\infty((0,T) \times \Omega)}^{1/2}\| \rho_{i,h} \|_{L^1((0,T) \times \Omega)}^{1/2} \leqslant T^{1/2}.$$
\end{proof}

To analyse the pressure field $p_h$, we recall the following lemma, \cite{MRCS,MS},

\begin{lem}{\cite[Lemma 3.5]{MS}}
\label{crossfiff-limit pressure}
Let $(p_h)_{h>0}$ be a bounded sequence in $L^2([0,T],H^1(\Omega))$ and $(\rho_h)_{h>0}$ a sequence of piecewise constant curves valued in $\Pa(\Omega)$ which satisfiy $W_2(\rho_h(t),\rho_h(s)) \leqslant C\sqrt{t-s-h}$ for all $s<t \in [0,T]$ and $\rho_h \leqslant C$ for a fixed constant $C$. Suppose that
$$ p_h \geqslant 0, \quad p_h(1-\rho_h)=0, \quad \rho_h \leqslant 1,$$
and that 
$$ p_h \rightharpoonup p \text{ weakly in } L^2([0,T],H^1(\Omega)) \text{  and  } \rho_h \rightarrow \rho \text{ uniformly in } \Pa(\Omega).$$
Then $p(1-\rho) =0$.
\end{lem}

Consequently, one has

\begin{prop}
\label{prop:convergence pressure}
There exists $p \in L^2([0,T],H^1(\Omega))$ such that $p_h$ converges weakly in $ L^2([0,T],H^1(\Omega))$ to $p$, where $p$ satisfies
$$ p \geqslant 0, \quad p(1-\rho_1 -\rho_2)=0, \quad \rho_1 +\rho_2 \leqslant 1 \text{  a.e. in } [0,T]\times \Omega.$$
In addition, $\rho_{i,h}\nabla p_h$ narrowly converges to $\rho_{i}\nabla p$.
\end{prop}

\begin{proof}

We apply Lemma \ref{crossfiff-limit pressure} to $\rho_h:=\rho_{1,h} + \rho_{2,h}$ and $p_h$. According to Proposition \ref{crossfiff-estimate pressure}, $p_h$ weakly converges in $L^2((0,T), H^1(\Omega))$ to $p$ such that 
\begin{eqnarray}
\label{crossfiff-condition pressure limit}
p \geqslant 0, \quad p(1-\rho_{1}-\rho_{2})=0, \quad \rho_{1}+\rho_{2} \leqslant 1.
\end{eqnarray}
Moreover, using the estimate on $p_h$, we know that $\nabla p_h$ weakly converges to $\nabla p$ in $L^2((0,T)\times \Omega)$. Then since $\rho_{i,h}$ strongly converges to $\rho_i$ in $L^2((0,T) \times \Omega)$ (Proposition \ref{prop:convergence rho}), by strong-weak convergence, we obtain that $\rho_{i,h}\nabla p_h$ narrowly converges to $\rho_{i}\nabla p$.

\end{proof}

\subsection{Existence of weak solutions to \eqref{crossfiff-system crowd}}
Arguing as in Proposition \ref{crossfiff-equation discrete entropie}, $(\rho_{1,h},\rho_{2,h})$ is solution to a discrete approximation of system \eqref{crossfiff-system crowd}.

\begin{prop}
\label{crossfiff-equation discrete crowd}
Let $h>0$, for all $T>0$, let $N$ such that $N=\lfloor \frac{T}{h} \rfloor $. Then for all $(\phi_1, \phi_2) \in \mathcal{C}^\infty_c ([0,T)\times \Rn)^2$ and for all $i \in \{1,2\}$,
\begin{eqnarray*}
\begin{split}
\int_0^T \int_{\Omega} &\rho_{i,h}(t,x)  \partial_t \phi_i(t,x) \,dxdt + \int_{\Omega} \rho_{i,0}(x) \phi_i(0,x) \, dx\\
&= h\sum_{k=0}^{N-1} \int_{\Omega} \nabla V_i(x) \cdot  \nabla \phi_i (t_{k},x) \rho_{i,h}^{k+1}(x)\,dx 
+h\sum_{k=0}^{N-1}\int_{\Omega} \nabla \rho_{i,h}^{k+1}(x) \cdot  \nabla \phi_i (t_{k},x)\,dx\\
&+h\sum_{k=0}^{N-1}\int_{\Omega} \nabla p_h^{k+1} \cdot  \nabla \phi_i(t_{k},x) \rho_{i,h}^{k+1}(x)\,dx
+\sum_{k=0}^{N-1}\int_{\Omega \times \Omega} \mathcal{R}[\phi_i(t_{k},\cdot)](x,y) d\gamma_{i,h}^k (x,y)\\
\end{split}
\end{eqnarray*}
where $t_k=hk$ ($t_N :=T$) and $\gamma_{i,h}^k$ is the optimal transport plan in $W_2(\rho_{i,h}^{k},\rho_{i,h}^{k+1})$.
Moreover, $ \mathcal{R}$ is defined such that, for all $\phi \in \mathcal{C}^\infty_c([0,T) \times \Rn)$,
$$ |\mathcal{R}[\phi](x,y)| \leqslant \frac{1}{2} \|D^2 \phi \|_{L^\infty ([0,T) \times \Rn)} |x- y|^2.$$

\end{prop}

Combining Propositions \ref{crossfiff-standart estimates crowd}, \ref{prop:convergence rho}, \ref{prop:convergence pressure} and \ref{crossfiff-equation discrete crowd}, the rest of the proof of Theorem \ref{crossfiff-existence crowd motion} is identical to the proof of Theorem \ref{crossfiff-existence entropy} in the previous section.

\begin{rem}
As in Remark \ref{rem: drop diffusion porous}, it is possible to drop one diffusion. Say we drop the individual Entropy for the second species, $\rho_2$. The difficulty is to pass to the limit in the nonlinear term $\rho_{2,h} \nabla p_h$. This term can be rewritten as 
$$ (\rho_{1,h} +\rho_{2,h}) \nabla p_h - \rho_{1,h} \nabla p_h. $$
Taking advantage of the definition of $p_h$, we deduce that $(\rho_{1,h} +\rho_{2,h}) \nabla p_h = \nabla p_h$ a.e and then converges weakly to  $\nabla p=(\rho_1 + \rho_2) \nabla p$ in $L^2((0,T) \times \Omega)$. Moreover, since $ \rho_{1,h}$ strongly converges in $L^2((0,T) \times \Omega)$ by Proposition \ref{prop:convergence rho} and $\nabla p_h$ converges weakly in $L^2((0,T) \times \Omega)$ we can pass to the limit in the second term by strong-weak convergence. Then we deduce that $\rho_{2,h} \nabla p_h$ weakly converges to $\rho_{2} \nabla p$. 
\end{rem}

\section{Systems with a common drift}
\label{crossfiff-section 5: link and uniqueness}

In this section, we focus on the special case where $\nabla V_1= \nabla V_2=: \nabla V \in L^\infty(\Omega)$. Although this asumption is very restrictive, it allows us to obtain better estimates on solutions (Proposition \ref{crossfiff-regularity pm} and Proposition \ref{prop:improvment regularity}) which are hard to get in the general case due to the lack of convexity of $\F_m(\rho_1 +\rho_2)$, see Remark \ref{rem: not convex}. Therefore, in this case, we will be able to prove the convergence of a solution to \eqref{systdiffcoup} to a solution to \eqref{crossfiff-system crowd}, when $m$ goes to $+\infty$. Moreover, under some regularity we give a $L^1$-contraction result for systems \eqref{crossfiff-system crowd} and \eqref{systdiffcoup}.

\begin{rem}
\label{rem: not convex}
It is well-known in the Wasserstein gradient flow theory that the $\lambda$-geodesic convexity of the functional implies a $W_2$-contraction of the flow. Unfortunately, as mentioned in \cite{KM}, in general, $(\rho_1,\rho_2) \in \Paa(\Omega)^2 \mapsto \F_m(\rho_1+\rho_2)$ is not displacement convex. Indeed, for $m=2$, we can rewrite the functional as
$$ \F_2(\rho_1+\rho_2)= \F_2(\rho_1) +\F_2(\rho_2) +2\int_\Omega \rho_1 \rho_2.$$
Let $\rho_2$ be a fixed density, we study the displacement convexity of $\rho \mapsto \F_2(\rho) +2\int_\Omega \rho_2 \rho$. We know, see \cite{MC}, that $\rho \in \Paa(\Omega) \mapsto \F_2(\rho)$ is displacement convex but $\rho \mapsto\int_\Omega \rho_2 \rho$ is displacment convex if $\rho_2$ is $\lambda$-convex.
\end{rem}

To overcome this lack of convexity, we need to obtain a stronger estimate, independent on $m$, on $\nabla  F_m'(\rho_{1,m} +\rho_{2,m})$, where $(\rho_{1,m},\rho_{2,m})$ is a solution to \eqref{systdiffcoup}. In the case of a common drift, this estimate can be found observing that $\rho_m:=\rho_{1,m}+\rho_{2,m}$ is the Wasserstein gradient flow of $\E + \V + \F_m$ and then, solves 
\begin{eqnarray}
\label{crossfiff-eq sum}
\partial_t \mu - \Delta \mu -\dive(\mu \nabla V) - \dive(\mu \nabla F_m'(\mu))=0,
\end{eqnarray}
with initial condition $\mu_{|t=0}=\rho_{1,0}+\rho_{2,0}$.
\begin{prop}
\label{crossfiff-regularity pm}
Let $(\rho_{1,m},\rho_{2,m})$ be a solution to \eqref{crossfiff-eq sum} in $L^2((0,T),H^1(\Omega))$ with $\nabla V_1= \nabla V_2=: \nabla V \in L^\infty(\Omega)$. Then $\rho_m:=\rho_{1,m}+\rho_{2,m}$ is unique and $F_m'(\rho_m)$ is bounded independently of $m$ in $L^2((0,T),H^1(\Omega))$, for all $T<+\infty$.   
\end{prop} 
\begin{proof}
 As we remark above, $\rho_m$ is solution to \eqref{crossfiff-eq sum}.
 By geodesic convexity of $\E$ and $\F_m$, we know that solution to \eqref{crossfiff-eq sum} is unique (see \cite{AGS}). 
To conclude, we reason as in \cite[Lemma 5.6]{GLM}. The proof is based on the flow interchange technique with the (smooth) solution to
$$ \left\{ \begin{array}{ll}
\partial_t \eta =\Delta\eta^{m-1} +\eps \Delta \eta & \text{ in } (0,T)\times \Omega,\\
(\nabla \eta^{m-1} +\eps \nabla \eta) \cdot \nu =0 & \text{ in } (0,T)\times \partial \Omega,\\
\eta_{|t=0}=\rho_{h,m}^k,&
\end{array}\right.$$
where $\rho_{h,m}^k$ is constructed using the JKO scheme. We obtain, when $\eps$ goes to $0$ and using a lower semi-continuity argument, $ \| \nabla F'_m(\rho_{m}) \|_{L^2((0,T),H^1(\Omega))} \leqslant C_T$, for all $T>0$, where $C_T$ is a constant independent on $m$. The $L^1$-estimate of $F_m'(\rho_m)$ and the Poincaré-Wirtinger inequality conclude the proof.

\end{proof}

Now, we show that $(\rho_{1,m},\rho_{2,m})$ converges to a solution to \eqref{crossfiff-system crowd}, $(\rho_{1,\infty},\rho_{2,\infty})$, as $m\nearrow+\infty$.

\begin{thm}
\label{crossfiff-convergence m}
Assume that the initial data satisfy $\rho_{1,0} + \rho_{2,0} \leqslant 1$. Up to a subsequence, as $m\rightarrow +\infty$, a solution to \eqref{systdiffcoup}, $ (\rho_{1,m},\rho_{2,m})$, converges strongly in $L^2((0,T)\times \Omega)$ to $ (\rho_{1,\infty},\rho_{2,\infty})$ and $p_m:=F_m'(\rho_{1,m}+\rho_{2,m})$ converges weakly in $L^2((0,T),H^1(\Omega))$ to $p_\infty$, where $ (\rho_{1,\infty},\rho_{2,\infty},p_\infty)$ is a solution to \eqref{crossfiff-system crowd}.
\end{thm}

\begin{proof}
First we prove the convergence of $\rho_{i,m}$. We start noticing that the estimate \eqref{crossfiff-estimation grad} does not depend on $m$ and then by Remark \ref{crossfiff-estimate distance ind m}, we have
$$ \| \rho_{i,m}^{1/2} \|_{L^2((0,T),H^1(\Omega))} \leqslant C_T \text{   and   } W_2(\rho_{i,m}(t),\rho_{i,m}(s))\leqslant C_T|t-s|^{1/2},$$
for all $t,s \leqslant T$ and where $C_T$ is a contant independent on $m$. Then using the Rossi-Savaré Theorem we obtain that $\rho_{i,m}$ converges to $\rho_{i,\infty}$ in $L^1((0,T)\times \Omega)$. In fact, $\rho_{i,m}$ converges strongly to $\rho_{i,\infty}$ in $L^2((0,T)\times \Omega)$. Indeed, for $m\gg 2$, $ \|\rho_{i,m} \|_{L^m((0,T) \times \Omega)}$ is uniformly bounded in $m$ so $(\rho_{i,m}^2)_m$ is uniformly integrable. Then, $\rho_{i,m}$ converges weakly in $L^2((0,T) \times \Omega)$  to $\rho_{i,\infty}$ and Vitali's convergence Theorem implies that $\|\rho_{i,m} \|_{L^2((0,T) \times \Omega)} =\|\rho_{i,m}^2 \|^{1/2}_{L^1((0,T) \times \Omega)}\rightarrow \|\rho_{i,\infty}^2 \|^{1/2}_{L^1((0,T) \times \Omega)}=\|\rho_{i,\infty} \|_{L^2((0,T) \times \Omega)}$. Furthermore,  $p_m$ converges weakly in $L^2((0,T),H^1(\Omega))$ to $p_\infty$, Proposition \ref{crossfiff-regularity pm}, and obviously $p_\infty \geqslant 0$. Consequently, we can pass to the limit in the weak formulation of the system \eqref{systdiffcoup} to obtain the weak formulation of sytem \eqref{crossfiff-system crowd}.
\\

To conclude the proof, it remains to prove that 
$$ \rho_{1,\infty}+\rho_{2,\infty} \leqslant 1 \text{   and   }  p_\infty(1-\rho_{1,\infty} - \rho_{2,\infty}) = 0 \qquad a.e.$$

We start to show that $ \rho_{1,\infty}+\rho_{2,\infty} \leqslant 1$. The argument is the same as in \cite[Lemma 4.3]{AKY}. The estimate \eqref{crossfiff-estimation entropie} does not depend on $m$ so we have
\begin{eqnarray}
\label{crossfiff-estimation sup 1}
\int_0^T \int_\Omega (\rho_{1,m}+\rho_{2,m}-1)_+^2 \,dxdt \leqslant \frac{2C}{m} \rightarrow 0,
\end{eqnarray}
when $m\rightarrow +\infty$, which implies that $ \rho_{1,\infty}+\rho_{2,\infty} \leqslant 1$ a.e.

To obtain the second part of the claim, we start proving 
$$ \int_0^T \int_\Omega p_m(1-\rho_{1,m}-\rho_{2,m}) \varphi \,dxdt \rightarrow \int_0^T \int_\Omega p_\infty (1-\rho_{1,\infty}-\rho_{2,\infty})\varphi \,dxdt,$$
for all $\varphi \in \mathcal{C}_c^\infty((0,T)\times \Omega)$. With the same argument as before, $\rho_{1,m}+\rho_{2,m} \rightarrow \rho_{1,\infty}+\rho_{2,\infty}$ strongly in $L^2((0,T)\times \Omega)$ and $p_m \rightharpoonup p_\infty$ weakly in $L^2((0,T)\times \Omega)$, then by strong-weak convergence, we obtain the result. 
Now, we show that
$$  \int_0^T \int_\Omega p_m(1-\rho_{1,m}-\rho_{2,m})\varphi \,dxdt \rightarrow 0,$$
for all nonnegative $\varphi \in \mathcal{C}_c^\infty((0,T)\times \Omega)$. We start splitting the integral,
\begin{eqnarray*}
\int_0^T \int_\Omega p_m(1-\rho_{1,m}-\rho_{2,m})\varphi\,dxdt &=& \iint_{\{\rho_{1,m}+\rho_{2,m} \leqslant 1 \}}p_m(1-\rho_{1,m}-\rho_{2,m})\varphi\,dxdt\\
& +&\iint_{\{\rho_{1,m}+\rho_{2,m} \geqslant 1 \}}p_m(1-\rho_{1,m}-\rho_{2,m})\varphi\,dxdt.
\end{eqnarray*}

Remark that, since $\rho_{1,m}+\rho_{2,m} \rightarrow \rho_{1,\infty}+\rho_{2,\infty}$ strongly in $L^1((0,T)\times \Omega)$, up to a subsequence, $\rho_{1,m}(t,x)+\rho_{2,m}(t,x) \rightarrow \rho_{1,\infty}(t,x)+\rho_{2,\infty}(t,x)$ $(t,x)$-a.e. Let $(t,x) \in [0,T] \times \Omega$ be a point where the convergence a.e. holds. If $\rho_{1,\infty}(t,x)+\rho_{2,\infty}(t,x)<1$, then $\rho_{1,m}(t,x)+\rho_{2,m}(t,x) \leqslant (1-\eps)$, for large $m$ and $p_m(t,x) \leqslant \frac{m}{m-1}(1-\eps)^{m-1}\rightarrow 0$, therefore $p_m(t,x)(1-\rho_{1,m}(t,x)-\rho_{2,m}(t,x))\rightarrow 0$. 
On the other hand, if $\rho_{1,\infty}(t,x)+\rho_{2,\infty}(t,x)=1$ and, for large $m$, $\rho_{1,m}(t,x)+\rho_{2,m}(t,x)\leqslant 1$, then $1-\rho_{1,m}(t,x)-\rho_{2,m}(t,x)\rightarrow 0$ and $p_m(t,x) \leqslant \frac{m}{m-1}$ remains bounded. 
Thus,  $p_m(t,x)(1-\rho_{1,m}(t,x)-\rho_{2,m}(t,x))\rightarrow 0$ a.e. and since on $\{\rho_{1,m}+\rho_{2,m} \leqslant 1 \}$, $\rho_{1,m}+\rho_{2,m}$ is bounded by $1$ and $p_m \leqslant \frac{m}{m-1} \leqslant 2,$ by Lebesgue convergence Theorem, we obtain
$$ \iint_{\{\rho_{1,m}+\rho_{2,m} \leqslant 1 \}}p_m(1-\rho_{1,m}-\rho_{2,m})\varphi\,dxdt \rightarrow 0.$$
The convergence of the second term is obtained by applying Cauchy-Schwarz inequality, \eqref{crossfiff-estimation sup 1} and Proposition \ref{crossfiff-regularity pm},
$$  \left| \iint_{\{\rho_{1,m}+\rho_{2,m} \geqslant 1 \}}p_m(1-\rho_{1,m}-\rho_{2,m})\varphi\,dxdt \right| \leqslant \| p_m \|_{L^2((0,T) \times \Omega)}\frac{C}{m^{1/2}}\rightarrow 0,$$
when $m \nearrow +\infty$. Then, for all $\varphi \in \mathcal{C}_c^\infty((0,T)\times \Omega)$,
$$  \int_0^T \int_\Omega p_\infty (1-\rho_{1,\infty}-\rho_{2,\infty})\varphi \,dxdt=0.$$
Since $p_\infty (1-\rho_{1,\infty}-\rho_{2,\infty}) \geqslant 0$, we conclude that $p_\infty (1-\rho_{1,\infty}-\rho_{2,\infty})=0$ a.e. in $(0,T) \times \Omega$.

\end{proof}

To end this section, we give a $L^1$-contraction result for $m \in [1,+\infty] $ under some regularity on solutions but first we establish maximum principle for $m \in [1, +\infty)$.

\begin{prop}
\label{prop:maximum principle}
Assume that $\rho_{i,0} + \rho_{2,0} \leqslant M_0$. For all $m\in [1, +\infty)$ and $T <+\infty$, there exists a constant $M_T>0$ such that $\| \rho_{1,m} + \rho_{2,m}\|_{L^\infty((0,T)\times \Omega)} \leqslant M_T$. In addition, we have $ \nabla \rho_{i,m} , \nabla F_m'(\rho_{1,m} + \rho_{2,m}) \in L^2((0,T) \times \Omega)$.
\end{prop}

\begin{proof}
It is well known that the solution $\mu$ to \eqref{crossfiff-eq sum} satisfies a maximum principle, see for instance \cite{O, A, PT, L_these, S}. Then by uniqueness of the solution, there exists $M_T$ such that $\| \rho_{1,m} + \rho_{2,m}\|_{L^\infty((0,T)\times \Omega)} \leqslant M_T$. We obtain then 
$$ | \nabla \rho_{i,m} | \leqslant 2M_T^{1/2} | \nabla \rho_{i,m}^{1/2} | \text{ and } |\rho_{i,m}\nabla F_m'(\rho_{1,m} + \rho_{2,m}) | \leqslant M_T|\nabla F_m'(\rho_{1,m} + \rho_{2,m}) |.$$

Since, $\nabla \rho_{i,m}^{1/2}$ and $\nabla F_m'(\rho_{1,m} + \rho_{2,m})$ are in $L^2((0,T) \times \Omega)$ (Proposition \ref{crossfiff-regularity pm}), the proof is concluded. 
\end{proof}

\begin{rem}
\label{Rem:density} 
In the sepcial case of a common drift, by Proposition \ref{prop:maximum principle}, we can improve the regularity of solutions to \eqref{systdiffcoup} in  Definition \ref{weak solution} if we start with $L^\infty$ initial conditions. Then, as in \cite[Remark 3.2 (a)]{KM}, we notice that, by density, we can consider test functions in $W^{1,1}((0,T),L^1(\Omega)) \cap L^2((0,T),H^1(\Omega))$ in Definition \ref{weak solution} for system \eqref{systdiffcoup} and system \eqref{crossfiff-system crowd}.
\end{rem}

\begin{thm}
\label{crossfiff-uniqueness thm}
Let $(\rho_{1,m}^1,\rho_{2,m}^1)$ and $(\rho_{1,m}^2,\rho_{2,m}^2)$ be two solutions to \eqref{systdiffcoup} (or \eqref{crossfiff-system crowd} if $m=+\infty$) with intial conditions $(\rho_{1,0}^1 ,\rho_{2,0}^1)$ and $(\rho_{1,0}^2 , \rho_{2,0}^2)$, respectively. Assume there exists $M_0 >0$ such that
$$\|\rho_{1,0}^1 + \rho_{2,0}^1\|_{L^\infty( \Omega)},\|\rho_{1,0}^2 + \rho_{2,0}^2\|_{L^\infty( \Omega)}  \leqslant M_0.$$
If $\partial_t \rho_{i,m}^1,\partial_t\rho_{i,m}^2 \in L^1((0,T)\times \Omega)$, then 
$$ \|\rho_{i,m}^1(t,\cdot) - \rho_{i,m}^2(t,\cdot) \|_{L^1(\Omega)} \leqslant \|\rho_{i,0}^1 - \rho_{i,0}^2 \|_{L^1(\Omega)}. $$
\end{thm}

\begin{proof}
First if $m<+\infty$, since $\rho_{1,m}+\rho_{2,m}$ solves \eqref{crossfiff-eq sum}, then it is unique and according to Proposition \ref{crossfiff-regularity pm}, $p_m:=F_m'(\rho_{1,m}+\rho_{2,m})$ is in $L^2((0,T),H^1(\Omega))$. Moreover, we have already shown in Theorem \ref{crossfiff-existence crowd motion} that the pressure $p_\infty$ associated to the constraint $\rho_{1,\infty}+\rho_{2,\infty} \leqslant 1 $  is in $L^2((0,T), H^1(\Omega))$ and, according to \cite{DMM}, $(\rho_{1,\infty}+\rho_{2,\infty},p_\infty)$ is unique. Then, for $m \in [1,+\infty]$, $\rho_{1,m}^i$ solves 
$$ \partial_t \rho_{1,m}^i -\Delta\rho_{1,m}^i -\dive(\rho_{1,m}^i (\nabla V + \nabla p_m)) =0.$$
Now, by the same argument as \cite{Ouniqueness,A}, we prove the $L^1$-contraction. We prove the result for $i=1$ and the argument is the same for $i=2$. We note $\Omega_T:=(0,T) \times \Omega$. Define the smooth function, for $z \in \R$, $f(z) = e^{-1/z}e^{-1/(1-z)}$ if $z \in (0,1)$ and $0$ otherwise and $M:=\|f\|_{L^\infty}$. Then for $\delta>0$, define the smooth function $\phi_\delta$ by
$$ \phi_\delta(z):= \frac{1}{Z} \int_0^{z/\delta} f(\xi) \, d\xi, \text{ where } Z := \int_0^{1} f(\xi) \, d\xi.$$
Consider $ \zeta_{\delta} := \phi_\delta (\rho_{1,m}^1-\rho_{1,m}^2).$
By definition, $\zeta_\delta \in W^{1,1}((0,T), L^1(\Omega))\cap L^2((0,T), H^1(\Omega)) \cap L^\infty (\Omega_T)$. Then taking $\zeta_\delta$ as an admissible test function in Definition \ref{weak solution}, see Remark \ref{Rem:density}, we obtain

$$ \iint_{\Omega_T} \partial_t(\rho_{1,m}^1-\rho_{1,m}^2)\zeta_{\delta}=-\iint_{\Omega_T}\left((\rho_{1,m}^1-\rho_{1,m}^2)(\nabla V+\nabla p_m) \cdot \nabla \zeta_{\delta}+\nabla ( \rho_{1,m}^1-\rho_{1,m}^2) \cdot \nabla \zeta_{\delta} \right) \,dxdt.$$

We introduce $\Omega_T^\delta:= \Omega_T \cap \{0 < \rho_{1,m}^1-\rho_{1,m}^2 < \delta\}$. Then by definition of $\zeta_{\delta}$ 

\begin{align*}
 \iint_{\Omega_T}& \partial_t(\rho_{1,m}^1-\rho_{1,m}^2)\zeta_{\delta}\\
 &= -\frac{1}{Z\delta}\iint_{\Omega_T^\delta} (\rho_{1,m}^1-\rho_{1,m}^2)(\nabla V +\nabla p_m) \cdot \nabla( \rho_{1,m}^1-\rho_{1,m}^2)f\left(\frac{\rho_{1,m}^1 - \rho_{1,m}^2}{\delta}\right)\,dxdt \\
 & - \frac{1}{Z\delta}\iint_{\Omega_T^\delta} |\nabla ( \rho_{1,m}^1-\rho_{1,m}^2)|^2 f\left(\frac{\rho_{1,m}^1 - \rho_{1,m}^2}{\delta} \right)\,dxdt.
  \end{align*}
 Young's inequality gives 
 \begin{align*}
 \iint_{\Omega_T}& \partial_t(\rho_{1,m}^1-\rho_{1,m}^2)\zeta_{\delta}\\
& \leqslant  \frac{M}{2Z\delta}\iint_{\Omega_T^\delta}(\rho_{1,m}^1-\rho_{1,m}^2)^2|\nabla V +\nabla p_m|^2\,dxdt\\
&-\frac{1}{2Z\delta}\iint_{\Omega_T^\delta}|\nabla ( \rho_{1,m}^1-\rho_{1,m}^2)|^2f\left(\frac{\rho_{1,m}^1 - \rho_{1,m}^2}{\delta}\right)\,dxdt\\
&\leqslant  \frac{M}{2Z} \|\nabla V +\nabla p_m\|_{L^2(\Omega_T)}^2 \delta.
 \end{align*}
 Then, when $\delta \searrow 0$, by Fatou's Lemma,
 $$ \iint_{\Omega_T \cap \{\rho_{1,m}^1-\rho_{1,m}^2 \geqslant 0\} } \partial_t(\rho_{1,m}^1-\rho_{1,m}^2)  \leqslant 0.$$
 Reversing the roles of $\rho_{1,m}^1$ and $\rho_{1,m}^2$, we have
$$ \iint_{\Omega_T} \partial_t(|\rho_{1,m}^1-\rho_{1,m}^2|) \leqslant 0,$$
which concludes the proof.

\end{proof}

\section{Numerical simulations}
\label{crossfiff-secction 6: simulations}
To end this paper, we use the algorithm introduced in \cite{BCL} to present numerical simulations in dimension 2 on the square $\Omega=\left[-\frac{1}{2}, \frac{1}{2}\right]^2$. Simulations are carried out using a $50 \times 50$ discretization in space with a time step $h=0.01$. The first system we study is the transport equation with common porous media congestion, without individual diffusions,
\begin{eqnarray}
\label{eq:porous without diffusion}
\partial_t \rho_i -\alpha_i\dive(\rho_i \nabla F_m'(\alpha_1\rho_1+\alpha_2\rho_2)) -\dive(\rho_i \nabla V_i)=0, \, i =1,2,
\end{eqnarray}
which, at least formally, is the gradient flow in Wasserstein space for the energy
$$ \E(\rho_1,\rho_2):= \int_\Omega V_1\rho_1 + \int_\Omega V_2\rho_2 + \int_\Omega F_m(\alpha_1\rho_1+\alpha_2\rho_2).$$
Arguing as in \cite{BCL},  setting $\phi=(\phi_1, \phi_2)$, $(D\phi_1, D\phi_2):=(\partial_t \phi_1, \nabla \phi_1, \partial_t \phi_2, \nabla \phi_2)$,  $q=(q_1, q_2)=(a_1, b_1, c_1, a_2, b_2, c_2)$,  $\sigma=(\sigma_1, \sigma_2)=((\mu_1, m_1, \tilm_1), (\mu_2, m_2, \tilm_2))$ and defining the convex set $K:=\{(a,b) \in\R^{n+1} \, : \, a+ \frac{1}{2 }|b|^2\leqslant 0\}$, one can rewrite one step of the JKO scheme, \eqref{crossfiff-scheme entropy}, with $\E$ as  a saddle-point problem for the augmented Lagrangian 
\[\begin{split}
L_r(\phi, q , \sigma)&=\sum_{i=1}^2 \int_{\Omega} \phi_i(0, x) \rho_{i,h}^k(x) \mbox{d} x+\sum_{i=1}^2 \int_0^1\int_{\Omega} \chi_K(a_i(t,x), b_i(t,x))  \mbox{d} x  \mbox{d} t\\
&+ \sum_{i=1}^2 \int_0^1\int_{\Omega} \Big(  (\mu_i, m_i)\cdot(  D\phi_i -(a_i, b_i) ) +\frac{r}{2} \vert  D \phi_i -(a_i, b_i)\vert^2 \Big)  \mbox{d} x  \mbox{d} t\\
&+  \sum_{i=1}^2 \int_{\Omega}    \Big(\frac{r}{2} \vert \phi_i(1,x)+c_i(x)\vert^2 \mbox{d} x - (\phi_i(1,x)+c_i(x)) \tilm_i(x)  \Big)\mbox{d}x \\
&+h\E^*\Big(\frac{c_1}{h}, \frac{c_2}{h}\Big),
\end{split}\]
where $\E^*$ is the Legendre tranform of $\E$ extended by $+\infty$ on $(-\infty,0]$. A saddle point of $L_r$ satisfies $\mu_i(1,\cdot)=\tilde{\mu}_i$ and the solution to one JKO step is $\rho_{i,h}^{k+1}=\tilde{\mu}_i$. Then, we use the augmented Lagrangian algorithm, ALG2-JKO, introduced in \cite{BCL} to compute numerically $(\rho_{1,h}^{k+1},\rho_{2,h}^{k+1})$ and we refer to \cite{BCL} for a detailed exposition.

Figure \ref{figure porous media} represents two populations crossing each other subject to porous media congestion with $\alpha_1=\alpha_2=1$ and $m=50$. Initial conditions are given by 
$$ \rho_{1,0} = \mathds{1}_{[-0.45,-0.15]^2} \text{ and } \rho_{2,0} = \mathds{1}_{[0.15,0.45]^2}.$$
The motion is imposed by potentials $V_1(x,y)= 4\| (x,y) - (0.3, 0.3) \|^2$ and $V_2(x,y)= 4\| (x,y) + (0.3, 0.3) \|^2$.
We remark that the two populations have the same behaviour and when they cross each other, the density has to spread. In Figure \ref{figure porous media weight}, we study the same behaviour but subject to the porous medium constraint on $\rho_1+2\rho_2$. We can see that the population where the constraint plays a higher role, $\rho_2$, has to deviate in order to let pass $\rho_1$ through.
Although the theory is not fully understood for system \eqref{eq:porous without diffusion} (see discusions in \cite{KM}), we notice that in Figures \ref{figure porous media} and \ref{figure porous media weight}, it seems that the unique discrete solutions behave numerically stable.\\

 In the two populations crowd motion model with linear diffusion, we saw that we can find a solution as the gradient flow of 
 $$ \E(\rho_1,\rho_2):= \int_\Omega (V_1 +\eps\log(\rho_1))\rho_1 + \int_\Omega (V_2+\eps\log(\rho_2))\rho_2 + \F_\infty(\alpha_1 \rho_1 +\alpha_2 \rho_2).$$
In this context, we use the same initial datas and potentials as previously. The small parameter $\eps=0.01$ in the simulations is taken to reduce the effect of the diffusion.
In Figure \ref{figure crowd motion 2}, we see two populations which cross each other. When they start to cross each other at time $t=0.05$, we remark that the density of $\rho_1$ and $\rho_2$ decrease and the sum is saturated. In this situation, individuals of both populations take the same space. 

Now assume that an individual of the second population takes twice the space than an individual of the first population. Then if we study the one population model (without interaction), populations $\rho_1$ and $\rho_2$ are subject to constraints $ \rho_1(x) \leqslant 1$ and $\rho_2(x) \leqslant \frac{1}{2}$. In our case, where populations interact each other,  $\rho_1$ and $\rho_2$ are subject to the common constraint $\rho_1(x) +2\rho_2(x) \leqslant 1$. Notice that when $\rho_1(x) =0$ or $\rho_2(x) =0$, we recover the expected behaviour, $ \rho_2(x) \leqslant \frac{1}{2}$ and $\rho_1(x) \leqslant 1$. In Figure \ref{figure crowd motion 2 weight}, we represent two populations crossing each other subject to this constraint. Immediately, the second population sprawls to saturate the constraints $\rho_2(x) \leqslant\frac{ 1}{2}$ and then when they start crossing the density of $\rho_1$ and $\rho_2$ decrease and we have $\rho_1(x) +2\rho_2(x)=1$.

 In Figures \ref{figure crowd motion obs} and \ref{figure crowd motion obs weight}, the same situations as in Figures \ref{figure crowd motion 2} and \ref{figure crowd motion 2 weight} are presented adding an obstacle in the middle of $\Omega$. This can be done using a potential with very high value in this area.

 \begin{figure}[h!]

\begin{tabular}{@{\hspace{0mm}}c@{\hspace{1mm}}c@{\hspace{1mm}}c@{\hspace{1mm}}c@{\hspace{1mm}}c@{\hspace{1mm}}c@{\hspace{1mm}}}

\centering
\includegraphics[ scale=0.175]{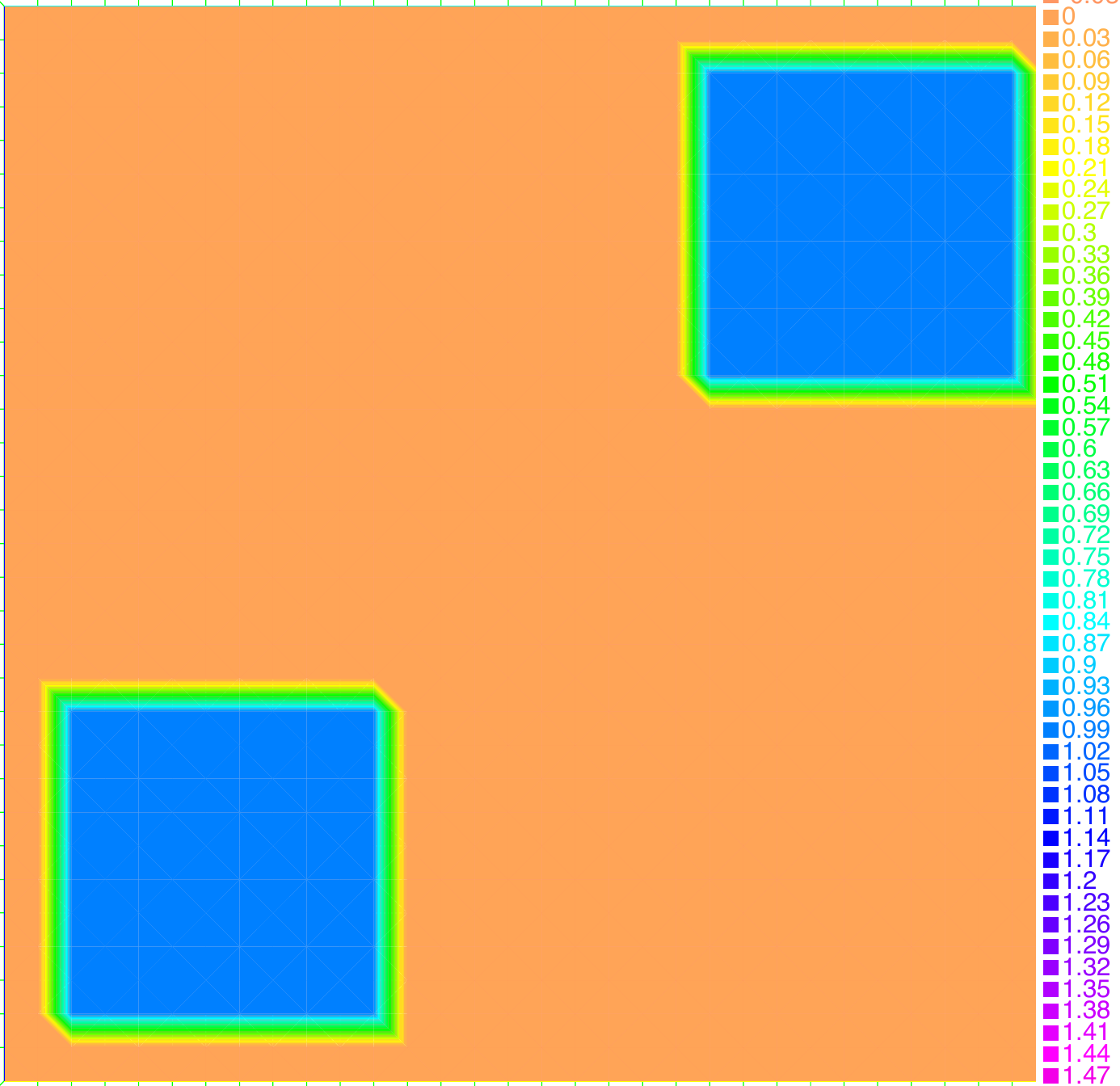}&
\includegraphics[ scale=0.175]{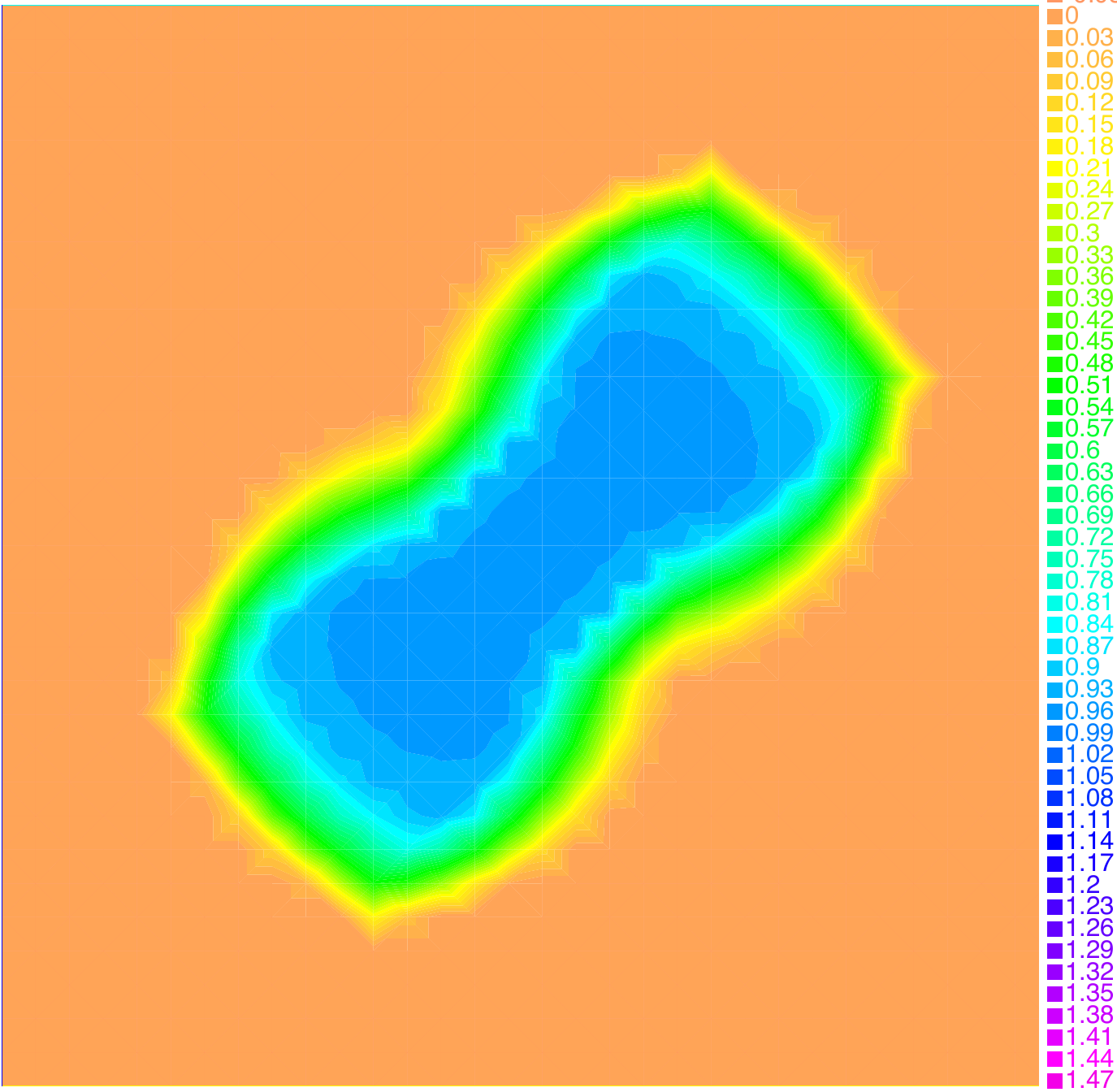}&
\includegraphics[ scale=0.175]{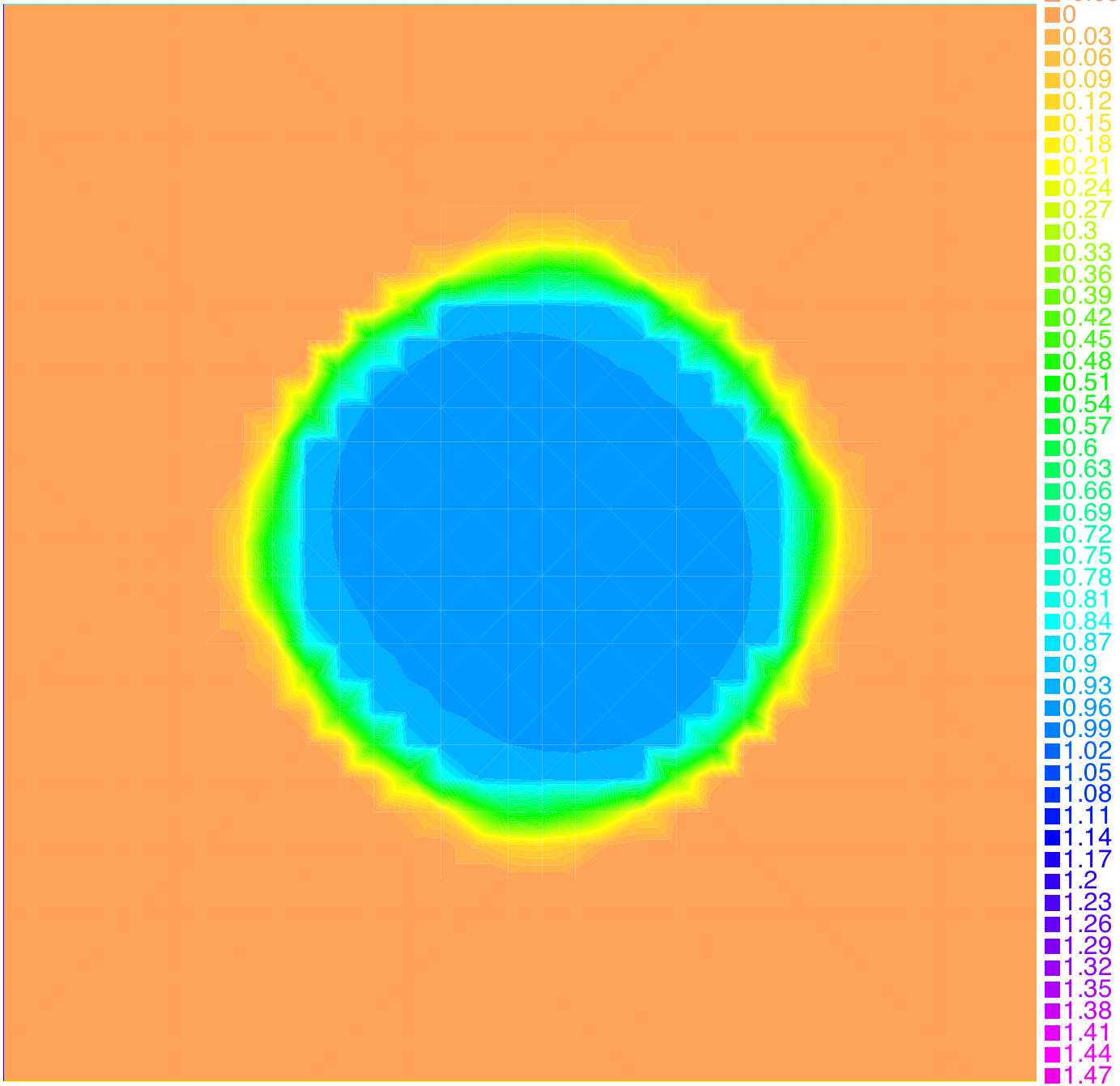}&
\includegraphics[ scale=0.175]{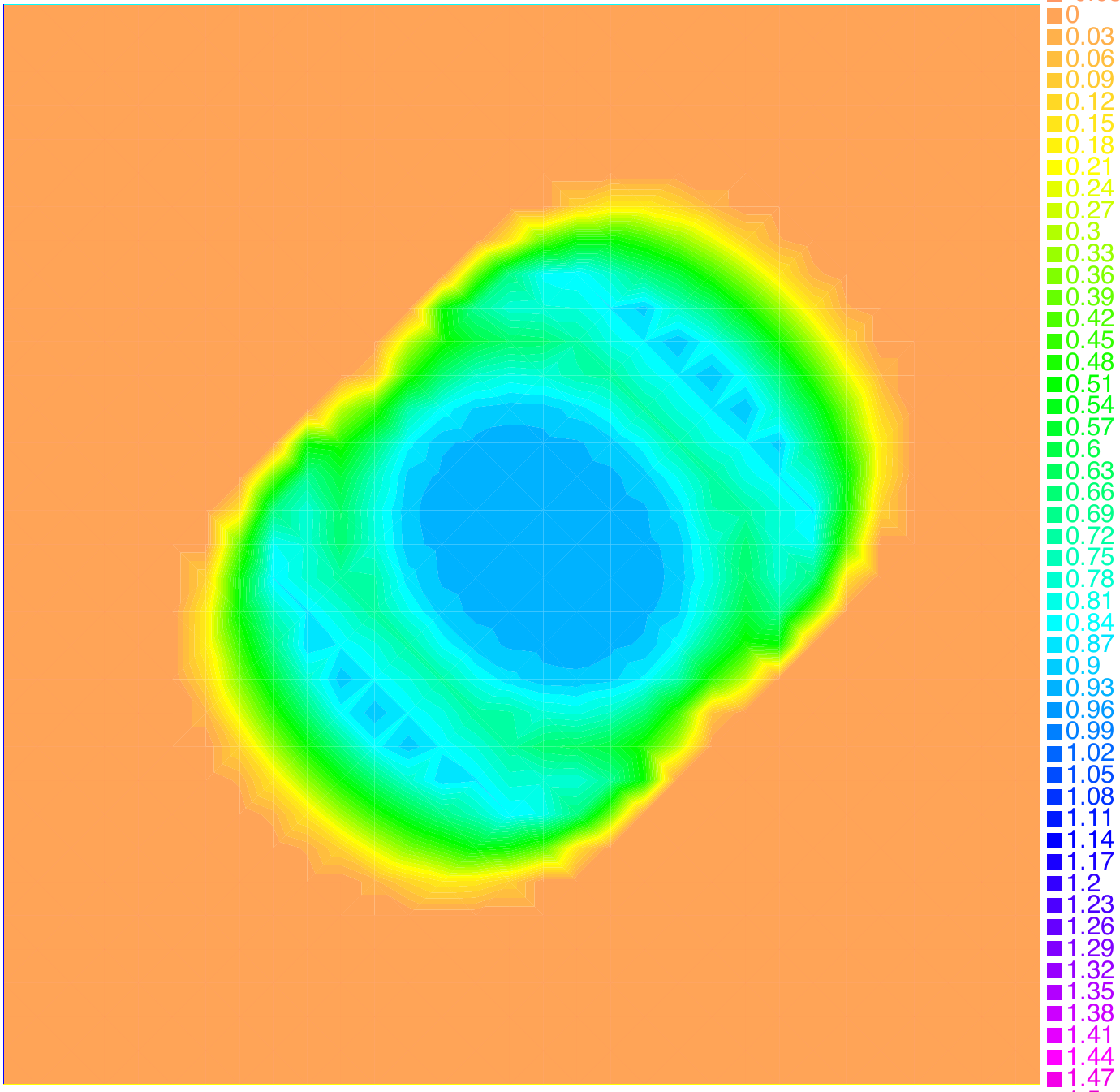}&
\includegraphics[ scale=0.175]{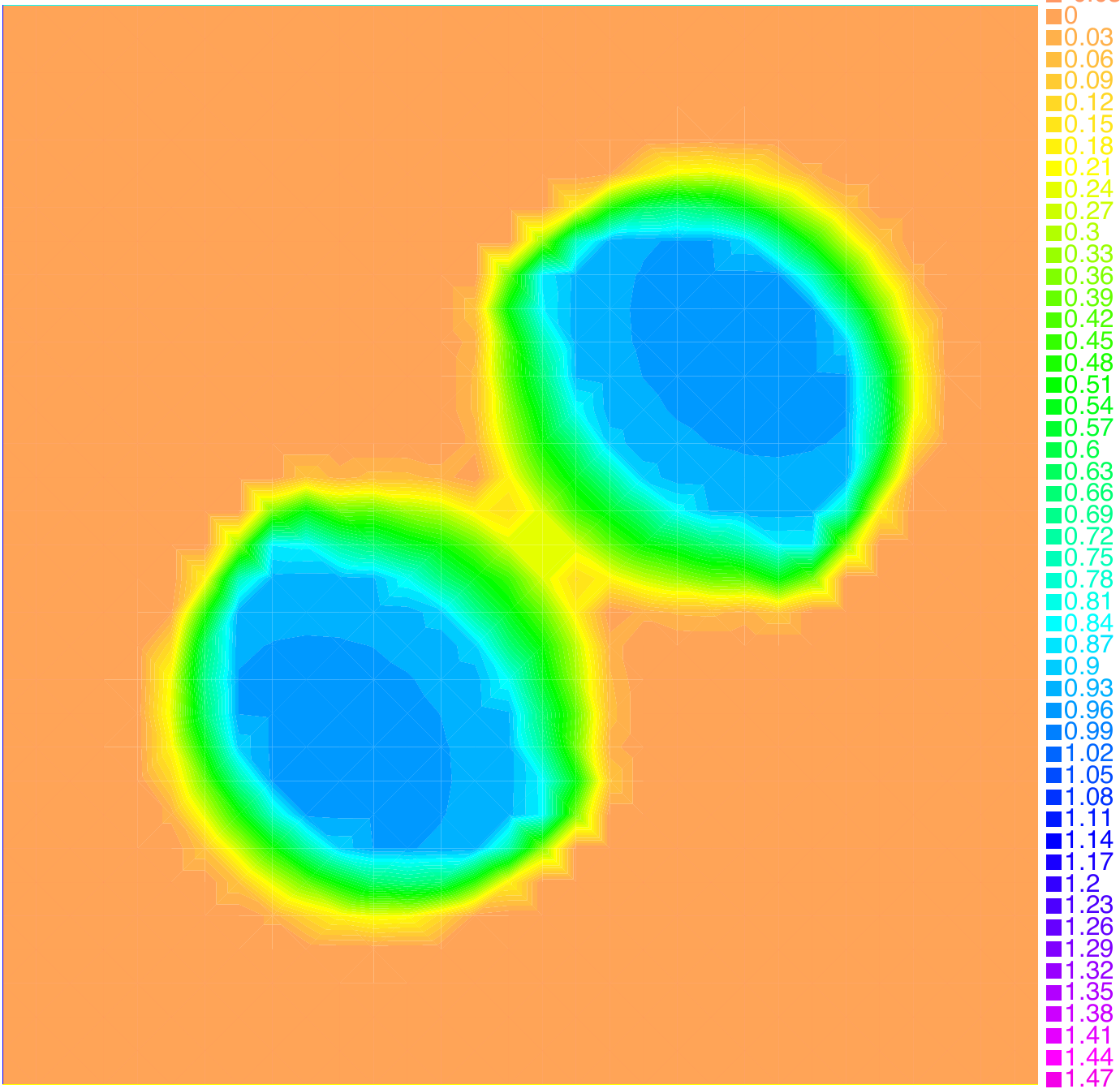}&
\includegraphics[ scale=0.175]{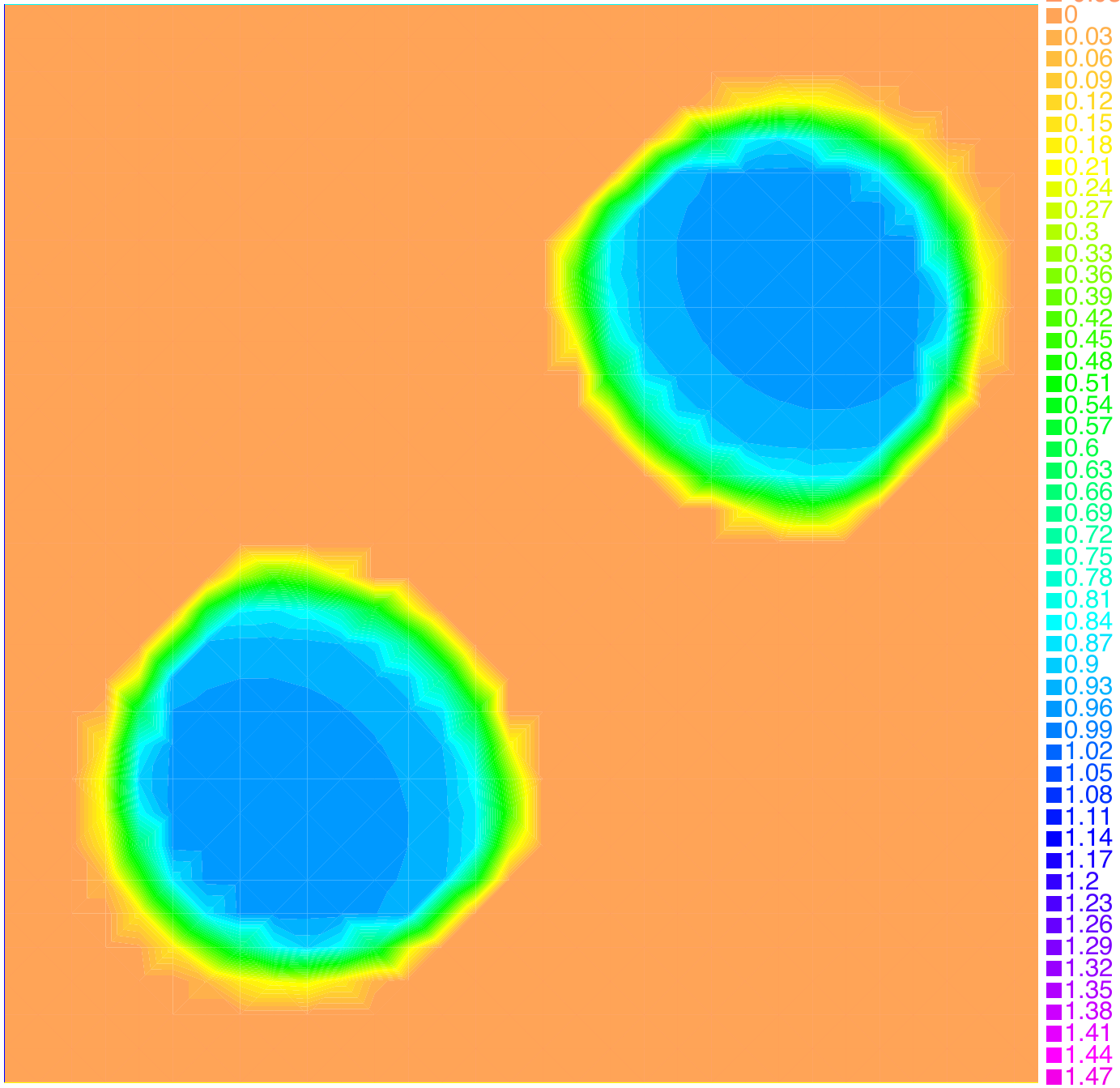}\\
\includegraphics[ scale=0.175]{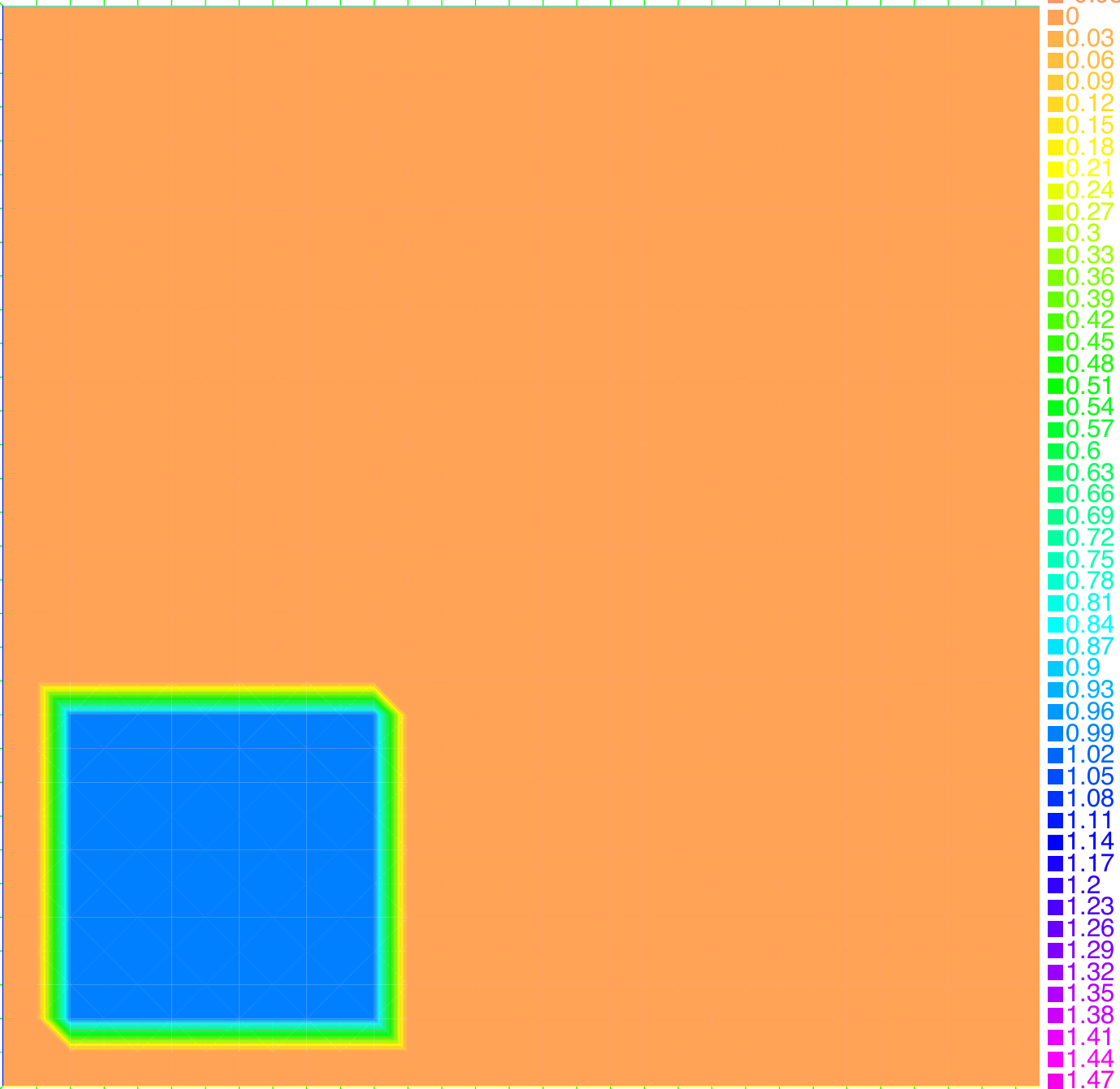}&
\includegraphics[ scale=0.175]{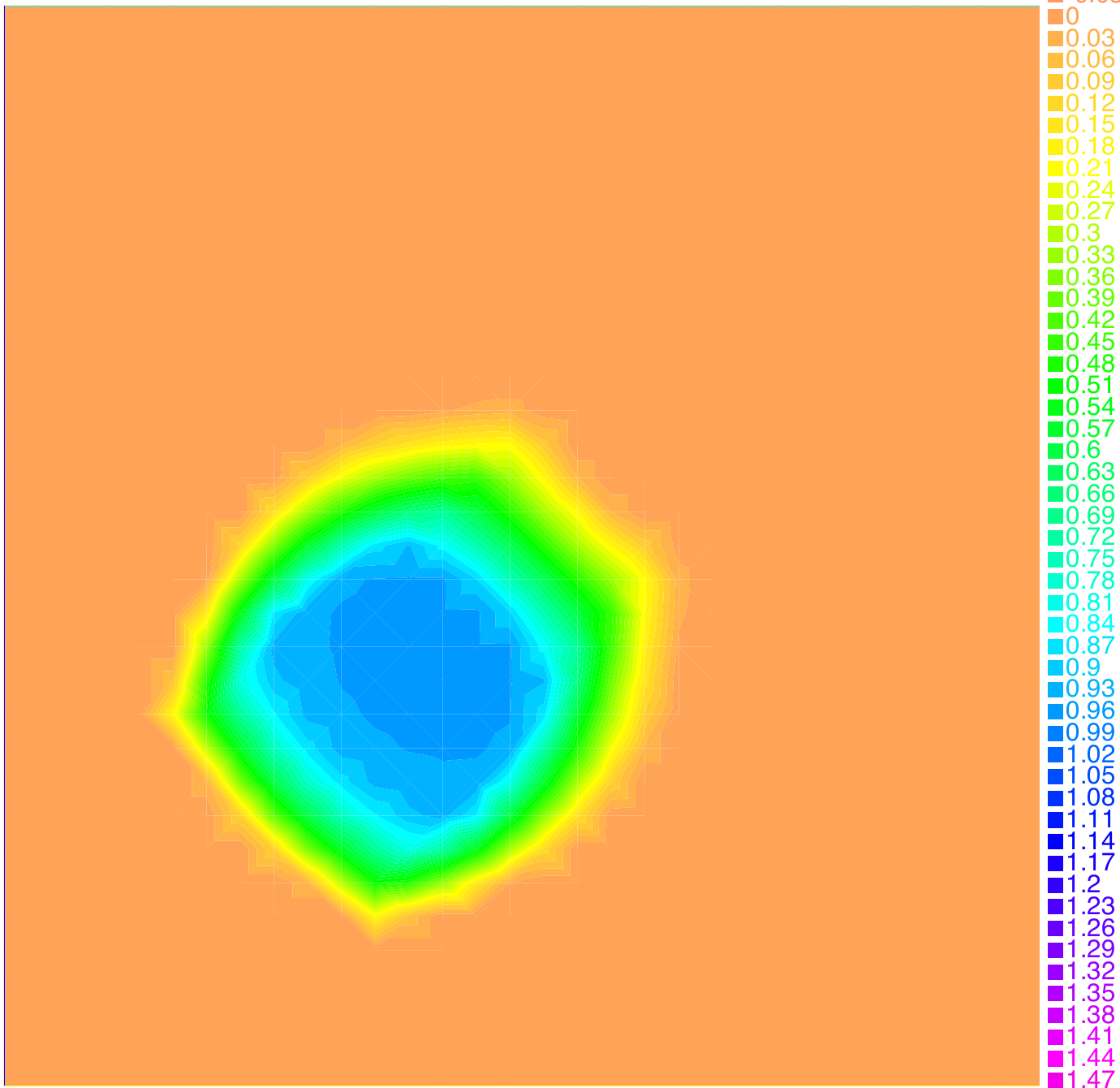}&
\includegraphics[ scale=0.175]{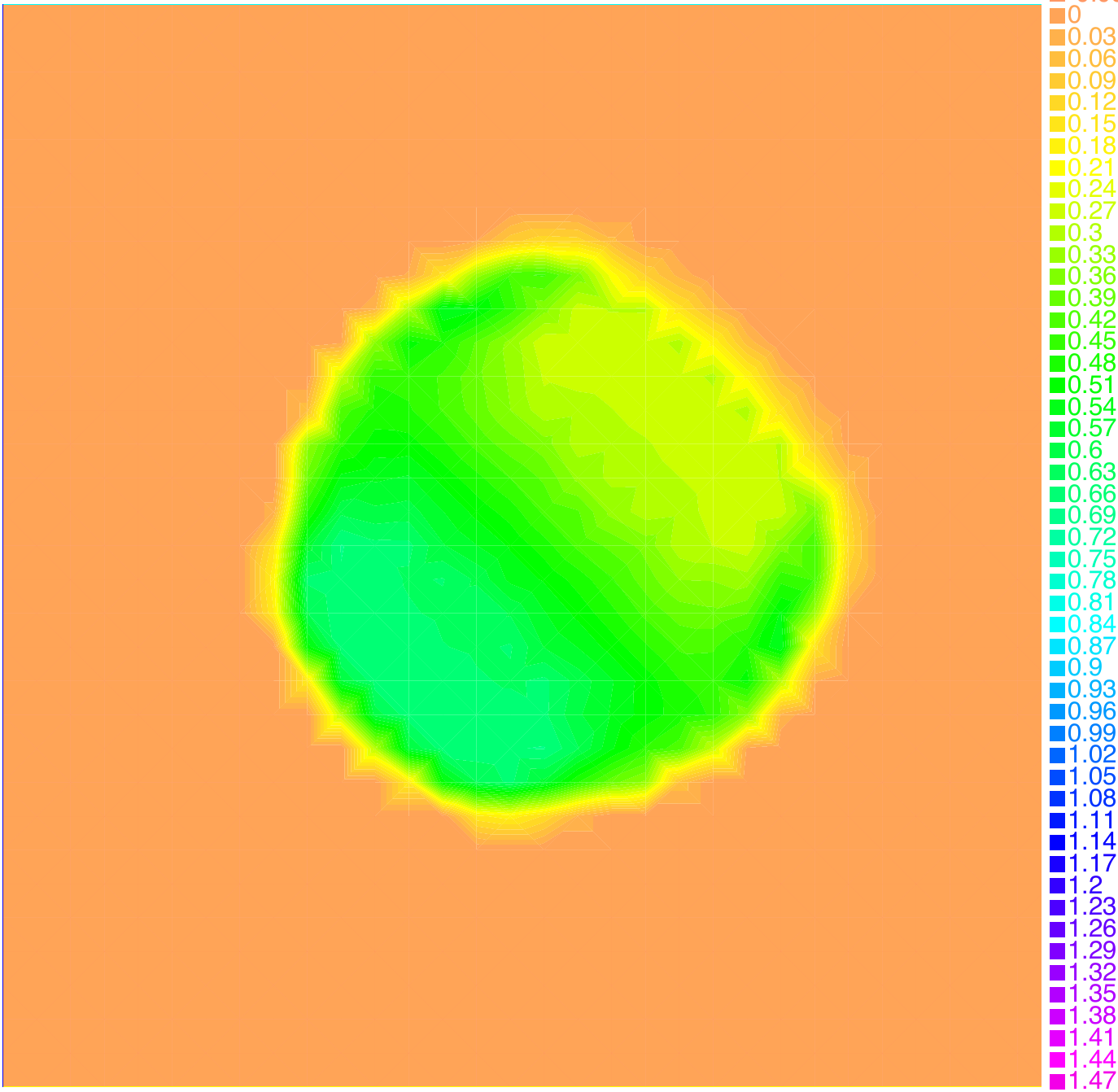}&
\includegraphics[ scale=0.175]{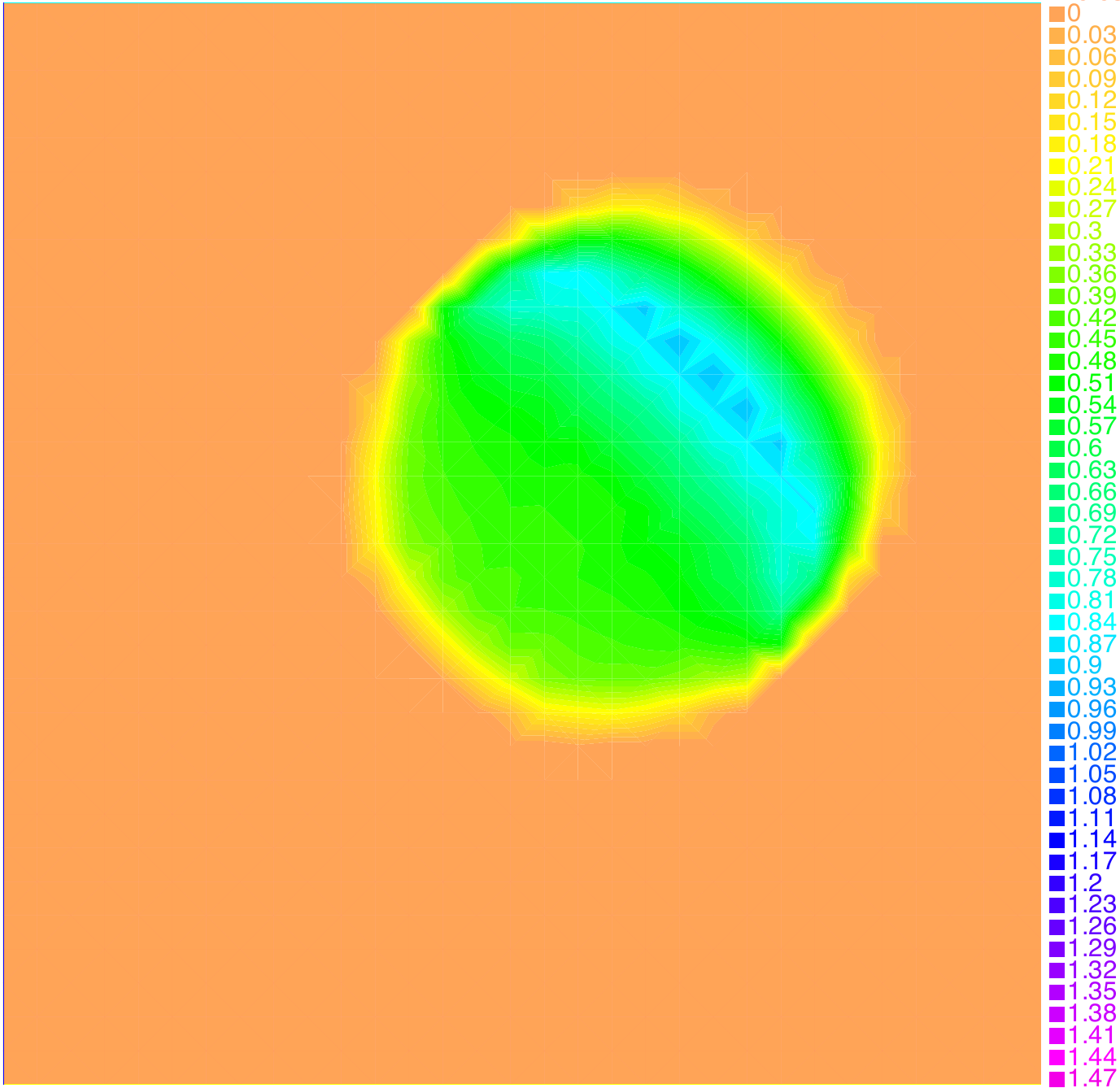}&
\includegraphics[ scale=0.175]{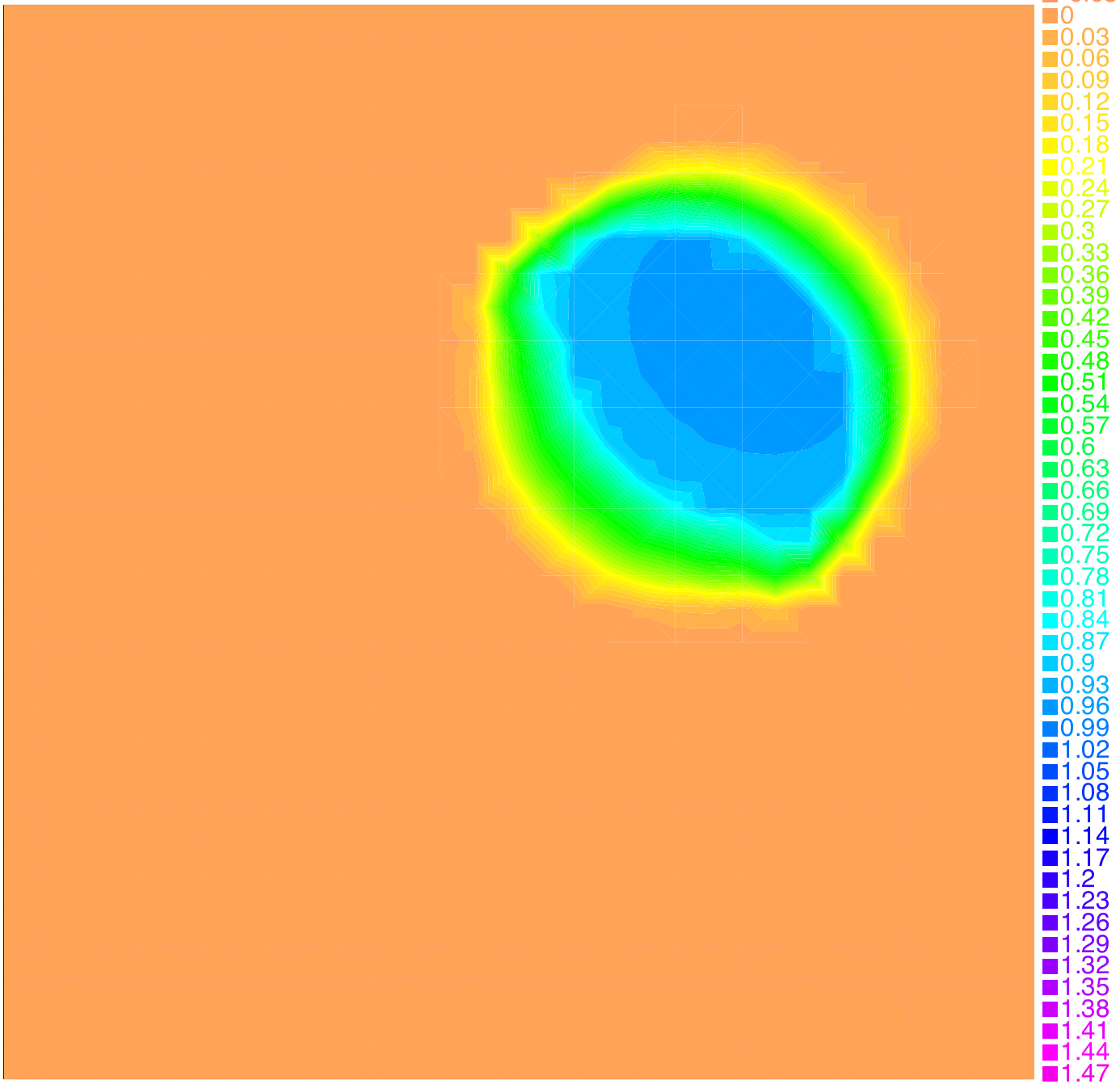}&
\includegraphics[ scale=0.175]{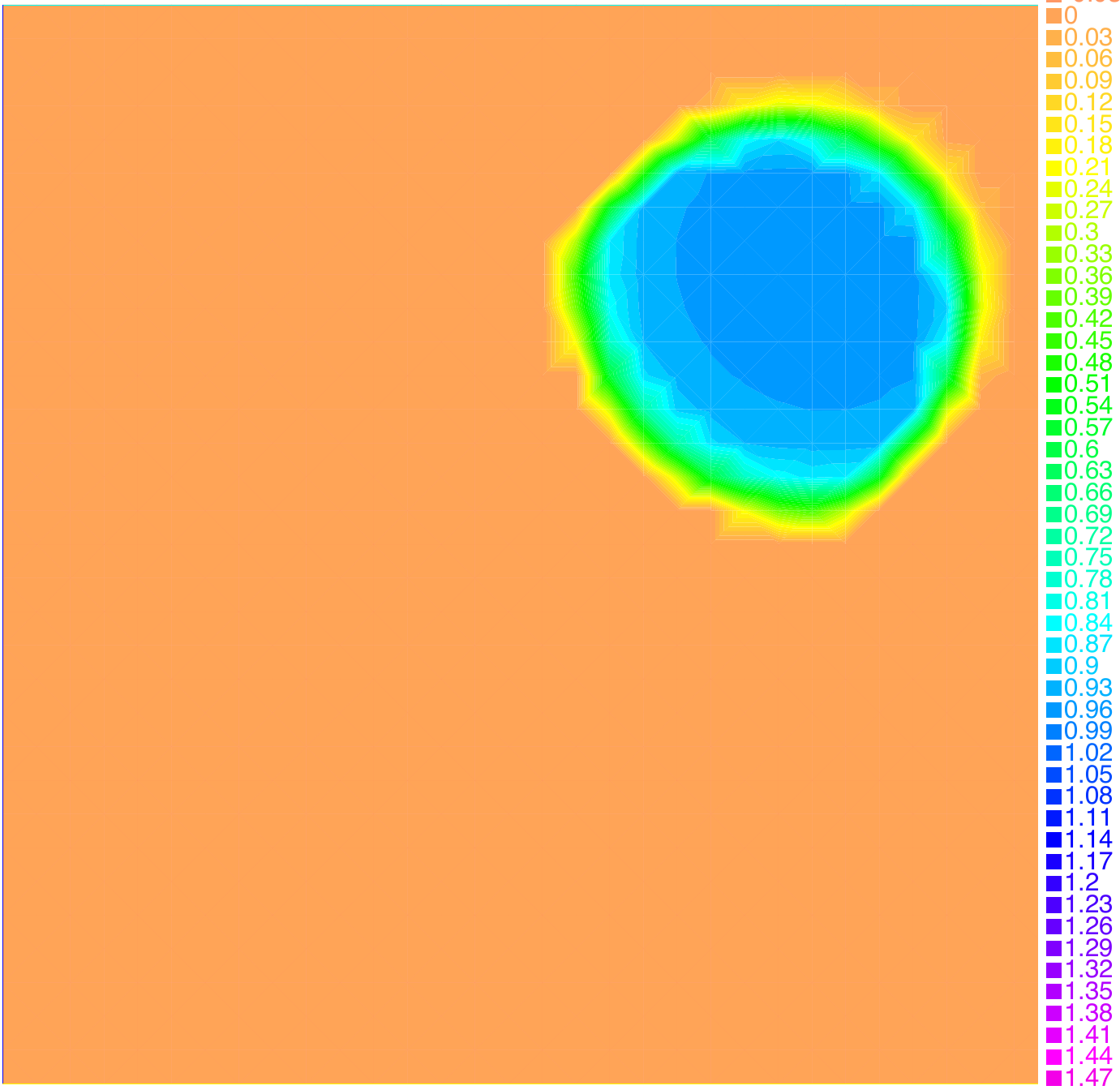}\\
$t=0$ & $t=0.05$ & $t=0.1$ & $t=0.15$ &$t=0.2$ & $t=0.3$\\
\end{tabular}
\caption{\textit{Evolution of two species crossing each other with porous media congestion, $m=50$. Top row: display of $\rho_1+\rho_2$. Bottom row: display of $\rho_1$.}}
\label{figure porous media}
\end{figure}

\begin{figure}[h!]

\begin{tabular}{@{\hspace{0mm}}c@{\hspace{1mm}}c@{\hspace{1mm}}c@{\hspace{1mm}}c@{\hspace{1mm}}c@{\hspace{1mm}}c@{\hspace{1mm}}}

\centering
\includegraphics[ scale=0.175]{c-dens-jko-mpw50-00.pdf}&
\includegraphics[ scale=0.175]{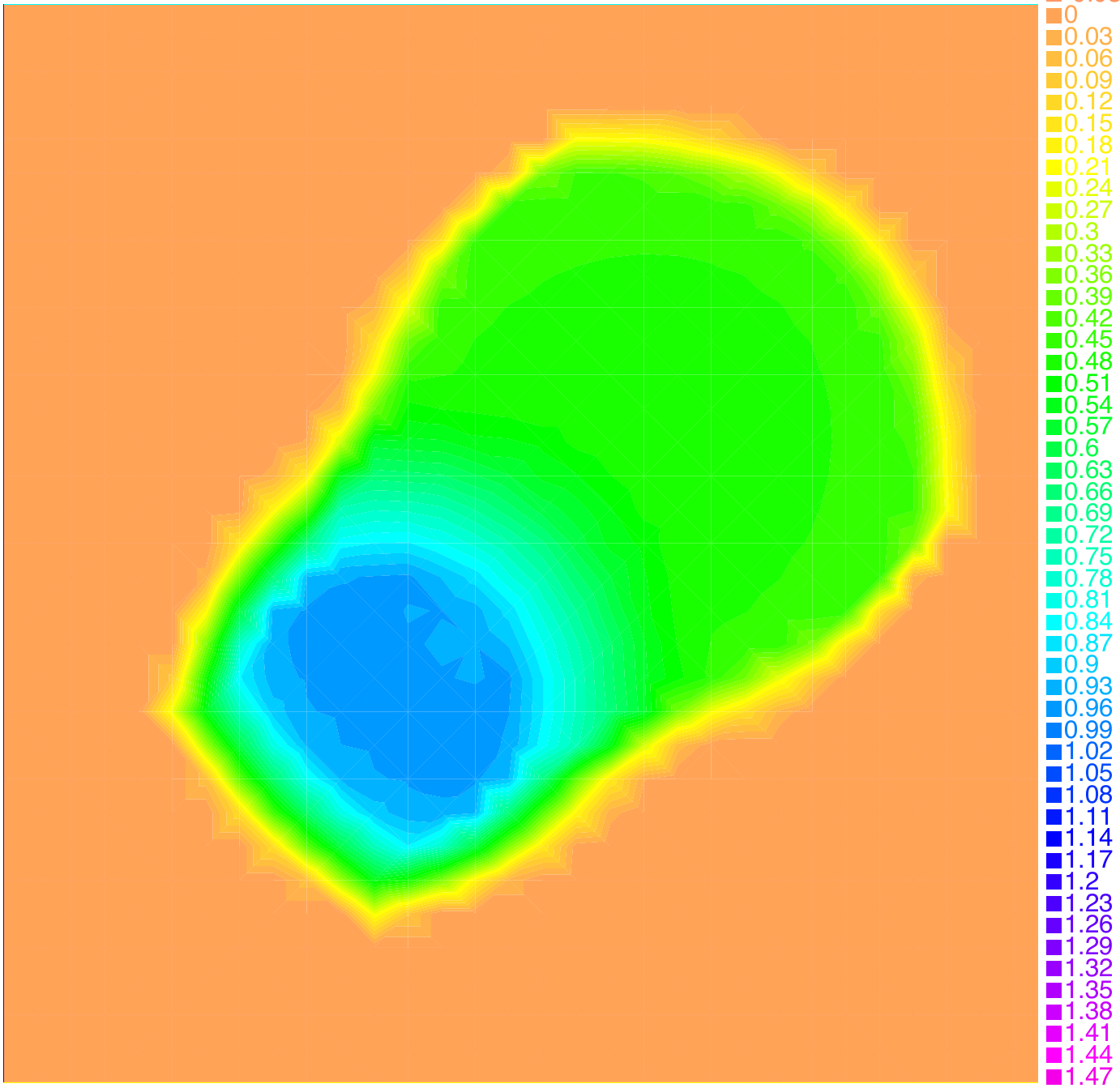}&
\includegraphics[ scale=0.175]{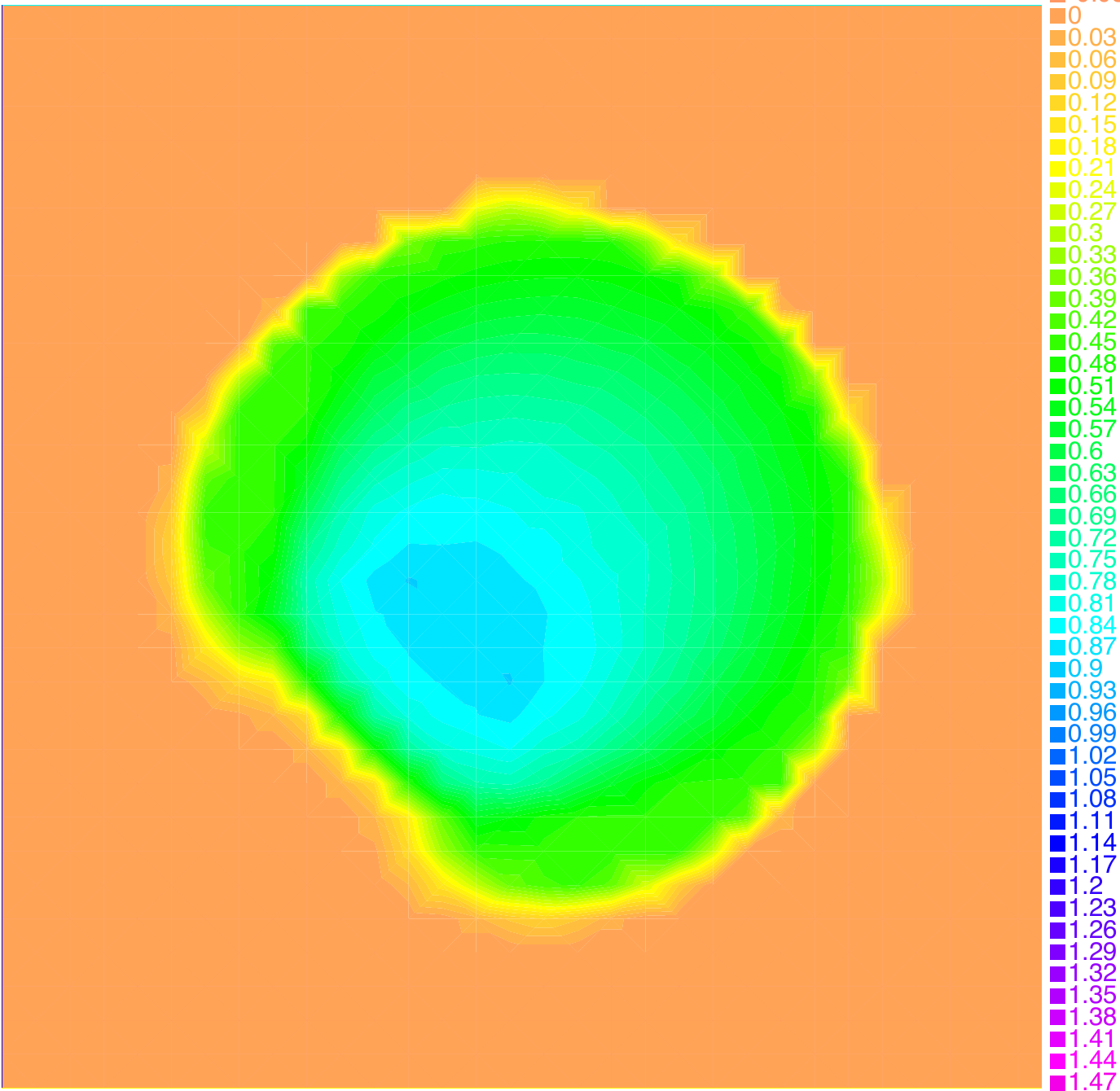}&
\includegraphics[ scale=0.175]{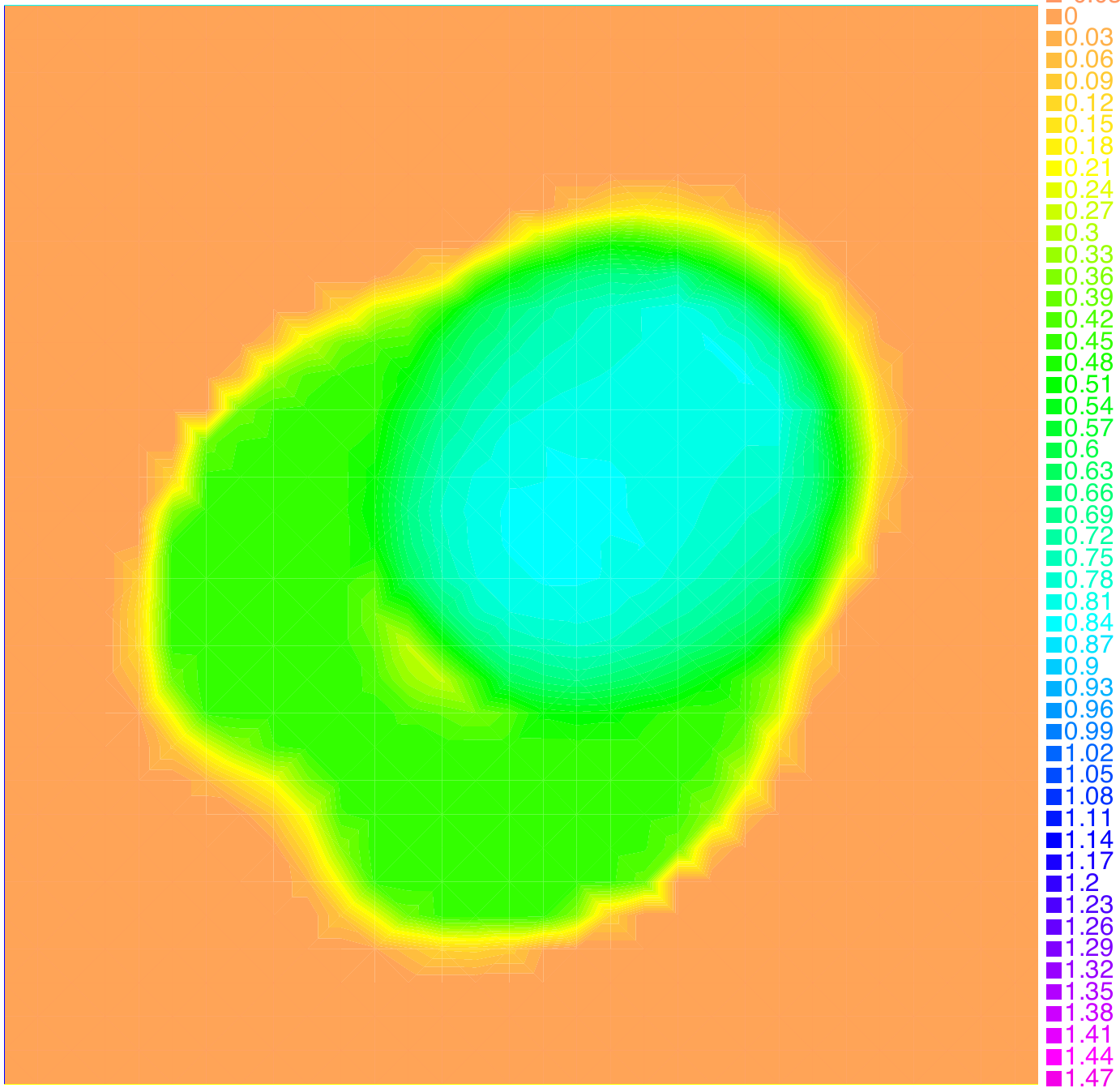}&
\includegraphics[ scale=0.175]{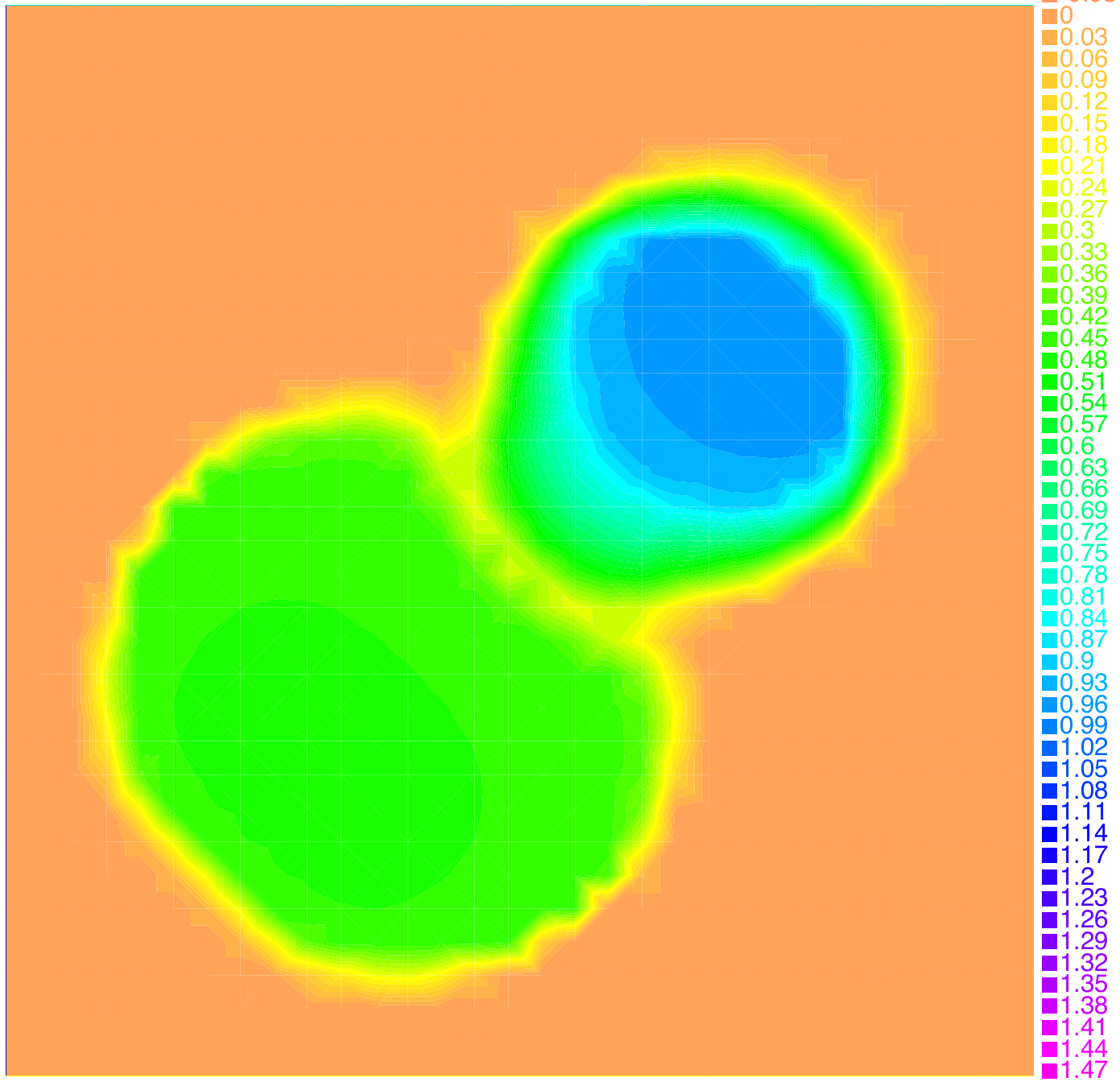}&
\includegraphics[ scale=0.175]{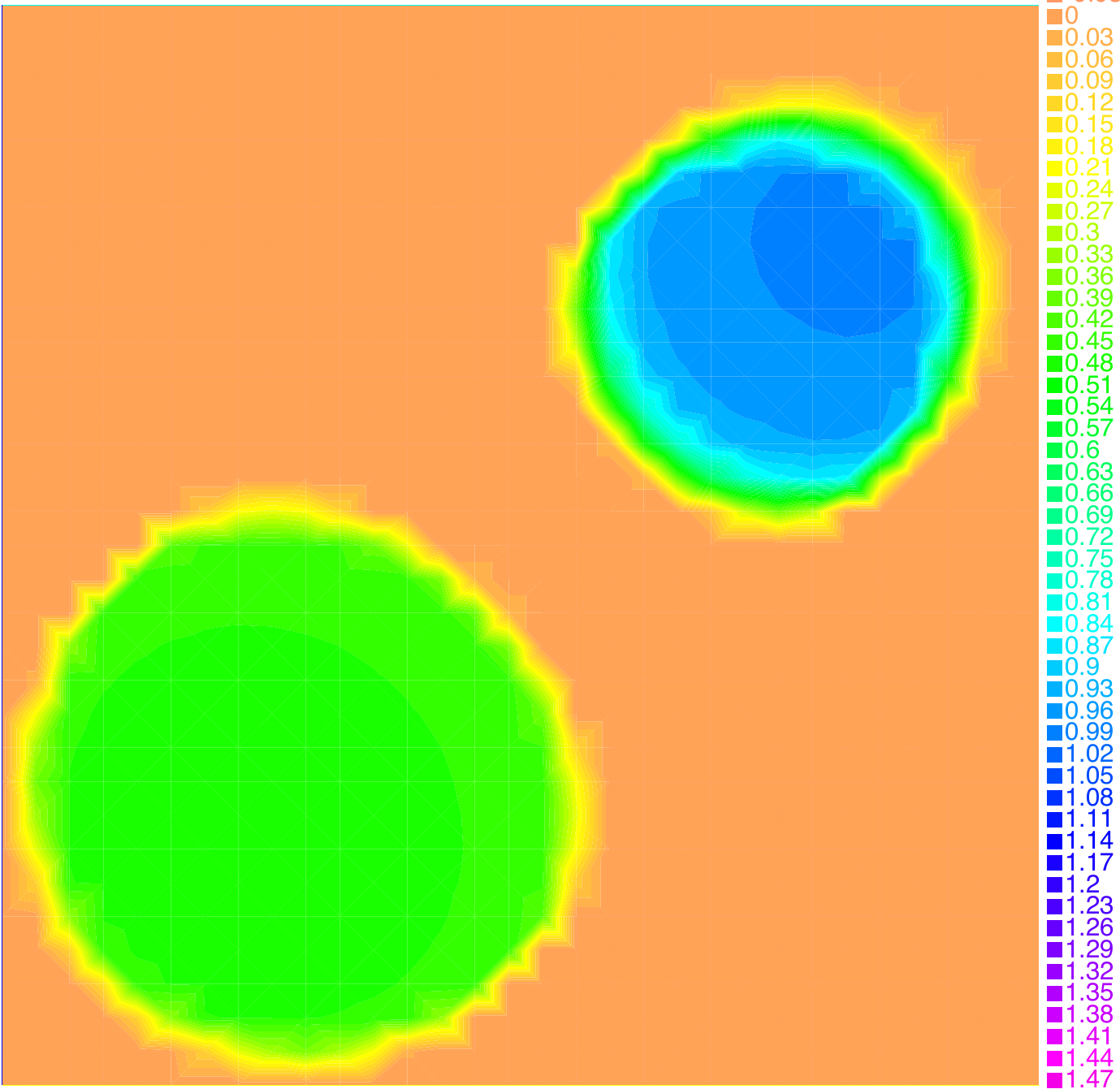}\\
\includegraphics[ scale=0.175]{a-dens-jko-mpw50-00.pdf}&
\includegraphics[ scale=0.175]{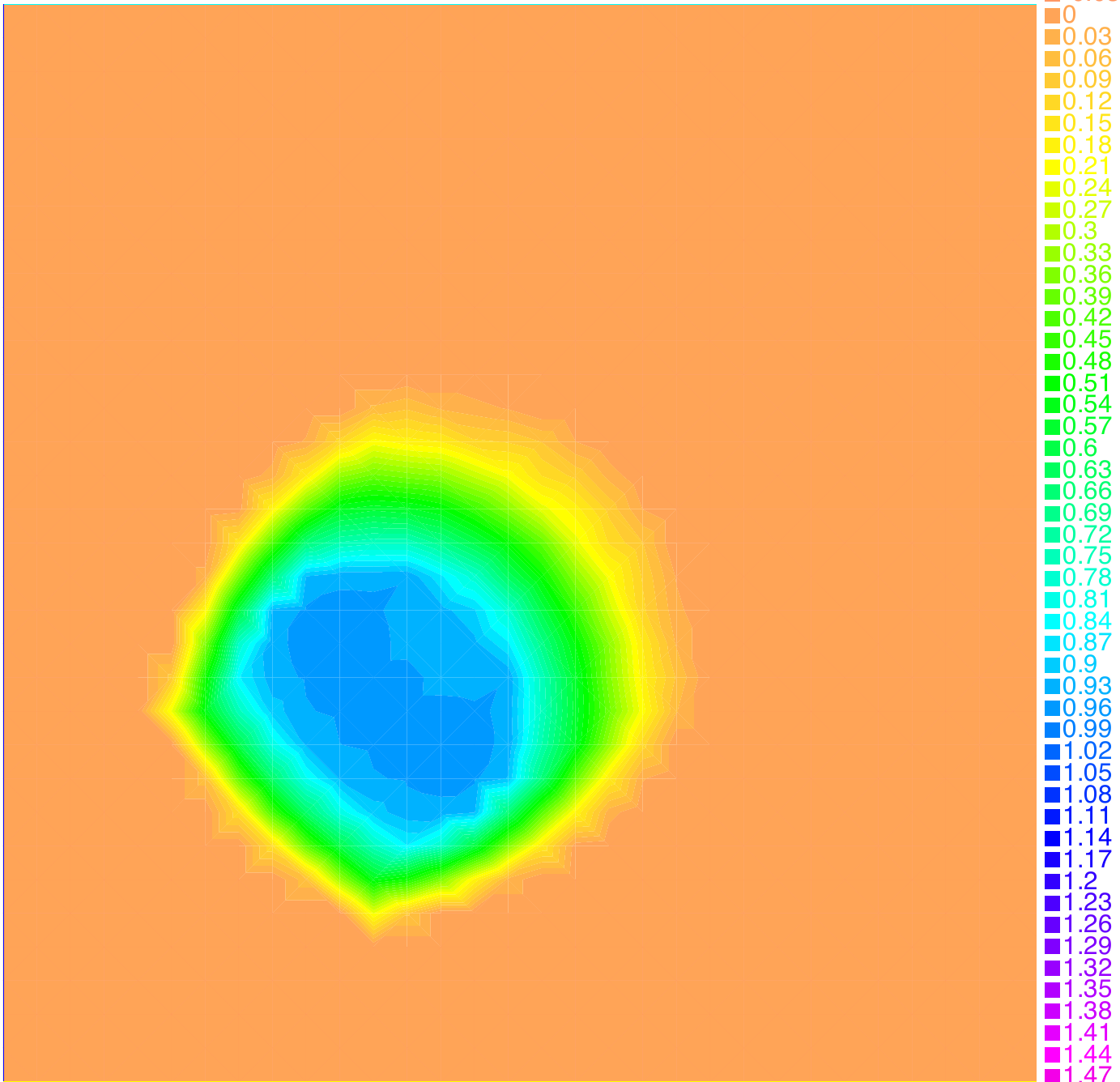}&
\includegraphics[ scale=0.175]{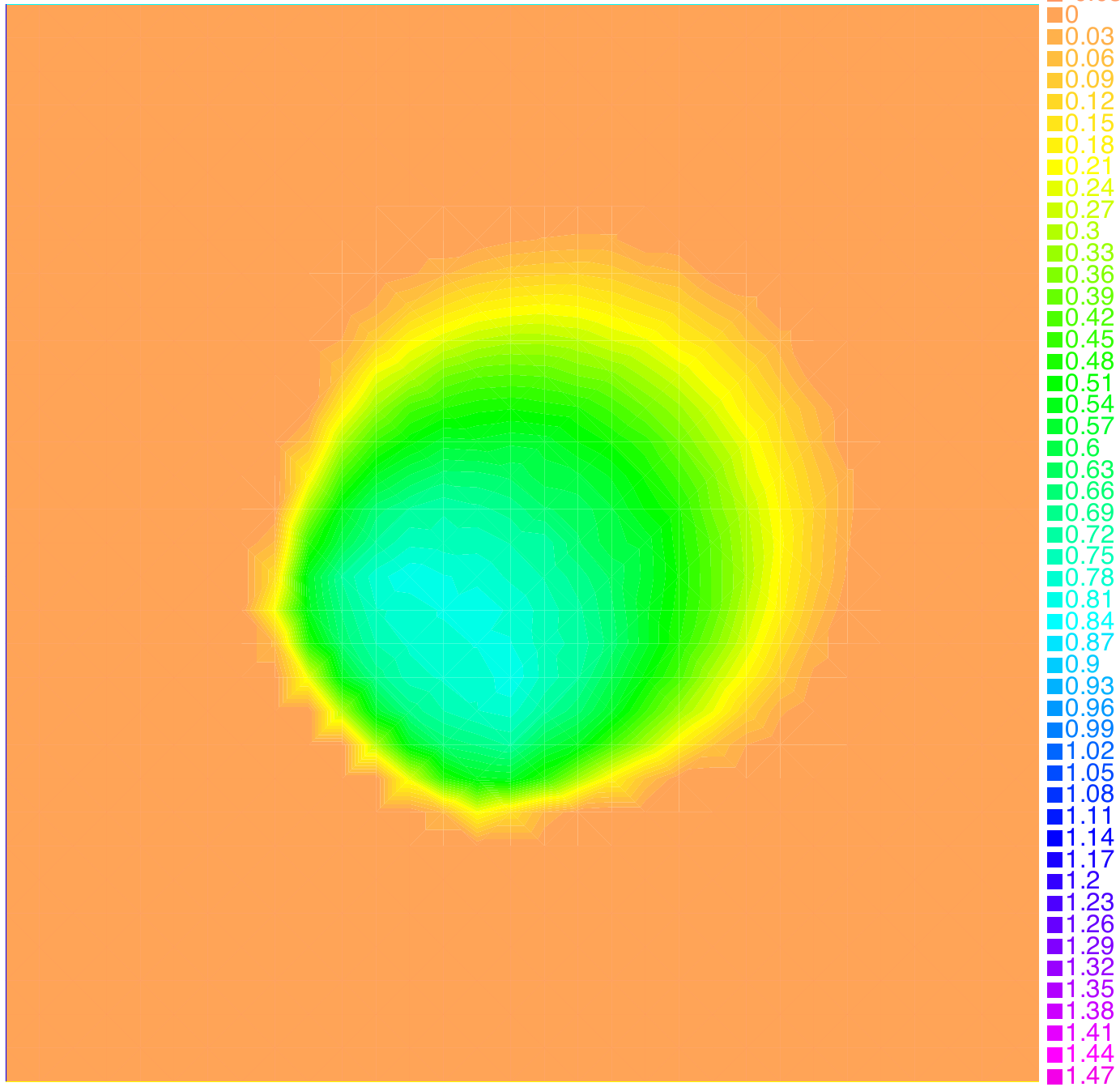}&
\includegraphics[ scale=0.175]{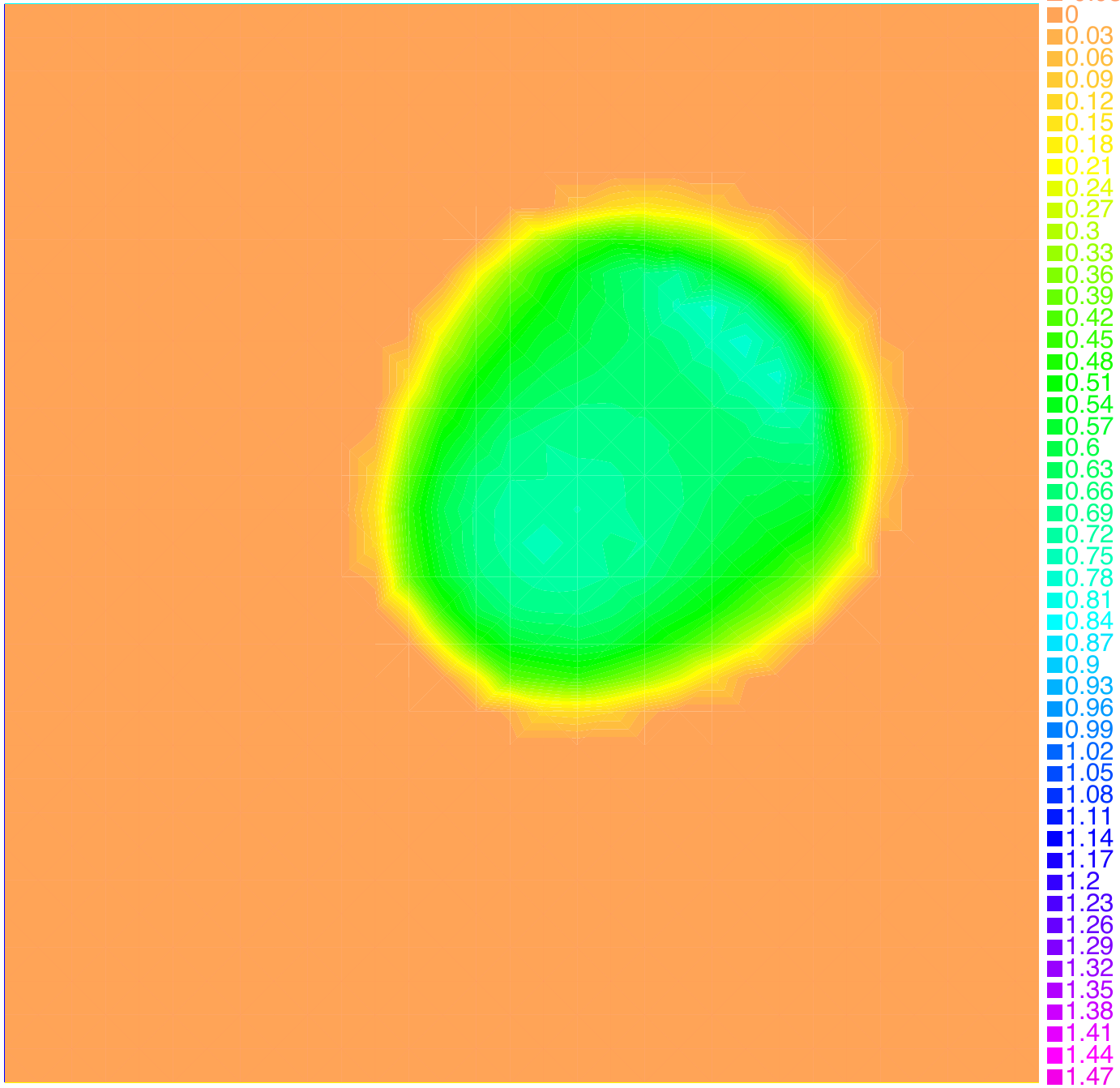}&
\includegraphics[ scale=0.175]{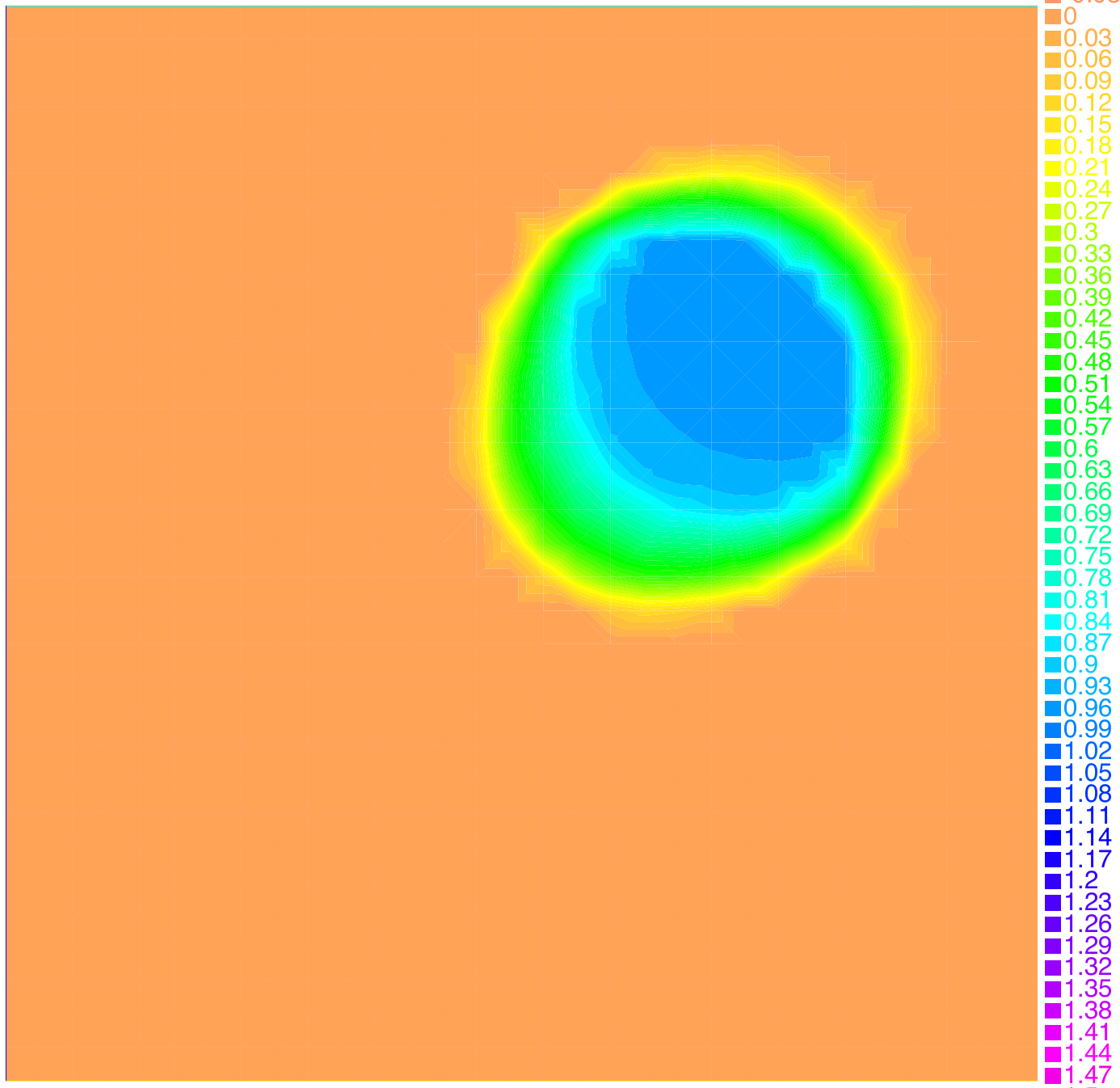}&
\includegraphics[ scale=0.175]{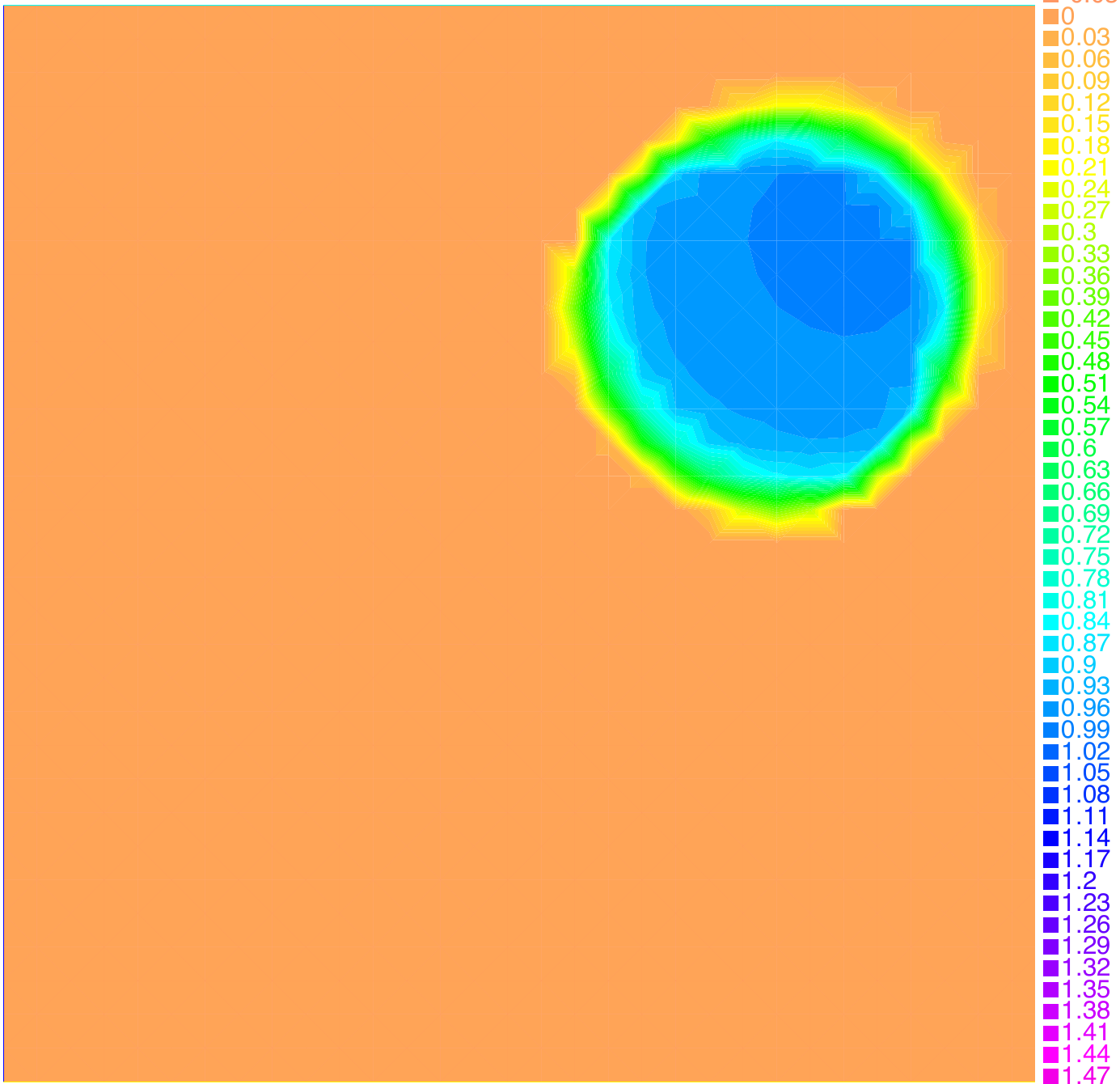}\\
\includegraphics[ scale=0.175]{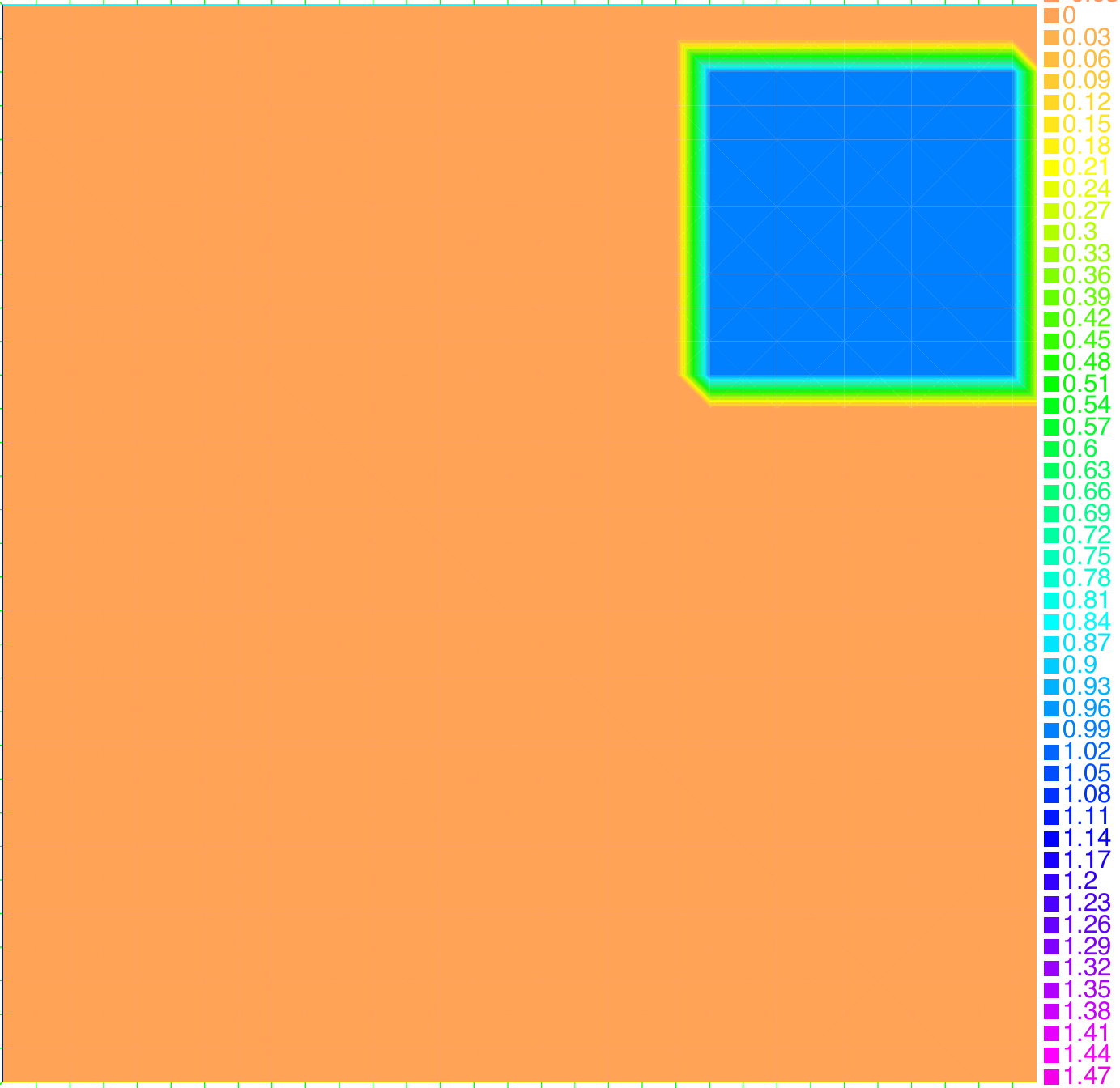}&
\includegraphics[ scale=0.175]{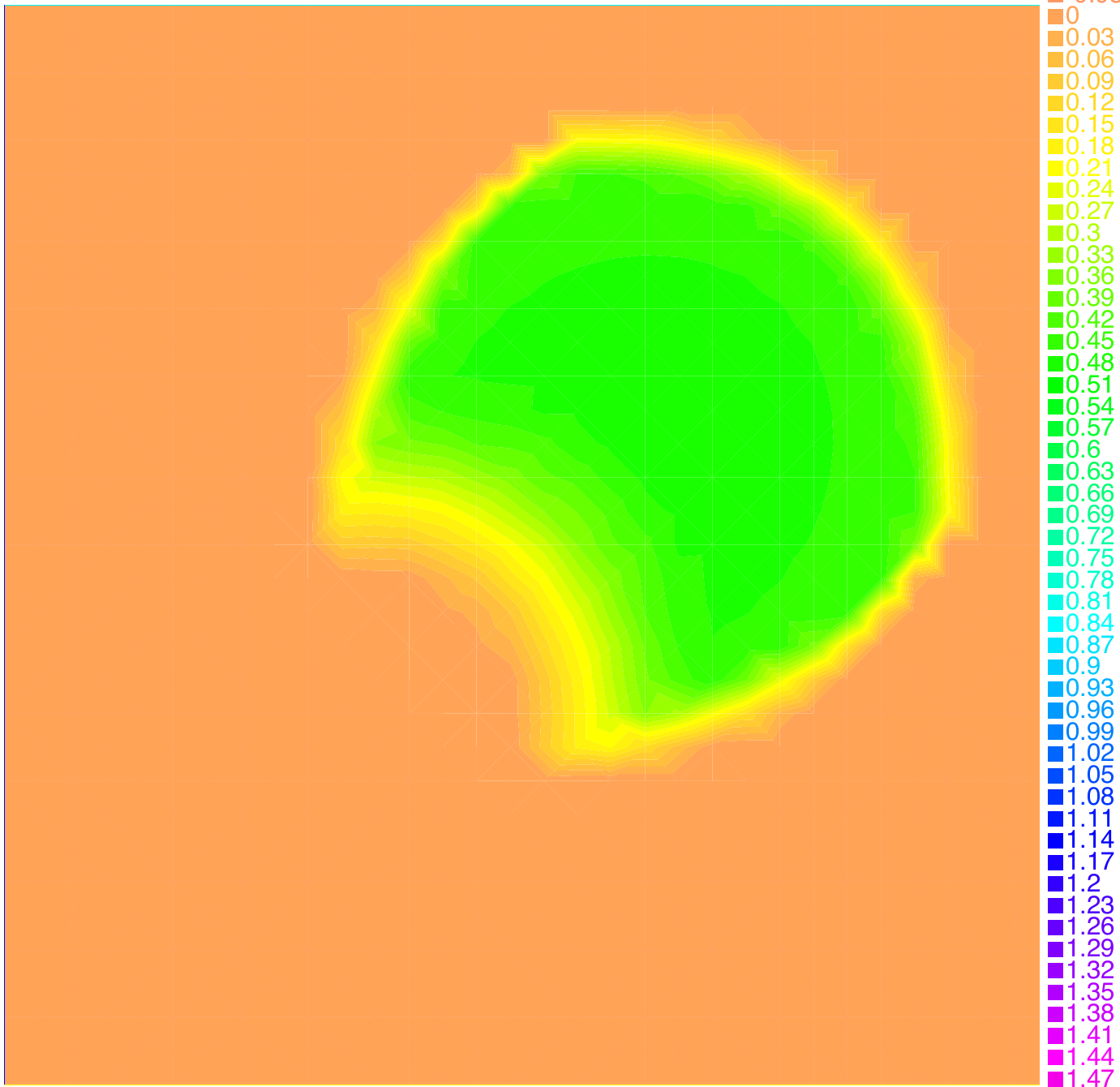}&
\includegraphics[ scale=0.175]{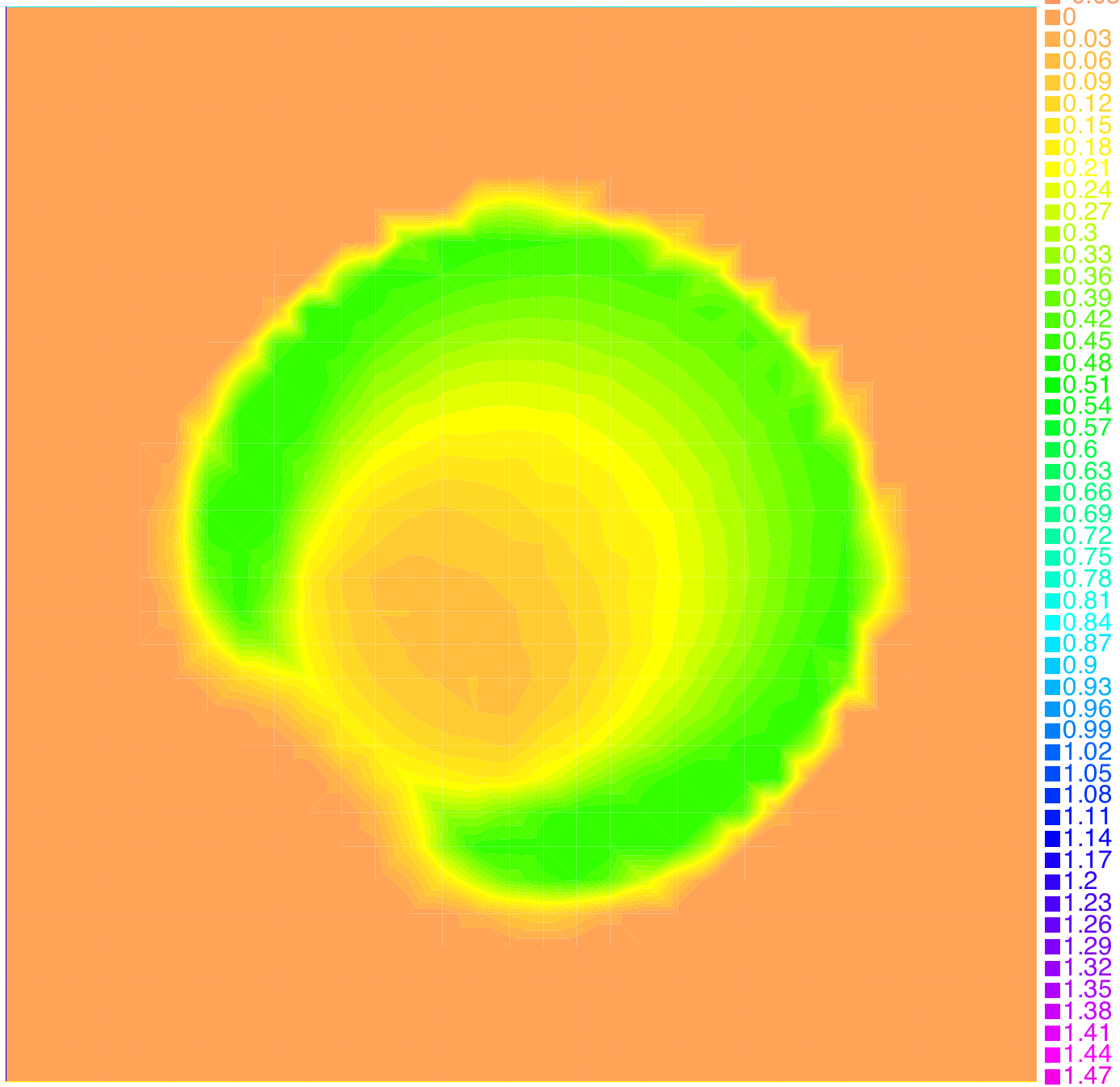}&
\includegraphics[ scale=0.175]{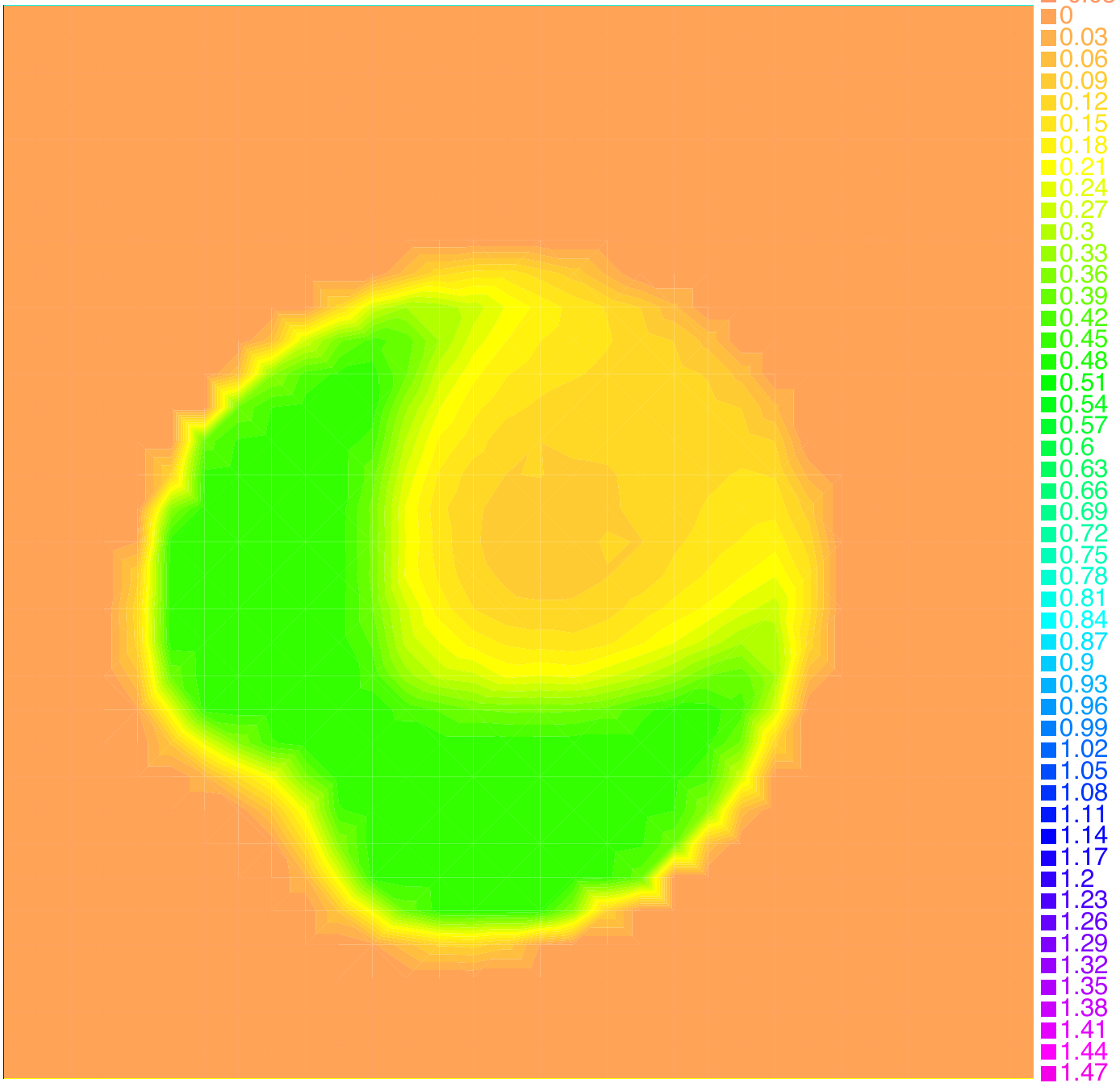}&
\includegraphics[ scale=0.175]{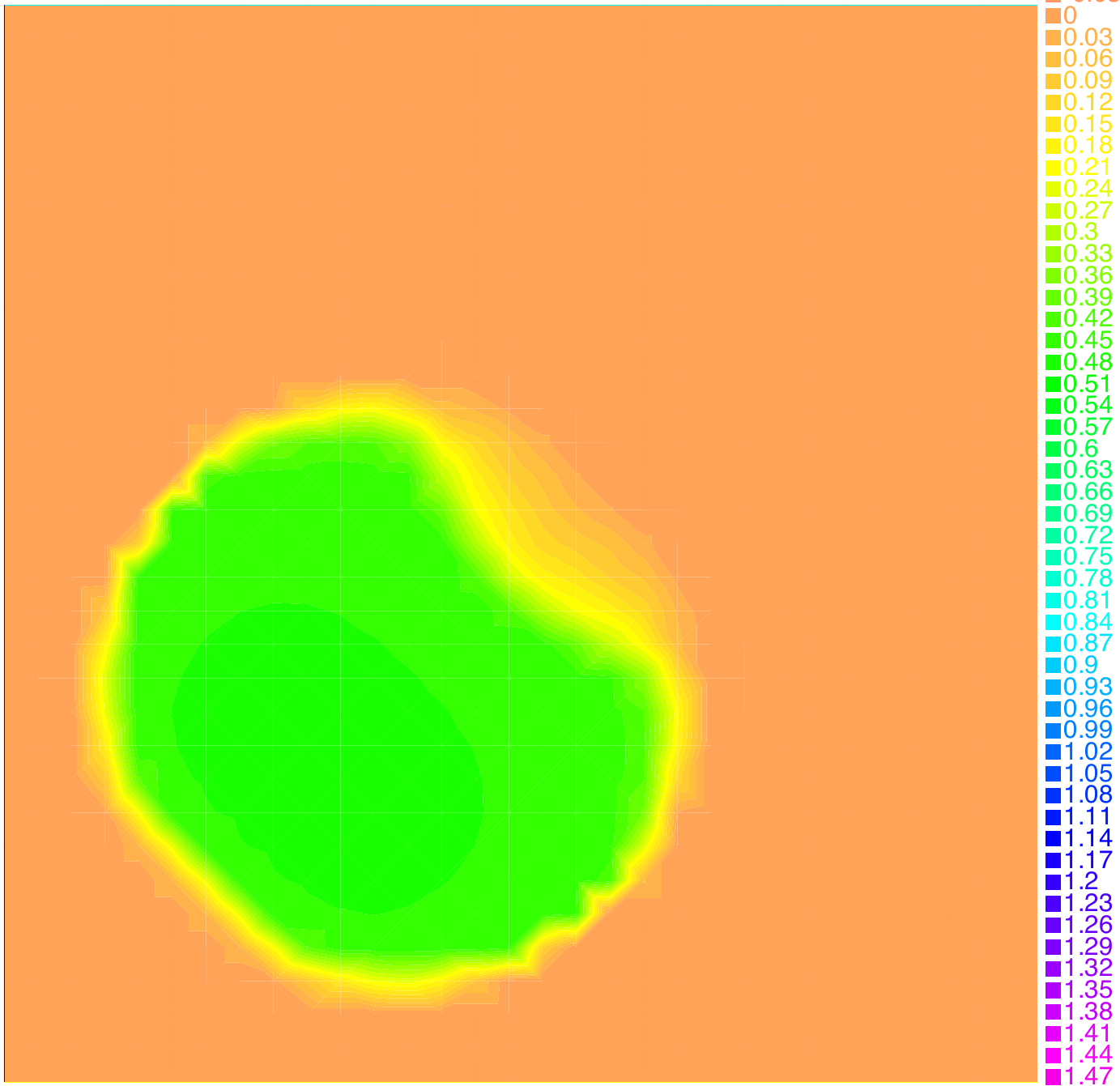}&
\includegraphics[ scale=0.175]{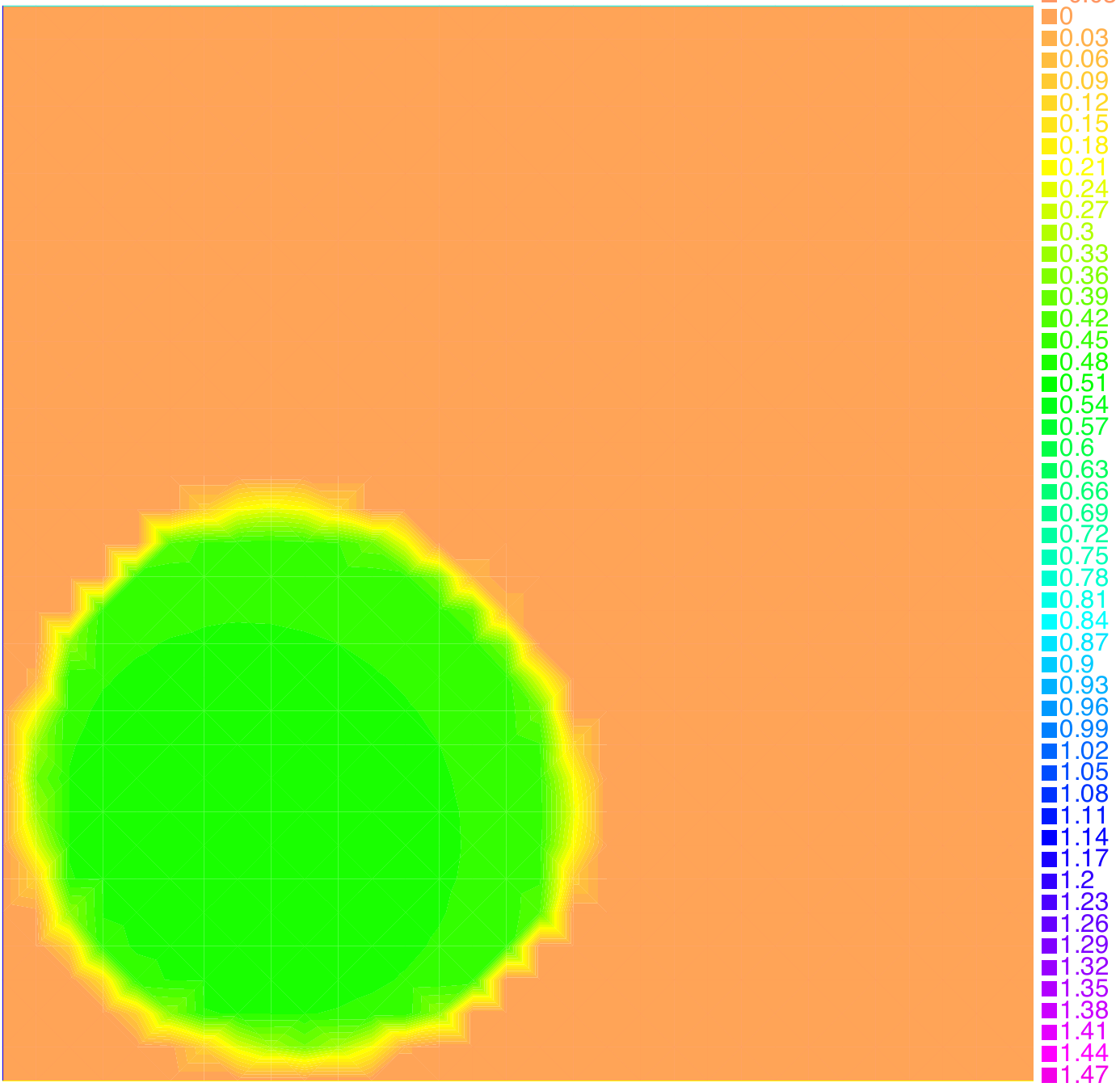}\\
$t=0$ & $t=0.05$ & $t=0.1$ & $t=0.15$ &$t=0.2$ & $t=0.3$\\
\end{tabular}
\caption{\textit{Evolution of two species crossing each other with weighted porous media congestion, $(\rho_1+2\rho_2)^m$, $m=50$. Top row: display of $\rho_1+\rho_2$. Middle row: display of $\rho_1$. Bottom row: display of $\rho_2$.}}
\label{figure porous media weight}
\end{figure} 

\begin{figure}[h!]

\begin{tabular}{@{\hspace{0mm}}c@{\hspace{1mm}}c@{\hspace{1mm}}c@{\hspace{1mm}}c@{\hspace{1mm}}c@{\hspace{1mm}}c@{\hspace{1mm}}}

\centering
\includegraphics[ scale=0.175]{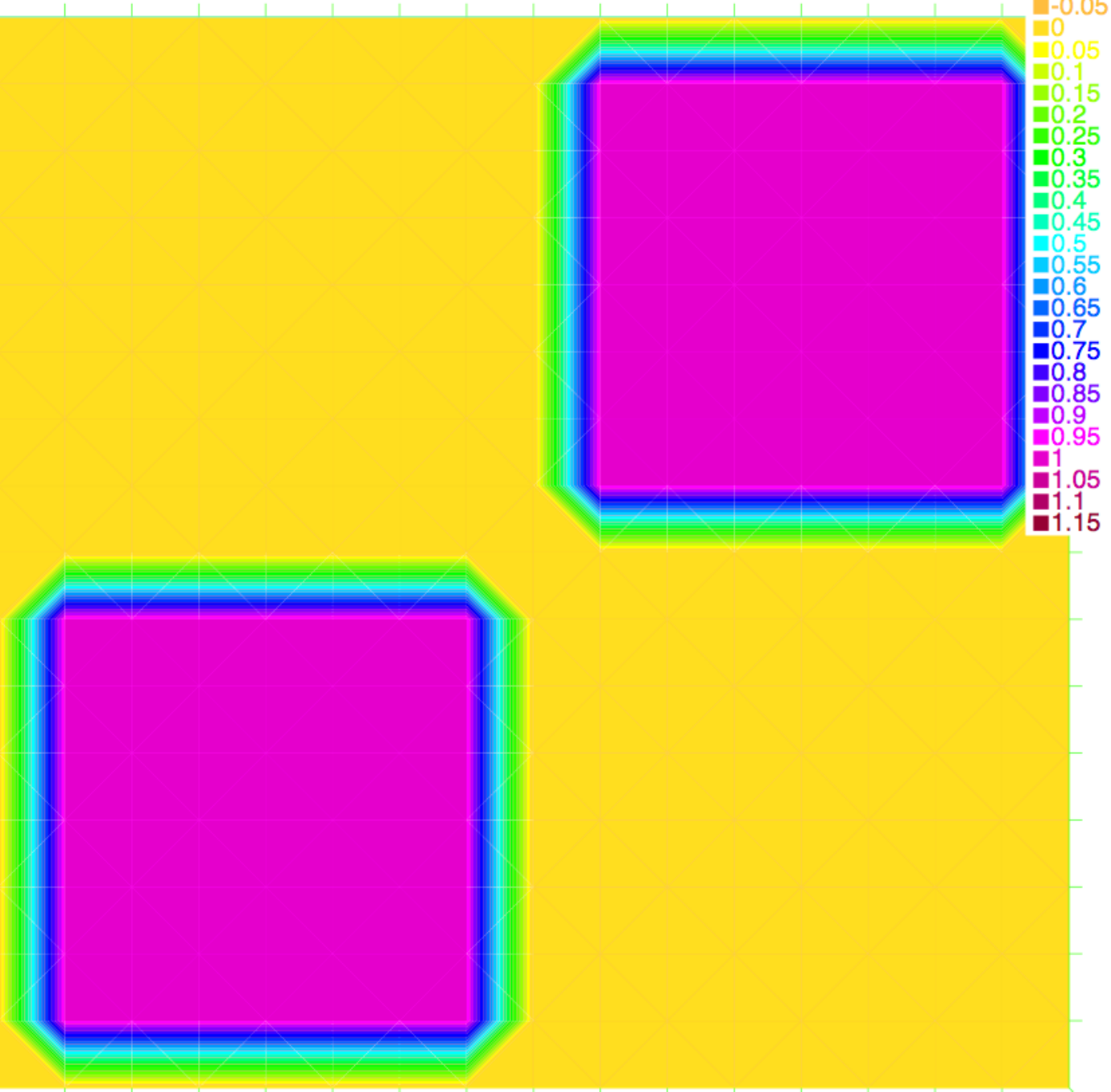}&
\includegraphics[ scale=0.175]{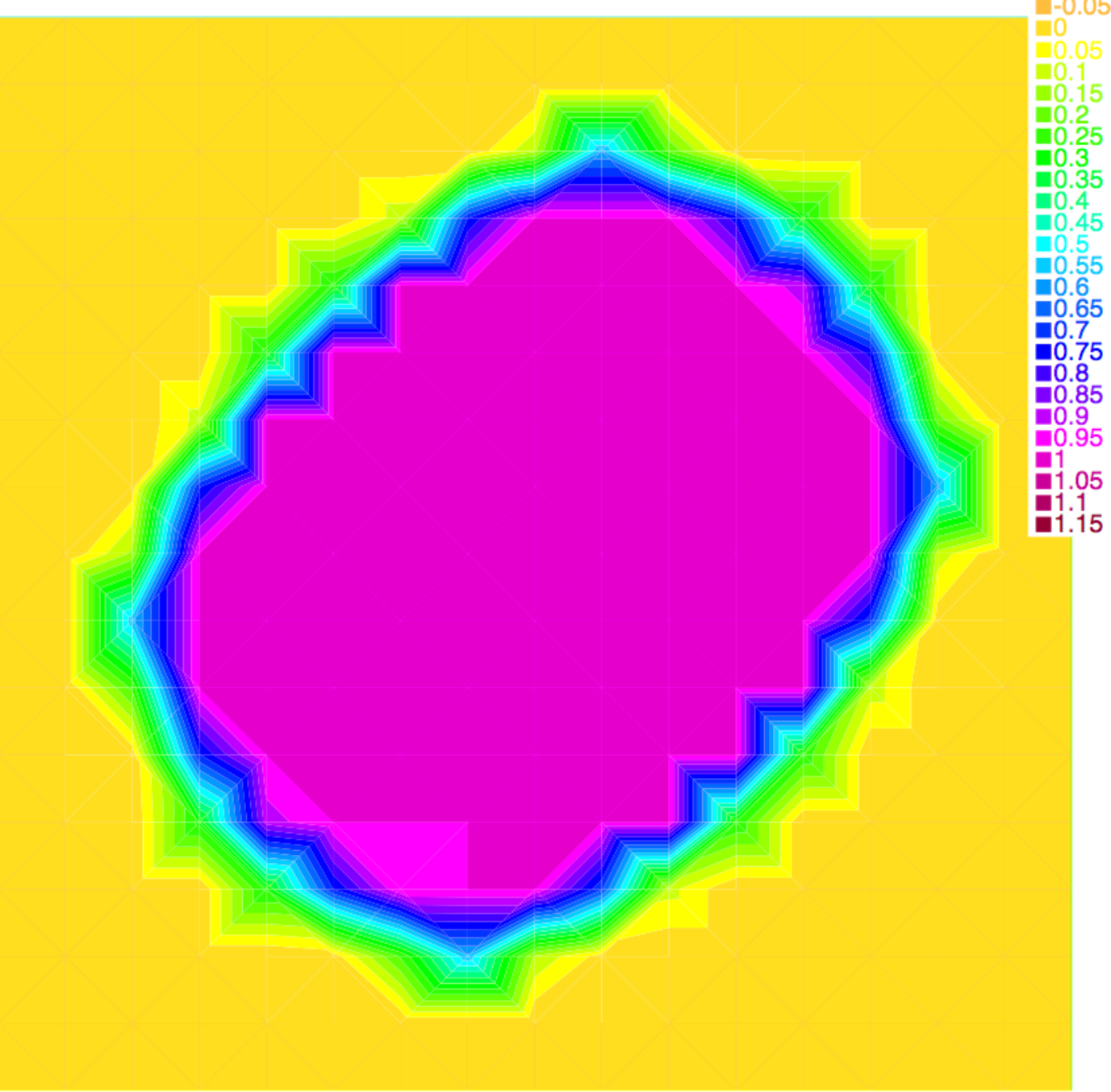}&
\includegraphics[ scale=0.175]{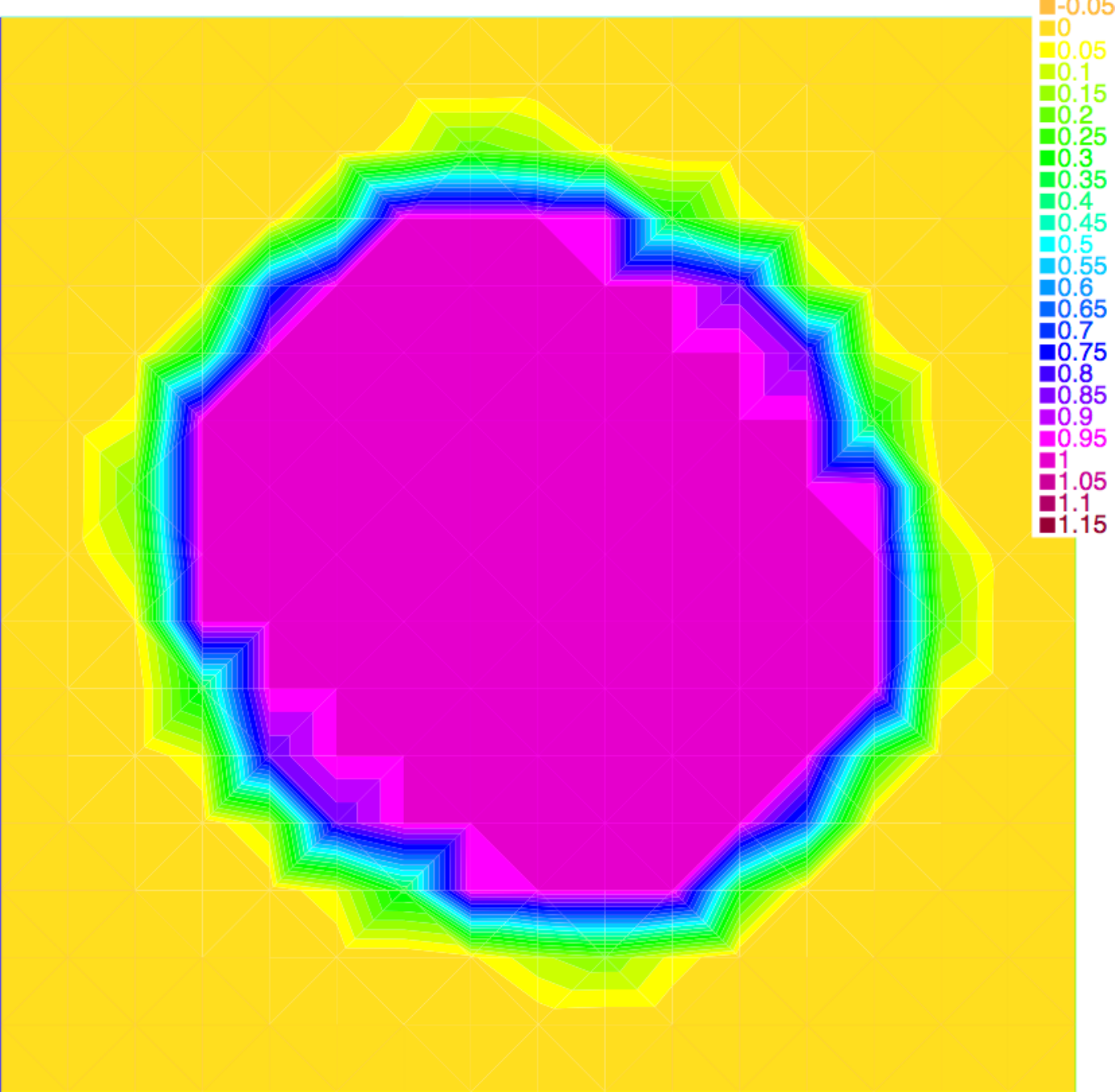}&
\includegraphics[ scale=0.175]{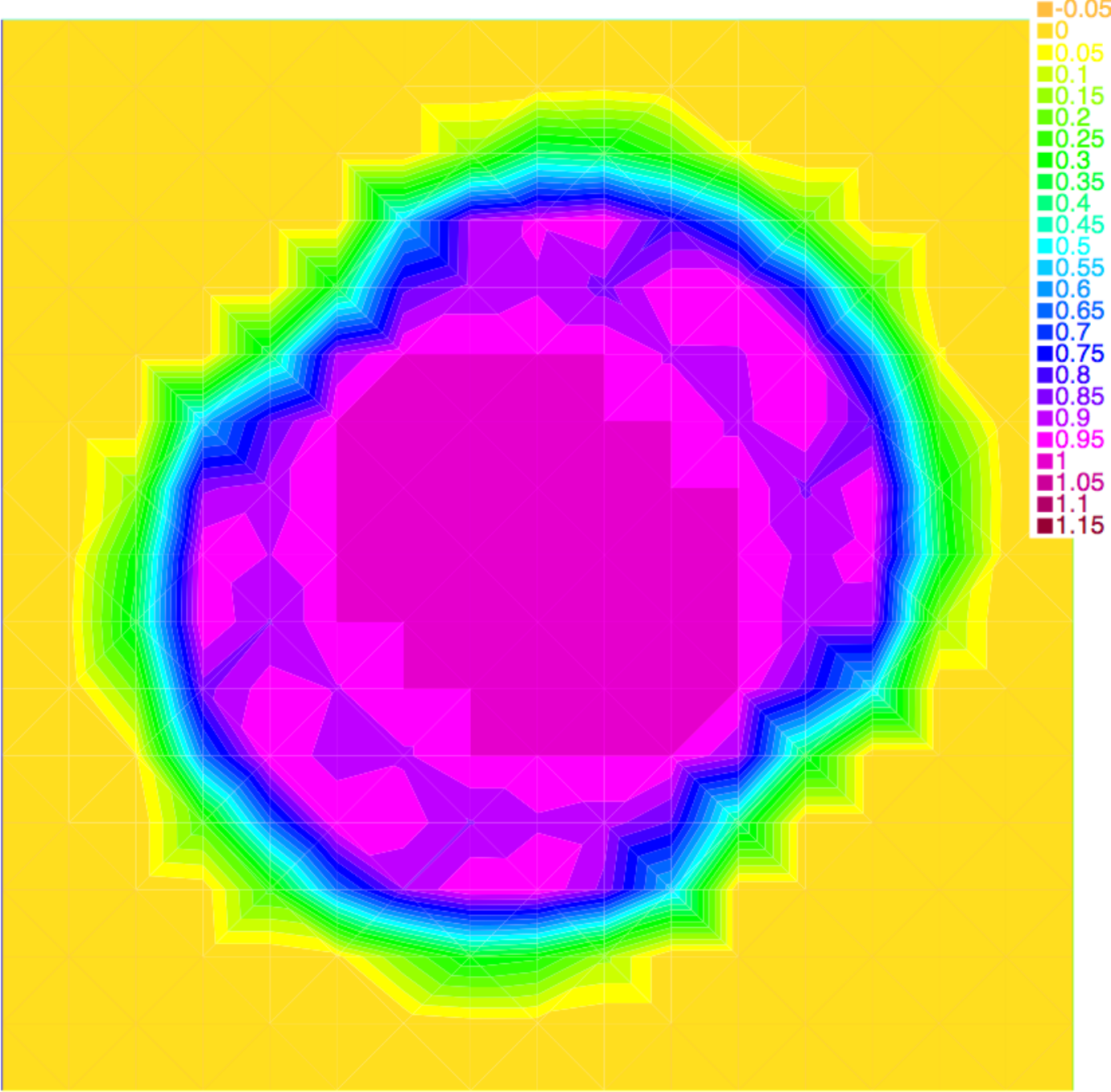}&
\includegraphics[ scale=0.175]{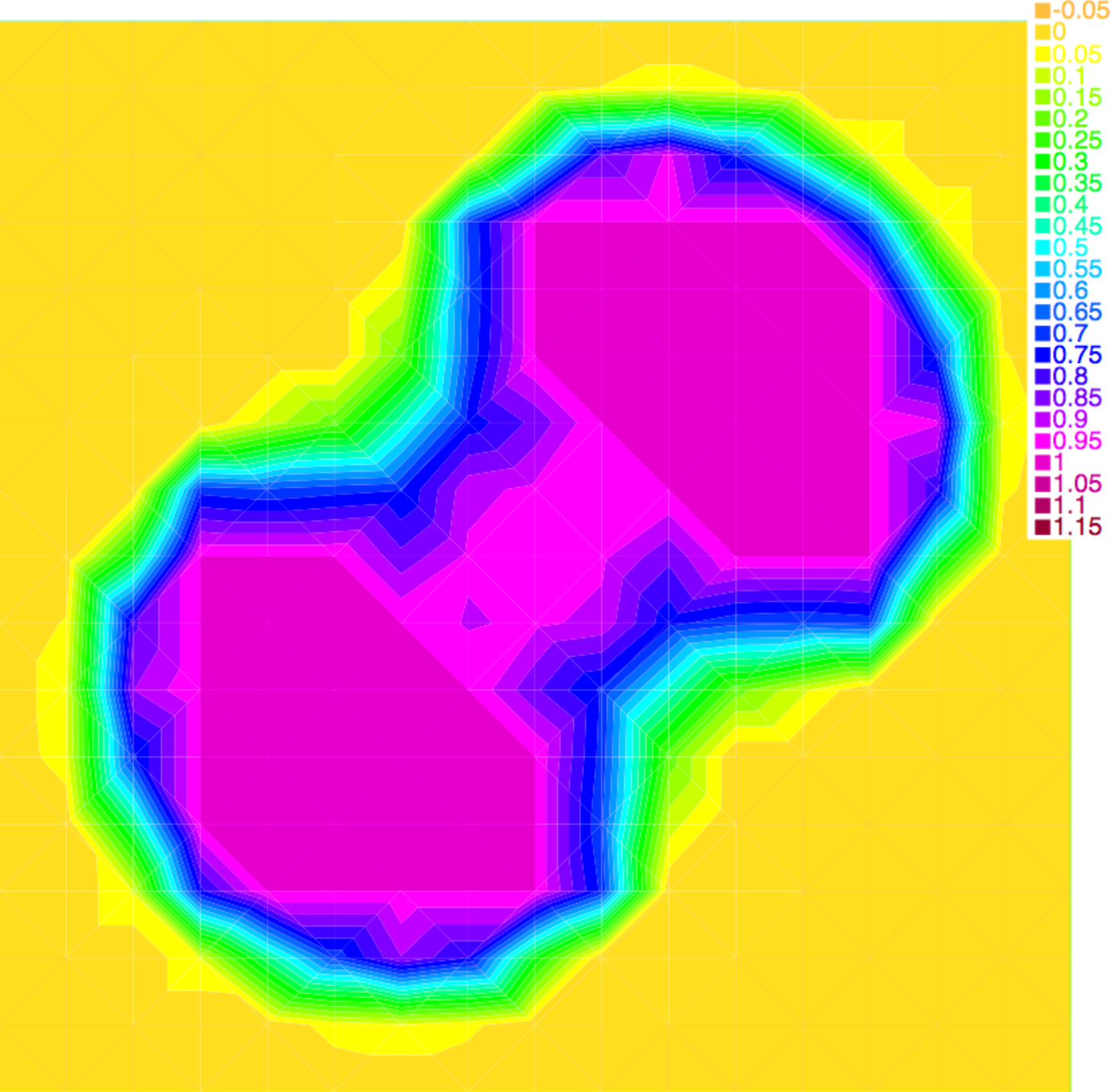}&
\includegraphics[ scale=0.175]{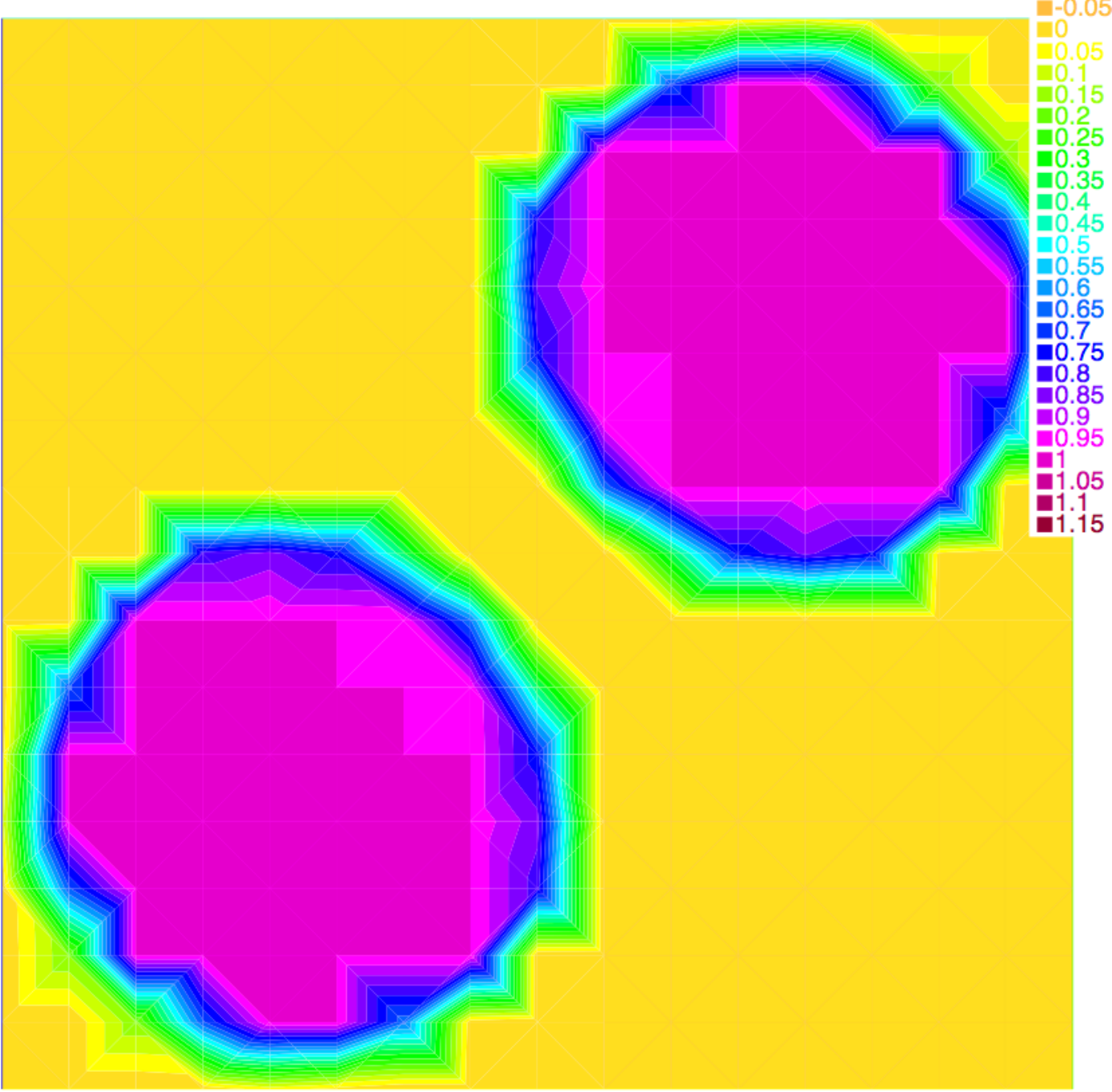}\\
\includegraphics[ scale=0.175]{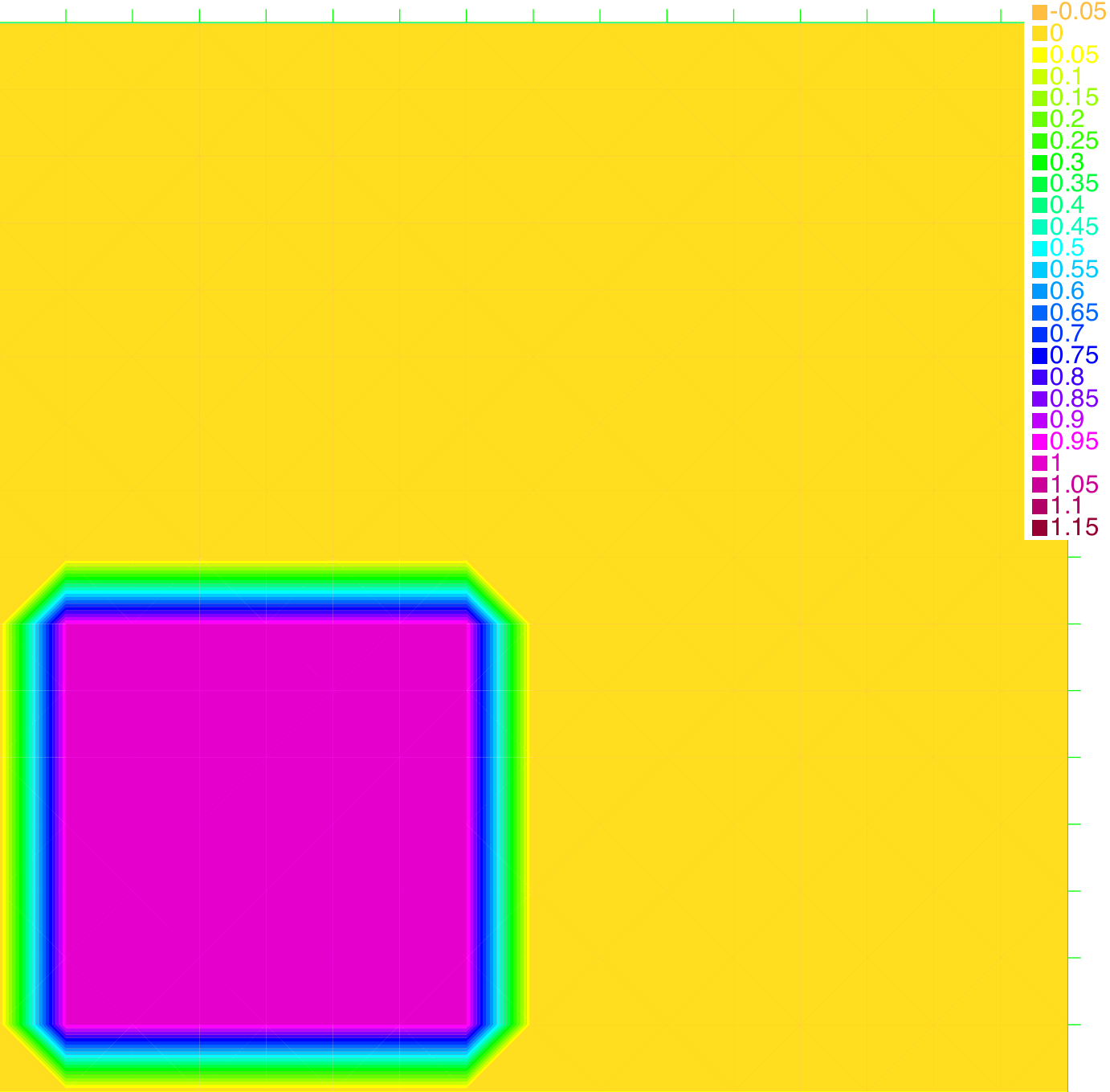}&
\includegraphics[ scale=0.175]{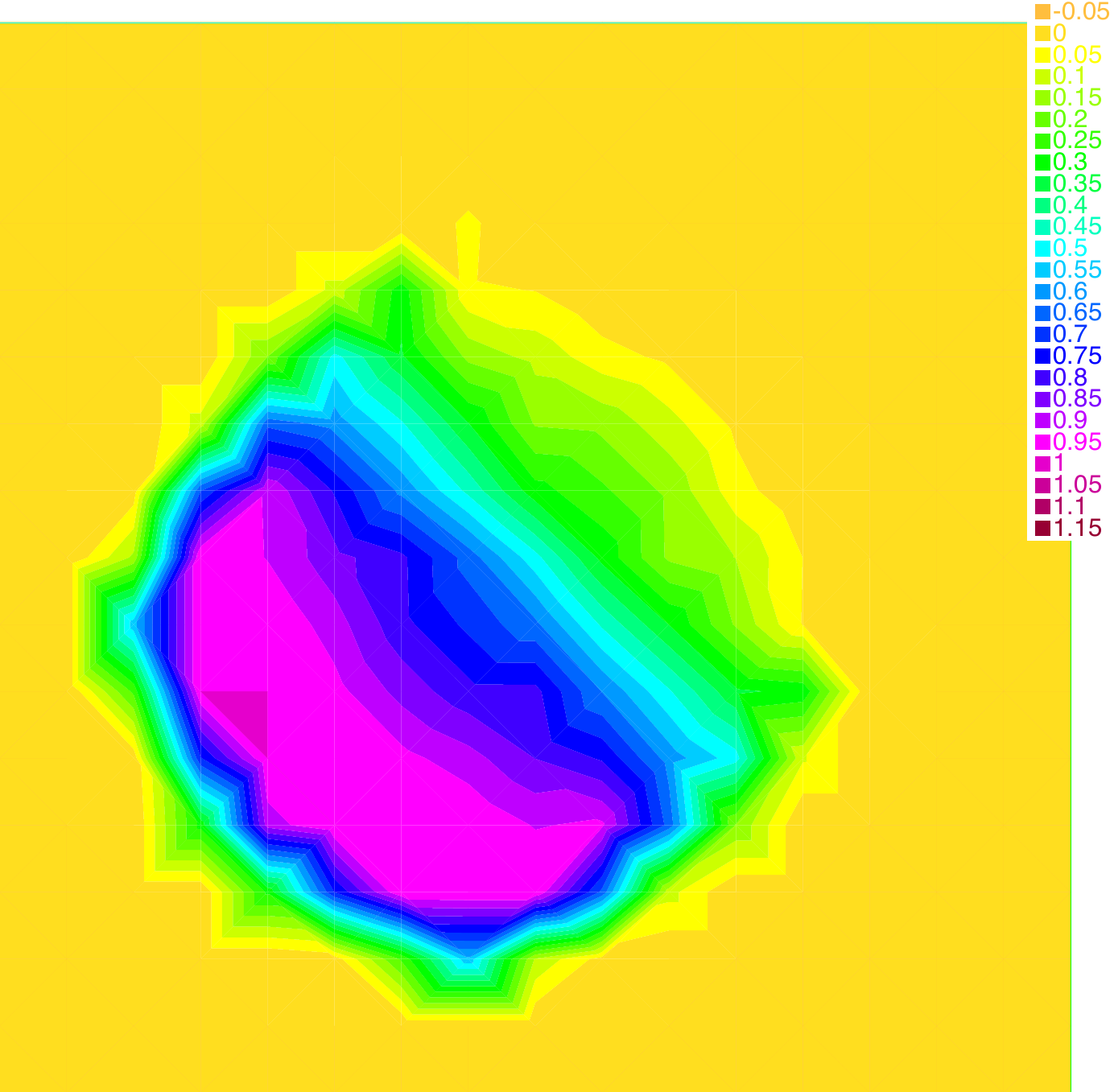}&
\includegraphics[ scale=0.175]{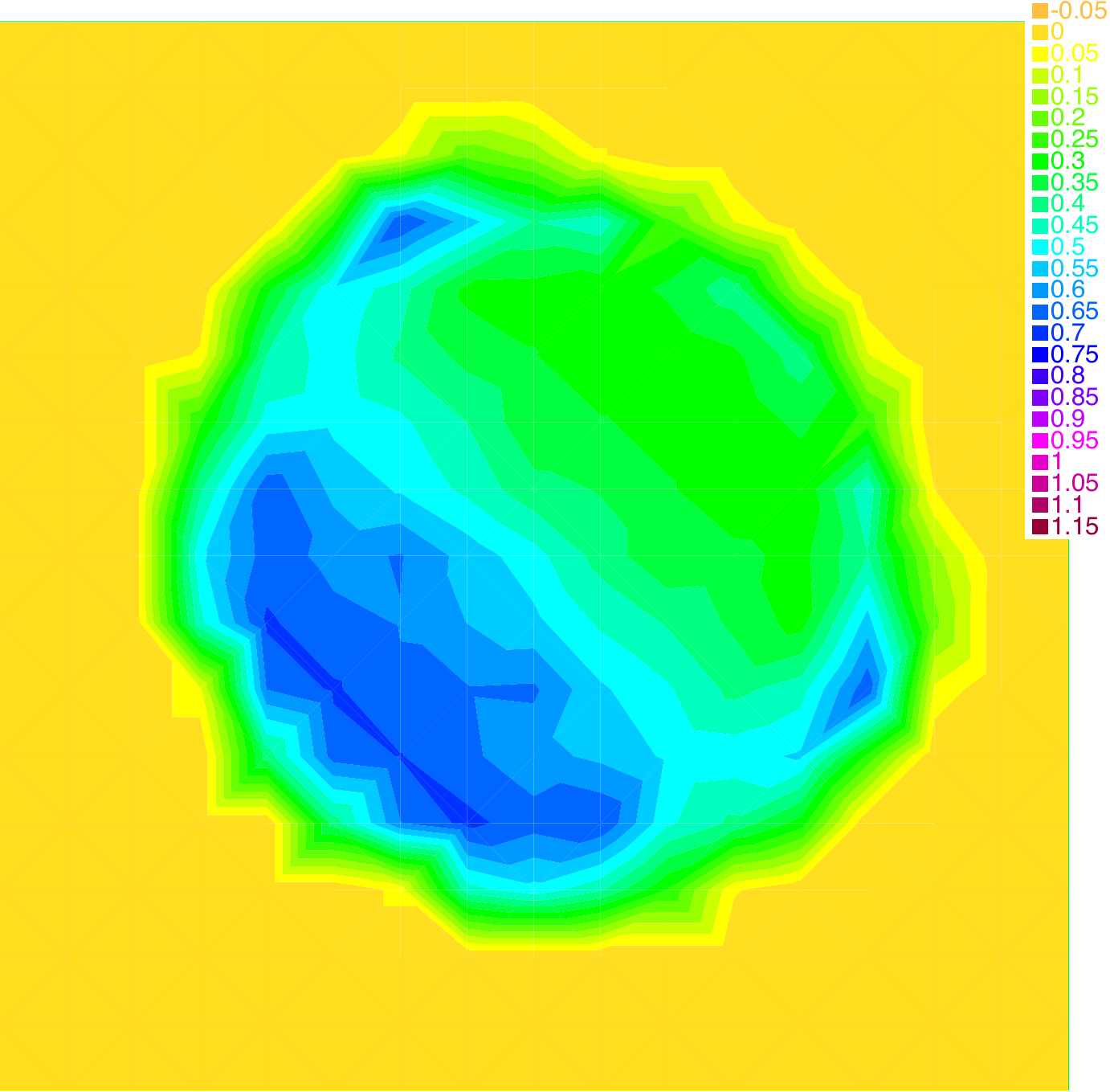}&
\includegraphics[ scale=0.175]{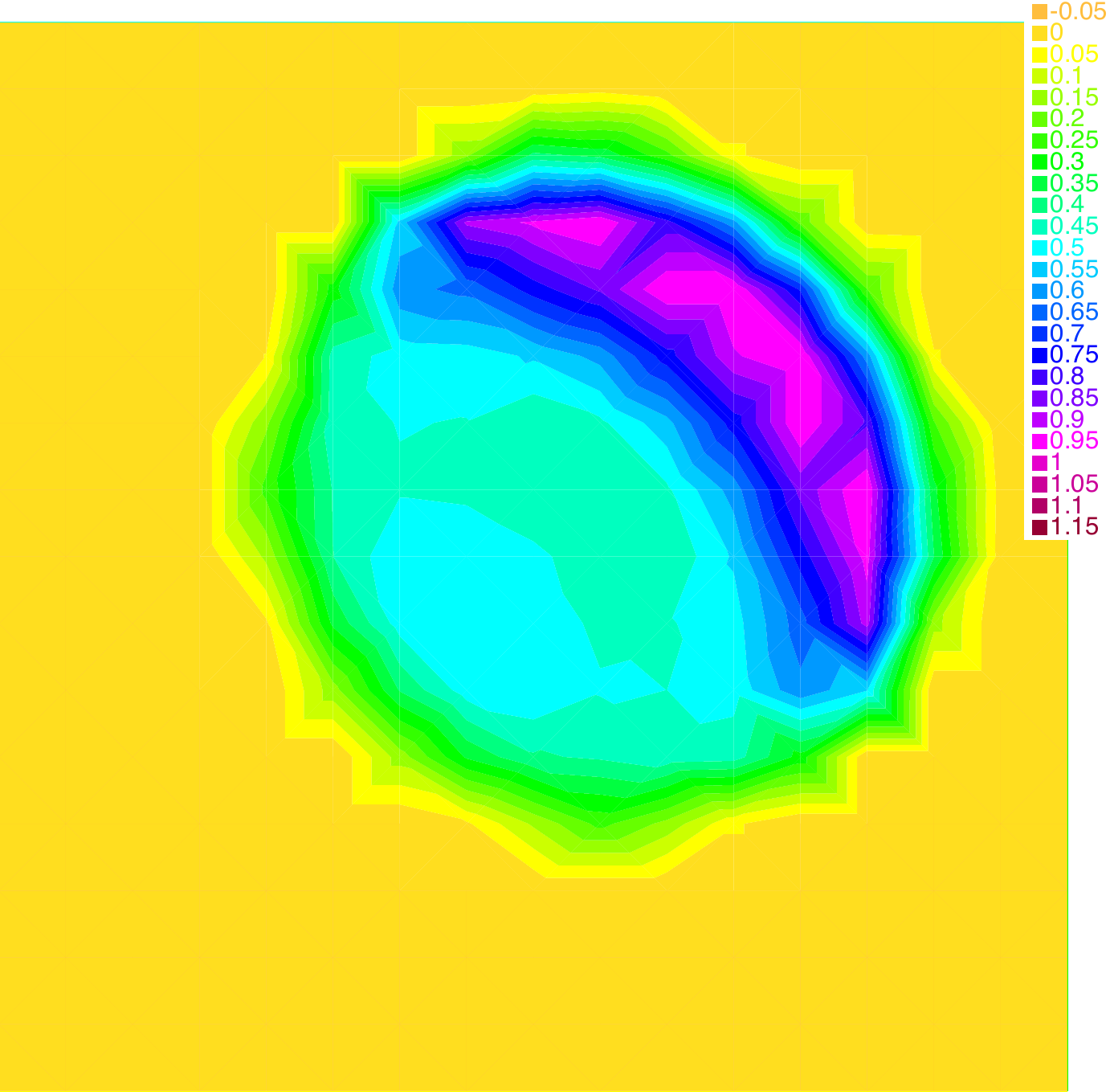}&
\includegraphics[ scale=0.175]{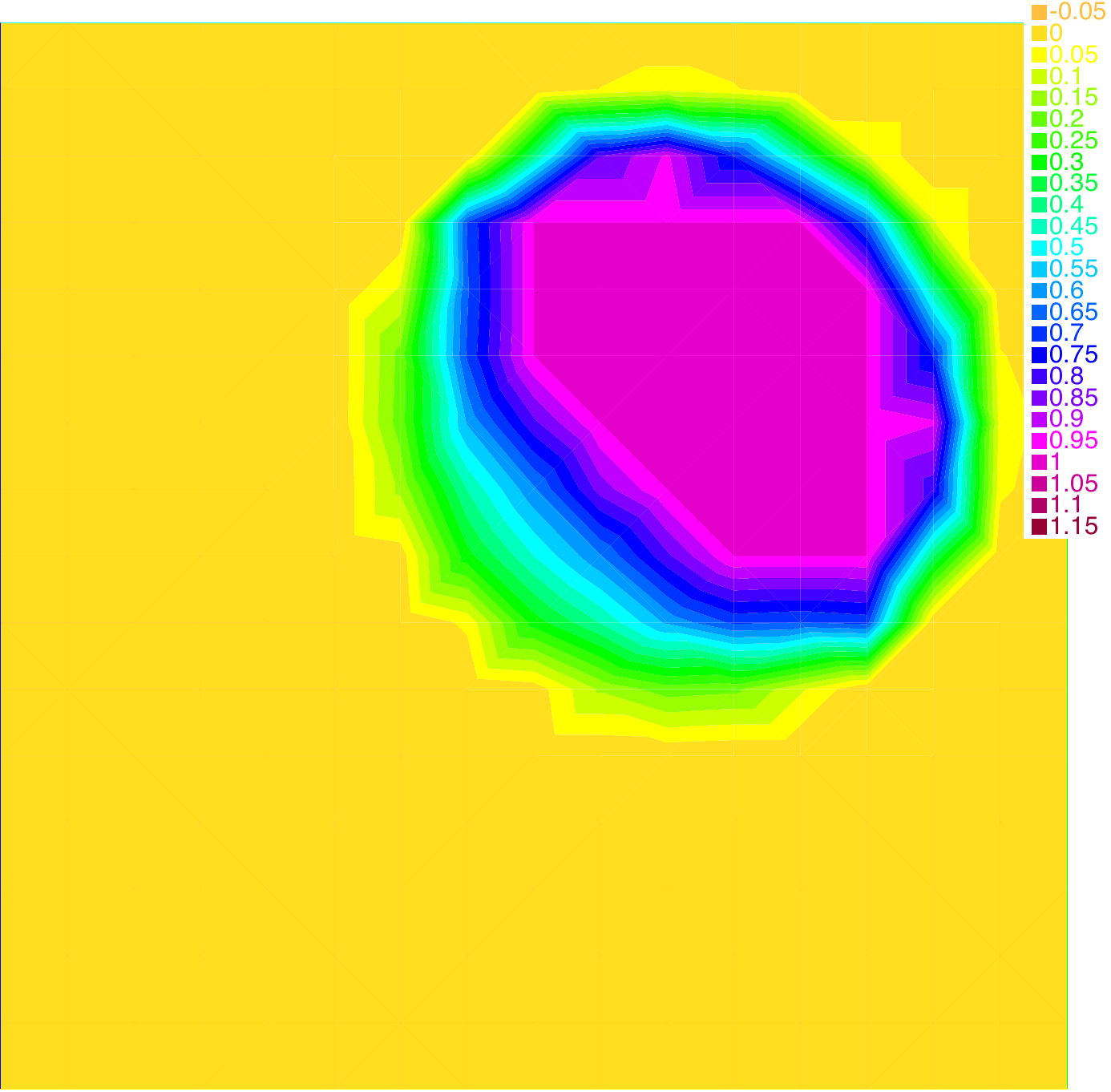}&
\includegraphics[ scale=0.175]{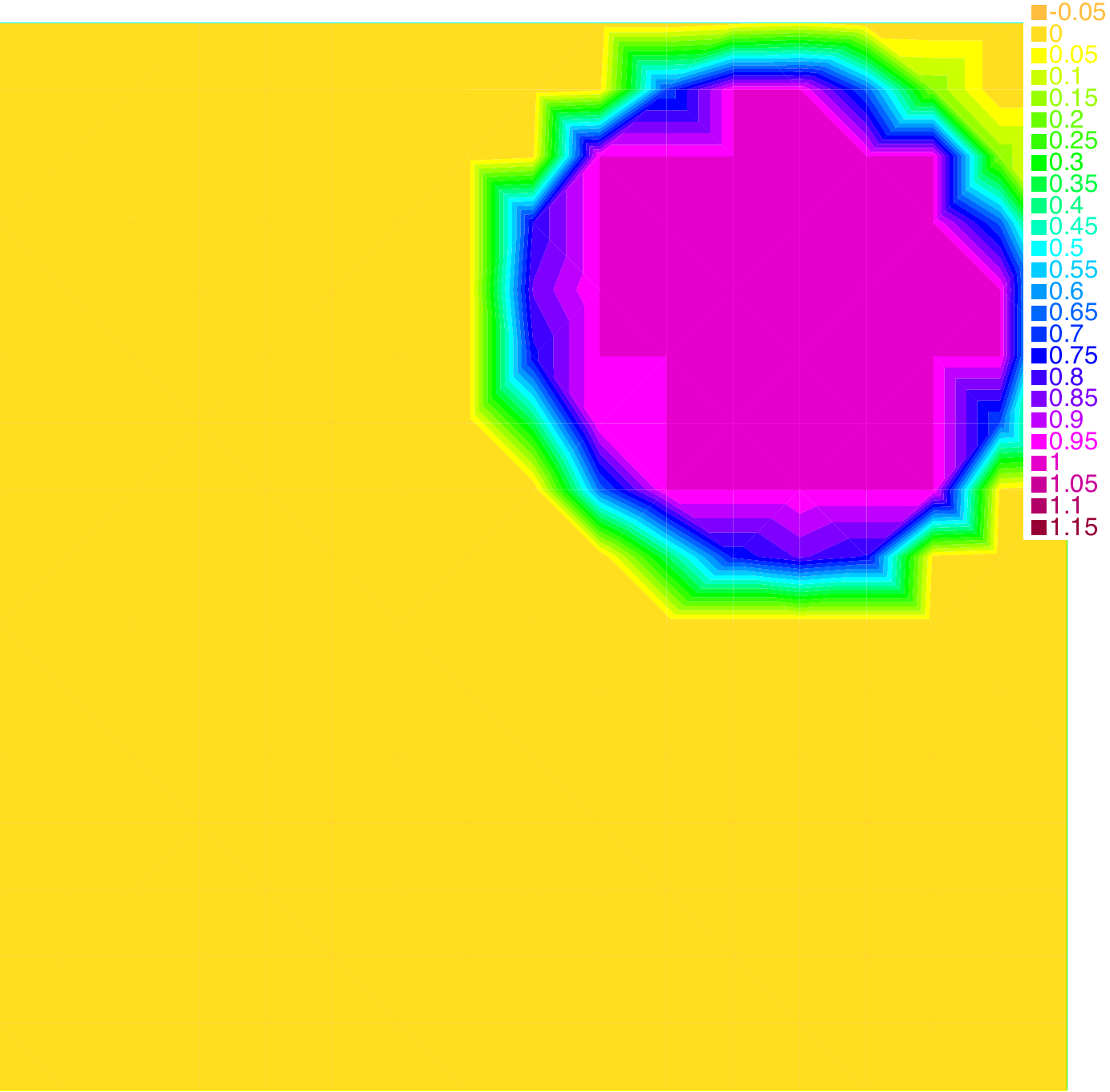}\\
$t=0$ & $t=0.05$ & $t=0.1$ & $t=0.15$ &$t=0.2$ & $t=0.3$\\
\end{tabular}
\caption{\textit{Evolution of two species crossing each other with density constraint. Top row: display of $\rho_1+\rho_2$. Bottom row: display of $\rho_1$.}}
\label{figure crowd motion 2}
\end{figure}

 \begin{figure}[h!]

\begin{tabular}{@{\hspace{0mm}}c@{\hspace{1mm}}c@{\hspace{1mm}}c@{\hspace{1mm}}c@{\hspace{1mm}}c@{\hspace{1mm}}c@{\hspace{1mm}}}

\centering
\includegraphics[ scale=0.175]{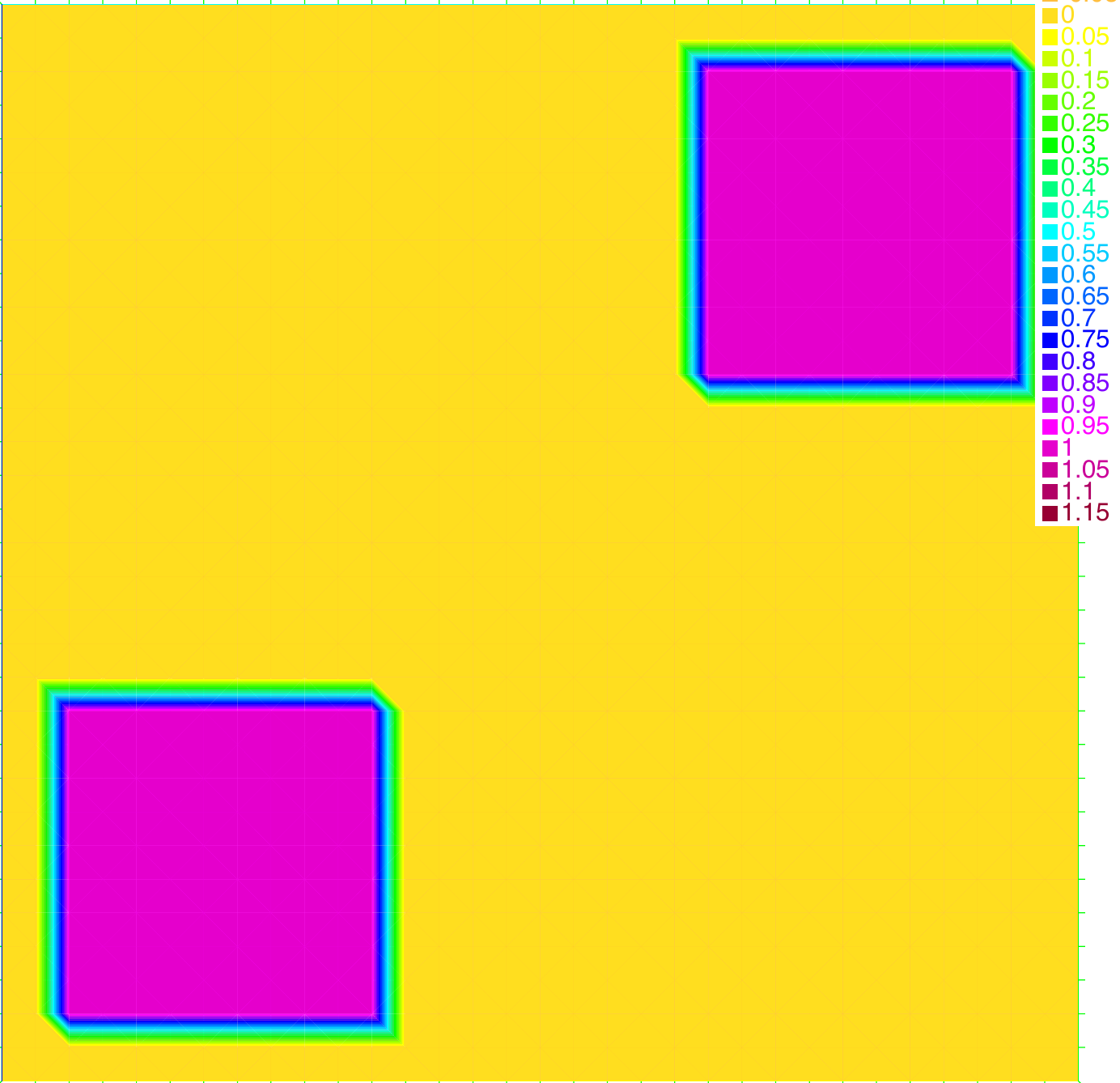}&
\includegraphics[ scale=0.175]{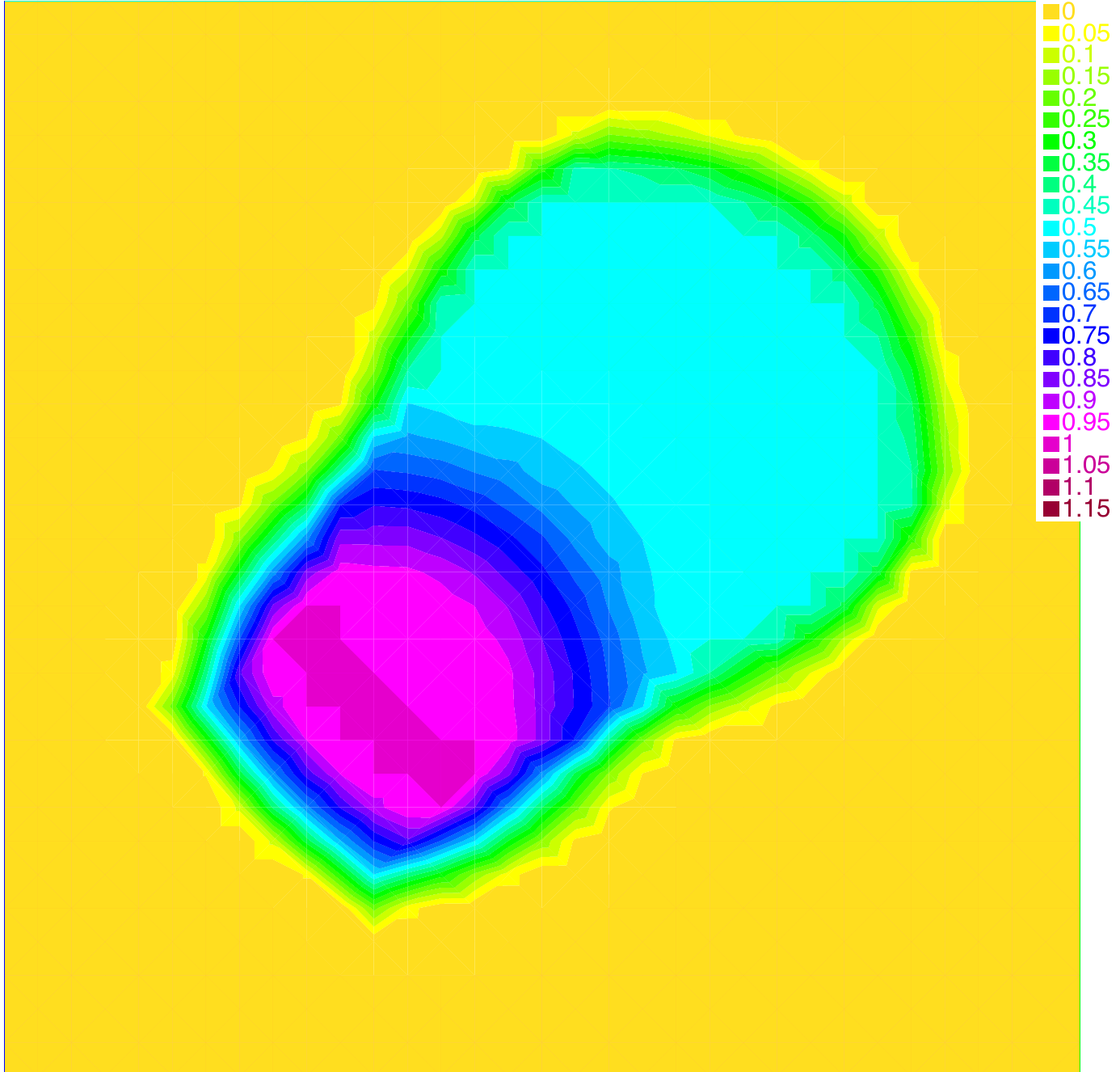}&
\includegraphics[ scale=0.175]{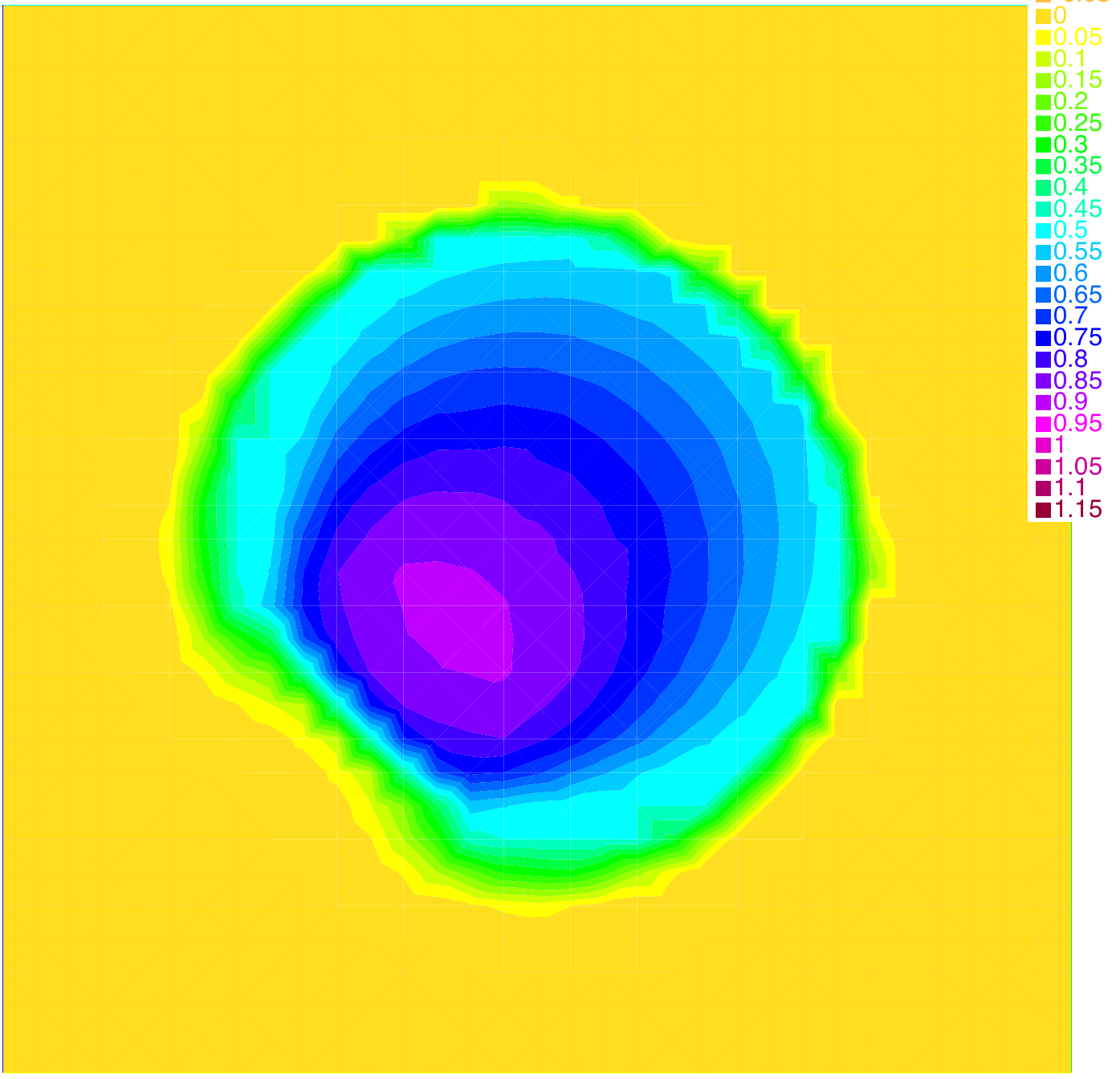}&
\includegraphics[ scale=0.175]{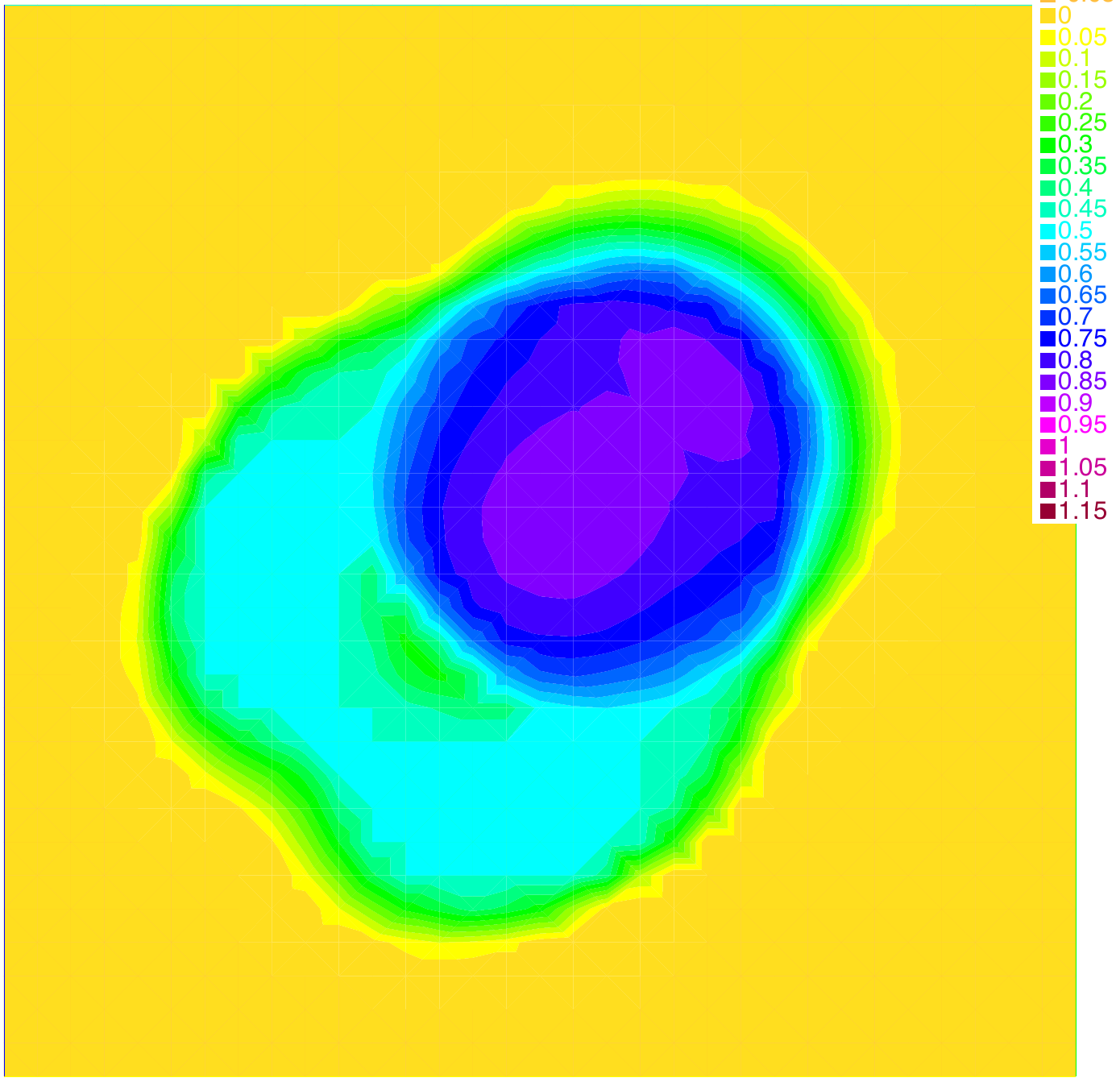}&
\includegraphics[ scale=0.175]{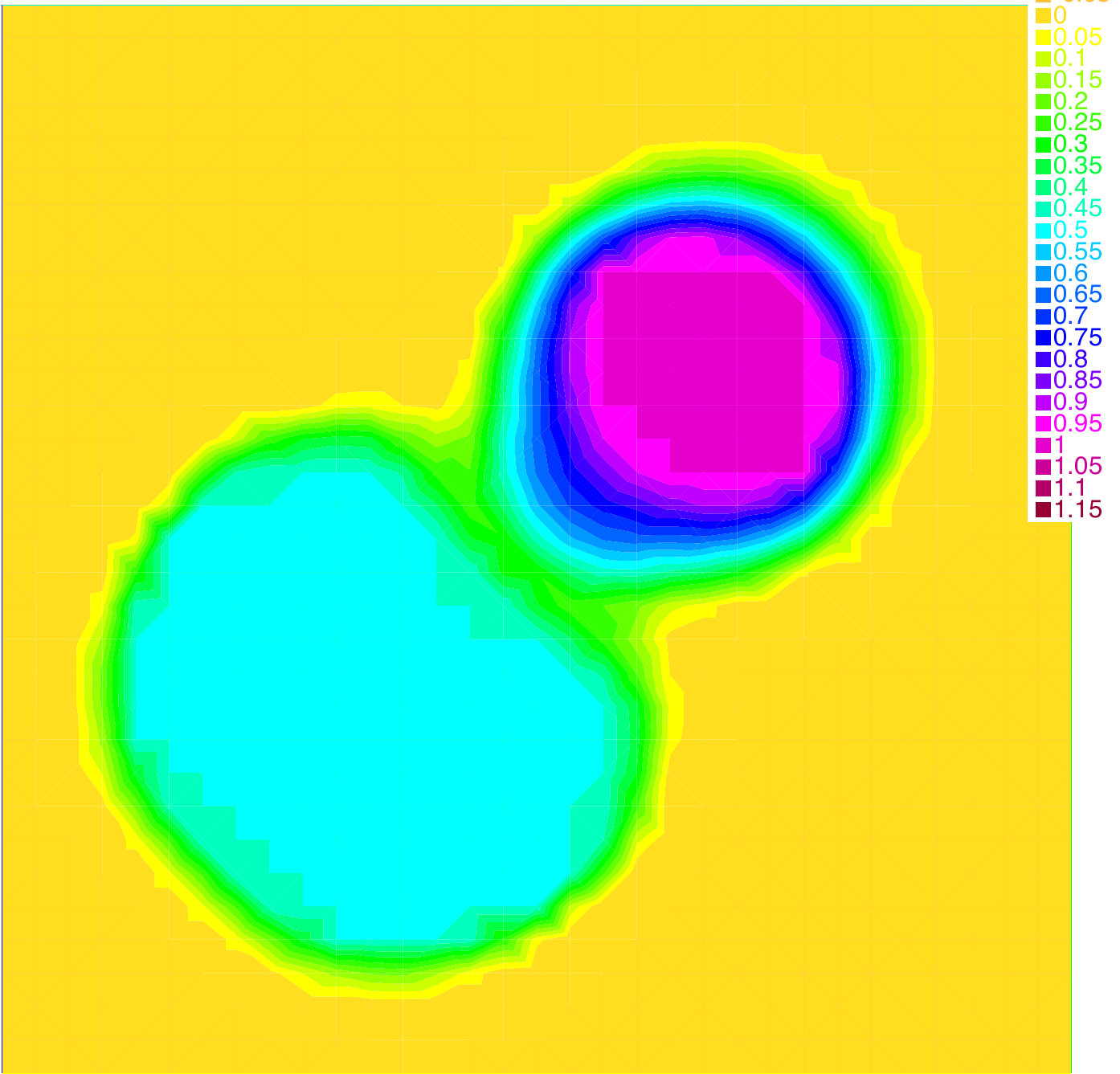}&
\includegraphics[ scale=0.175]{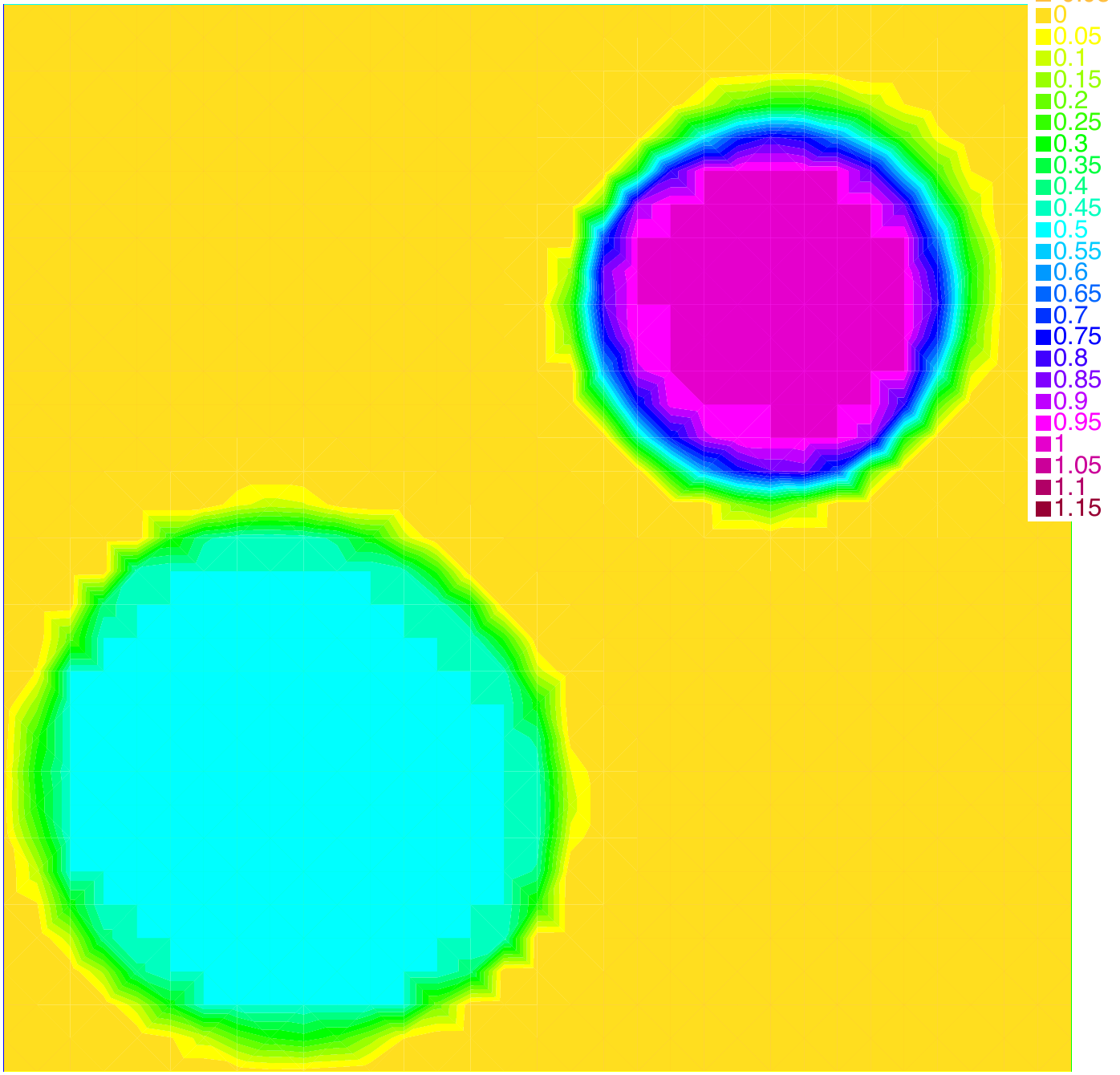}\\
\includegraphics[ scale=0.175]{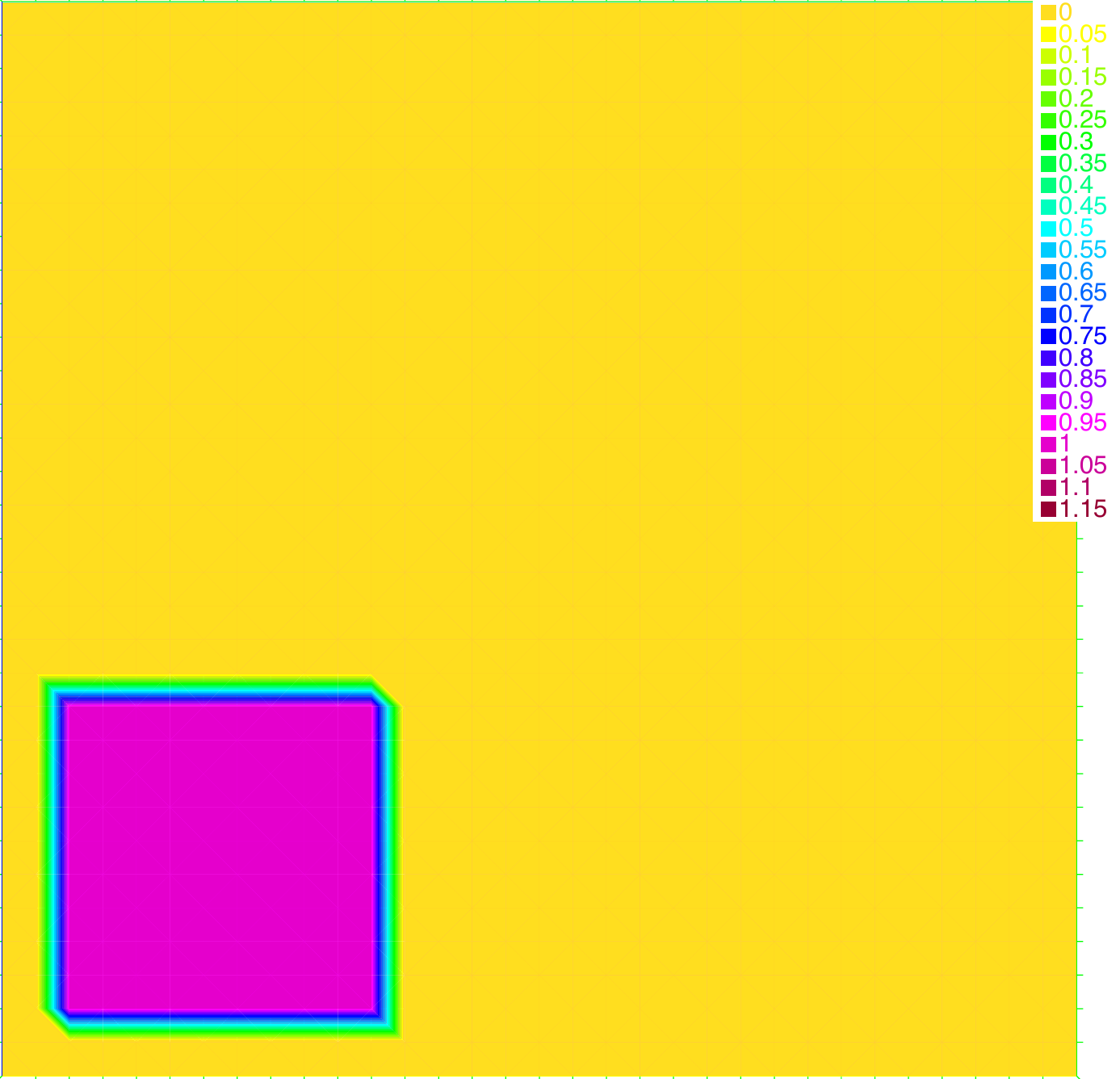}&
\includegraphics[ scale=0.175]{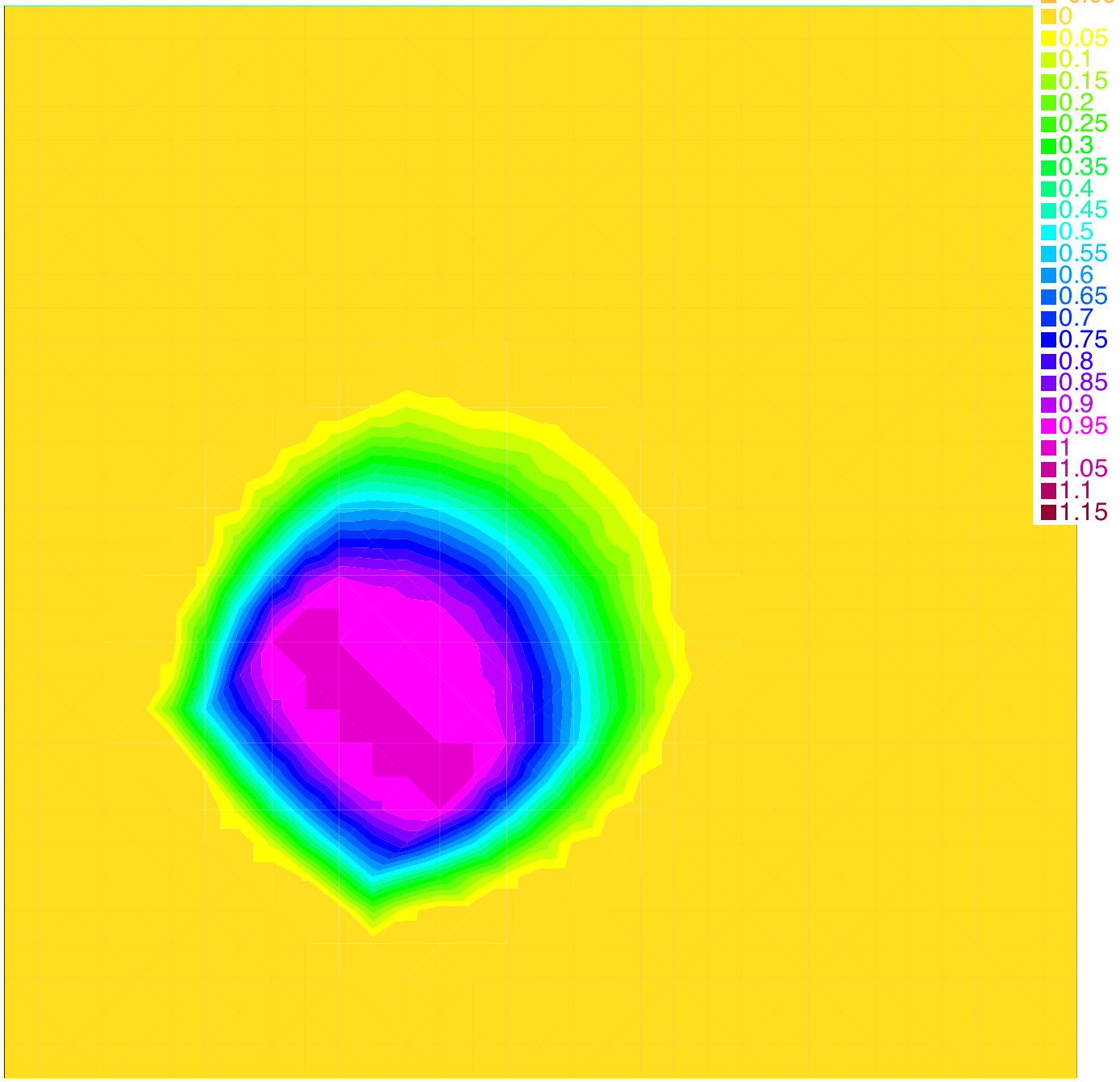}&
\includegraphics[ scale=0.175]{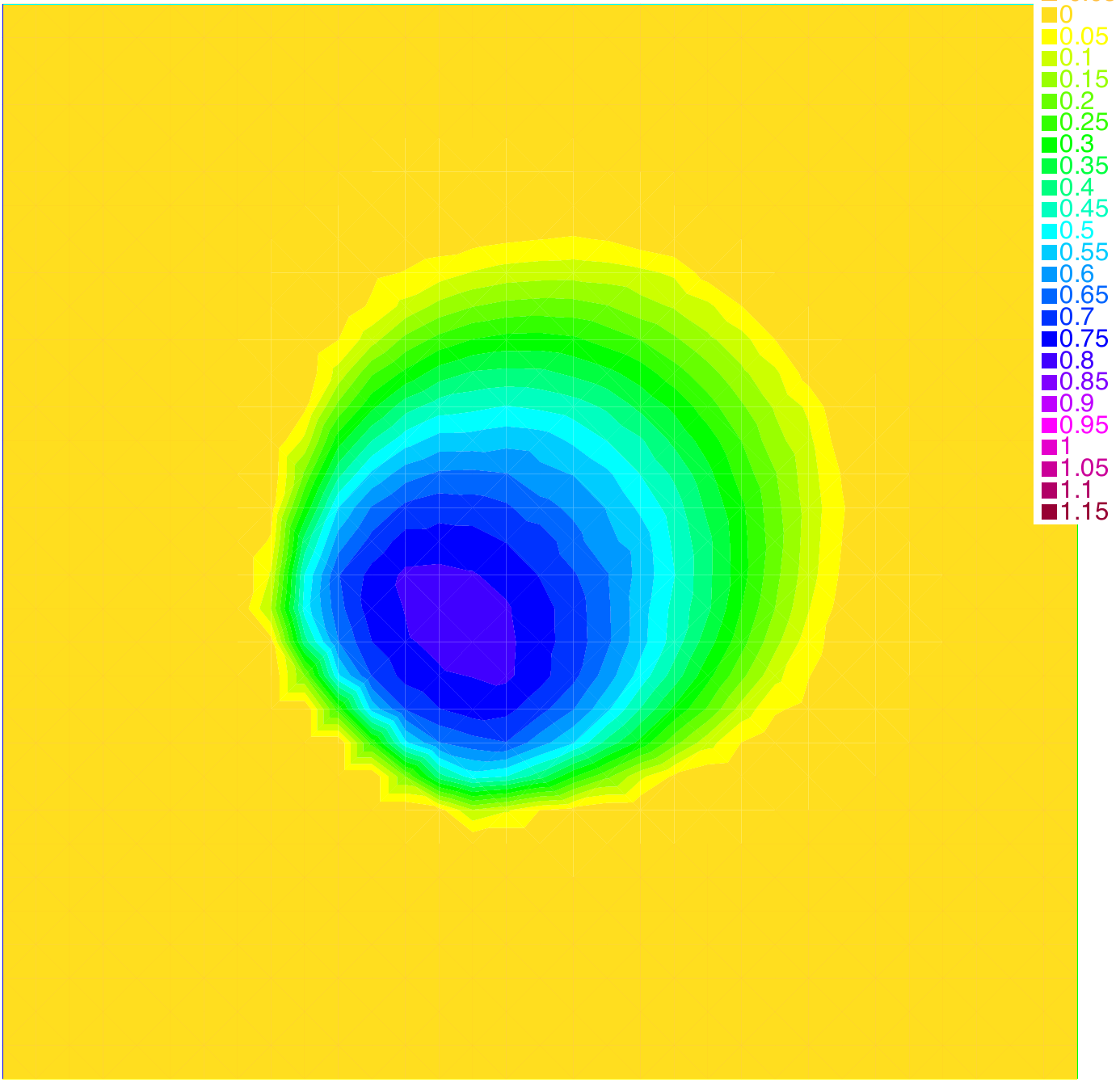}&
\includegraphics[ scale=0.175]{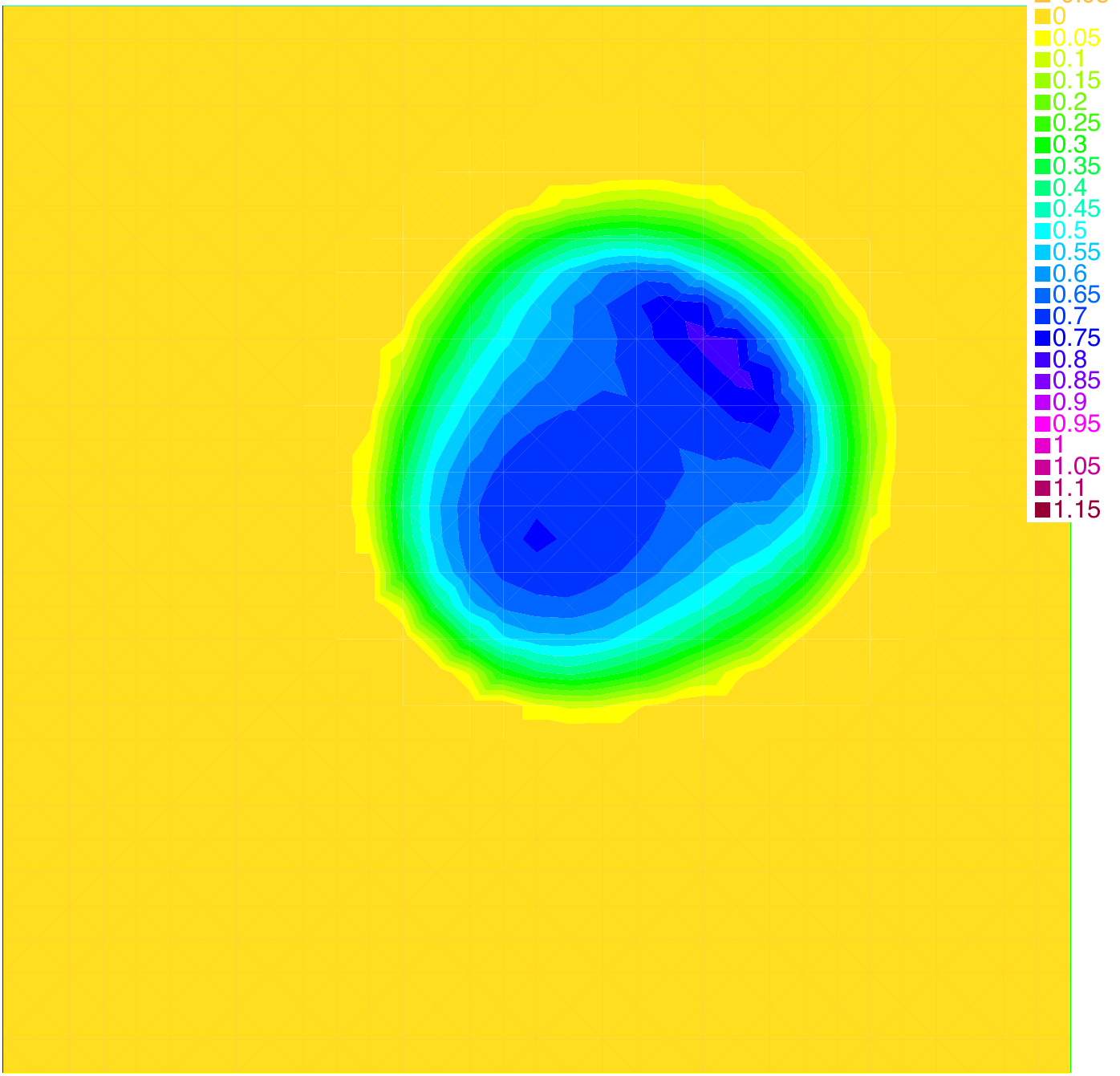}&
\includegraphics[ scale=0.175]{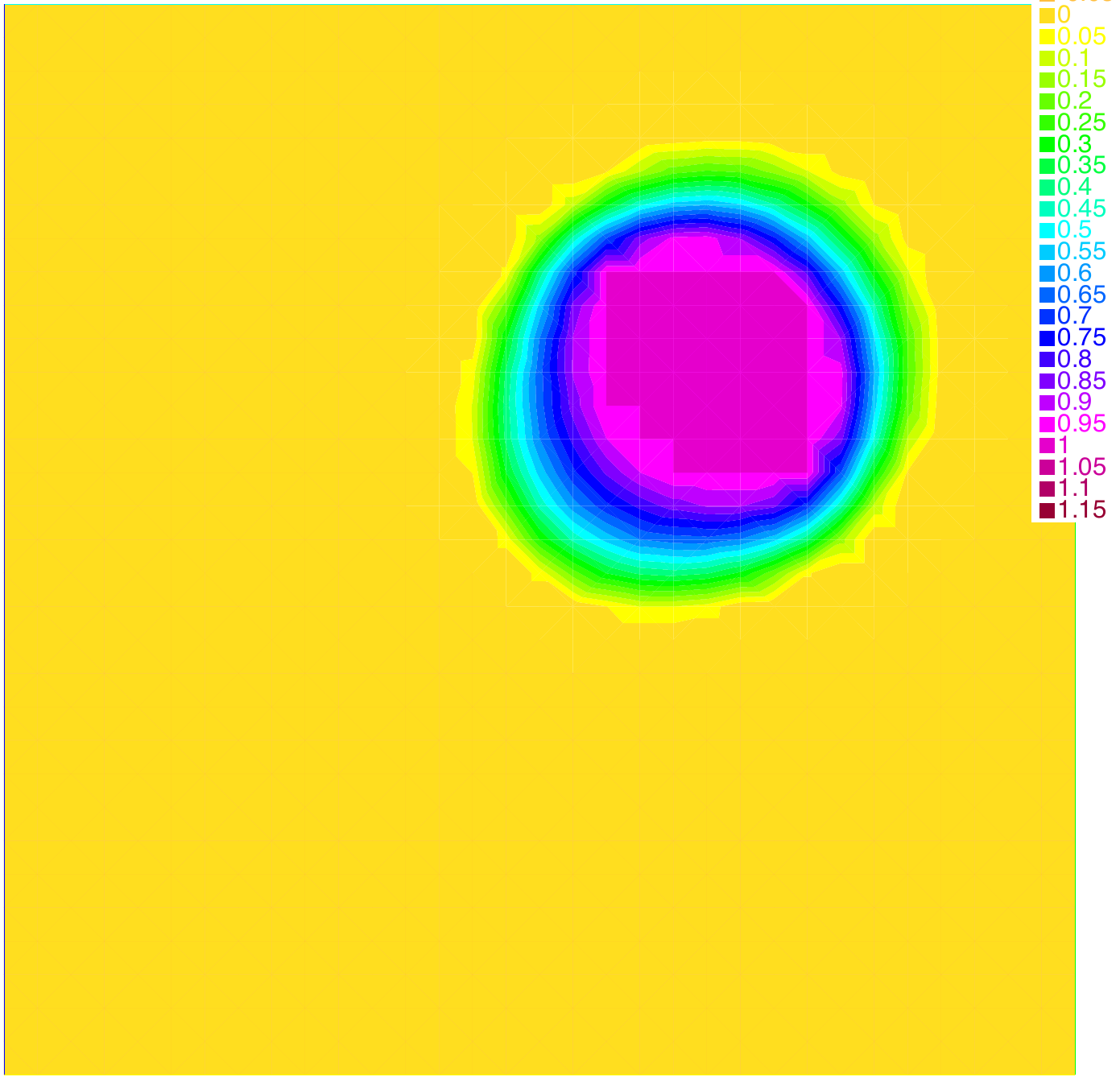}&
\includegraphics[ scale=0.175]{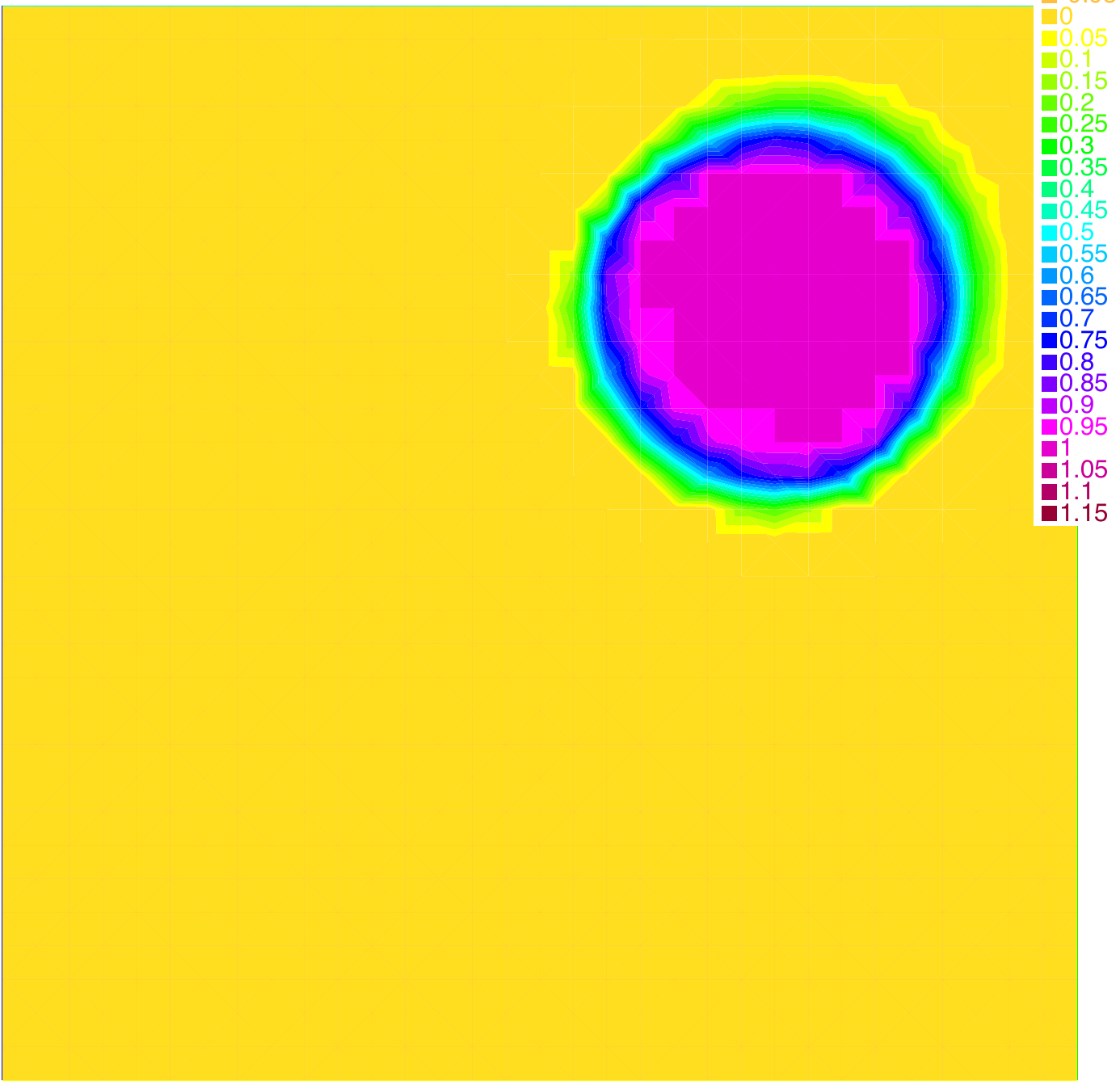}\\
\includegraphics[ scale=0.175]{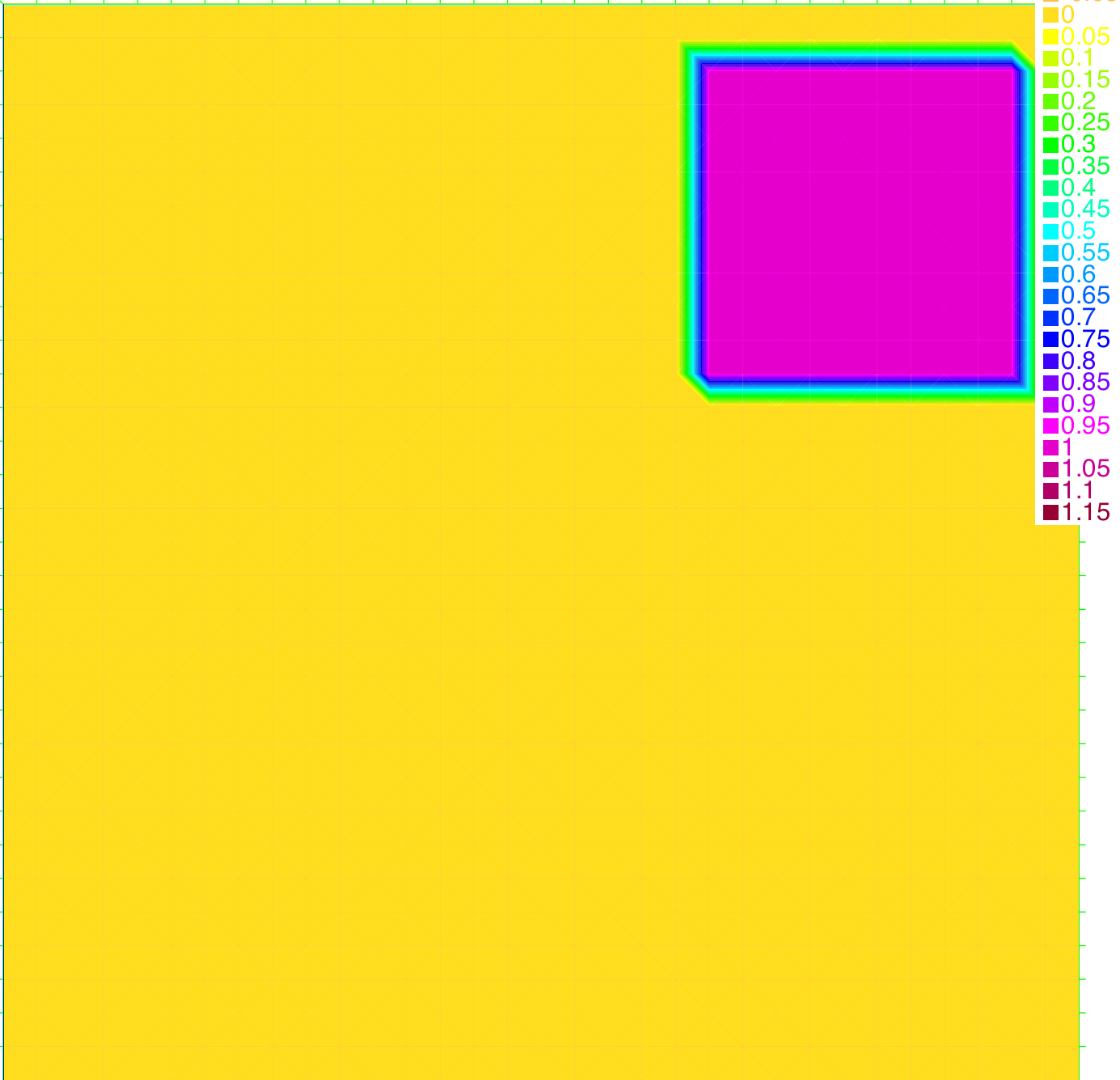}&
\includegraphics[ scale=0.175]{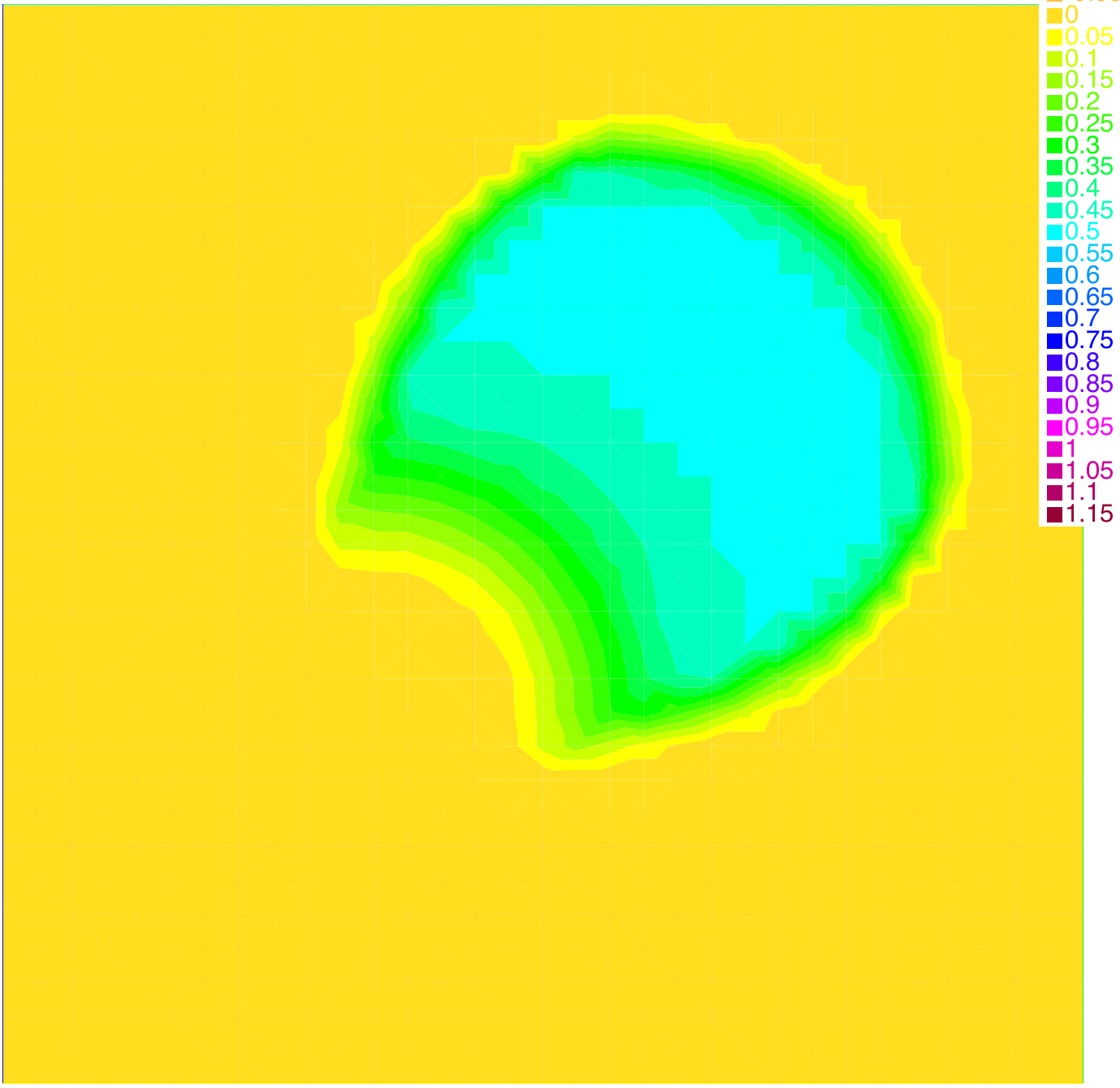}&
\includegraphics[ scale=0.175]{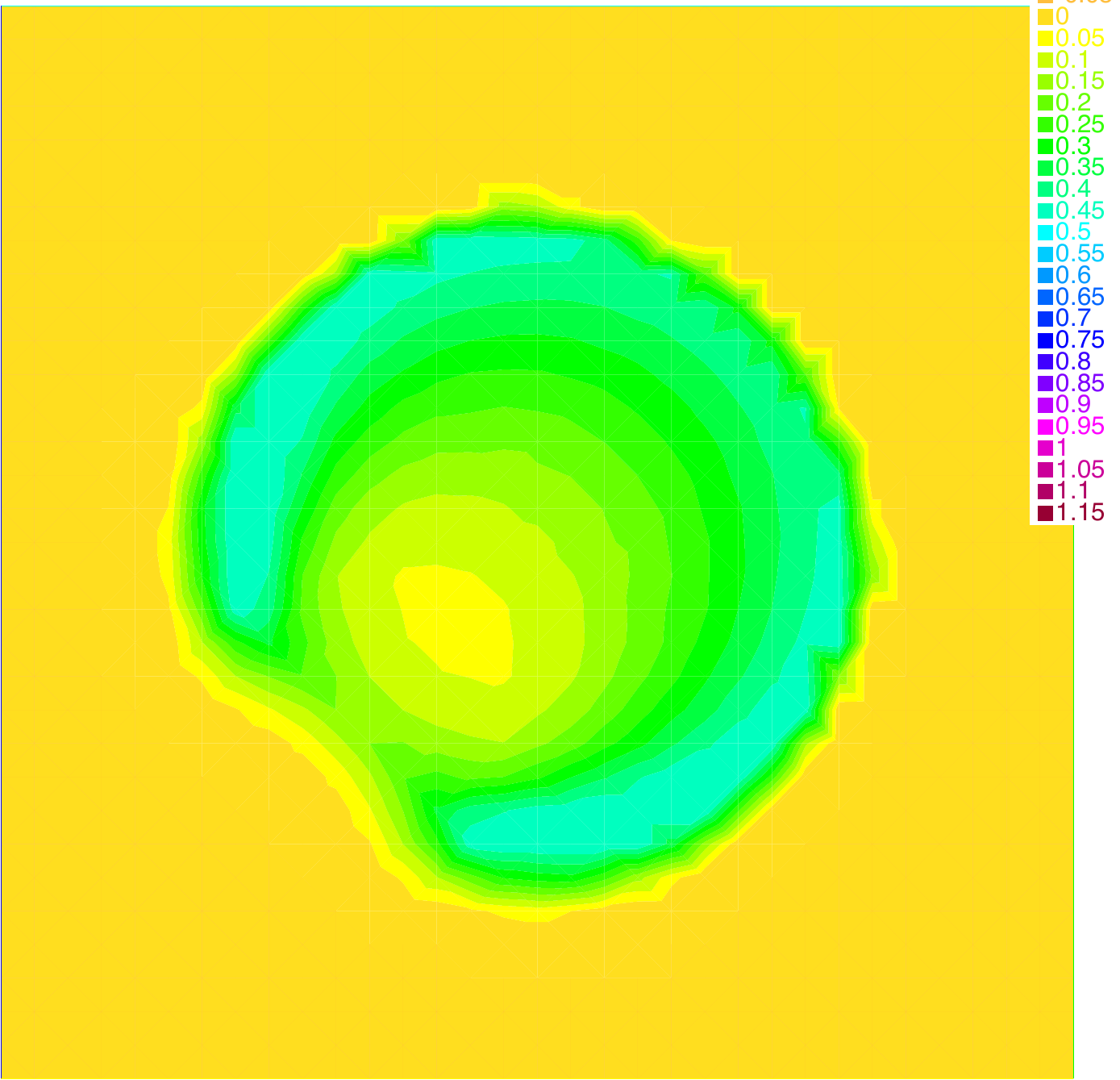}&
\includegraphics[ scale=0.175]{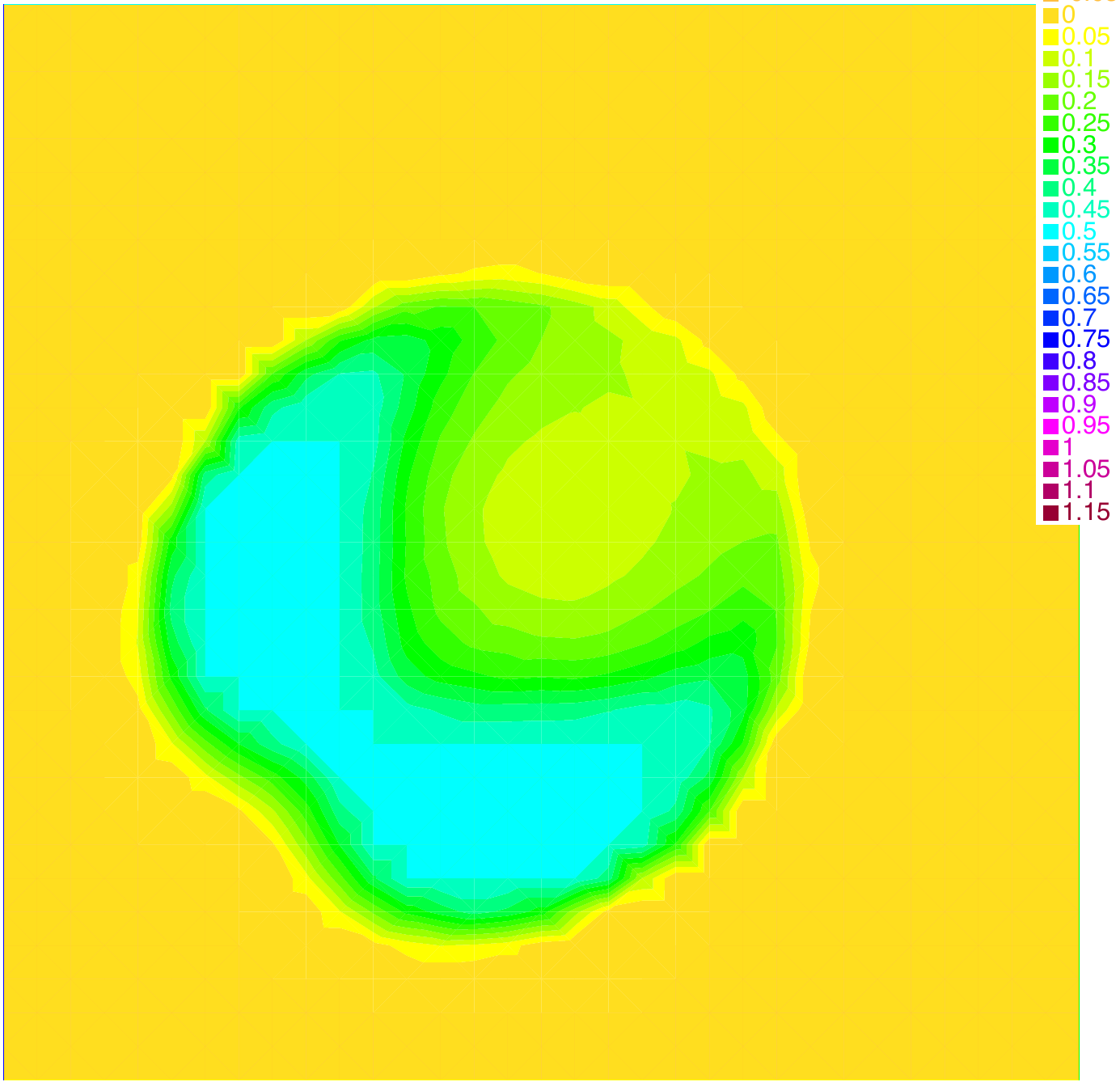}&
\includegraphics[ scale=0.175]{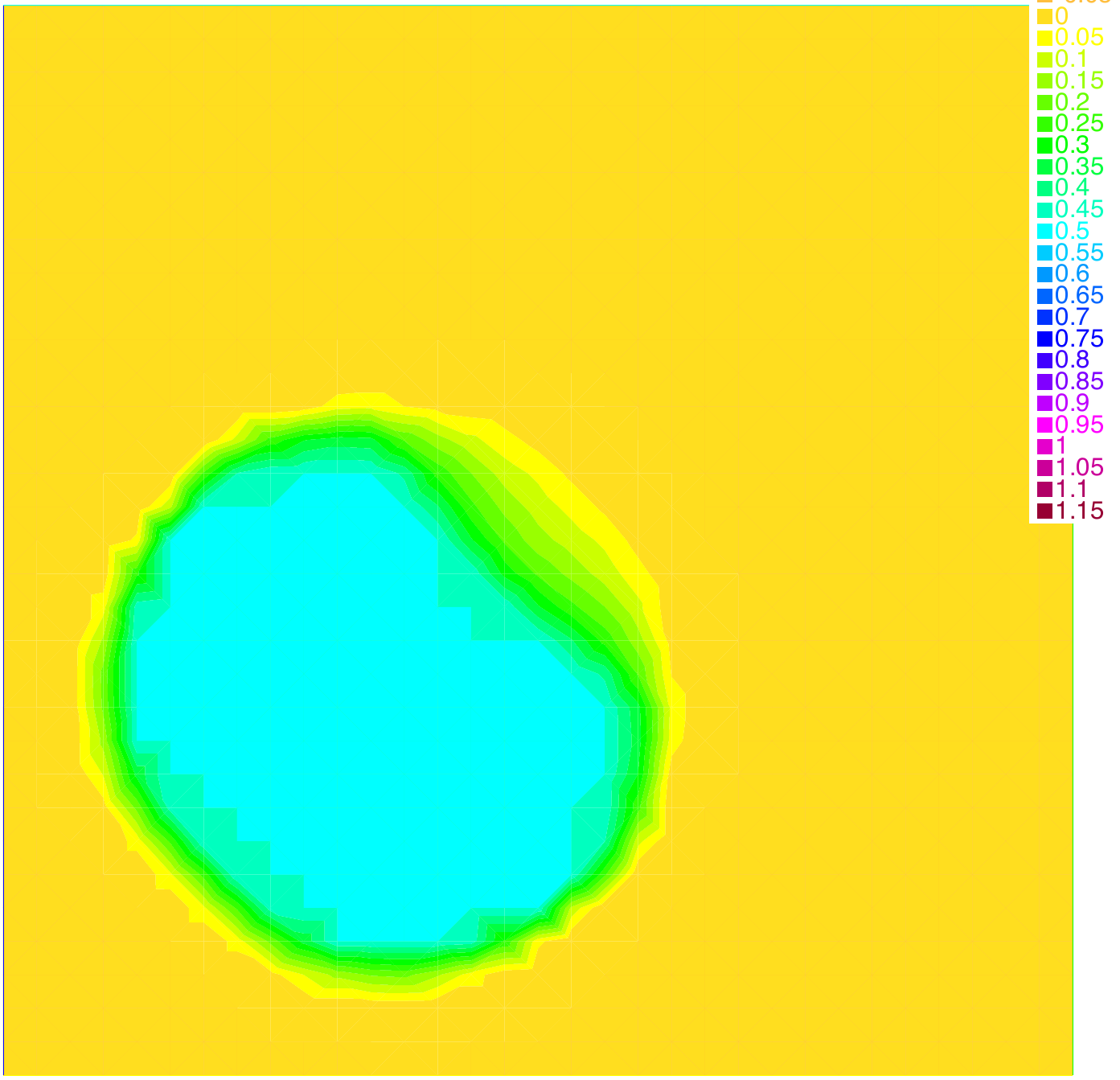}&
\includegraphics[ scale=0.175]{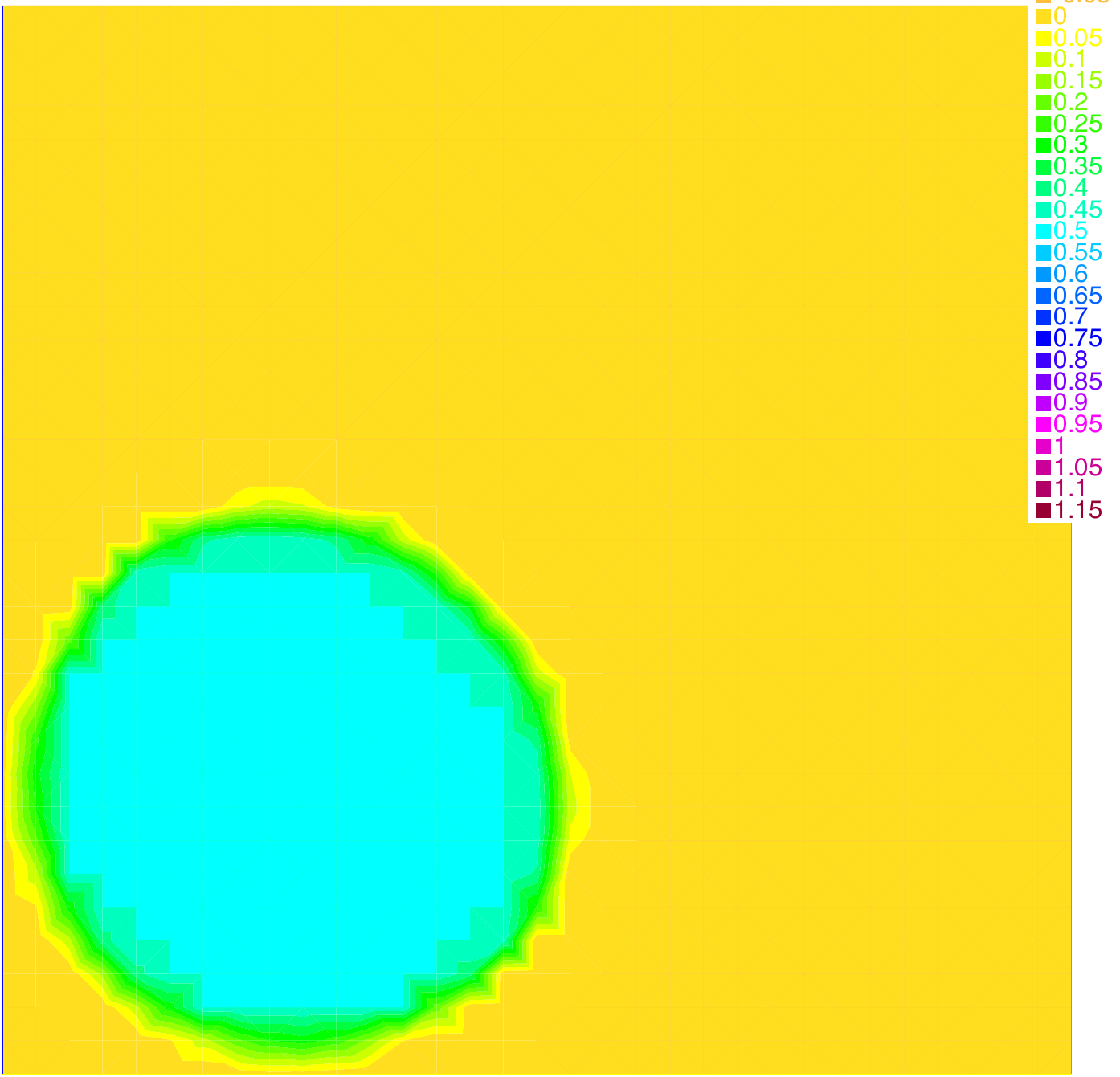}\\
$t=0$ & $t=0.05$ & $t=0.1$ & $t=0.15$ &$t=0.2$ & $t=0.3$\\
\end{tabular}
\caption{\textit{Evolution of two species crossing each other with weighted density constraint, $\rho_1+2\rho_2 \leqslant 1$. Top row: display of $\rho_1+\rho_2$. Middle row: display of $\rho_1$. Bottom row: display of $\rho_2$.}}
\label{figure crowd motion 2 weight}
\end{figure}

\newpage
\begin{figure}[h!]

\begin{tabular}{@{\hspace{0mm}}c@{\hspace{1mm}}c@{\hspace{1mm}}c@{\hspace{1mm}}c@{\hspace{1mm}}c@{\hspace{1mm}}c@{\hspace{1mm}}}
\centering
\includegraphics[ scale=0.175]{c-dens-jko-wfoule-00.pdf}&
\includegraphics[ scale=0.175]{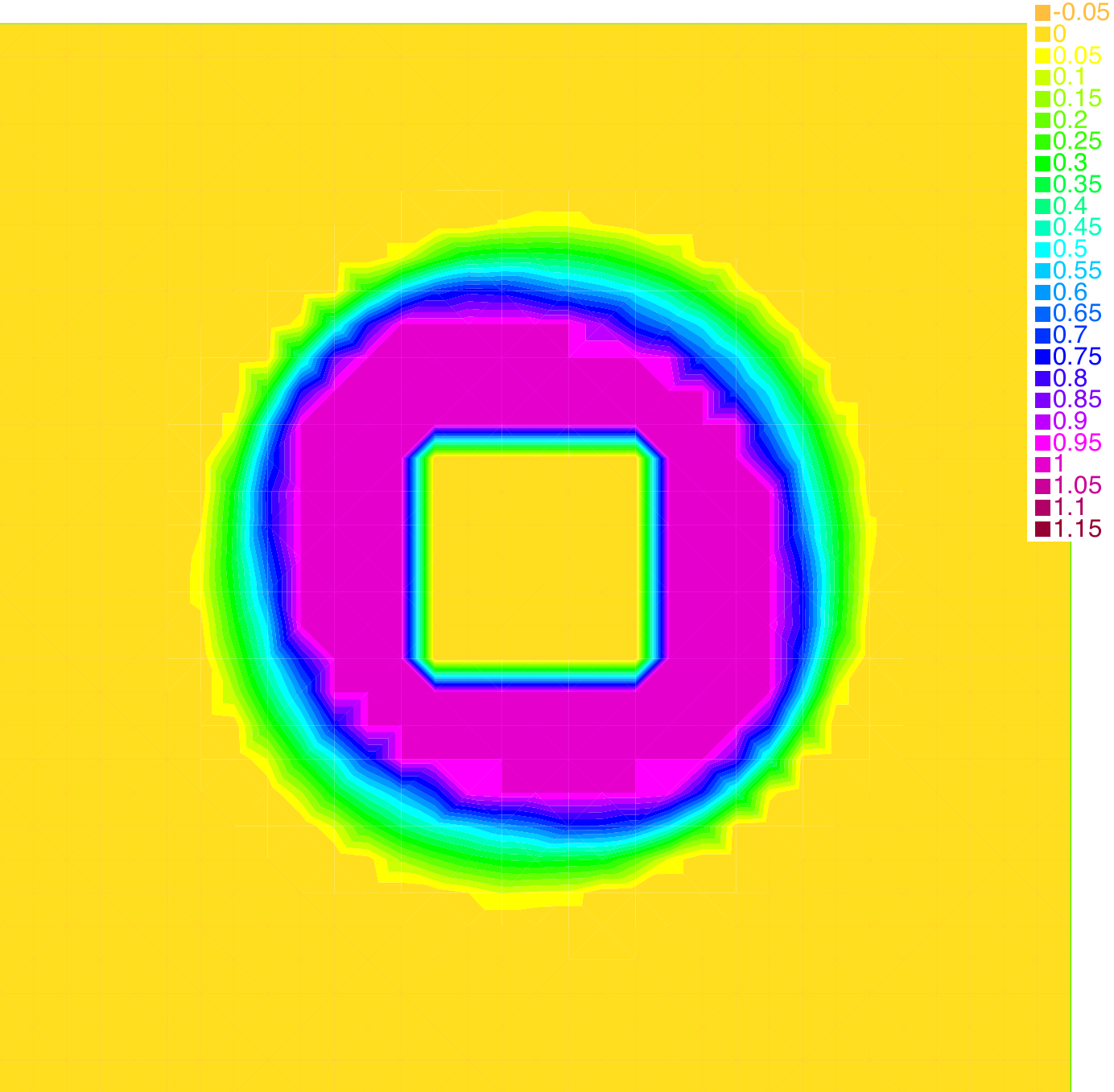}&
\includegraphics[ scale=0.175]{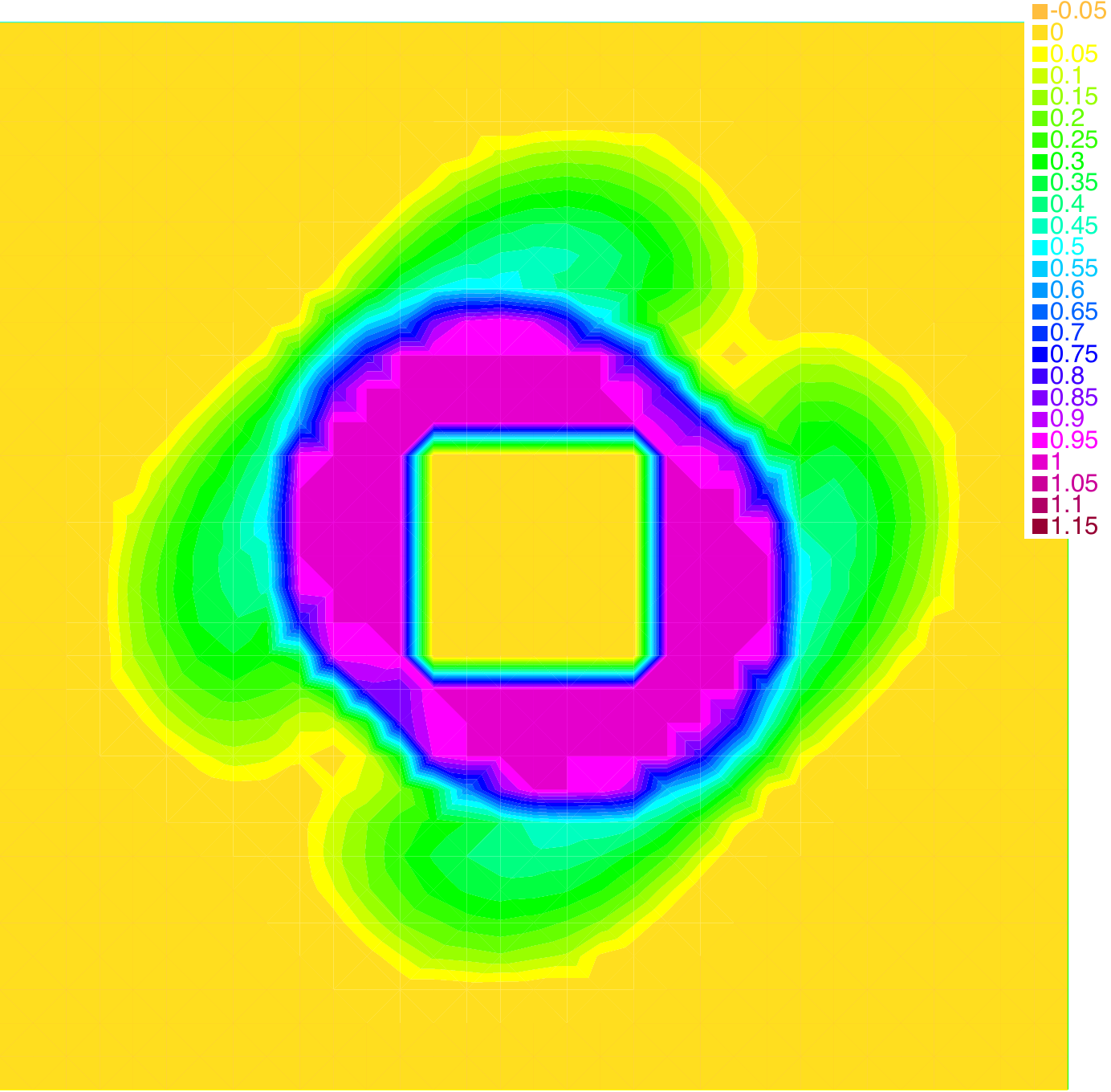}&
\includegraphics[ scale=0.175]{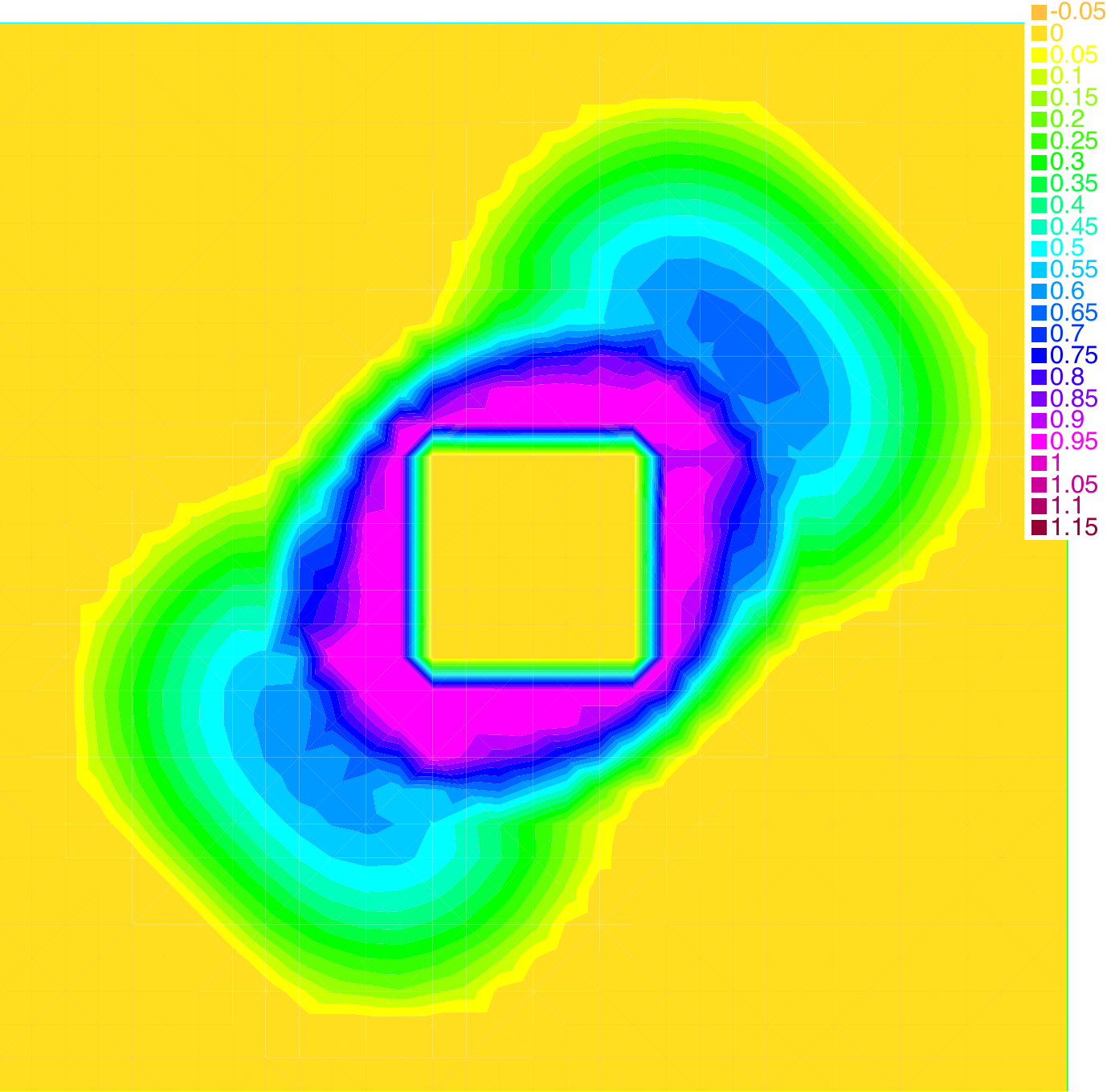}&
\includegraphics[ scale=0.175]{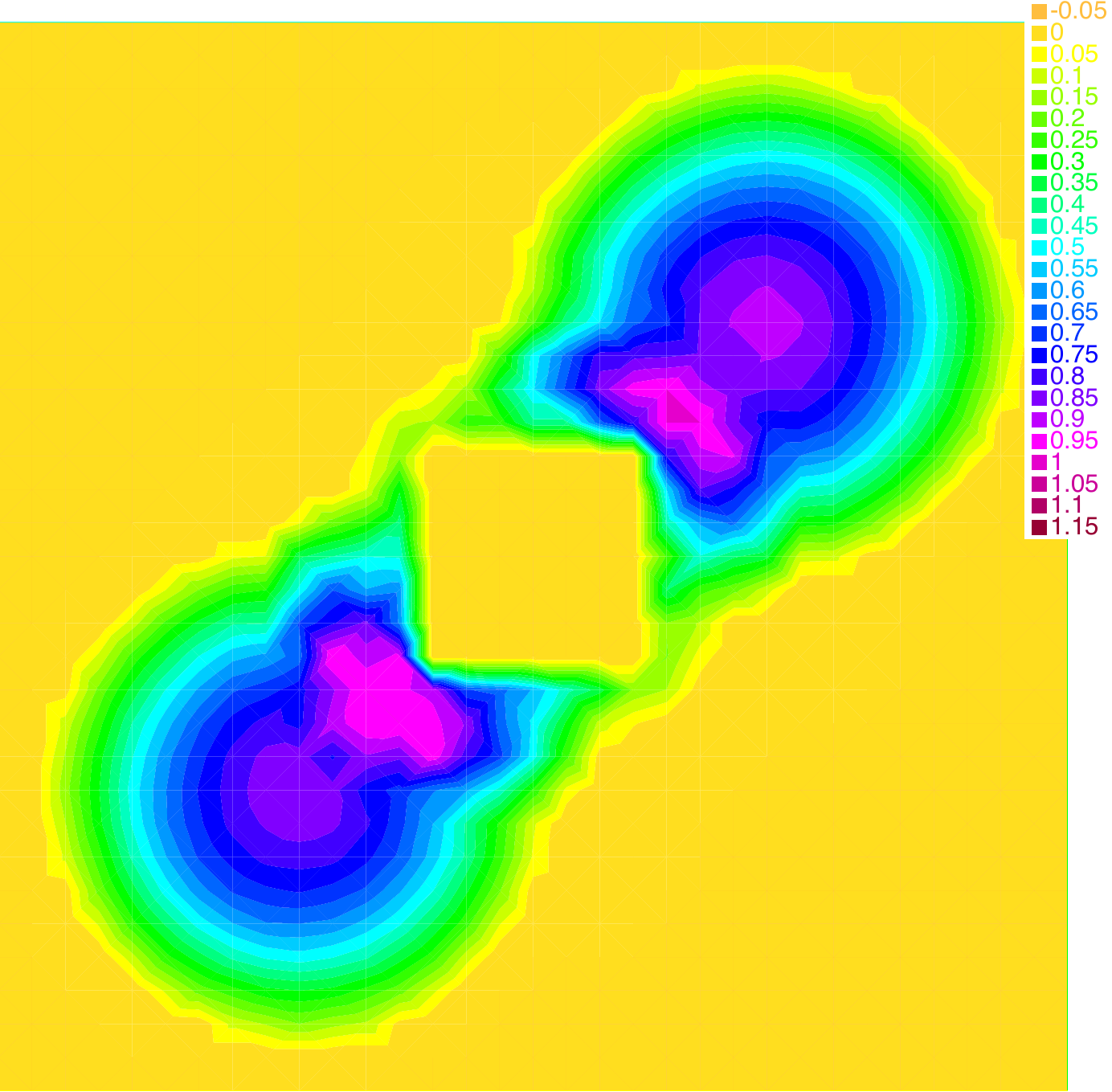}&
\includegraphics[ scale=0.175]{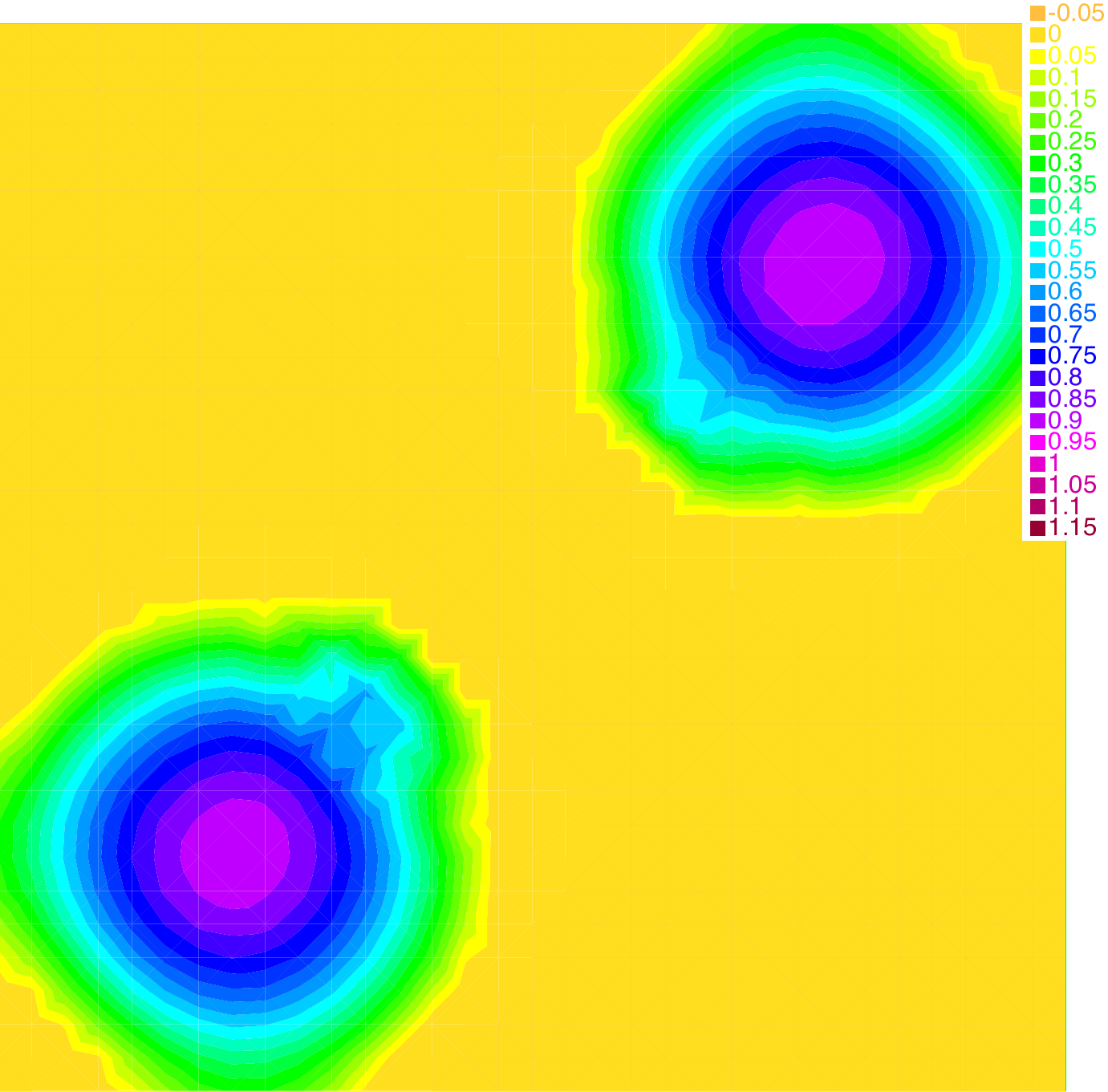}\\
\includegraphics[ scale=0.175]{b-dens-jko-wfoule-00.pdf}&
\includegraphics[ scale=0.19]{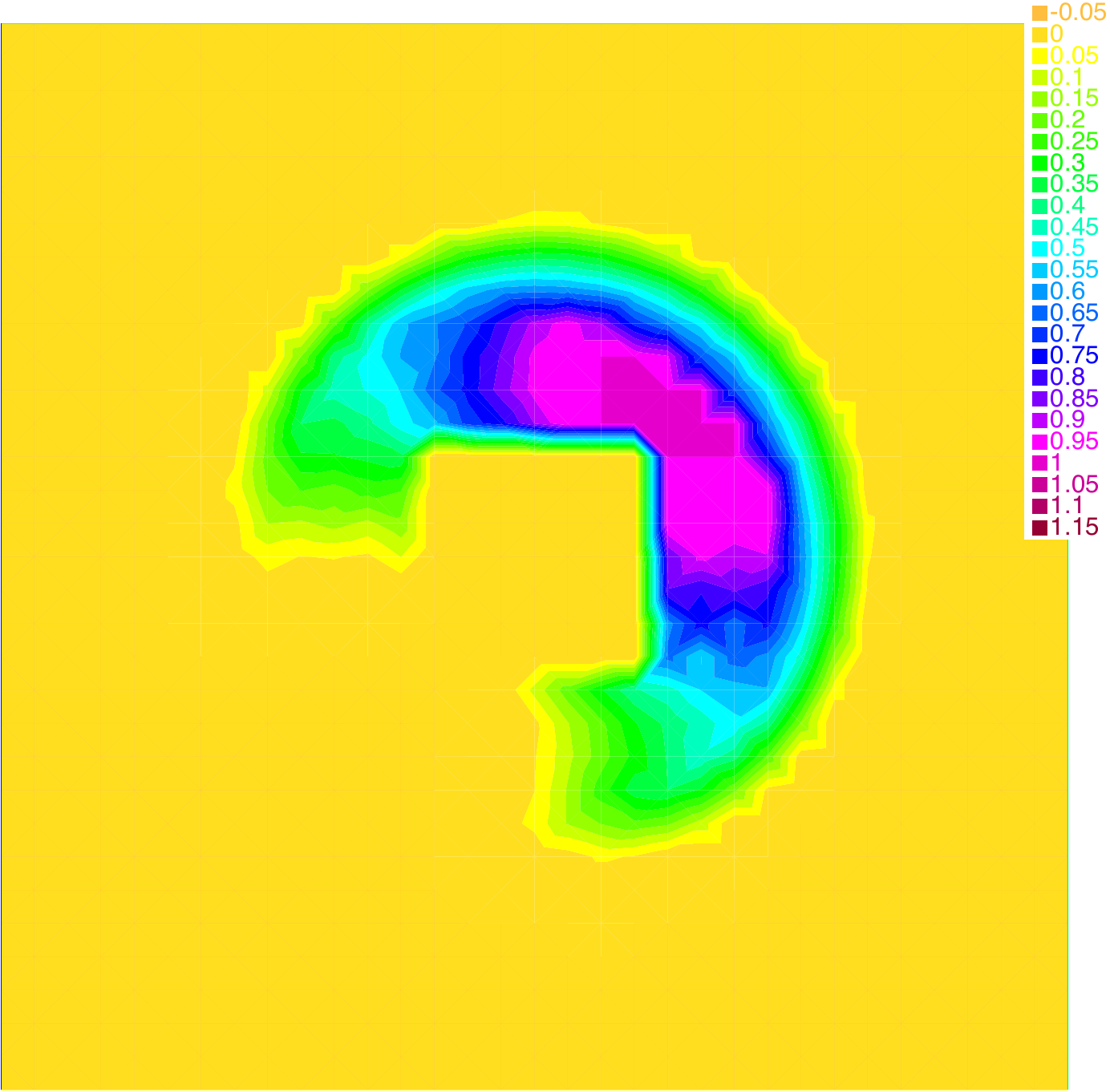}&
\includegraphics[ scale=0.175]{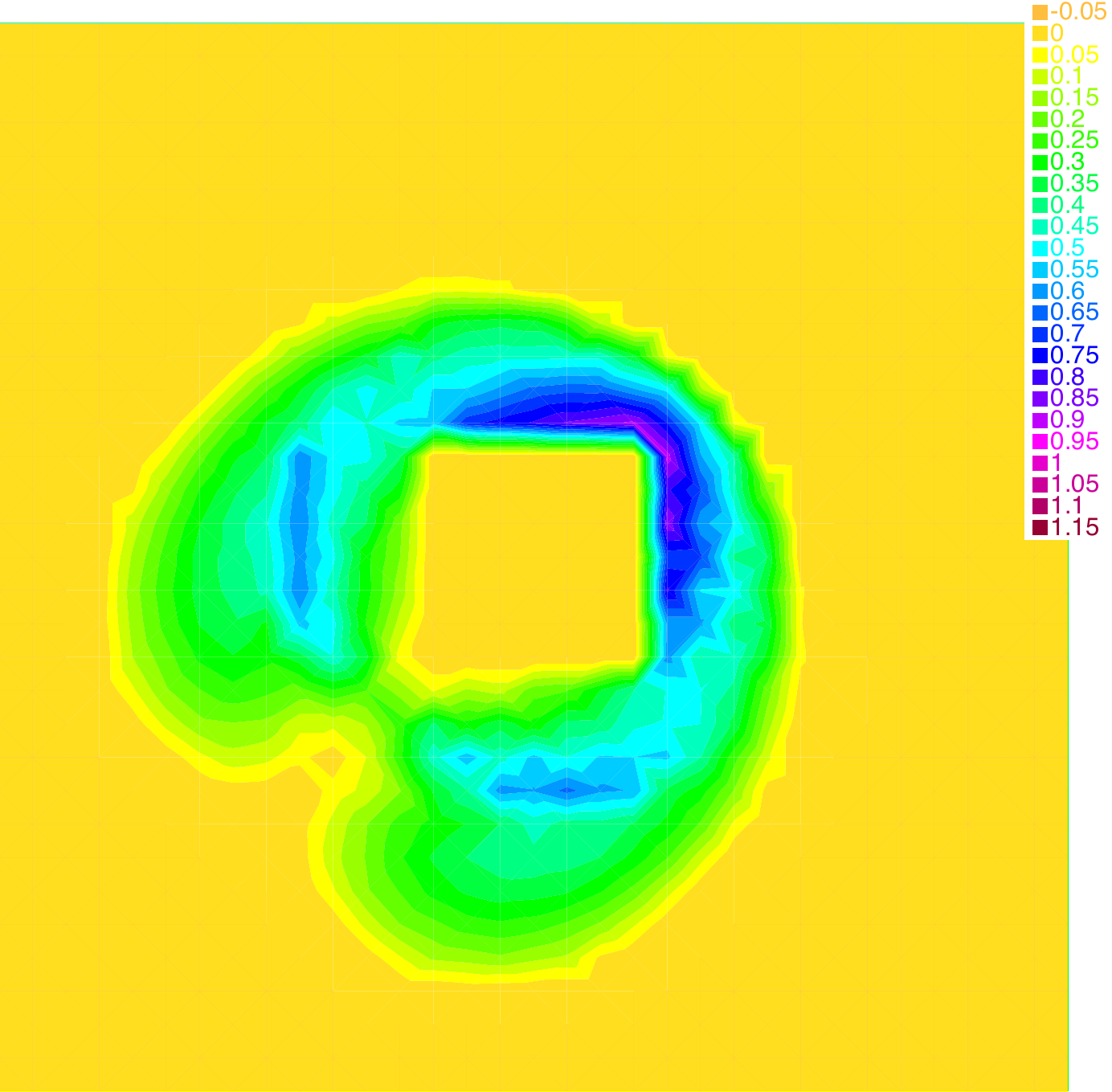}&
\includegraphics[ scale=0.175]{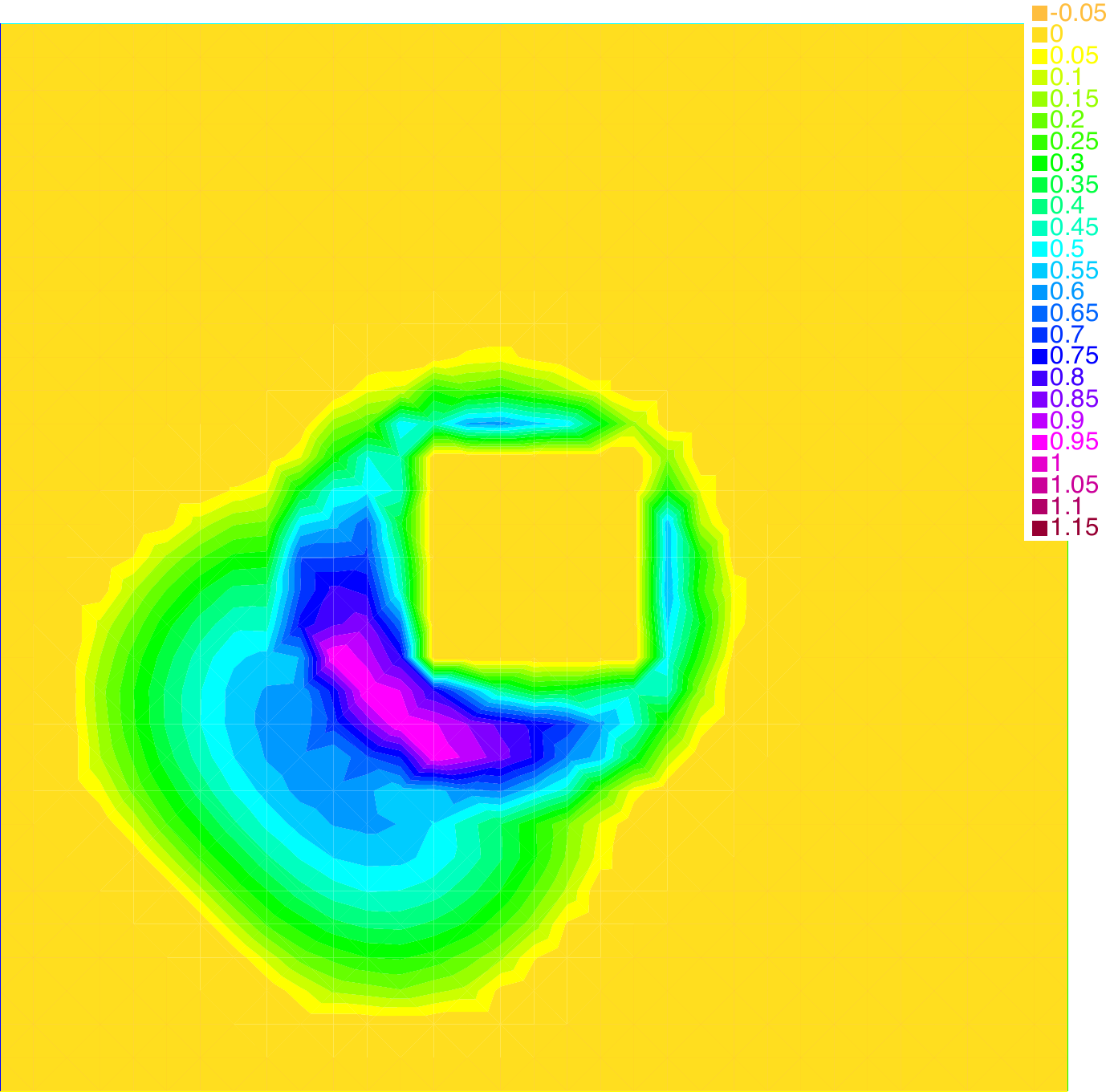}&
\includegraphics[ scale=0.175]{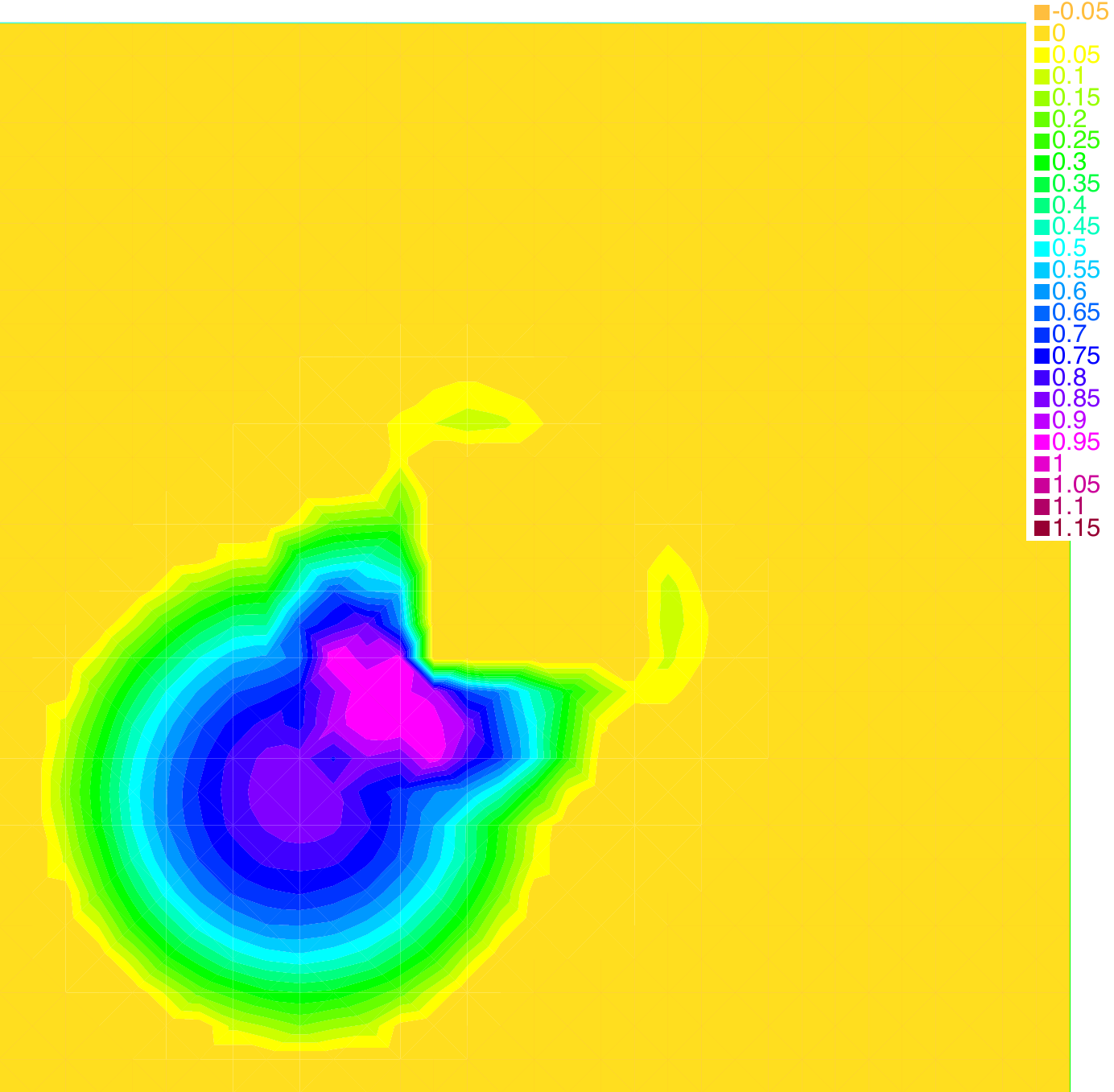}&
\includegraphics[ scale=0.175]{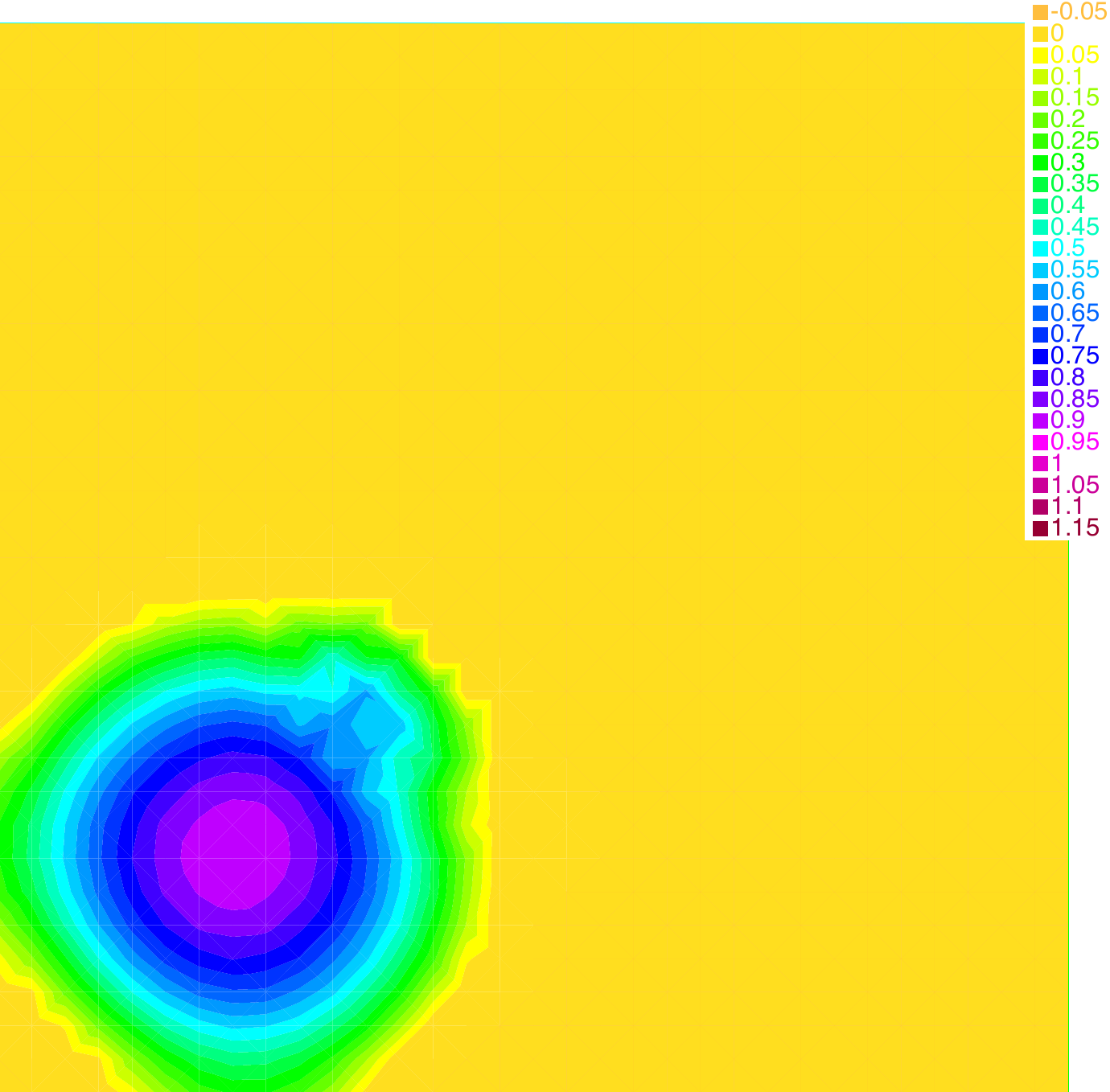}\\
$t=0$ & $t=0.1$ & $t=0.2$ & $t=0.3$ &$t=0.4$ & $t=0.5$\\
\end{tabular}
\caption{\textit{Evolution of two species crossing each other with density constraint and an obstacle. Top row: display of $\rho_1+\rho_2$. Bottom row: display of $\rho_1$.}}
\label{figure crowd motion obs}
\end{figure}

\begin{figure}[h!]

\begin{center}
\begin{tabular}{@{\hspace{0mm}}c@{\hspace{1mm}}c@{\hspace{1mm}}c@{\hspace{1mm}}c@{\hspace{1mm}}c@{\hspace{1mm}}c@{\hspace{1mm}}}
\includegraphics[ scale=0.175]{c-dens-jko-wfoule-00.pdf}&
\includegraphics[ scale=0.175]{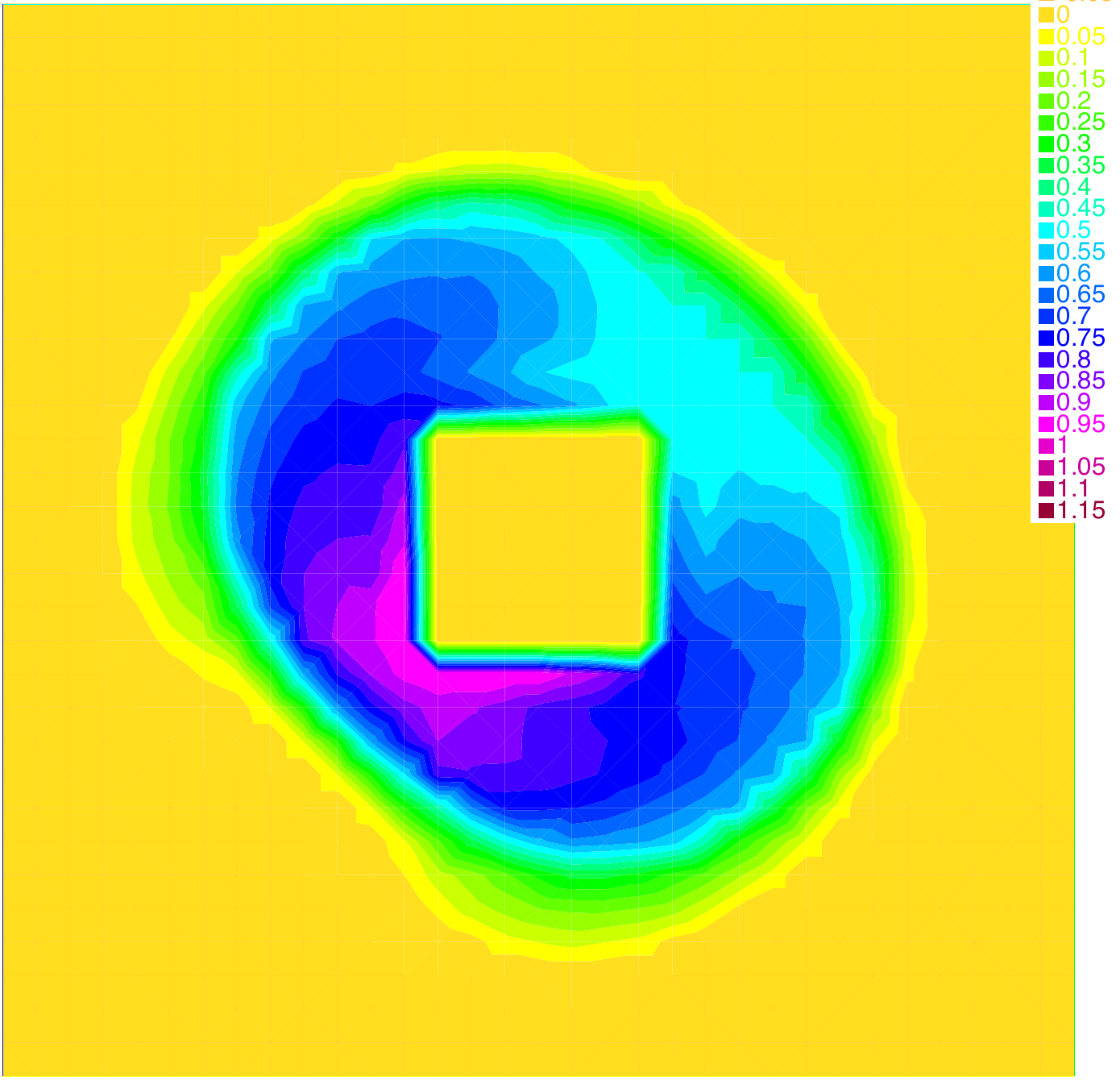}&
\includegraphics[ scale=0.175]{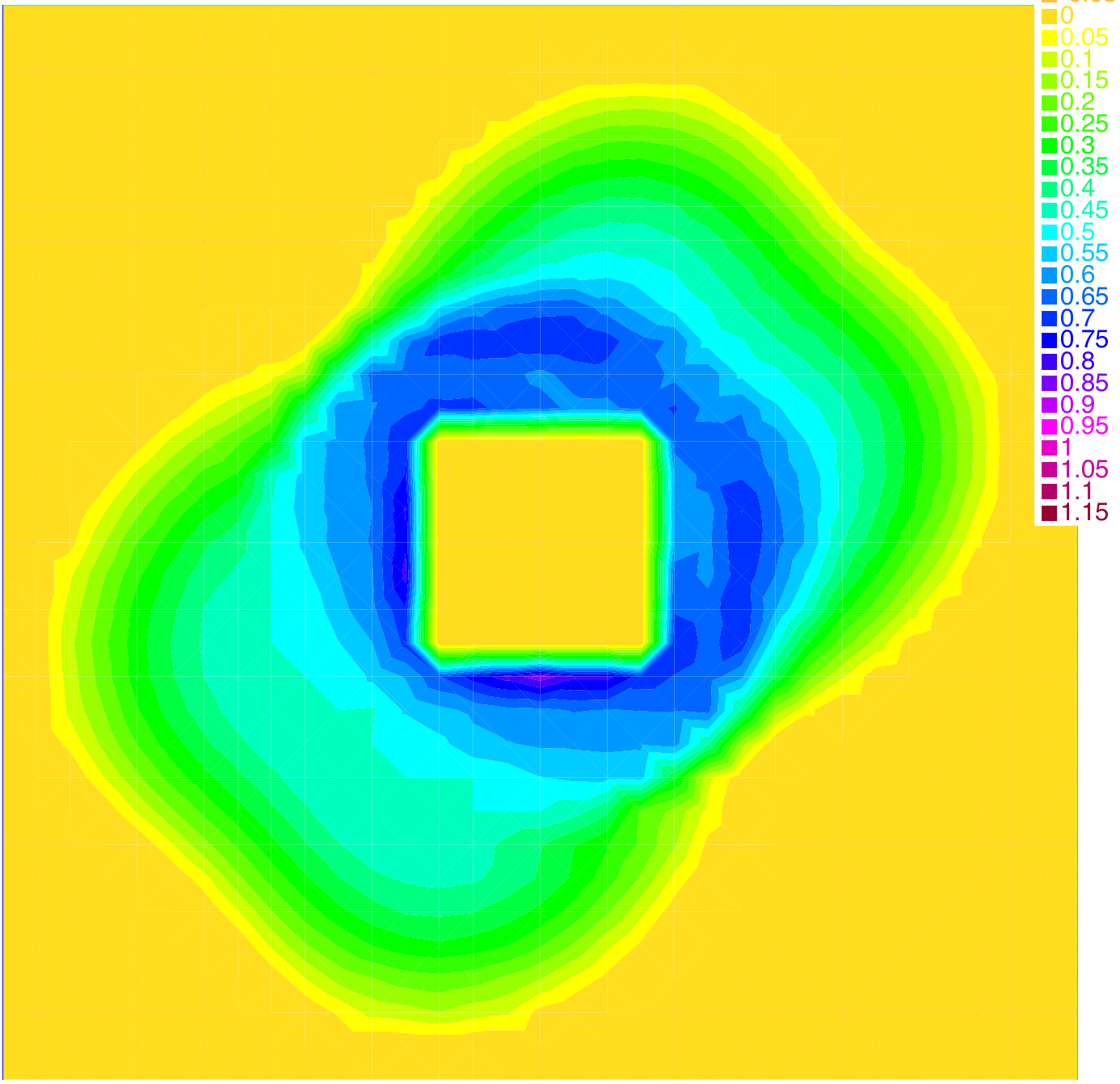}&
\includegraphics[ scale=0.175]{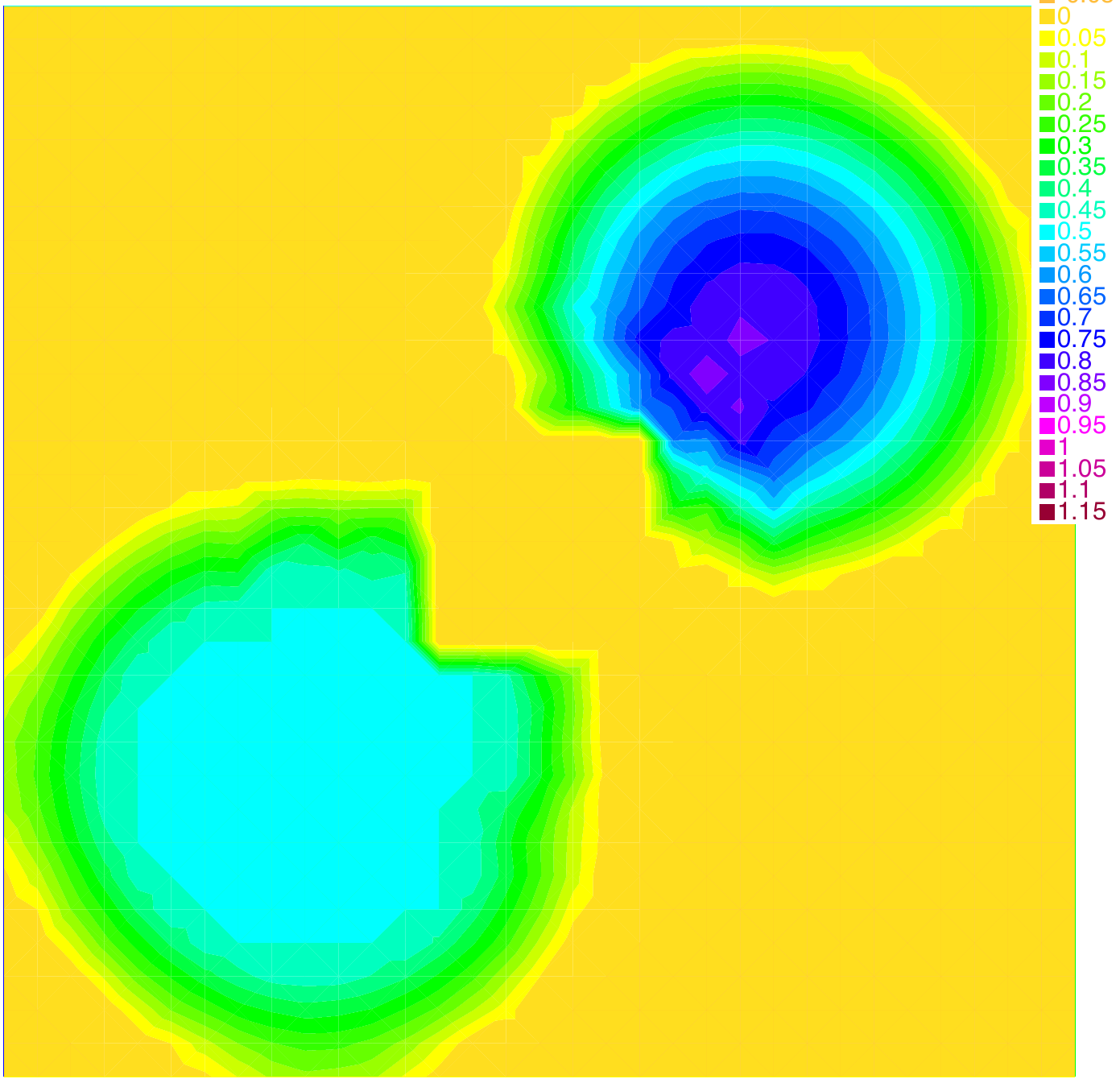}&
\includegraphics[ scale=0.175]{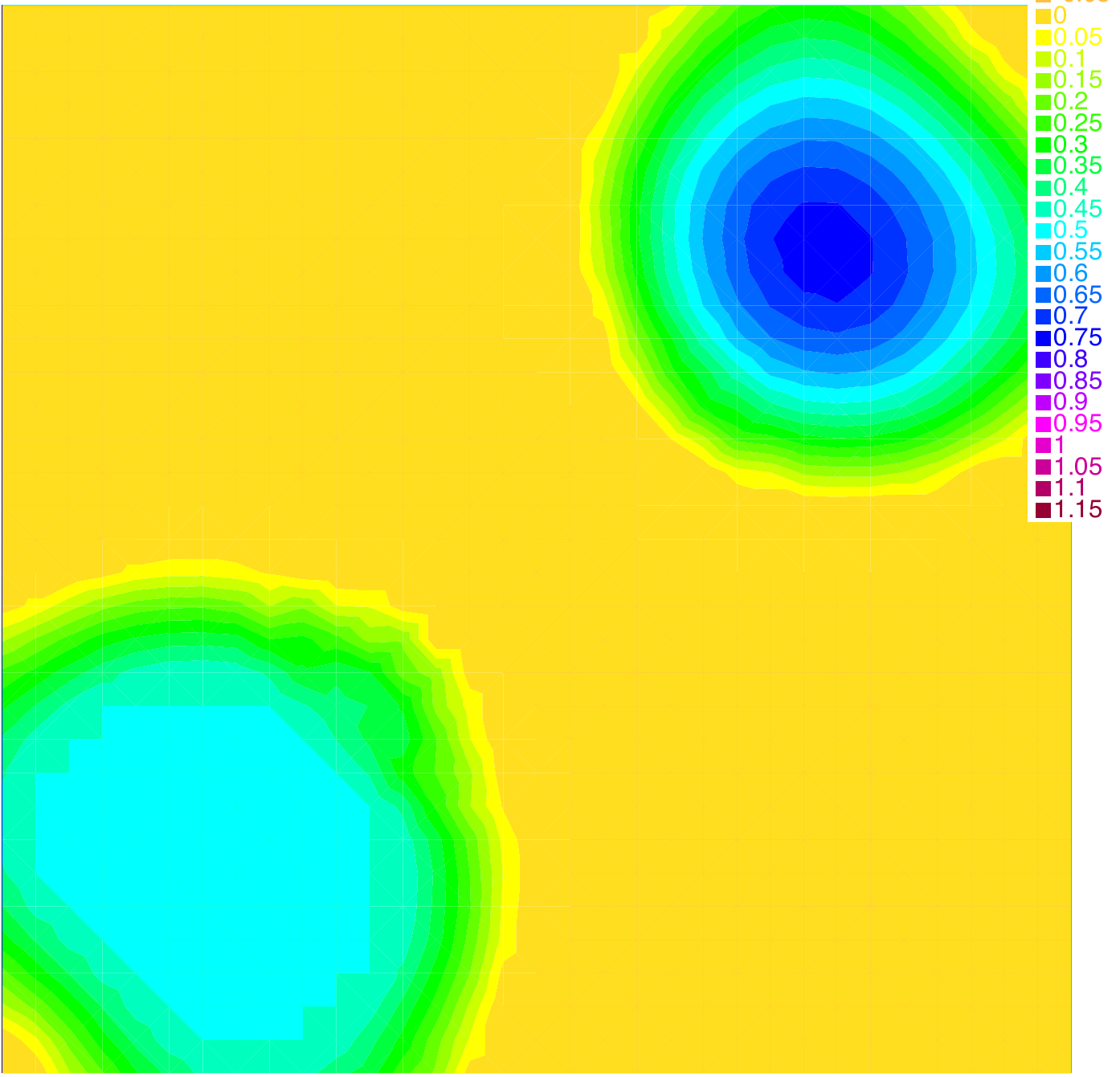}&
\includegraphics[ scale=0.175]{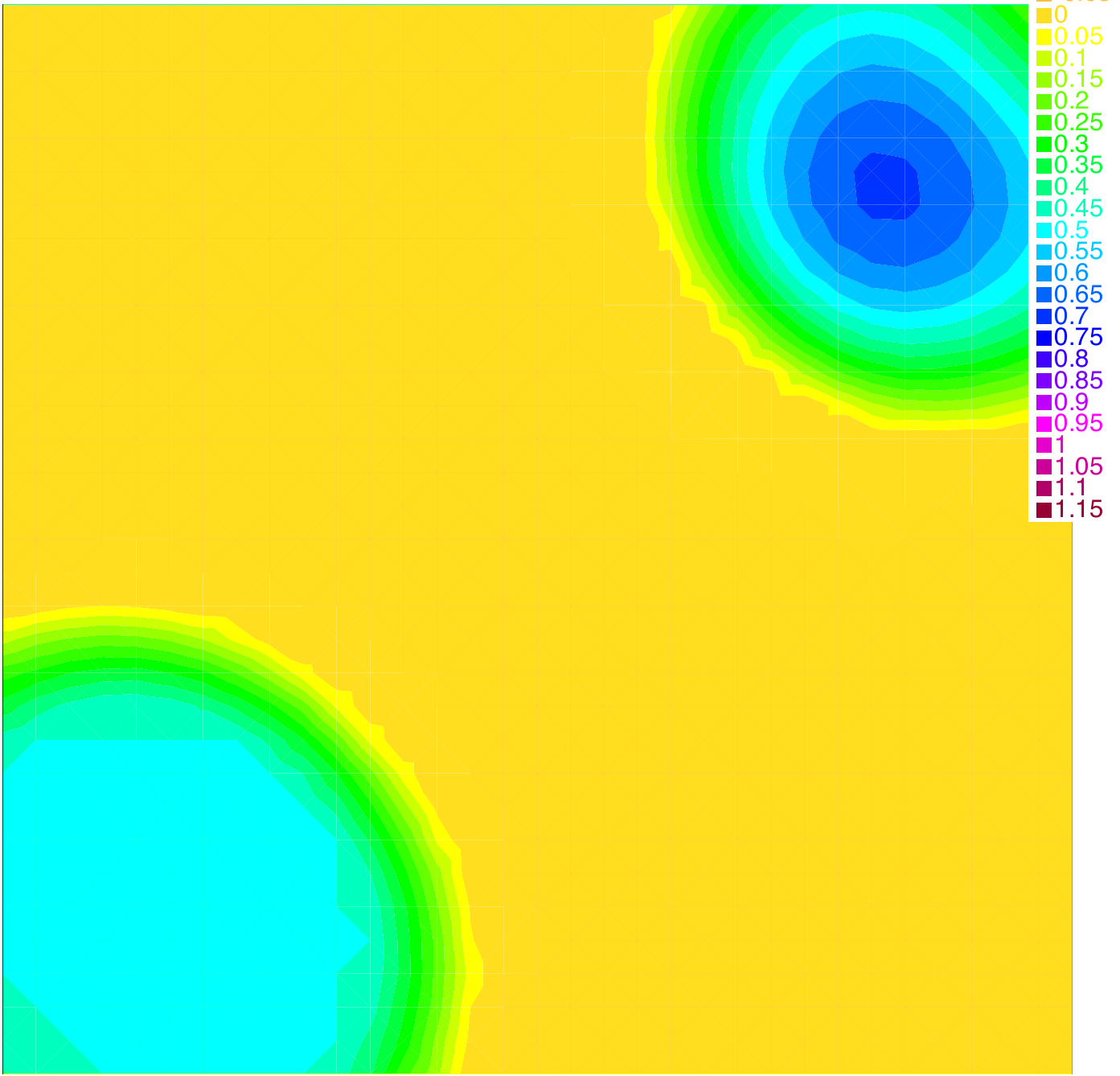}\\
\includegraphics[ scale=0.175]{a-dens-jko-wfoule-00.pdf}&
\includegraphics[ scale=0.175]{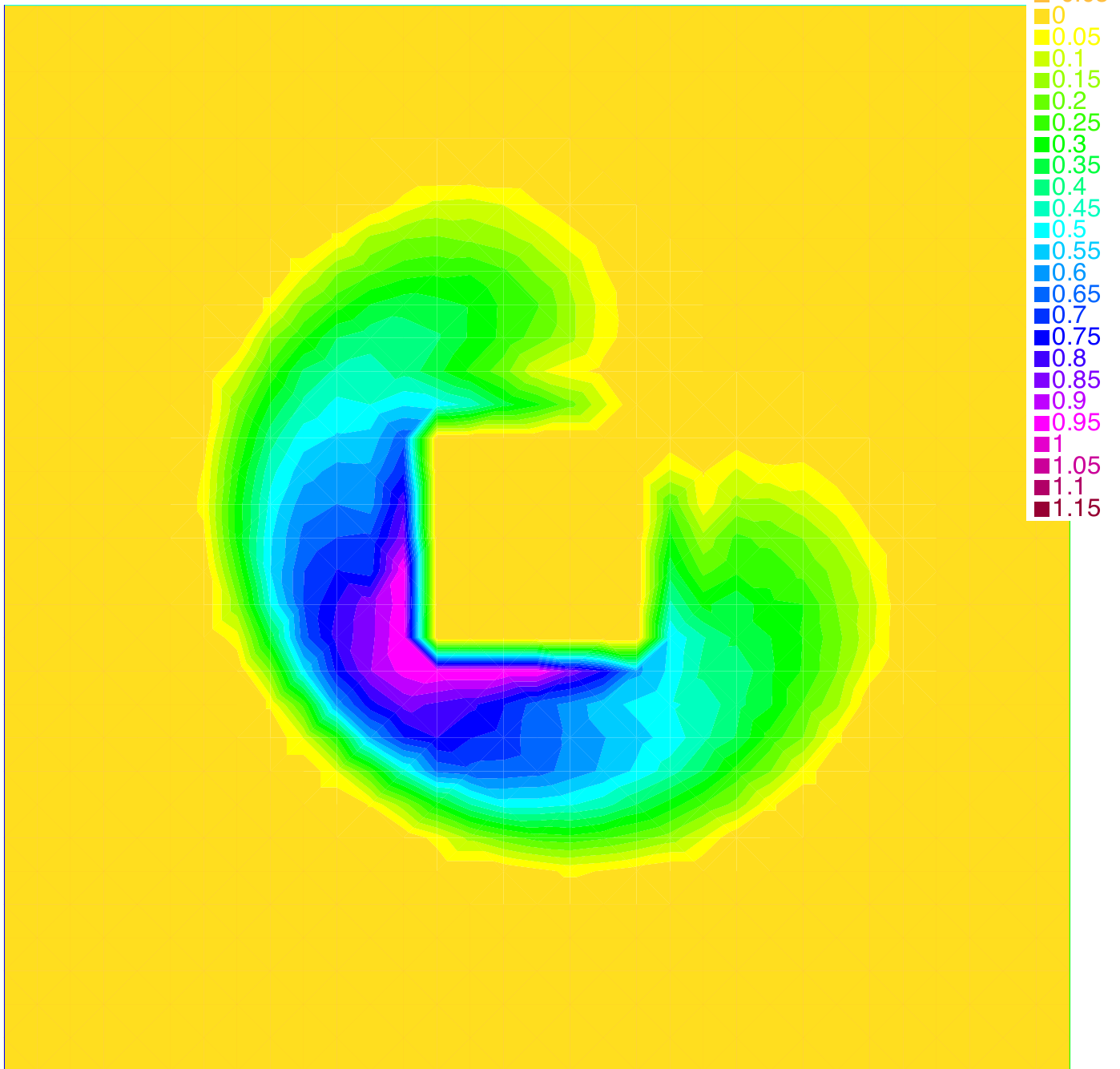}&
\includegraphics[ scale=0.175]{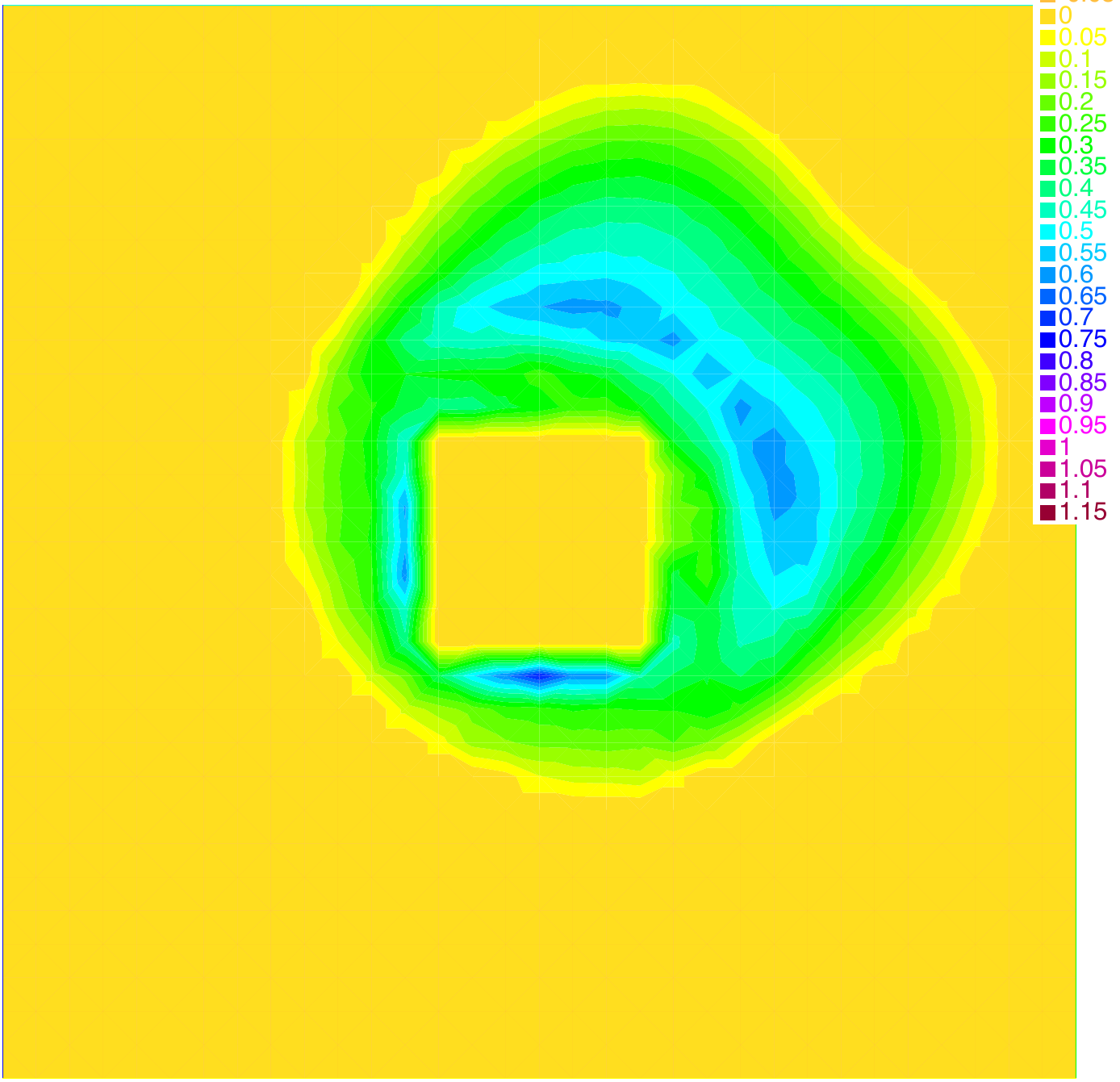}&
\includegraphics[ scale=0.175]{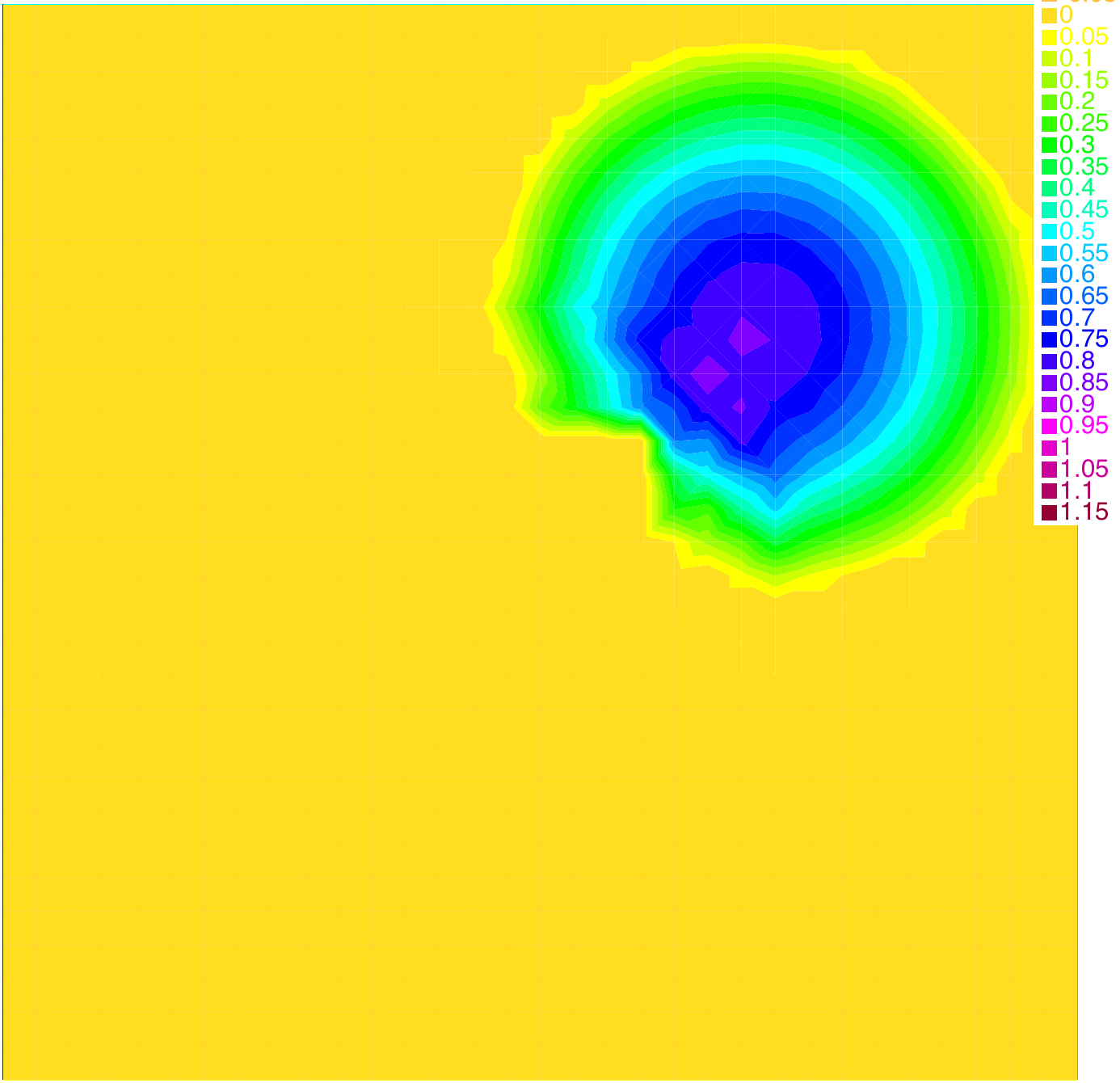}&
\includegraphics[ scale=0.175]{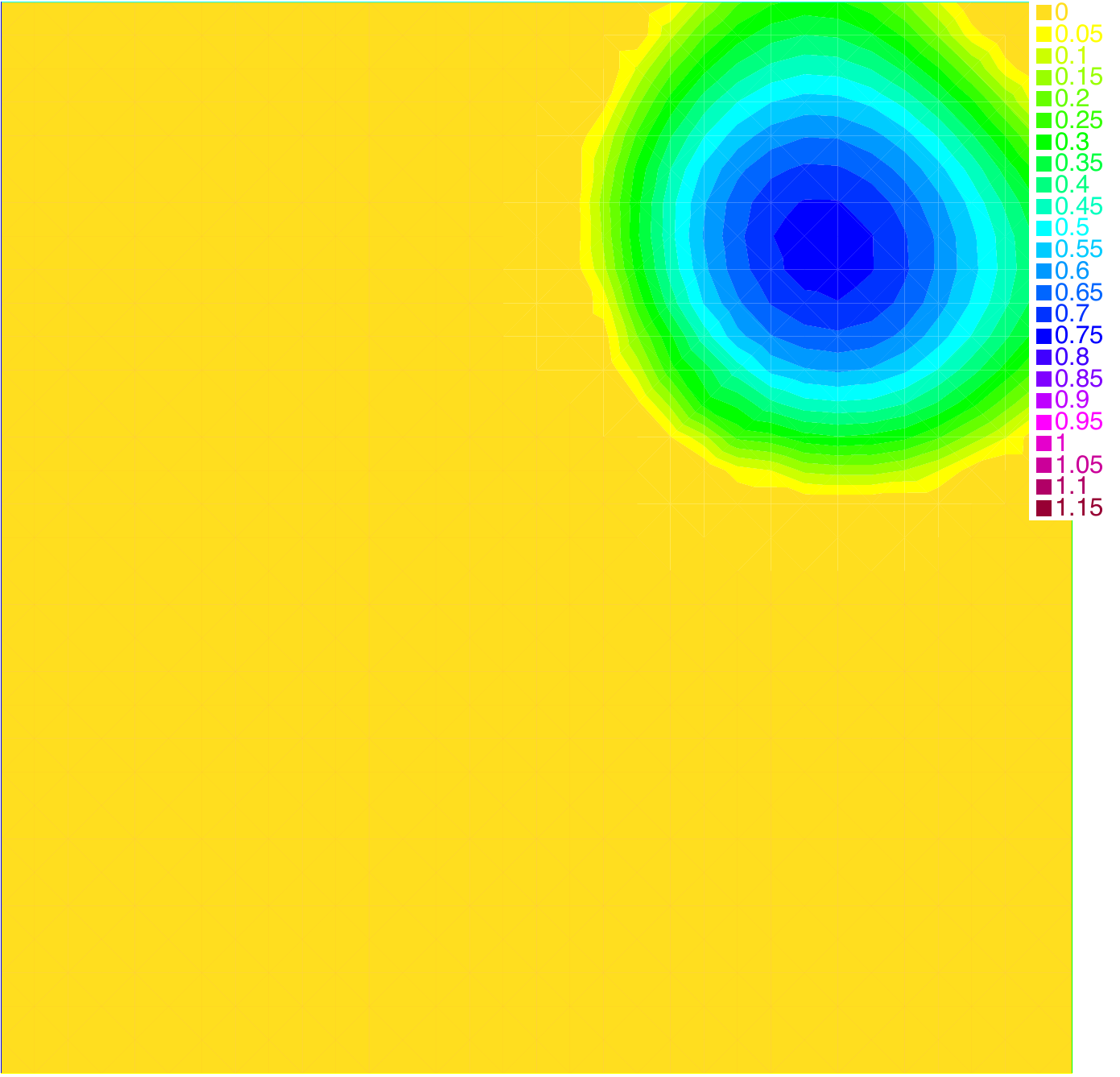}&
\includegraphics[ scale=0.175]{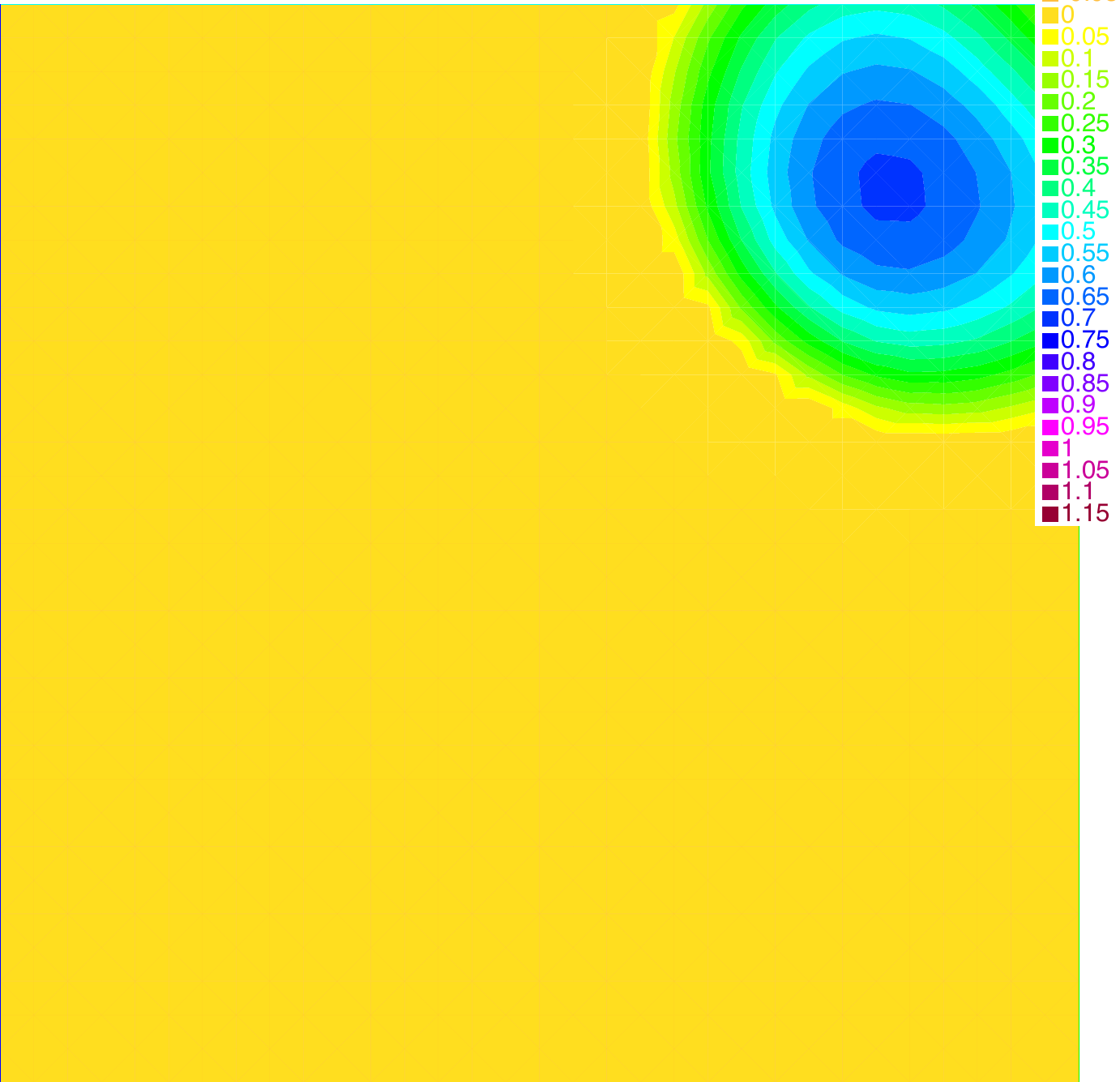}\\
\includegraphics[ scale=0.175]{b-dens-jko-wfoule-00.pdf}&
\includegraphics[ scale=0.175]{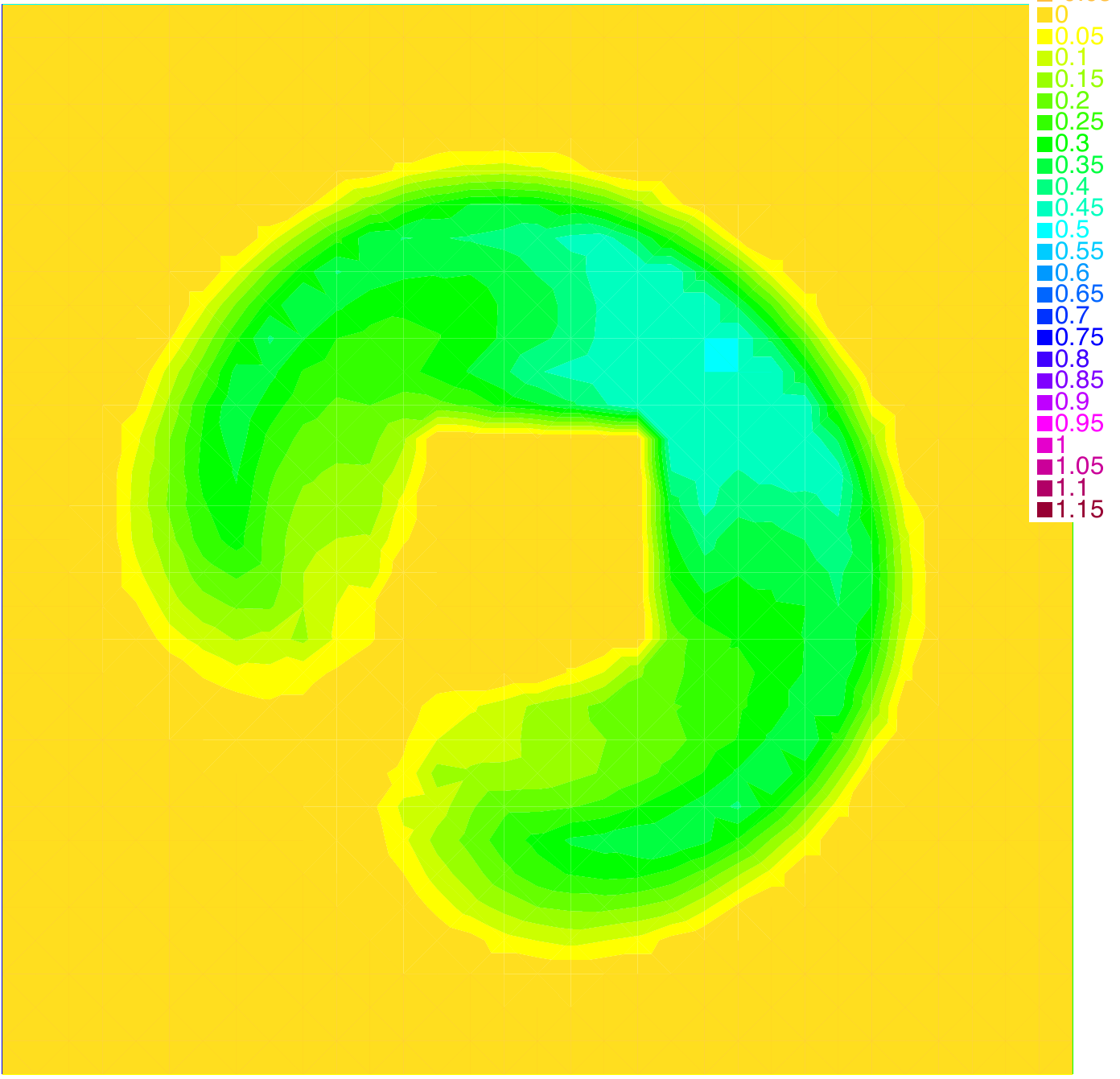}&
\includegraphics[ scale=0.175]{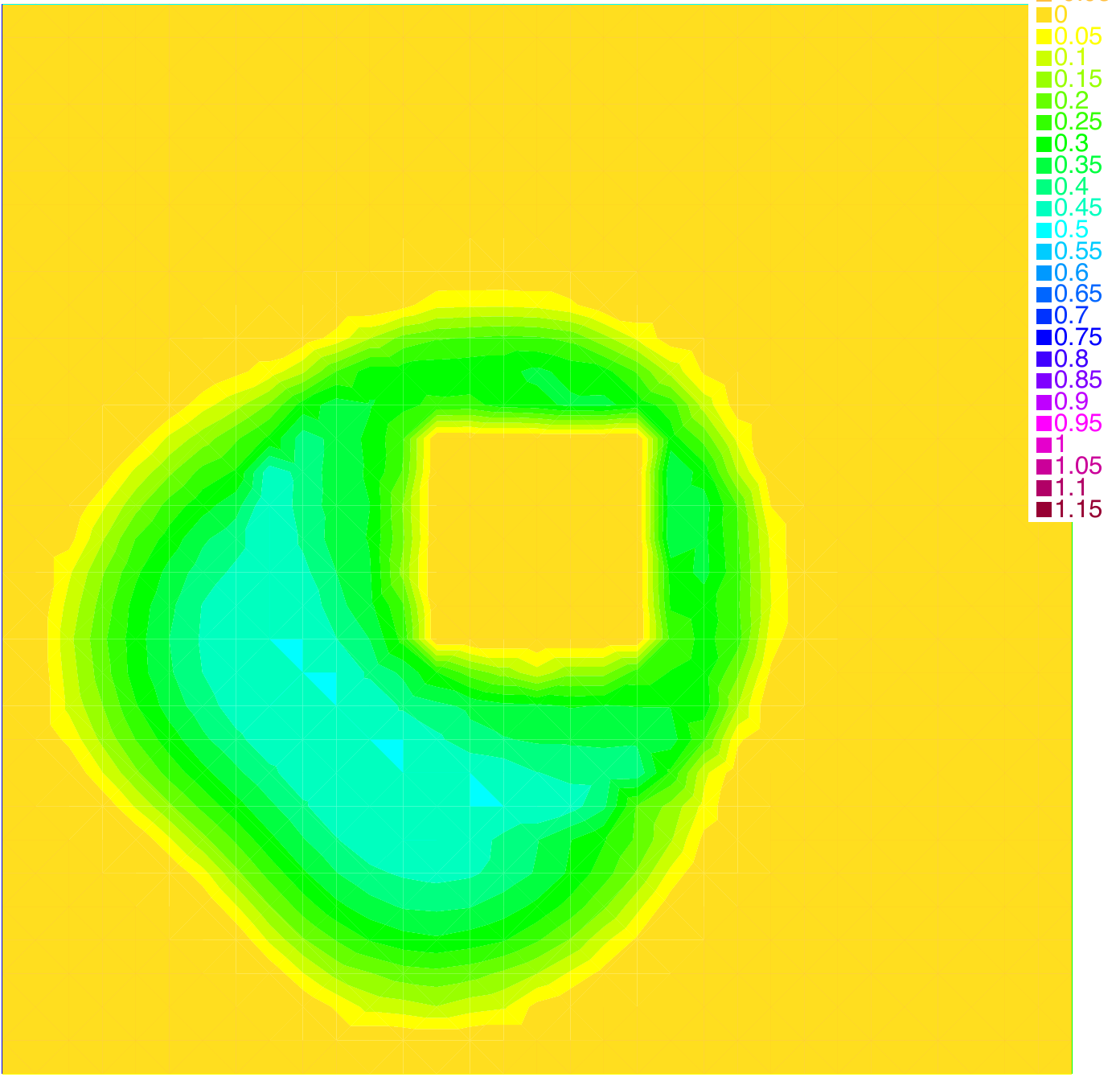}&
\includegraphics[ scale=0.175]{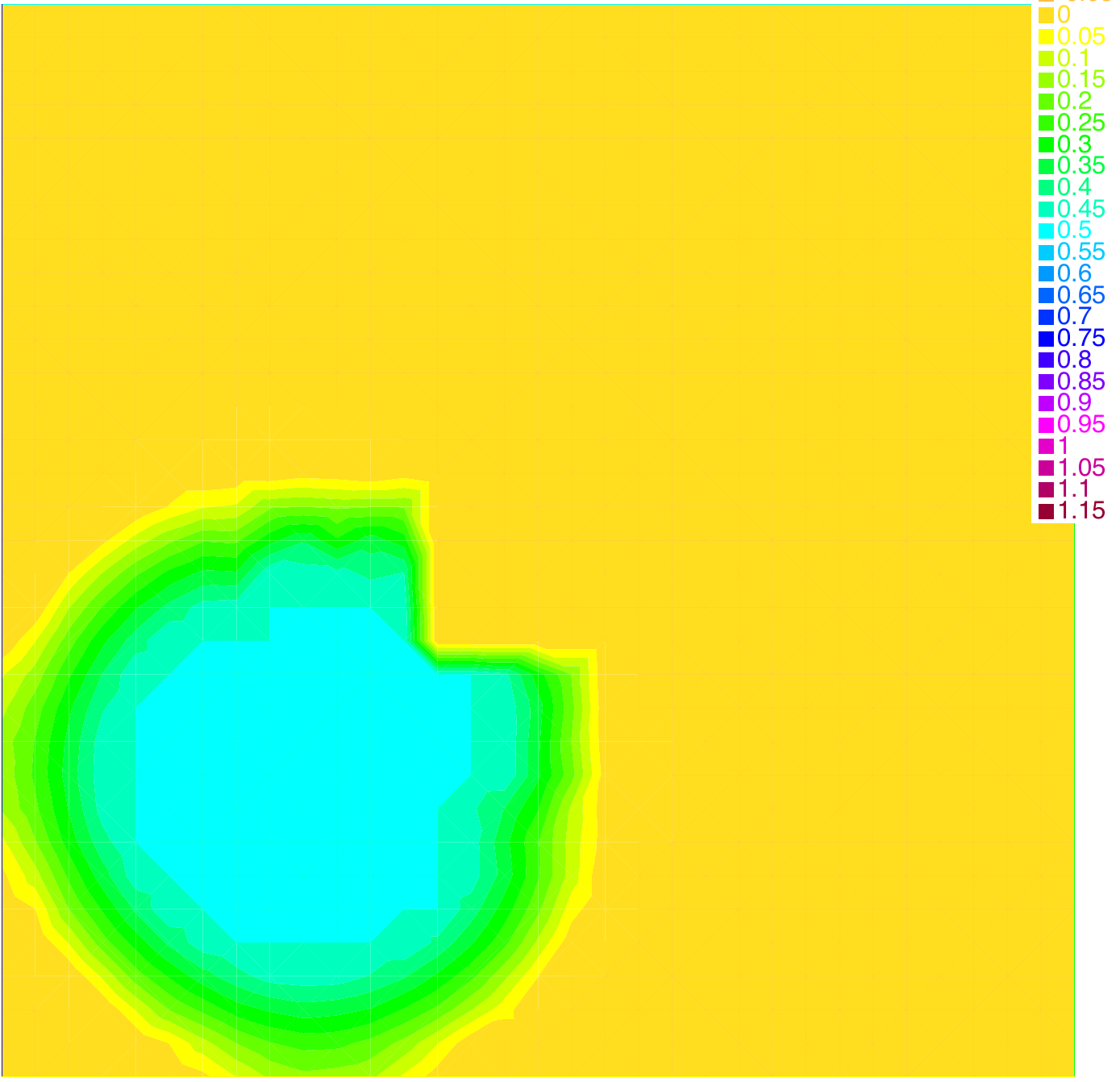}&
\includegraphics[ scale=0.175]{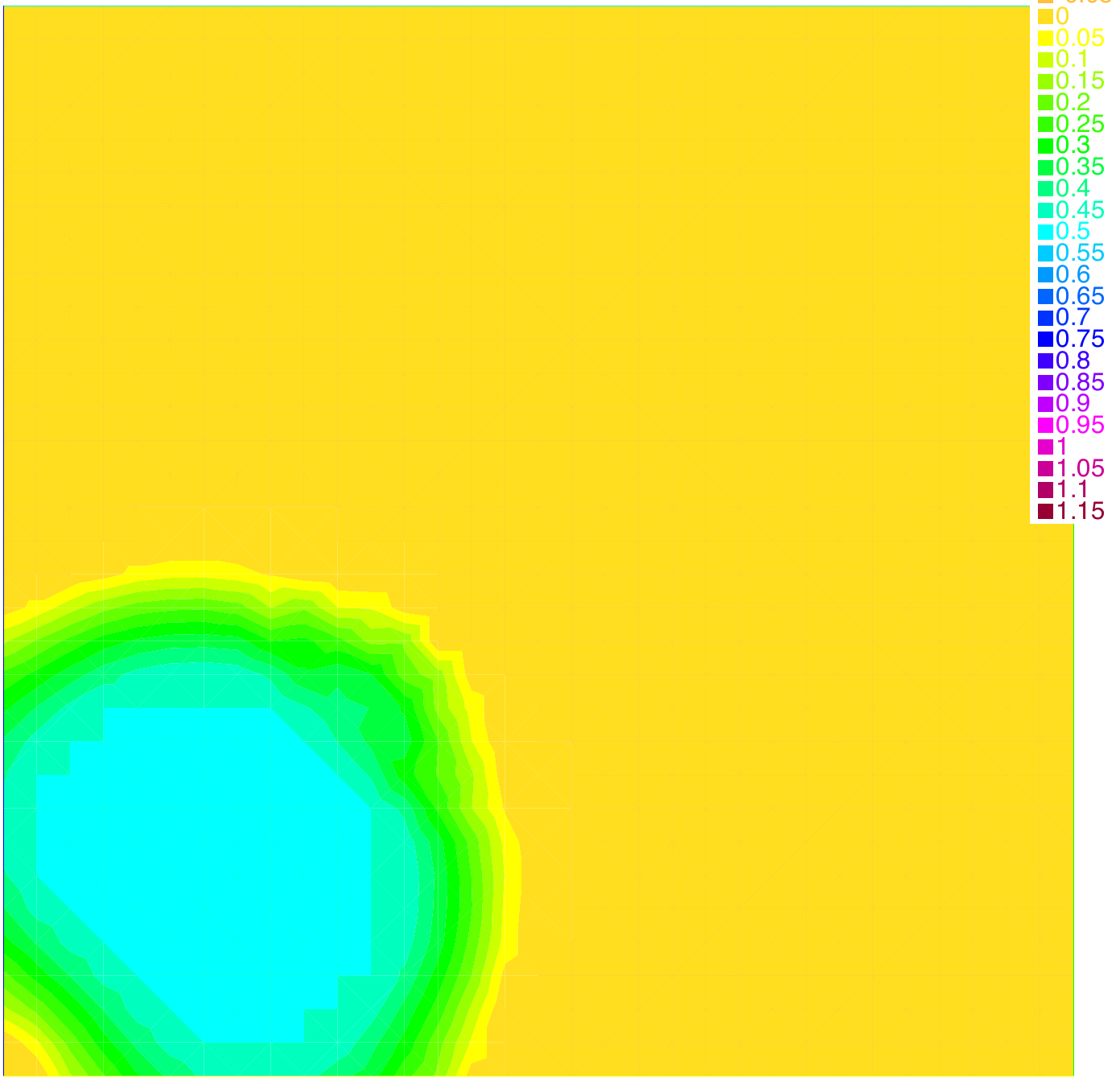}&
\includegraphics[ scale=0.175]{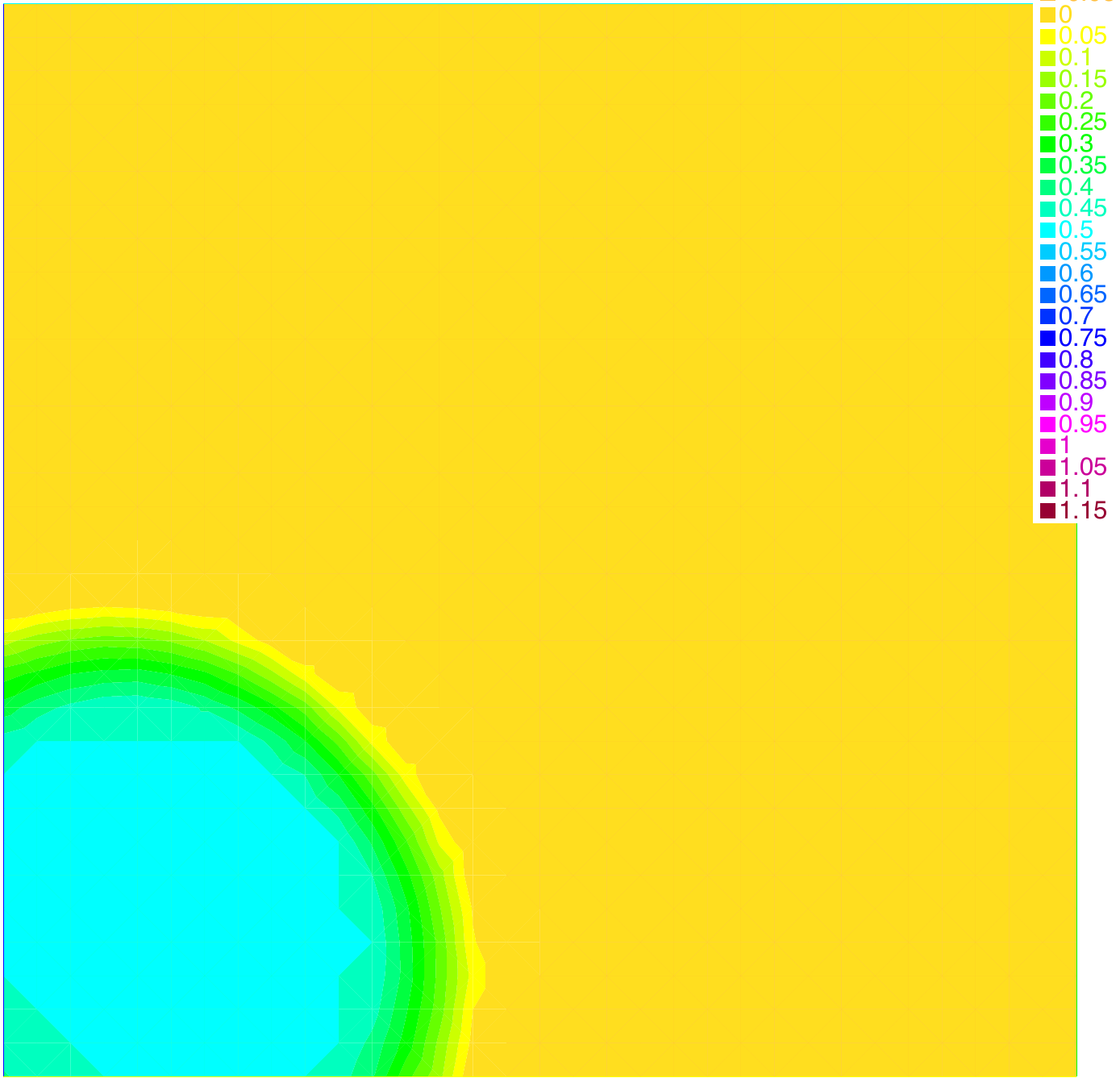}\\
$t=0$ & $t=0.1$ & $t=0.2$ & $t=0.3$ &$t=0.4$ & $t=0.5$\\
\end{tabular}
\end{center}

\caption{\textit{Evolution of two species crossing each other with weighted density constraint, $\rho_1+2\rho_2 \leqslant 1$, and an obstacle. Top row: display of $\rho_1+\rho_2$. Middle row: display of $\rho_1$. Bottom row: display of $\rho_2$.}}
\label{figure crowd motion obs weight}
\end{figure}
 
\subsection*{Acknowledgements} The author gratefully thanks G. Carlier for suggesting this problem and for fruitful discussions about this work.

\bibliographystyle{plain}
\bibliography{Laborde_bib(2)}

\begin{thebibliography}{10}

\bibitem{A}
M.~Agueh.
\newblock Existence of solutions to degenerate parabolic equations via the
  {M}onge-{K}antorovich theory.
\newblock {\em Adv. Differential Equations}, 10(3):309--360, 2005.

\bibitem{AKY}
D.~Alexander, I.~Kim, and Y.~Yao.
\newblock Quasi-static evolution and congested crowd transport.
\newblock {\em Nonlinearity}, 27(4):823--858, 2014.

\bibitem{AGS}
L.~Ambrosio, N.~Gigli, and G.~Savar{\'e}.
\newblock {\em Gradient flows in metric spaces and in the space of probability
  measures}.
\newblock Lectures in Mathematics ETH Z\"{u}rich. Birkh\"{a}user Verlag, Basel,
  2005.

\bibitem{BE}
A.~Bakhta and V.~Ehrlacher.
\newblock {Cross-diffusion systems with non-zero flux and moving boundary
  conditions}.
\newblock November 2016.
\newblock preprint.

\bibitem{guittet}
J.-D. Benamou, Y.~Brenier, and K.~Guittet.
\newblock Numerical analysis of a multi-phasic mass transport problem.
\newblock {\em Contemporary Mathematics}, 353:1--18, 2004.

\bibitem{BCL}
J.-D. Benamou, G.~Carlier, and M.~Laborde.
\newblock An augmented {L}agrangian approach to {W}asserstein gradient flows
  and applications.
\newblock In {\em Gradient flows: from theory to application}, volume~54 of
  {\em ESAIM Proc. Surveys}, pages 1--17. EDP Sci., Les Ulis, 2016.

\bibitem{B}
Y.~Brenier.
\newblock Polar factorization and monotone rearrangement of vector-valued
  functions.
\newblock {\em Comm. Pure Appl. Math.}, 44(4):375--417, 1991.

\bibitem{BS}
G.~Buttazzo and F.~Santambrogio.
\newblock A model for the optimal planning of an urban area.
\newblock {\em SIAM J. Math. Anal.}, 37(2):514--530, 2005.

\bibitem{CGM}
C.~Canc\`es, T.~O. Gallou\"et, and L.~Monsaingeon.
\newblock Incompressible immiscible multiphase flows in porous media: a
  variational approach.
\newblock {\em Anal. PDE}, 10(8):1845--1876, 2017.

\bibitem{CL1}
G.~Carlier and M.~Laborde.
\newblock A splitting method for nonlinear diffusions with nonlocal,
  nonpotential drifts.
\newblock {\em Nonlinear Analysis: Theory, Methods \& Applications}, 150:1 --
  18, 2017.

\bibitem{CS}
G.~Carlier and F.~Santambrogio.
\newblock A variational model for urban planning with traffic congestion.
\newblock {\em ESAIM Control Optim. Calc. Var.}, 11(4):595--613, 2005.

\bibitem{clm2}
R.~M. Colombo, M.~Garavello, and M.~L{\'e}cureux-Mercier.
\newblock A class of nonlocal models for pedestrian traffic.
\newblock {\em Math. Models Methods Appl. Sci.}, 22(4):1150023, 34, 2012.

\bibitem{clm}
R.~M. Colombo and M.~L{\'e}cureux-Mercier.
\newblock Nonlocal crowd dynamics models for several populations.
\newblock {\em Acta Math. Sci. Ser. B Engl. Ed.}, 32(1):177--196, 2012.

\bibitem{crippalm}
G.~Crippa and M.~L{\'e}cureux-Mercier.
\newblock Existence and uniqueness of measure solutions for a system of
  continuity equations with non-local flow.
\newblock {\em NoDEA Nonlinear Differential Equations Appl.}, 20(3):523--537,
  2013.

\bibitem{DMMRC}
J.~Dambrine, N.~Meunier, B.~Maury, and A.~Roudneff-Chupin.
\newblock A congestion model for cell migration.
\newblock {\em Commun. Pure Appl. Anal.}, 11(1):243--260, 2012.

\bibitem{DS}
S.~Daneri and G.~Savar{\'e}.
\newblock Eulerian calculus for the displacement convexity in the {W}asserstein
  distance.
\newblock {\em SIAM J. Math. Anal.}, 40(3):1104--1122, 2008.

\bibitem{DF}
D.~G. De~Figueiredo.
\newblock {\em Lectures on the Ekeland variational principle with applications
  and detours}.
\newblock Springer Berlin, 1989.

\bibitem{DLMT}
L.~Desvillettes, T.~Lepoutre, A.~Moussa, and A.~Trescases.
\newblock On the entropic structure of reaction-cross diffusion systems.
\newblock {\em Comm. Partial Differential Equations}, 40(9):1705--1747, 2015.

\bibitem{DFF}
M.~Di~Francesco and S.~Fagioli.
\newblock Measure solutions for non-local interaction {PDE}s with two species.
\newblock {\em Nonlinearity}, 26(10):2777--2808, 2013.

\bibitem{DFM}
M.~Di~Francesco and D.~Matthes.
\newblock Curves of steepest descent are entropy solutions for a class of
  degenerate convection-diffusion equations.
\newblock {\em Calc. Var. Partial Differential Equations}, 50(1-2):199--230,
  2014.

\bibitem{DMM}
S.~Di~Marino and A.~R. M\'esz\'aros.
\newblock Uniqueness issues for evolution equations with density constraints.
\newblock {\em Math. Models Methods Appl. Sci.}, 26(9):1761--1783, 2016.

\bibitem{GLM}
T.~O. Gallou{\"e}t, M.~Laborde, and L.~Monsaingeon.
\newblock {An unbalanced optimal transport splitting scheme for general
  advection-reaction-diffusion problems}.
\newblock {\em to appear in ESAIM Control. Optim. Calc. Var.}, 2017.

\bibitem{JKO}
R.~Jordan, D.~Kinderlehrer, and F.~Otto.
\newblock The variational formulation of the {F}okker-{P}lanck equation.
\newblock {\em SIAM J. Math. Anal.}, 29(1):1--17, 1998.

\bibitem{J}
A.~J{\"u}ngel.
\newblock The boundedness-by-entropy method for cross-diffusion systems.
\newblock {\em Nonlinearity}, 28(6):1963--2001, 2015.

\bibitem{JZ}
A.~J\"ungel and N.~Zamponi.
\newblock A cross-diffusion system derived from a {F}okker-{P}lanck equation
  with partial averaging.
\newblock {\em Z. Angew. Math. Phys.}, 68(1):Art. 28, 15, 2017.

\bibitem{KM}
I.~{Kim} and A.~R. {M{\'e}sz{\'a}ros}.
\newblock On nonlinear cross-diffusion systems: an optimal transport approach.
\newblock {\em Calc. Var. Partial Differential Equations}, 57(3):79, 2018.

\bibitem{KMVcrossdiff}
S.~Kondratyev, L.~Monsaingeon, and D.~Vorotnikov.
\newblock A fitness-driven cross-diffusion system from population dynamics as a
  gradient flow.
\newblock {\em J. Differential Equations}, 261(5):2784--2808, 2016.

\bibitem{L_these}
M.~Laborde.
\newblock {\em Interacting particles systems, Wasserstein gradient flow
  approach}.
\newblock PhD thesis, Paris-Dauphine University, 2016.

\bibitem{L}
M.~Laborde.
\newblock On some nonlinear evolution systems which are perturbations of
  {W}asserstein gradient flows.
\newblock In {\em Topological optimization and optimal transport}, volume~17 of
  {\em Radon Ser. Comput. Appl. Math.}, pages 304--332. De Gruyter, Berlin,
  2017.

\bibitem{LM}
P.~Lauren\c{c}ot and B.-V. Matioc.
\newblock A gradient flow approach to a thin film approximation of the {M}uskat
  problem.
\newblock {\em Calc. Var. Partial Differential Equations}, 47(1-2):319--341,
  2013.

\bibitem{LPR}
T.~Lepoutre, M.~Pierre, and G.~Rolland.
\newblock Global well-posedness of a conservative relaxed cross diffusion
  system.
\newblock {\em SIAM J. Math. Anal.}, 44(3):1674--1693, 2012.

\bibitem{MMCS}
D.~Matthes, R.~J. McCann, and G.~Savar{\'e}.
\newblock A family of nonlinear fourth order equations of gradient flow type.
\newblock {\em Comm. Partial Differential Equations}, 34(10-12):1352--1397,
  2009.

\bibitem{MRCS}
B.~Maury, A.~Roudneff-Chupin, and F.~Santambrogio.
\newblock A macroscopic crowd motion model of gradient flow type.
\newblock {\em Math. Models Methods Appl. Sci.}, 20(10):1787--1821, 2010.

\bibitem{MRCS1}
B.~Maury, A.~Roudneff-Chupin, and F.~Santambrogio.
\newblock Congestion-driven dendritic growth.
\newblock {\em Discrete Contin. Dyn. Syst.}, 34(4):1575--1604, 2014.

\bibitem{MRCSV}
B.~Maury, A.~Roudneff-Chupin, F.~Santambrogio, and J.~Venel.
\newblock Handling congestion in crowd motion modeling.
\newblock {\em Netw. Heterog. Media}, 6(3):485--519, 2011.

\bibitem{MV}
B.~Maury and J.~Venel.
\newblock {\em Handling of Contacts in Crowd Motion Simulations}, pages
  171--180.
\newblock Springer Berlin Heidelberg, 2009.

\bibitem{MC}
R.~J. McCann.
\newblock A convexity principle for interacting gases.
\newblock {\em Adv. Math.}, 128(1):153--179, 1997.

\bibitem{MS}
A.~R. M\'esz\'aros and F.~Santambrogio.
\newblock Advection-diffusion equations with density constraints.
\newblock {\em Anal. PDE}, 9(3):615--644, 2016.

\bibitem{O}
F.~Otto.
\newblock {\em Double Degenerate Diffusion Equations as Steepest Descent}.
\newblock Preprint. Bonn University, 1996.

\bibitem{Ouniqueness}
F.~Otto.
\newblock {$L^1$}-contraction and uniqueness for quasilinear elliptic-parabolic
  equations.
\newblock {\em J. Differential Equations}, 131(1):20--38, 1996.

\bibitem{PT}
L.~Petrelli and A.~Tudorascu.
\newblock Variational principle for general diffusion problems.
\newblock {\em Appl. Math. Optim.}, 50(3):229--257, 2004.

\bibitem{RS}
R.~Rossi and G.~Savar\'{e}.
\newblock Tightness, integral equicontinuity and compactness for evolution
  problems in {B}anach spaces.
\newblock {\em Ann. Sc. Norm. Super. Pisa Cl. Sci. (5)}, 2(2):395--431, 2003.

\bibitem{Sa}
F.~Santambrogio.
\newblock Gradient flows in {W}asserstein spaces and applications to crowd
  movement.
\newblock In {\em Seminaire: {E}quations aux {D}\'eriv\'ees {P}artielles.
  2009--2010}, S\'emin. \'Equ. D\'eriv. Partielles, pages Exp. No. XXVII, 16.
  \'Ecole Polytech., Palaiseau, 2012.

\bibitem{S}
F.~Santambrogio.
\newblock {\em Optimal Transport for Applied Mathematicians}.
\newblock Progress in Nonlinear Differential Equations and Their Applications
  87. Birkasauser Verlag, Basel, 2015.

\bibitem{V1}
C.~Villani.
\newblock {\em Topics in optimal transportation}, volume~58 of {\em Graduate
  Studies in Mathematics}.
\newblock American Mathematical Society, Providence, RI, 2003.

\bibitem{V2}
C.~Villani.
\newblock {\em Optimal transport}, volume 338 of {\em Grundlehren der
  Mathematischen Wissenschaften [Fundamental Principles of Mathematical
  Sciences]}.
\newblock Springer-Verlag, Berlin, 2009.
\newblock Old and new.

\end{thebibliography}
\end{document}